\documentclass[11pt,english]{amsart}

\usepackage[T1]{fontenc}
\usepackage{float}

\usepackage[english]{babel}
\usepackage{amssymb,url,xspace}
\usepackage{graphicx}
\usepackage{subfig}
\usepackage{xcolor}
\usepackage{textcmds}
\usepackage[section]{placeins}
\usepackage{hyperref}
\usepackage[skip=0.333\baselineskip]{caption}
\usepackage{subcaption}

\usepackage{amsmath}
\usepackage{amssymb}
\usepackage{amsfonts}
\usepackage{amsthm}
\usepackage[inkscapelatex=false]{svg}
\usepackage{calc}
\usepackage{xcolor}
\usepackage{epigraph}
\usepackage{dsfont}
\usepackage{bm}
\usepackage{enumitem}
\usepackage{float}
\usepackage{relsize}
\usepackage{tikz-cd}
\usepackage{physics}
\usepackage[new]{old-arrows}
\usepackage{etoolbox}

\makeatletter
\pretocmd{\section}{\addtocontents{toc}{\protect\addvspace{15\p@}}}{}{}
\pretocmd{\subsection}{\addtocontents{toc}{\protect\addvspace{5\p@}}}{}{}
\makeatother

\makeatletter
\def\th@remark{%
  \thm@headfont{\bfseries}%
  \normalfont 
  \thm@preskip\topsep \divide\thm@preskip\tw@
  \thm@postskip\thm@preskip
}
\makeatother

\theoremstyle{plain}
\newtheorem{theo}{Theorem}
\newtheorem{coro}[theo]{Corollary}

\newtheorem{prop}[theo]{Proposition}
\newtheorem{lemm}[theo]{Lemma}

\theoremstyle{definition}
\newtheorem{defi}[theo]{Definition}

\theoremstyle{remark}
\newtheorem{rema}[theo]{Remark}
\newtheorem{exem}{Example}

\numberwithin{equation}{section}
\numberwithin{theo}{section}

\newcommand{\lk}{\text{lk}}
\newcommand{\interior}[1]{%
  {\kern0pt#1}^{\mathrm{o}}%
}
\def\HF{\widehat{\text{HF}}}
\def\hfk{\text{hfk}}
\def\deg{\text{deg}}
\def\min{\text{min}}
\def\max{\text{max}}
\def\spin{\text{Spin}}
\newcommand{\bu}{\bullet}
\newcommand{\ci}{\circ}
\def\HFKh{\widehat{\text{HFK}}}
\def\CFKh{\widehat{\text{CFK}}}
\def\CFh{\widehat{\text{CF}}}

\def\dim{\text{dim}}
\def\wind{\text{wind}}
\def\wrap{\text{wrap}}
\def\lk{\text{lk}}
\def\A{\mathcal{A}}
\def\I{\mathcal{I}}
\def\id{\mathbb{I}}
\def\len{\text{len}}
\def\diam{\text{diam}}
\def\dist{\text{dist}}
\def\th{\text{th}}
\newcommand{\chapquote}[3]{\begin{quotation} \textit{#1} \end{quotation} \begin{flushright} - #2 \end{flushright} }

\begin{document}

\title[]{Twisting, Stabilization and Bordered Floer homology}
\date{\today}

\author{Soheil Azarpendar}
\address{Mathematical Institute, University of Oxford, Andrew Wiles Building,
		Radcliffe Observatory Quarter, Woodstock Road, Oxford, OX2 6GG, UK}
\email{azarpendar@maths.ox.ac.uk}


\begin{abstract}
Consider an unknot $c$ in $S^3$ and a knot $K$ in ${S^3-N(c)}$. Twisting the knot $K$ along $c$, or equivalently applying $\frac{1}{m}$-surgery on $c$, produces a family of knots $\{K_m\}_{m \in \mathbb{Z}}$. We use bordered Floer homology and the theory of immersed curve invariants to show that for $|m|\gg0$, total dimension of $\HFKh(K_m)$, $\tau(K_{m})$ and thickness of $K_{m}$ are linear functions of $m$. Furthermore, we prove that the extremal coefficients of the Alexander polynomial and extremal knot Floer homologies of $K_m$ stabilize as $m$ goes to infinity. This generalizes results of Chen, Lambert-Cole, Roberts, Van Cott and the author on coherent twist families.  
\end{abstract}

\maketitle
\tableofcontents
\pagebreak

\section{Introduction}\label{Section:Introduction}
\chapquote{Any time you have a framing, take it to the limit}
{Matthew Hedden}{}

\subsection{Introducing the Twisting problem}\hfill\\

Consider an unknot $c$ in $S^3$ and a knot $K$ in ${S^3-N(c)}$. We can also think of $K$ as a knot within the standard solid torus $S^1 \times D^2 \subset S^3$, with $c$ representing (a small push out of) the meridian $\{*\} \times \partial D^2$. Let $\wind_{c}(K)$ and $\wrap_{c}(K)$ denote the winding number and wrapping number of $K$ within the solid torus, respectively. We assume that the geometric intersection of $K$ and $\{*\} \times D^2$ is equal to $\wrap_{c}(K)$. Note that $0 \leq \wind_{c}(K) \leq \wrap_{c}(K).$ We say that $c$ links $K$ \emph{coherently} when $\lk(K,c)=\wind_{c}(K)=\wrap_{c}(K).$\\

In this paper, we examine the family of knots $\{K_m\}_{m \in \mathbb{Z}}$, where $K_0=K$ and $K_m \subset S^3$ is obtained by performing $(-\frac{1}{m})$ surgery on $c$. This is equivalent to applying the $m$-th power of the meridional Dehn twist $$\tau_{\partial D^2} : S^1 \times D^2 \longrightarrow S^1 \times D^2$$ 
to $K$, viewed as a knot in $S^1 \times D^2$. We refer to  $\{K_m\}_{m \in \mathbb{Z}}$ as a \emph{twist family} of knots. If $c$ links $K$ coherently, we call $\{K_m\}$ a \emph{coherent twist family}.\\

Numerous authors have investigated how the invariants of $K_m$ behave as $m \rightarrow \pm \infty$. These studies include results on hyperbolic geometry \cite{Thurston1979TheGA}, Seifert genus \cite{Baker2016DehnFA}, Jones polynomial \cite{Champanerkar2004OnTM} and Khovanov homology \cite{Lee2023StableKH}.\\

Lambert-Cole \cite{LambertCole2016TwistingMA} proved that for a coherent twist family of knots $\{K_m\}$ with winding number $2$, both the Alexander polynomial and knot Floer homology stabilize as $m \rightarrow \pm \infty$. Chen \cite{Chen2022TwistingsAT} further proved that the extremal coefficients of the Alexander polynomial stabilize for any coherent twist family of knots. Boninger \cite{Boninger2024TwistedKA} extended Chen's result while analyzing the behaviour of the perturbed Alexander invariant in coherent twist families. The first author \cite{Azarpendar2023TwistingAS} showed that Lambert-Cole's argument generalizes to coherent twist families with higher winding numbers, proving a stabilization result for extremal knot Floer homolgies. This stabilization result, unlike that of Lanmbert-Cole, included a potential unknown shift in the absolute Maslov grading.\\

\subsection{Summary of the main results}\hfill\\

In this paper, we study the stabilization of knot Floer homology using the framework of bordered Floer homology. This perspective enables us to access and analyze features of knot Floer homology that were previously out of reach. Our first main result concerns the growth of the total dimension of knot Floer homology, as stated in Theorem \ref{Theorem:totaldim}.

\begin{theo}\label{Theorem:totaldim}
    Let $\{K_m\}$ be a twist family of knots. Assume $\wind_{c}(K) \neq 0$. Then, there exist fixed integers $D,d$ such that for sufficiently large $m$, the dimension of $\HFKh(K_m)$ is given by:
    $$\dim(\HFKh(K_m)) = D \cdot |m| - d.$$
\end{theo}

As previously discussed, prior work has addressed the stabilization of extremal knot Floer homologies under twisting. In this work, we establish stronger results. Specifically, we extend the known stabilization theorems from coherent twist families to all twist families with nonzero winding numbers. Furthermore, we precisely compute the shift in absolute Maslov grading, demonstrating that this shift vanishes for coherent twist families. These results are encapsulated in Theorem \ref{Theorem:extremalhfk}.\\ 

Prior to stating Theorem \ref{Theorem:extremalhfk}, we introduce some additional notations. Let $E :=S^3 - N(K \cup c)$ repreresnt the complement of $K$ in the solid torus. Restricting the meridian disk of solid torus to $E$, we obtain a planar surface denoted by $\hat{D}$. Let $x$ denote the Thurston norm on $H_2(E,\partial E)$. We define an invariant $F_K$ of $K$ (as a knot inside the solid torus) as follows:
$$F_K := \frac{1}{4} (\lk(K,c) - x([\hat{D}])+1 )(\lk(K,c) - x([\hat{D}])-1).$$
Note that when $c$ links $K$ coherently $x([\hat{D}]) = \lk(K,c)-1$, and $F_K=0$. 

\begin{theo}\label{Theorem:extremalhfk}
    Let $\{K_m\}$ be a twist family of knots, and assume $\wind_{c}(K) \neq 0$. For any $k \in \mathbb{N}$, there exists $N \in \mathbb{N}$ such that for any $m \geq N$ and $0\leq j \leq k-1$, we have
    $$\widehat{HFK}(K_m,-g(K_m)+j)\cong \widehat{HFK}(K_{m+1},-g(K_{m+1})+j) \ [F_K],$$
    where $[F_K]$ denotes decreasing the Maslov grading by $F_K$. 
\end{theo}

Regarding the stabilization of the extremal coefficients of the Alexander polynomial, Daren Chen pointed out to us that the main results in this area follow readily from the Torres formula. With this in mind, we offer alternative proofs arising from the perspective of knot Floer homology. While this section does not contain new results, some of the discussions may shed light on later developments (particularly Theorem \ref{Theorem:thickness}). They may also be of interest for future work on the topic.\\

Nonetheless, we prove the stabilization of the extremal coefficients of the Alexander polynomial. In addition, we show that the total number of jumps in the sequence of its coefficients stabilizes. This implies that the sequence of coefficients becomes periodic, except at a fixed number of positions. These results are presented in Theorems \ref{Theorem:extremalAlexander} and \ref{Theorem:jumpsAlexander}.

\begin{theo}\label{Theorem:extremalAlexander}
       Let $\{K_m\}$ be a twist family of knots and assume ${\wind_{c}(K) \neq 0}$. There exists a Laurent series $h_{L}(t)$ with finitely many terms of negative degree in $t$, and an integer $ l \in \{-x([\hat{D}]), \cdots, 0\}$, such that the following holds:\\
       
       For any $k \in \mathbb{N}$, there exists $N \in \mathbb{N}$ such that for any $m \geq N$, the first $k$ terms (in the increasing order of degree) of $\Delta_{K_m}$ agree with the first $k$ terms of 
    $t^{m \delta} \ h_{L}(t)$ where 
    $$\delta = \frac{l}{2}\wind_{c}(K).$$
\end{theo}

\begin{theo}\label{Theorem:jumpsAlexander}
     Let $\{K_m\}$ be a twist family of knots, and assume $\wind_{c}(K) \neq 0$. Let $\{a_{m,i}\}_{i \in \mathbb{Z}}$ represent the series of coefficients of the Alexander polynomial of $K_m$, i.e.,
     $$\Delta_{K_m}(t)= \sum_{i \in \mathbb{Z}} \alpha_{m,i} \cdot t^i.$$
     Define $j_{m}(K)$ as follows: 
     $$j_{m}(K) : =|\{ i \ | \ \alpha_{m,i+\wind_c(K)} - \alpha_{m,i} \neq 0 \}|.$$ Then, there exists $j_{\infty}(K) \in \mathbb{N}$, such that for sufficiently large $m$ we have  
     $$j_{m}=j_{\infty}(K).$$ 
\end{theo}

We can also analyze the behavior of Heegaard Floer invariants in a twist family. These results appear in Theorems \ref{Theorem:tau}, \ref{Theorem:thickness}. Theorem \ref{Theorem:tau} generalizes results of Van Cott \cite{Cott2008OzsvthSzabAR} and Roberts \cite{Roberts2011EXTENDINGVC}.

\begin{theo}\label{Theorem:tau}
    Let $\{K_m\}$ be a twist family of knots. Assume $\wind_{c}(K) \neq 0$. Then, there exist fixed integers $T,t$ such that for sufficiently large $m$, the tau invariant of $K_m$ is given by:
    $$\tau(K_m) = T \cdot m + t.$$
\end{theo}

\begin{theo}\label{Theorem:thickness}
    Let $\{K_m\}$ be a twist family of knots. Assume $\wind_{c}(K) \neq 0$. Then, there exist fixed integers $W,w$ such that for sufficiently large $m$, the knot Floer thickness of $K_m$ is given by:
    $$\text{th}(K_m) = W \cdot m + w.$$
\end{theo}

\subsection{Organization of the paper}\hfill\\

This paper is organized as follows. In Section \ref{Section:background}, we briefly review background material from bordered Floer theory and the theory of immersed curve invariants. Section \ref{Section:Immeresedcurve} applies the immersed curve invariant to prove Theorem \ref{Theorem:totaldim}. In Section \ref{Section:Boxtensor}, we use the pairing theorem in bordered Floer homology to analyze the properties of the box tensor product, which yields the knot Floer homology of $K_{m}$. The remaining sections are devoted to proving the results stated in the introduction, building on the setup and constructions developed in Section \ref{Section:Boxtensor}.\\

In the following, we provide a more detailed overview of the paper. This aims to give readers familiar with bordered Floer theory and knot Floer homology a clearer sense of the main ideas and techniques, while also serving as a road map for those less familiar with the technical details. That said, some parts of the overview may be difficult to fully grasp without reference to the more thorough explanations provided later in the paper.\\ 

We set $\omega$ to be $-\lk(K,c)$. In this paper, without the loss of generality, we assume that $\lk(K,c) > 0$.\\  

In Section \ref{Section:background}, we briefly recall key tools from bordered Floer theory and the theory of immersed curve invariants. This section may also serve as a concise survey of both theories for interested readers.\\

The main strategy is to reinterpret the Dehn surgery description of the twisting family ${K_m}$ as a gluing of the following bordered three-manifolds:
$$(H , \mu_{c} , \lambda_{c}) \ \text{and} \ (H'_{\frac{1}{m}}, \mu_{c}, \lambda_{c})$$
where:
\begin{itemize}
    \item $H := S^3 \setminus N(c)$ is the solid torus containing the knot $K$;
    \item $\lambda_{c} \subset \partial H$ is a Seifert longitude of $c$, and in this setting it serves as a meridian of $H$;
    \item $\mu_c \subset \partial H$ is a meridian of $c$, which acts as a longitude on $\partial H$;
    \item $H'_{\frac{1}{m}}$ is a solid torus whose meridian is homologous to $[\mu_c] - m \cdot [\lambda_c]$.\\
\end{itemize}
  
In Section \ref{Section:Immeresedcurve}, we apply the surgery-to-gluing translation described above, together with the computation of the immersed curve invariant of $(H'_{\frac{1}{m}}, \mu_c, \lambda_c)$ from Example \ref{Example:twistedmeridian}, and the pairing theorem for immersed curves (Theorem \ref{Theorem:immeresedpairing}), to establish Theorem \ref{Theorem:totaldim}.\\

From the bordered translation of the twisting problem, we extract certain type $D$ and type $A$ invariants. We assume that the pair $(\mu_c, \lambda_c)$ parametrizes $\partial H$ compatibly with its orientation. This parametrization induces the reversed orientation on $\partial H'_{\frac{1}{m}}$. Accordingly, we consider:
\begin{itemize}
    \item A doubly-pointed bordered Heegaard diagram $(\mathcal{H}_{K},z,w)$ representing the pair $(H,K)$, which determines
    $$\widehat{\text{CFA}}(\mathcal{H}_{K},z)\simeq \widehat{\text{CFA}}(H) \ \ \text{and}  \ \  \widehat{\text{CFA}}(\mathcal{H}_{K},z,w) \simeq \widehat{\text{CFA}}(H,K)$$
    \item A bordered Heegaard diagram $(\mathcal{H}'_{\frac{1}{m}},z')$ for $H'_{\frac{1}{m}}$, which yields:
    $$\widehat{\text{CFD}}(\mathcal{H}'_{\frac{1}{m}},z')$$
    This computation is done in Example \ref{Example:twistedmeridian}.
\end{itemize}

By the pairing theorems for bordered invariants (Theorems \ref{Theorem:pairing} and \ref{Theorem:pairingknots}), we have chain homotopy equivalences: 
$$\CFKh(S^3,K_m) \simeq \widehat{\text{CFA}}(\mathcal{H}_{K}, z, w) \boxtimes \widehat{\text{CFD}}(\mathcal{H}'_{\frac{1}{m}} , z'),$$
$$\CFh(S^3) \simeq \widehat{\text{CFA}}(\mathcal{H}_{K}, z) \boxtimes \widehat{\text{CFD}}(\mathcal{H}'_{\frac{1}{m}} , z').$$
We denote the absolute Maslov and Alexander gradings on these complexes by $h_m$ and $a_m$, respectively.\\

The remainder of this paper is devoted to analyzing these two box tensor products. Throughout, we assume that both $\widehat{\text{CFA}}(\mathcal{H}_K, z)$ and $\widehat{\text{CFA}}(\mathcal{H}_K, z, w)$ are bounded. This assumption is not restrictive, as it can be arranged up to homotopy equivalence.\\

In Subsection \ref{Subsection:GenstructureofBoxtensor}, we construct a basis $\bm{C^{m}}$ for the two box tensor products. This basis is arranged in a specific circular pattern, illustrated in Figure \ref{Figure:Boxtensorscheme}, and decomposes into a central black box $C^{\bu}$ and a collection of white boxes $C^{\circ}_{1}, \dots, C^{\circ}_{m}$. We define two distinct inclusion maps,
$$\Phi_{m}, \Phi'_{m}: \bm{C^{m}} \rightarrow \bm{C^{m+1}}$$ 
which endow the family $\{\bm{C^{m}}\}_{m \in \mathbb{Z}_{\geq 0}}$ with two nested sequence structures, as depicted in Figures \ref{Figure:Inclusion1} and \ref{Figure:Inclusion2}. This framework greatly simplifies many of our subsequent discussions on stabilization..\\

We also describe the differentials in the box tensor product, as detailed in Proposition \ref{Proposition:boxtensordifferential} and illustrated in Figure \ref{Figure:Edgedecomposition}. Throughout, we adopt a graph-theoretic presentation for all our chain complexes, including the bordered invariants.\\

In Subsection \ref{Subsection:Gradingsinboxtensor}, we analyze the relative gradings of elements in the basis $\bm{C^{m}}$ and their behavior as $m$ varies. Moreover, we compute the grading differences between elements in $\bm{C^{m}} \setminus C^{\bu}$ and their natural shifts.\\

The following three subsections primarily focus on the chain complex
$$\widehat{\text{CFA}}(\mathcal{H}_{K}, z, w) \boxtimes \widehat{\text{CFD}}(\mathcal{H}'_{\frac{1}{m}} , z')$$
since several results depend on the Alexander grading.\\

In Subsection \ref{Subsection:Minimalandhomogenous}, we begin analyzing elements of the chain complexes by viewing them as subsets of the basis $\bm{C^{m}}$. Working over the field $\mathbb{F}_{2}$, there is a natural one-to-one correspondence:
$$I \subseteq \bm{C^{m}} \longleftrightarrow S_I=\sum_{\nu \in I} \nu.$$
This perspective allows us to define a \emph{minimality} condition, which serves as a refined version of homogeneity and plays a key role in our later structural arguments.\\

In Subsection \ref{Subsection:lengthsizecycles}, we use the circular arrangement of the basis $\bm{C^{m}}$ to define a natural metric on the set of basis elements. This allows us to introduce the notion of \emph{length} (a refinement of diameter) for subsets of $\bm{C^{m}}$, which quantifies how spread out the elements of a subset can be along the circular ordering. We then establish a uniform upper bound on the length of subsets corresponding to minimal cycles, providing a key constraint for our later stability arguments.\\ 

In Subsection \ref{Subsection:lengthsizerelations}, we study \emph{relations} between cycles i.e. equalities of the form 
$$\partial S_{R} = S_{I_1} + \cdots + S_{I_r}.$$
where each $S_{I_j}$ is a cycle. We define a minimality condition for such relations and establish upper bounds on the length of minimal relations, analogous to the bounds on minimal cycles. While the homology is determined by the full spaces of cycles and relations, we show that the minimal cycles and minimal relations suffice to generate these spaces. This allows us to restrict attention to a well-controlled subset of data in our subsequent analysis.\\

The next two subsections focus on the analysis of the complex 
$$\widehat{CFK}(S^3) \simeq \widehat{\text{CFA}}(\mathcal{H}_{K}, z) \boxtimes \widehat{\text{CFD}}(\mathcal{H}'_{\frac{1}{m}} , z').$$
Since the homology of this complex is isomorphic to $\HF(S^3)$, it is one-dimensional. Our goal is to examine the representatives of its unique generator, expressed in terms of the basis $\bm{C^{m}}$.\\

Note that $\widehat{\text{CFA}}(\mathcal{H}_{K}, z)$ is simply the type A invariant of a solid torus with meridian $\alpha_2$. We introduce an alternative bordered Heegaard diagram $(\mathcal{H}_{\infty}, z)$ for this manifold, as shown in Figure \ref{Figure:Mazurisotoped}, and consider the corresponding type A module $\widehat{\text{CFA}}(\mathcal{H}_{\infty}, z)$, illustrated in Figure \ref{Figure:HinfinitytypeA}.\\

In Subsections \ref{Subsection:HomologyofboxtensorS^3:Algebra}, we use the $\A_{\infty}$ homotopy equivalences 
$$f : \widehat{\text{CFA}}(\mathcal{H}_{\infty},z) \rightarrow \widehat{\text{CFA}}(\mathcal{H}_{K},z),$$
to study representatives of the non-trivial homology class. We examine the grading double cosets of the basis elements that appear in such representatives. This analysis allows us to identify certain fixed (i.e., stabilized) elements with Maslov grading zero. These elements serve as reference points for determining absolute Maslov gradings in later arguments, including the proof of Theorem \ref{Theorem:extremalhfk}.\\

We also show that any such representative
$$S_{I_{m_{0}}} \in \widehat{\text{CFA}}(\mathcal{H}_{K}, z) \boxtimes \widehat{\text{CFD}}(\mathcal{H}'_{\frac{1}{m_{0}}} , z')$$
can be extended to a family of representatives $\{S_{I_{m}}\}$ for all $m \in \mathbb{Z}_{\geq 0}$.  While the sequence ${I_m}$ may not be nested, we are able to track how it evolves with $m$, as described in Equation \ref{Equation:Twostabilizationformula}. This structural control is crucial to the arguments in Section \ref{Section:tau}.\\

In Subsection \ref{Subsection:HomologyofboxtensorS^3:Combinatorics}, we introduce notations that facilitate comparison of gradings among basis elements appearing in a representative of the non-trivial homology class, by analyzing the differentials (i.e., edges) connecting them. While these concepts are not particularly deep, they simplify several proofs in Section \ref{Section:tau}.\\

Section \ref{Section:ExtermalknotFloerstabilization} is primarily devoted to the proof of Theorem \ref{Theorem:extremalhfk} in Subsection \ref{Subsection:ProofofTheoremextremalhfk}. In Subsection \ref{Subsection:ballsandshiftmaps}, we introduce the closed metric balls around the black box, denoted by $B^{\bu}_{t}$, and the shift maps $R_{\pm}$ (see Figure \ref{Figure:Blackboxball}). These closed balls capture the stable portion of the chain complex (illustrated in Figure \ref{Figure:Nested}) and form the foundation for nearly all stabilization arguments.\\

By analyzing the action of the shift maps on differentials and null-homologous classes, Subsection \ref{Subsection:extremahhfkfallinball} establishes that the extremal knot Floer homologies are supported within a fixed closed ball. Then, in Subsection \ref{Subsection:ProofofTheoremextremalhfk}, we complete the proof of Theorem \ref{Theorem:extremalhfk} by showing that these extremal homologies can be fully determined from the \emph{data} of that fixed closed ball, namely, the differentials and grading double coset representatives associated of its elements.\\

Section \ref{Section:Alexanderpolynomial} is devoted to proving Theorems \ref{Theorem:extremalAlexander} and \ref{Theorem:jumpsAlexander}. Both follow from a stabilization result (Proposition \ref{Proposition:Alexanderjumpsequence}) concerning the sequence of \emph{differences} (or \emph{jumps}) in the coefficients of the Alexander polynomial:
$$d_{m,i} := \alpha_{m,i} - \alpha_{m,i+\omega} = \chi(\HFKh(K_{m},i)) - \chi(\HFKh(K_{m},i+\omega)).$$
To compare $\widehat{HFK}(K_{m}, i)$ and $\widehat{HFK}(K_{m}, i+\omega)$, we first decompose the knot Floer homology into the subspace classes represented by cycles located \emph{far} from the black box and its complement subspace. The notion of \emph{far} can mean outside a fixed closed ball, but more generally refers to a broader class, see Figure \ref{Figure:Alexanderjumpsdefects}). This is the \emph{far decomposition} introduced in Subsection \ref{Subsection:Fardecomposition}. In Subsection \ref{Subsection:Shiftsonfar}, we prove that the shift maps induce isomorphisms on the subspaces of far homology classes, shifting the Alexander grading by $\omega$ while preserving the Maslov grading modulo two.\\

This reduces the analysis to the complementary subspaces, consisting of non-far classes. In Subsection \ref{Subsection:Stabilizationofcomplement}, we show that these classes are supported within a fixed closed ball and stabilize as an ungraded vector space. The argument closely parallels that used in the proof of Theorem \ref{Theorem:extremalhfk}. Building on this, in Subsection \ref{Subsection:stabilizationofAlexanderjump}, we analyze the behavior of gradings as $m \to \infty$ to establish Proposition \ref{Proposition:Alexanderjumpsequence}. Finally, in Subsection \ref{Subsection:proofsofalexanderstabilization}, we deduce Theorems \ref{Theorem:extremalAlexander} and \ref{Theorem:jumpsAlexander} as consequences of this proposition.\\

In Section \ref{Section:tau}, we focus on proving Theorem \ref{Theorem:tau}. We begin by recalling the definition of the $\tau$ invariant and rewriting it as
$$\tau(K_m) = \underset{S_{I}\in \mathcal{HZ}_{m}}{\min} \Big( \underset{\nu \in I}{\max} \ a_{m}(\nu) \Big),$$
where 
$$\mathcal{HZ}_{m} \ \subset \ \widehat{\text{CFA}}(\mathcal{H}_{K}, z) \boxtimes \widehat{\text{CFD}}(\mathcal{H}'_{\frac{1}{m}} , z'),$$
denotes the set of homogeneous representatives of the nontrivial homology class.
To compute $\tau(K_m)$, we focus on the set $\mathcal{P}_{m}$ as the set of elements that realize the maximal Alexander grading in some cycle in $\mathcal{HZ}_{m}$. The main strategy is to show that this set stabilizes within a fixed closed ball, as stated in Lemma \ref{Lemma:taupotentials}. In Subsection \ref{Subsection:Proofoftau}, we then conclude the proof of Theorem \ref{Theorem:tau} from this stabilization.\\

In Subsection \ref{Subsection:Maslovgradingzeroelements}, we use the grading machinery developed in Subsection \ref{Subsection:Gradingsinboxtensor} to determine the subset of elements in $\bm{C^{m}}$ with Maslov grading zero, and their properties. Although this material naturally fits within Section \ref{Section:Boxtensor}, we chose to include it in Section \ref{Section:tau} instead, as the proof techniques align more closely with the stabilization arguments used in Sections \ref{Section:ExtermalknotFloerstabilization} and \ref{Section:Alexanderpolynomial}.\\

Subsection \ref{Subsection:Stabilizationofmaximizers} is devoted to the proof of Lemma \ref{Lemma:taupotentials}. We begin by applying the results from Subsection \ref{Subsection:Maslovgradingzeroelements} to show that any element that could potentially realize the maximal Alexander grading in a subset $I$ for $S_{I} \in \mathcal{HZ}_{m}$ lies within a fixed closed ball $B^{\bu}_{L_{\tau}}$ ,as established in Proposition \ref{Proposition:taupotentials-bounded} and Corollary \ref{Corollary:placeofmaximizer}.\\

Next, we consider the collection of subsets given by $I \cap B^{\bu}_{L_{\tau}}$, where $I$ comes from cycles $S_{I} \in \mathcal{HZ}_{m}$. By Lemma \ref{Lemma:homotopyequivgenerator2} in Subsection \ref{Subsection:HomologyofboxtensorS^3:Algebra}, we have a distinguished element $\nu_{0} \in I \cap C^{\bu} $. Comparing the Alexander gradings of other elements in $I \cap B^{\bu}_{L_{\tau}}$ relative to $\nu_0$, we can further restrict this collection to those elements that may potentially realize the maximal Alexander grading. This restricted family is precisely what matters for the computation of $\mathcal{P}_{m}$.\\

Finally, we prove that this collection of subsets of potential maximizers stabilizes as $m$ grows. This is done in Lemmas \ref{Lemma:generatorinpositivehalfballstablizes} and \ref{Lemma:generatorinnegativehalfballstablizes}. To this end, we split the closed ball $B^{\bu}_{L_{\tau}}$ to a positive and negative half-balls (as defined in Subsection \ref{Subsection:ballsandshiftmaps}). The stabilization over the positive half-ball follows readily from the results in Subsection \ref{Subsection:HomologyofboxtensorS^3:Algebra}, especially Equation \ref{Equation:Twostabilizationformula}. In contrast, stabilization over the negative half-ball requires the combinatorial tools developed in Subsection \ref{Subsection:HomologyofboxtensorS^3:Combinatorics}.\\

Finally, in Section \ref{Section:thickness}, we prove Theorem \ref{Theorem:thickness}. This proof further analyzes the far decomposition of the knot Floer homology. In Subsection \ref{Subsection:Shiftequivariantbasis}, we construct a \emph{shift-equivariant} basis for the subspace of far homology classes. This is expected both from the results of Subsection \ref{Subsection:Shiftsonfar} and the theme of Proposition \ref{Proposition:boxtensordifferential} (also illustrated in Figure \ref{Figure:Edgedecomposition}), which shows that the chain complex (its elements and differentials) behaves shift-equivariantly away from the black box.\\

In Subsection \ref{Subsection:deltagradingstab}, as a straightforward corollary of the results in Subsection \ref{Subsection:Gradingsinboxtensor}, we observe that the $\delta$-grading changes linearly along any shift orbit. Consequently, the maximizers and minimizers of the $\delta$-grading within the far homology subspace lie at the two extremes of a shift orbit and remain close to the black box. Combining this with the stabilization of the complement subspaces from Subsection \ref{Subsection:Stabilizationofcomplement}, we conclude that the overall maximizer and minimizer of the $\delta$-grading also stabilize, as stated in Lemma \ref{Lemma:deltagradingmaximizer}. This leads to an immediate proof of Theorem \ref{Theorem:thickness} in Subsection \ref{Subsection:proofofthickness}.\\

\subsection{Future work}\hfill\\

There are several promising directions for future work. We briefly highlight a few below. As previously mentioned, Baker and Taylor \cite{Baker2016DehnFA} proved a stabilization result for the Seifert genus in a twisting family. This result, stated in Theorem \ref{Theroem:BakerTaylor}, arises from their broader study of how the Thurston norm behaves under Dehn filling.

\begin{theo}\label{Theroem:BakerTaylor}\cite{Baker2016DehnFA}
    There is a constant $G=G(K,c)$ such that 
    $$2g(K_{m}) = (x([\hat{D}])\cdot \lk(K,c))\cdot m+ 2G $$
    for sufficiently large $m$.
\end{theo}

We have used their result multiple times throughout this work. It is natural to expect that a careful investigation of the relationship between the bordered invariant $\widehat{\text{CFA}}(S^{3} \setminus N(c), K)$ and the Thurston norm on $S^{3} \setminus N(K \cup c)$ could yield an alternative proof of this stabilization result from a bordered Floer perspective.\\

Another promising direction for future work is to study the twisting problem in the case where $\wind_{c}(K) = 0$. As is apparent throughout this paper, both the statements and the techniques rely crucially on the assumption that $\wind_{c}(K) \neq 0$.\\

In addition, it is possible that the arguments presented here can be simplified using more modern tools. As Chen Zhang pointed out to us, work of Ian Zemke on a bordered perspective of the link surgery formula \cite{zemke2021bordered} could potentially be used to study the twisting problem. Similarly, advances in the theory of immersed curves, such as Hanselman's work \cite{hanselman2023knotfloerhomologyimmersed}, might help avoid the combinatorial intricacies of working directly with chain complexes.\\

This project began as the author’s first encounter with both bordered Floer theory and the theory of immersed curves. As a result, the methods used here were designed to involve only the minimal necessary machinery from the two theories, focusing instead on a direct combinatorial analysis of the relevant complexes.\\

Finally, although we have made efforts to refine the exposition, through improvements in structure, notation, and presentation, we are aware that, given the length and scope of the paper, there remains room for improvement. We welcome comments and suggestions from readers and plan to revise the paper further in future versions.

\section*{Acknowledgement}

It is a pleasure to thank my advisor, András Juhász, whose patience, support, and guidance made this project possible. I am especially grateful to Claudius Zibrowius, who served as a mentor throughout the project. His insights were invaluable not only in helping me learn the relevant theory, but also in sustaining my motivation through his interest in the work. I also wish to thank Matthew Hedden, Jacob Rasmussen, Sunkyung Kang, and Chen Zhang for their valuable feedback, insightful suggestions, and interest in the project.

\section{Background}\label{Section:background}
\subsection{Type D and Type A invariants}\hfill\\

We focus on \emph{bordered manifolds with torus boundary}. Let $M$ be an orientable 3-manifold with torus boundary, and select two oriented, essential, simple closed curves $\alpha_1$ and $\alpha_2$ on $\partial M$ with $|\alpha_1 \cap \alpha_2| = 1$. As we will discuss, to define a bordered invariant of type $A$, we require the oriented intersection number $\alpha_1 \cdot \alpha_2$ to be $+1$. Conversely, to define a type $D$ invariant, we require this intersection to be negative. The triple $(M, \alpha_1, \alpha_2)$ is called a bordered three-manifold.\\

Lipshitz, Oszváth and Thurston \cite{Lipshitz2008BorderedHF} associated a \emph{type D structure} $\widehat{\text{CFD}}(M,\alpha_1,\alpha_2)$, and a \emph{type A structure} $\widehat{\text{CFA}}(M,\alpha_1,\alpha_2)$, over the \emph{torus algebra} $\A$ to bordered three-manifolds. Both structures are invariant up to certain notion of homotopy equivalence. We define these terms in the following paragraphs. \\

The torus algebra $\A$ is given by the path algebra of the quiver shown in Figure \ref{Figure:quiver}. As a vector space over $\mathbb{F}_2$, the torus algebra has a basis consisting of two idempotent elements $\iota_0$ and $\iota_1$, and six elements: $\rho_1,\rho_2,\rho_3,\rho_{12} =\rho_{1} \rho_{2},\rho_{23}=\rho_{2} \rho_{3},\rho_{123}=\rho_{1}\rho_{2}\rho_{3}$. We occasionally use the notation $\rho_{\varnothing}$ to denote the unit element $\mathds{1} = \iota_0 + \iota_1$. \\

Let $\I$ represent the ring of idempotents of $\A$. Unless stated otherwise, all of the tensor products in this section are over $\I$. We use $\mu : \A \otimes \A \rightarrow \A$ to denote the multiplication of the torus algebra. Note that torus algebra is associative. \\

A type D structure over $\A$ is a unital left $\I$-module $N$ equipped with an $\I$-linear (differential) map $\delta : N \rightarrow \A \otimes N$. The map $\delta$ must  satisfy a compatibility condition analogous to $\partial^2 =0$. This condition is vanishing of the map $(\mu \otimes \id) \circ (\id \otimes \delta) \circ \delta$ also described in the following sequence: 
$$N \xrightarrow[]{\delta} \A \otimes N \xrightarrow[]{\id_{\A} \otimes \delta} \A \otimes \A \otimes N \xrightarrow[]{\mu \otimes \id_{N}} \A \otimes N.$$
For a type D structure $N$, we can define a family of maps $\delta_{k} : N \rightarrow \A^{\otimes k} \otimes N$. Let $\delta_0 = \id_{N}$, $\delta_1 := \delta$, and then for $k \geq 2$ inductively define :
$$\delta_{k} := (\id^{\otimes (k-1)} \otimes \delta_1) \circ \delta_{k-1}.$$
We call a type D structure \emph{bounded} if $\delta_k =0$ for sufficiently large $k$. We call a type D structure \emph{reduced} if there are no elements $x$ with a summand $\rho_{\varnothing} \otimes y$ in $\delta(x)$. In this paper, we work exclusively with bounded type D structures. This condition is not very restrictive as any type D structure is homotopic to a bounded one (see \cite{Hanselman2016BorderedFH}).\\

There is a convenient graph-theoretic shorthand for describing type D structures. We work with decorated graphs, where each vertex is labelled with exactly one element from $\{\bu,\ci\}$, and every edge is directed and labeled with exactly one element from $\{\varnothing, 1,2,3,12,23,123\}$. The edge and vertex labels must be compatible with the quiver in Figure \ref{Figure:quiver}, leading to the following possible arrangements: 
$$\bu \xrightarrow[]{\varnothing} \bu, \ \ci \xrightarrow[]{\varnothing} \ci, \ \bu \xrightarrow[]{1} \ci, \ \bu \xrightarrow[]{3} \ci, \ \ci \xrightarrow[]{2} \bu, \ \bu \xrightarrow[]{12} \bu, \ \ci \xrightarrow[]{23} \ci, \
 \bu \xrightarrow[]{123} \ci$$

We associate a decorated graph $\Gamma_{N}$ to the type D structure $N$. Since $N$ is a left $\I$-module, it splits as $N = \iota_0 N \oplus \iota_1 N$. Let $B_N$ be a basis for $N$, formed from the disjoint union of a basis $B_{N_0}$ for $\iota_0 N $, and a basis $B_{N_1}$ for $\iota_1 N $. For each element of $B_N$, assign a corresponding vertex in $\Gamma_N$, where elements of $B_{N_0}$ are labeled with $\bu$ and elements of $B_{N_1}$ are labelled with $\ci$. We use the same notation for both an element in the basis $B_N$ and its associated vertex.\\

Next, we describe the edges that originate from any $x \in B_N$. The map $\delta(x)$ can be uniquely expanded in terms of the basis $B_N$; that is, there exist $x_1,\cdots,x_k \in B_N$ and $I_1,\cdots,I_k \in \{\varnothing, 1,2,3,12,23,123\}$ such that: 
$$\delta(x) = \sum_{1\leq j \leq k} \rho_{I_j} \otimes x_{j}.$$
For each $1 \leq j \leq k$, we consider an edge in $\Gamma_N$, extending from $x$ to $x_j$ with the label $\rho_{I_j}$. An example of such a decorated graph is depicted in Figure \ref{Figure:typeDgraph}.\\

\begin{figure}[h]
\centering
\begin{center}
\includegraphics[scale=1]{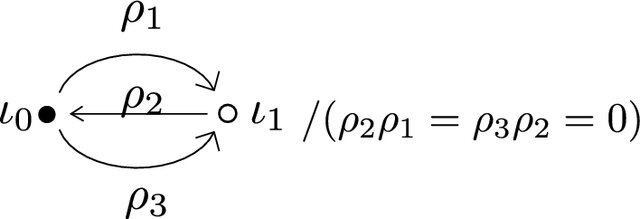}
\end{center}
\caption{Quiver used to define torus algebra $\A$ \cite{Chen2019KnotFH}}\label{Figure:quiver}
\end{figure}

\begin{figure}[h]
\centering
\begin{center}
\includegraphics[scale=0.15]{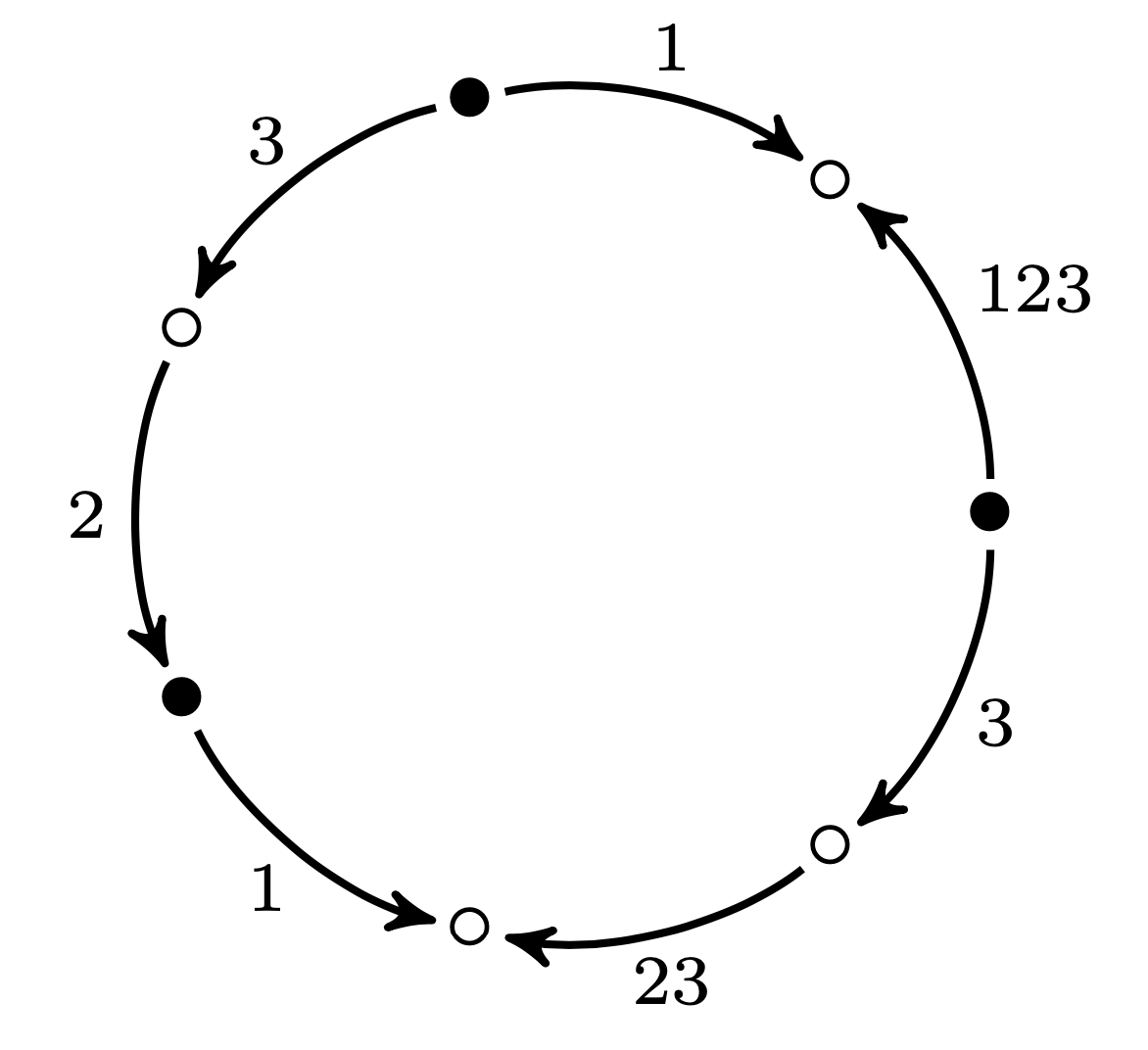}
\end{center}
\caption{Decorated graph of a type D structure \cite{Hanselman2016BorderedFH}}\label{Figure:typeDgraph}
\end{figure}

Higher-order $\delta_{k}$ are determined by directed paths in $\Gamma_N$. The type D structure $N$ is bounded if and only if $\Gamma_N$ contains no directed cycles. Similarly, $N$ is reduced if and only if $\Gamma_N$ has no edge with label $\varnothing$.\\

A type A structure is defined as a right $A_{\infty}$-module over $\A$. This means that a type A structure is a right unital $\I$-module $P$, with a family of maps $m_{i+1} : P \otimes \A^{\otimes i} \rightarrow P$, such that for any $x \in P$ and $a_1 \otimes \cdots \otimes a_{n-1} \in \A^{\otimes (n-1)}$, the following holds:
$$\sum_{i=1}^{n} m_{n-i}(m_{i}(x \otimes a_1 \otimes \cdots \otimes a_{i-1})\otimes \cdots \otimes a_{n-1})+\sum_{i=1}^{n-2} m_{n-1}(a_1 \otimes \cdots \otimes a_ia_{i+1} \otimes \cdots \otimes a_{n-1}) = 0$$
and 
$$m_2(x,\mathds{1} )=x,$$
$$m_i(x, \cdots , \mathds{1}, \cdots) =0 \ \text{for} \ i \geq 3.$$
We call a type A structure \emph{bounded} if $m_k =0$ for sufficiently large $k$. A type A structure is called \emph{reduced} if $m_1=0$.\\

Similarly, it is possible to give a graph-theoretic description of type A structures, but the graphs and the edge labels will be more complicated. In the special case of the bordered invariant $\widehat{\text{CFA}}(M,\alpha_1,\alpha_2)$, Hedden and Levine \cite{Hedden2012SplicingKC} used the duality of the type A and type D structures to offer a simplified graphical representation of the type A invariants. Their method is described briefly below.\\

Assume that $\widehat{\text{CFD}}(M, \alpha_1, \alpha_2)$ is reduced. Starting with a decorated graph $\Gamma_{D}$ representing $\widehat{\text{CFD}}(M, \alpha_1, \alpha_2)$, we rewrite the edge labels according to the following bijection: $$1 \leftrightarrow 3, 2 \leftrightarrow 2, 3 \leftrightarrow 1$$ which gives us the decorated graph $\Gamma_{A}$. The vertices of $\Gamma_{A}$ are labelled with $\{\bu,\ci\}$, and the edges are labelled with $\{1,2,3,32,21,321\}$. We construct $\widehat{\text{CFA}}(M, \alpha_1, \alpha_2)$ from $\Gamma_{A}$ as follows.\\

The vertices and their labels provide a basis for $\widehat{\text{CFA}}(M, \alpha_1, \alpha_2)$ as a right $\I$-module. To compute the maps $m_{n+1}$, we consider all directed paths that begin at a vertex $x$. For a directed path from $x$ to $y$, we add a $y$ summand to $m_{n+1}(x \otimes \rho_{I_1} \otimes \cdots \otimes \rho_{I_n})$, where $I_1, \cdots, I_n$ are defined as follows. Let $(J_1, \cdots, J_l)$ be the sequence of edge labels in the path. Concatenate these labels to form a word $J_1 \cdots J_l$, then regroup the word into $I_1 \cdots I_n$, such that $n$ is the minimum integer where each $I_t$ is an element of $\{1,2,3,32,21,321\}$. Following this algorithm for all the directed paths provides the complete $\A_{\infty}$-module structure of $\widehat{\text{CFA}}(M, \alpha_1, \alpha_2)$.\\

\begin{figure}[h]
\centering
\begin{center}
\includegraphics[scale=0.25]{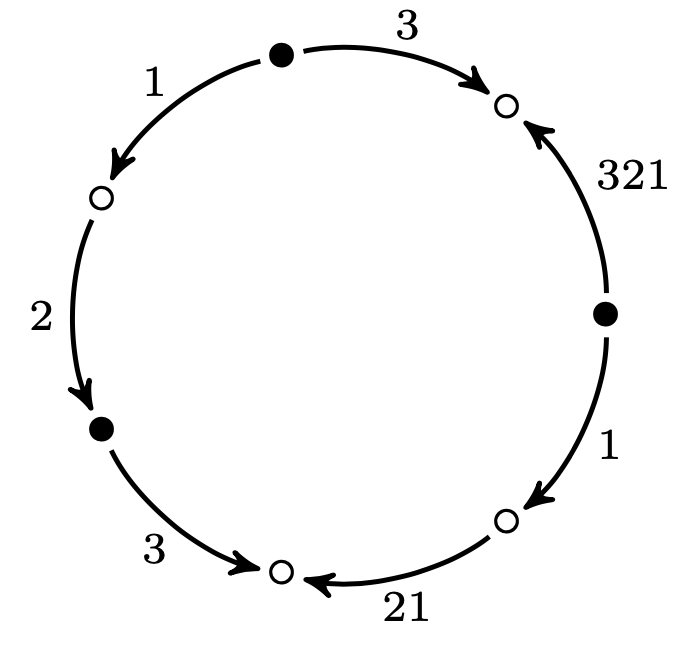}
\end{center}
\caption{Decorated graph of a type A structure \cite{Hanselman2016BorderedFH}}\label{Figure:typeAgraph}
\end{figure}

A type A structure $P$ and a type $D$ structure $N$ can be paired to form the \emph{box tensor product} $P \boxtimes N$. As a $\mathbb{F}_2$ vector space, $P \boxtimes N$ is isomorphic to $P \otimes_{\I} N$. The differential which turns $P \boxtimes N$ to a chain complex is defined as follows:
\begin{equation}\label{Equation:Boxtensordifferential}
\partial^{\boxtimes}(x \otimes y)= \sum_{i=0}^{\infty} (m_{i+1} \otimes \id_{N})(x \otimes \delta_{i}(y)).
\end{equation}
Note that boundedness of $N$ ensures that the sum in the above equation is finite, and therefore the differential is well-defined.\\

As previously discussed, to bordered manifolds $(M, \alpha_1, \alpha_2)$ and $(M', \alpha'_1, \alpha'_2)$, we can associate bordered invariants $\widehat{\text{CFA}}(M,\alpha_1,\alpha_2)$ and $\widehat{\text{CFD}}(M', \alpha'_1, \alpha'_2)$. We can glue two bordered manifolds using an orientation reversing homeomorphism ${h:\partial M' \rightarrow \partial M}$ which is specified by $h(\alpha'_1)=\alpha_1$ and $h(\alpha'_2)=\alpha_2$. This gives us a closed 3-manifold $Y := M \cup_{h} M'$. Lipshitz, Oszváth and Thurston \cite{Lipshitz2008BorderedHF} proved a pairing theorem which relates the Heegaard Floer homology of $M$ to bordered invariants. 

\begin{theo}\label{Theorem:pairing}
If $\widehat{\text{CFD}}(M', \alpha'_1, \alpha'_2)$ is bounded, then we have 
$$\widehat{CF}(Y)\simeq \widehat{\text{CFA}}(M,\alpha_1,\alpha_2) \boxtimes \widehat{\text{CFD}}(M', \alpha'_1, \alpha'_2),$$
where $\widehat{CF}(Y)$ denotes a Heegaard Floer chain complex for $Y$. 
\end{theo}

The definition of bordered invariants of $
(M, \alpha_1, \alpha_2)$ is based on a \emph{pointed bordered Heegaard diagram} $(\mathcal{H},z) = ((\Sigma, \bm{\alpha}, \bm{\beta}), z)$ and involves counting $J$-holomorphic curves for construction of $\delta$ and $m_{i}$ maps. The general definitions are not required for this paper, and we refer the reader to the works of Lipshitz, Oszváth and Thurston \cite{Lipshitz2008BorderedHF}. Later in this section, we provide these definitions for the special case where $\Sigma$ has genus one.\\

Similar to the Heegaard Floer homology, we can also define invariants for knots inside bordered 3-manifolds as well. For a knot $K$ within a bordered 3-manifold $(M, \alpha_1, \alpha_2)$, we define bordered invariants $\widehat{\text{CFA}}(M, K)$ and $\widehat{\text{CFD}}(M, K)$. These invariants are constructed using a \emph{doubly-pointed bordered Heegaard diagram} $(\mathcal{H},z, w) = ((\Sigma, \bm{\alpha}, \bm{\beta}), z, w)$ for the pair $(M,K)$.  The construction of the maps $\delta$ and $m_i$ follows the same method as for $(\mathcal{H},z)$, except that some of the curves are excluded based on their interaction with the additional basepoint $w$. In the graph-theoretic description, $\widehat{\text{CFD}}(M, K)$ (resp.~$\widehat{\text{CFA}}(M, K)$) is obtained by deleting some of the edges of a certain graph representing $\widehat{\text{CFD}}(M, \alpha_1, \alpha_2)$ (resp.~$\widehat{\text{CFA}}(M, \alpha_1, \alpha_2)$). \\

Lipshitz, Oszváth and Thurston \cite{Lipshitz2008BorderedHF} generalized the pairing theorem to bordered invariants of knots. 

\begin{theo}\label{Theorem:pairingknots}
Let $(M, \alpha_1, \alpha_2)$ and $(M', \alpha'_1, \alpha'_2)$ be bordered 3-manifolds, and let $K$ be a knot in $M$. Let $Y$ to denote $M \cup_{h} M'$. Then the following holds: 
\begin{enumerate}
\item If $\widehat{\text{CFD}}(M', \alpha'_1, \alpha'_2)$ is bounded,
$$\widehat{CFK}(Y,K) \simeq \widehat{\text{CFA}}(M,K) \boxtimes \widehat{\text{CFD}}(M', \alpha'_1, \alpha'_2).$$

\item If $\widehat{\text{CFD}}(M,K)$ is bounded,
$$\widehat{CFK}(Y,K) \simeq \widehat{\text{CFA}}(M', \alpha'_1, \alpha'_2) \boxtimes \widehat{\text{CFD}}(M,K).$$
\end{enumerate}
\end{theo}

We can assign a grading to type D and type A structures associated with bordered $3$-manifolds with torus boundary. This grading takes values in certain coset spaces of a non-commutative group $G$, defined as follows: 
$$G = \{(m ; i,j) \ | \ m,i,j \in \frac{1}{2}\mathbb{Z}, i+j \in \mathbb{Z}\}$$
with group law given by: 
$$(m_1;i_1,j_1) \cdot (m_2;i_2,j_2) = (m_1+m_2+\begin{vmatrix}
    i_1       & j_1 \\
    i_2       & j_2
\end{vmatrix}; i_1+i_2 , j_1+j_2).$$
In this context, $m$ is known as the \emph{Maslov component}, and $(i,j)$ is referred to as the $\text{\emph{Spin}}^{\mathbb{C}}$ \emph{component}. When dealing with bordered invariants of knots inside manifolds with torus boundary, we must extend this group to an enhanced version, $\widetilde{G}=G \times \mathbb{Z}$, where the additional $\mathbb{Z}$ factor in called the \emph{Alexander factor}.\\

The following two central elements of $\widetilde{G}$ will be relevant to us: 
$$\lambda = (1;0,0;0) \ , \ \mu=(0;0,0;1).$$

The torus algebra is graded by $\widetilde{G}$ (and $G$) as follows :
$$gr(\iota_0) = gr(\iota_1) = (0;0,0;0),$$
$$gr(\rho_1)=(-\frac{1}{2};\frac{1}{2}, -\frac{1}{2};0),$$
$$gr(\rho_2)=(-\frac{1}{2};\frac{1}{2}, \frac{1}{2};0),$$
$$gr(\rho_3)=(-\frac{1}{2};-\frac{1}{2}, \frac{1}{2};0),$$
$$gr(\rho_{12})=gr(\rho_1)gr(\rho_2), gr(\rho_{23})=gr(\rho_2)gr(\rho_3), gr(\rho_{123})=gr(\rho_1)gr(\rho_2)gr(\rho_3).$$

To determine the grading of a type D invariant $\widehat{\text{CFD}}(M, \alpha_1, \alpha_2)$, we first select a reference generator, $\bm{z_0}$, and assign it the grading $(0; 0,0)$. The grading is then determined by the following rules:
$$gr(\delta(x)) = \lambda^{-1} gr(x), $$ 
$$gr(a \otimes x) := gr(a) gr(x).$$
In terms of the graph-theoretic description of the type D invariant, these rules can be summarized in the following table.
\begin{figure}[h]
\begin{center}
\includegraphics[width=\textwidth]{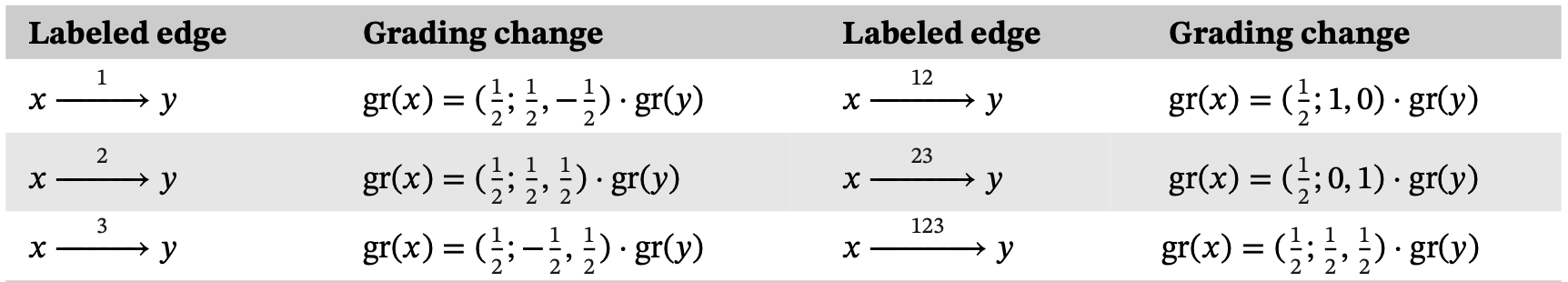}
\end{center}
\caption{Rules of grading a type D structure \cite{Hanselman2018HeegaardFH}}\label{Figure:typeDgrading}
\end{figure}

Defining the grading of a type A invariant $\widehat{\text{CFD}}(M, \alpha_1, \alpha_2)$ follows a similar procedure to that of type D. First, a reference generator must be chosen. Then, for any $x,y \in \widehat{\text{CFA}}(M, \alpha_1, \alpha_2)$ and $a_1 \otimes \cdots \otimes a_n \in \A^{\otimes n}$, if $y$ appears in $m_{n+1}(x \otimes a_1 \otimes \cdots \otimes a_n)$, the grading satisfies:
\begin{equation}\label{Equation:typeAgrading}
gr(y) = \lambda^{-1+n}gr(x)\cdot gr(a_1) \cdots gr(a_n).  
\end{equation}

The graph-theoretic interpretation of grading rules of the type A invariant is depicted in Figure \ref{Figure:typeAgrading}. \\

\begin{figure}[h]
\begin{center}
\includegraphics[width=\textwidth]{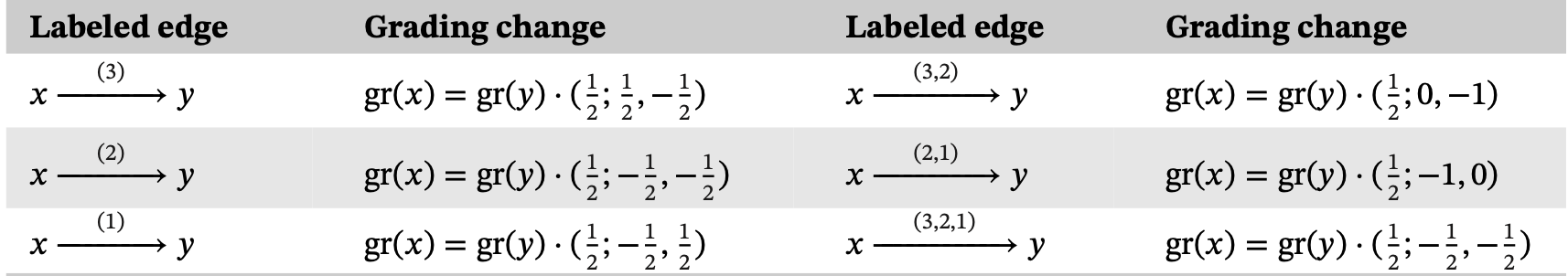}
\end{center}
\caption{Rules of grading a type A structure \cite{Hanselman2018HeegaardFH}}\label{Figure:typeAgrading}
\end{figure}

For this paper, we work with knots in solid torus, and hence, the grading rules described are sufficient for determining the grading of all elements in the associated bordered invariants.\\

To compute the Alexander factor of gradings, we need the definition of domains in bordered Heegaard diagrams. See \cite{Chen2019KnotFH} for more details.\\

It turns out that a grading which takes values in $G$ is not always well-defined. Instead, a subgroup $P(\bm{z_0}) \subset G$ exists such that the gradings on $\widehat{\text{CFD}}$ (resp.~$\widehat{\text{CFA}}$) are defined modulo the right (resp.~left) action of $P(\bm{z_0})$. Specifically, the gradings maps are as follows: 
$$gr : \widehat{\text{CFA}}(M, \alpha_1, \alpha_2) \rightarrow P(\bm{z_0}) \backslash G, $$
$$gr : \widehat{\text{CFA}}(M,K) \rightarrow P(\bm{z_0}) \backslash \widetilde{G}, $$
$$gr : \widehat{\text{CFD}}(M, \alpha_1, \alpha_2) \rightarrow G/P(\bm{z_0}), $$
$$gr : \widehat{\text{CFD}}(M,K) \rightarrow \widetilde{G}/P(\bm{z_0}). $$
Here, $P(\bm{z_0})$ records the gradings of the periodic domains in the bordered Heegaard diagram. When $M$ is a rational homology solid torus, $P(\bm{z_0})$ is a cyclic group, and it can be identified using the rules described earlier.\\

The gradings on type $D$ and type $A$ structures, induce a grading on the box tensor product as follows: 
$$gr:\widehat{\text{CFA}}(M, \alpha_1, \alpha_2) \boxtimes \widehat{\text{CFD}}(M', \alpha'_1, \alpha'_2) \rightarrow P(\bm{z_0}) \backslash G / P(\bm{z'_0})$$
$$gr : x\otimes y \rightarrow  gr(x) \cdot gr(y)$$
$$gr:\widehat{\text{CFA}}(M, K) \boxtimes \widehat{\text{CFD}}(M', \alpha'_1, \alpha'_2) \rightarrow P(\bm{z_0}) \backslash \widetilde{G} / (P(\bm{z'_0});0)$$
$$gr : x\otimes y \rightarrow gr(x) \cdot (gr(y);0)$$
We can relate these gradings to the Maslov and Alexander gradings in Heegaard Floer homology as follows. Two elements $$x_1 \otimes y_1 , x_2 \otimes y_2 \in \widehat{\text{CFA}}(M, K) \boxtimes \widehat{\text{CFD}}(M', \alpha'_1, \alpha'_2)$$ belong to the same $\text{\emph{Spin}}^{\mathbb{C}}$ structure if and only if there exist integers $h$ and $a$ such that: 
$$gr(x_1 \otimes y_1) = gr(x_2 \otimes y_2) \lambda^{h} \mu^{a}$$
where the equality holds in the double coset space. 
In this case, $h$ is equal to the relative Maslov grading $M(x_1 \otimes y_1) - M(x_2 \otimes y_2)$, and $a$ is equal to the relative Alexander grading $A(x_1 \otimes y_1) - A(x_2 \otimes y_2)$.\\

We get a similar result for the relative Masolv grading  on $$\widehat{\text{CFA}}(M, \alpha_1, \alpha_2) \boxtimes \widehat{\text{CFD}}(M', \alpha'_1, \alpha'_2)$$ by dropping the Alexander factor and the $\mu^{A}$.\\

As mentioned earlier, both bordered invariants are defined up to a specific notion of homotopy equivalence. We will now recall the definition of homotopy equivalence for type A invariants, which will be useful in later sections. See \cite{Lipshitz2008BorderedHF} for information about type D invariants.\\

Let $P$ and $P'$ be right $\A_{\infty}$-modules over $\A$. A homomorphism $f$ of $\A_{\infty}$-modules is a collection of maps
$$f_i : P \otimes \A^{\otimes(i-1)} \rightarrow P'$$
indexed by $i \geq 1$, satisfying the following compatibility conditions for all $n\geq 1$ , $x \in P$ and $a_1 \otimes \cdots \otimes a_{n-1} \in \A^{\otimes(n-1)}$: 
$$\sum_{i+j=n+1} m'_{i}(f_{j}(x \otimes a_1 \otimes \cdots \otimes a_{j-1})\otimes \cdots \otimes a_{n-1}) + $$
$$\sum_{i+j=n+1} f_{i}(m_{j}(x \otimes a_1 \otimes \cdots \otimes a_{j-1})\otimes \cdots \otimes a_{n-1})+$$
$$\sum_{l=1}^{n-2}f_{n-2}(x \otimes a_1 \otimes \cdots \otimes a_{l}a_{l+1} \otimes \cdots \otimes a_{n-1}) = 0,$$
and the unital condition for all $i>1$:
$$f_{i}(x \otimes a_1 \otimes \cdots \otimes \mathds{1} \otimes \cdots a_{i-1})=0.$$
We also refer to such collection of maps as an $\A_{\infty}$ homomorphism.\\

We can also define a composition operation for $\A_{\infty}$ homomorphisms. If $f$ is an $\A_{\infty}$ homomorphism from $P$ to $P'$, and $g$ is an $\A_{\infty}$ homomorphism from $P'$ to $P''$, we can form the composite $g \circ f$, defined by
$$(g \circ f)_{n}(x \otimes a_1 \otimes \cdots \otimes a_{n-1}) := \sum_{i+j=n+1} g_{j}(f_{i}(x \otimes a_1 \otimes \cdots \otimes a_{i-1}) \otimes \cdots \otimes a_{n-1}),$$
for all $n\geq 1$ , $x \in P$ and $a_1 \otimes \cdots \otimes a_{n-1} \in \A^{\otimes(n-1)}$.\\

Next, we define $\A_{\infty}$ homotopy equivalence. To do so, we first introduce the concept of a \emph{null homotopic} $\A_{\infty}$ homomorphism as follows.\\

An $\A_{\infty}$ homomorphism $f : P \rightarrow P'$ is \emph{null homotopic} if there exists a collection of maps 
$$H_i: P \otimes \A^{\otimes(i-1)} \rightarrow P'$$
which satisfies the following two conditions:
\begin{enumerate}[leftmargin=*]
\item For all $n\geq 1$ , $x \in P$ and $a_1 \otimes \cdots \otimes a_{n-1} \in \A^{\otimes(n-1)}$, 
$$f_{n}(x \otimes a_1 \otimes \cdots \otimes a_{n-1}) = \sum_{i+j=n+1} m'_{i}(H_{j}(x \otimes a_1 \otimes \cdots \otimes a_{j-1})\otimes \cdots \otimes a_{n-1}) + $$
$$\sum_{i+j=n+1} H_{i}(m_{j}(x \otimes a_1 \otimes \cdots \otimes a_{j-1})\otimes \cdots \otimes a_{n-1})+$$
$$\sum_{l=1}^{n-2}H_{n-2}(x \otimes a_1 \otimes \cdots \otimes a_{l}a_{l+1} \otimes \cdots \otimes a_{n-1}).$$
\item For all $i>1$,
$$H_i(x \otimes a_1 \otimes \cdots \otimes \mathds{1}\otimes \cdots \otimes a_{i-1})=0$$
\end{enumerate}
Now two $\A_{\infty}$ homomorphisms $f,g : P \rightarrow P'$ are \emph{homotopic} if $f-g$ is null homotopic. An $A_{\infty}$ homomorphism $h : P \rightarrow P'$ is an $\A_{\infty}$ homotopy equivalence if there exists an $A_{\infty}$ homomorphism $h' : P' \rightarrow P$ such that $h \circ h'$ and $h' \circ h$ are homotopic to identity.\\ 

Now let $f : P \rightarrow P'$ be an $\A_{\infty}$ homomorphism, and let $N$ be a bounded type $D$ structure. Then $f$ induces a chain map $f \boxtimes \id_{N} : P \boxtimes N \rightarrow P' \boxtimes N$ defined as follows for any $x \in P, y \in N$:
$$(f \boxtimes \id_{N})(x \otimes y) := \sum_{k=0}^{\infty} (f_{k+1} \otimes \id_{N}) \circ (x \otimes \delta_{k}(y)).$$

Lipshitz, Oszváth and Thurston \cite{Lipshitz2008BorderedHF} proved that if $f$ is an $\A_{\infty}$ homotopy equivalence then $f \boxtimes \id_{N}$ is also a homotopy equivalence.\\

Now that we introduced the notion of $\A_{\infty}$ homotopy equivalence, we can also state the invariance theorem of type A invariants. As we mentioned, definition of type A invariant relies on a choice of a bordered Heegaard diagram. For a bordered Heegaard diagram $\mathcal{H}$, let $\widehat{\text{CFA}}(\mathcal{H})$ denote the associated $\A_{\infty}$-module.

\begin{theo}\cite{Lipshitz2008BorderedHF}\label{Theorem:typeAinvariance}
For any two choices of provincially admissible Heegaard diagrams $\mathcal{H}_1$ and $\mathcal{H}_2$ for a bordered manifold $(M,\alpha_1,\alpha_2)$, and any two choices of reference generators $\bm{z^{1}_0}$ and $\bm{z^{2}_0}$, there is an isomorphism of $G$-sets 
$$\phi : P(\bm{z^{1}_0})\backslash G \rightarrow P(\bm{z^{2}_0}) \backslash G,$$
and an $\A_{\infty}$ homotopy equivalence of $\A_{\infty}$ modules 
$$\widehat{f} : \widehat{\text{CFA}}(\mathcal{H}_1) \rightarrow \widehat{\text{CFA}}(\mathcal{H}_2)$$
$$\widehat{f} = \{\widehat{f}_{i+1}: \widehat{\text{CFA}}(\mathcal{H}_1) \otimes \A^{\otimes i} \rightarrow \widehat{\text{CFA}}(\mathcal{H}_2)\}$$
such that, for $x \in \widehat{\text{CFA}}$ and $a_1\otimes \cdots \otimes a_l \in \A^{\otimes l}$,
$$gr(\widehat{f}_{l+1}(x \otimes a_1 \otimes \cdots \otimes a_l)=\phi(gr(x))gr(a_1)\cdots gr(a_l) \lambda^{l}.$$
\end{theo}

\subsection{Immersed Curve Invariants}\hfill\\

We now turn our attention to the immersed curve invariants and briefly recall some of the basic facts about them.\\ 

Hanselman, Rasmussen and Watson \cite{Hanselman2016BorderedFH} introduced a geometric interpretation of bordered invariants of 3-manifolds with torus boundary in terms of the \emph{immersed curve invariant}. We are going to briefly recall the definition and basic properties of these invariants.\\ 

For a compact oriented 3-manifold with torus boundary $M$ with a basepoint $z \in \partial M$, the \emph{immersed curve invariant} $\widehat{HF}(M)$ is a collection of closed immersed curves $\{\gamma_1, \cdots, \gamma_n\}$ in $\partial M \setminus z$ each decorated with a local system $(k_i, A_i)$ consisting of a finite dimensional vector space over $\mathbb{F}_2$ of dimension $k_i$, and an automorphism represented by a $k\times k$ matrix $A_i$. $\widehat{HF}(M)$ is invariant of $M$ up to regular homotopy of the curves and the isomorphism of the local systems.\\

To construct $\widehat{HF}(M)$, imagine the $\partial M \setminus z$ as the surface $T=\mathbb{R}^2 / \mathbb{Z}^2$ punctured at $z=(1-\epsilon,1-\epsilon)$. The images of the $y$ and $x$-axes in $T$ will be referred to as $\alpha_1$ and $\alpha_2$ respectively. Consider the decorated graph $\Gamma_{D}$ representing the $\widehat{\text{CFD}}(M, \alpha_1, \alpha_2)$, and for now assume that it is reduced. We embed the vertices of $\Gamma_D$ into $T$ so that the $\bu$ vertices are distinct points on $\alpha_1$ and the $\ci$ vertices are distinct points on $\alpha_2$. We then embed each edge in $T$ according to its label, as shown in Figure \ref{Figure:typeDimmeresed}.  In some cases this process gives a collection of immersed curves in $T$ which will be the immersed curve invariant $\widehat{HF}(M)$. In general, the process gives us an immersed train track as seen in example of Figure \ref{Figure:traintrackexample}. Hanselman, Rasmussen and Watson \cite{Hanselman2016BorderedFH} showed that one can pick a particularly nice basis of the type D invariant so that this process ends up in a collection of immersed curves.\\ 

\begin{figure}[h]
\centering
\begin{center}
\includegraphics[scale=0.4]{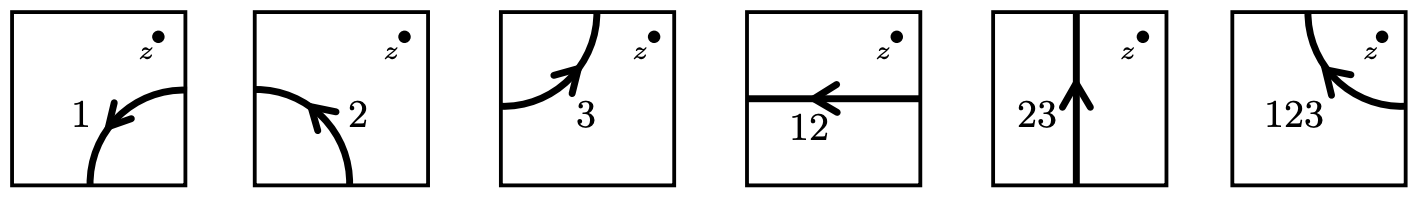}
\end{center}
\caption{Embedding the decorated graph in $T$ \cite{Hanselman2016BorderedFH}}\label{Figure:typeDimmeresed}
\end{figure}

\begin{figure}[h]
\centering
\begin{center}
\includegraphics[scale=0.3]{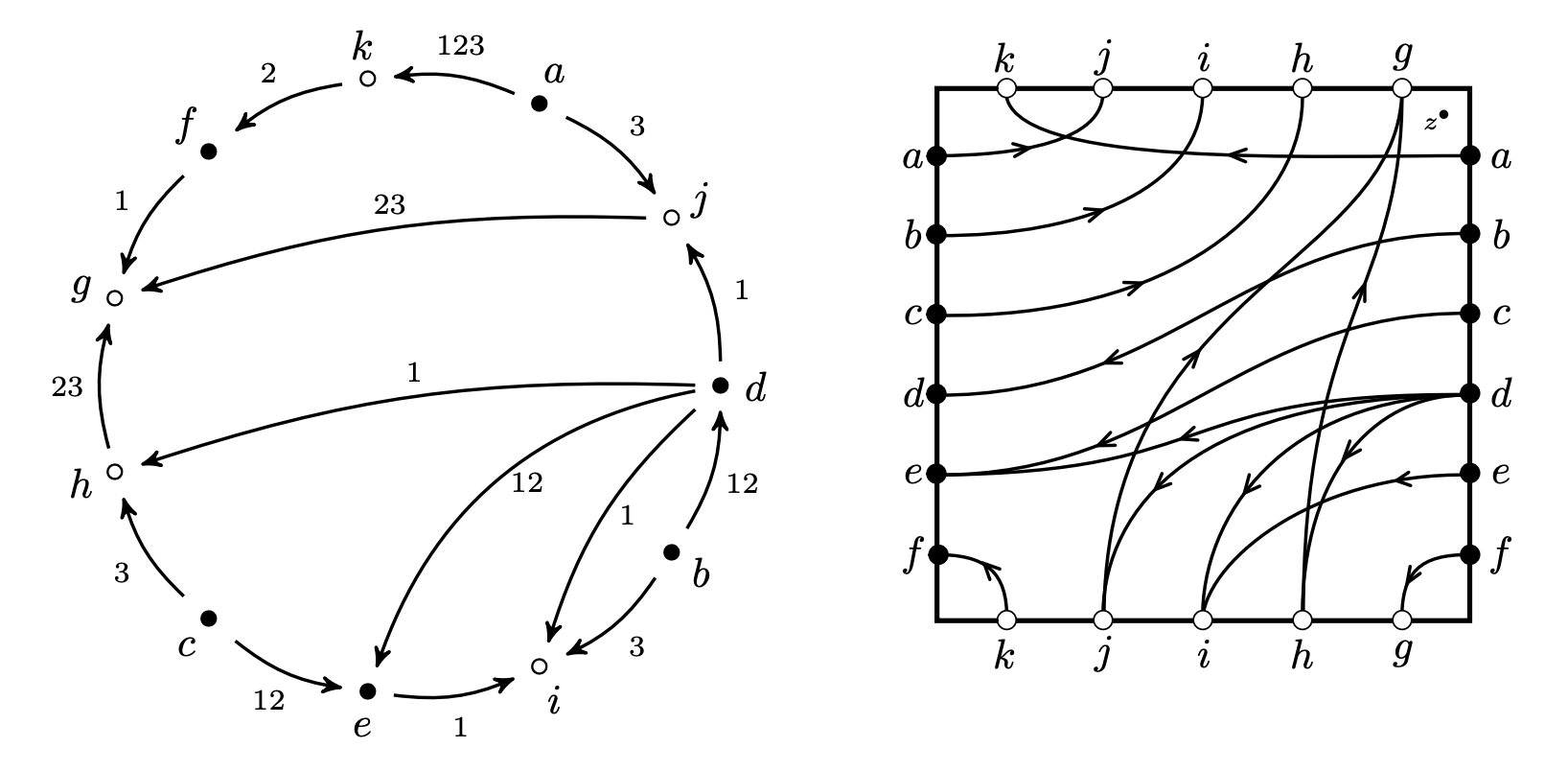}
\end{center}
\caption{Example of embedding a decorated graph as a train track \cite{Hanselman2016BorderedFH}}\label{Figure:traintrackexample}
\end{figure}

Hanselman, Rasmussen and Watson \cite{Hanselman2016BorderedFH} also proved a pairing theorem for the immersed curve invariants which we will recall in Theorem \ref{Theorem:immeresedpairing}. Similar to before, assume that $M$ and $M'$ are 3-manifolds with torus boundary and $h:\partial M' \rightarrow \partial M$ is an orientation reversing homeomorphism. Let $Y := M \cup_{h} M'$. Let $\bar{h}$ denote the composition of $h$ with the elliptic involution of $\partial M$.

\begin{theo}\label{Theorem:immeresedpairing}
Let $\bm{\gamma}, \bm{\gamma'}$ denote $\widehat{HF}(M),\widehat{HF}(M')$ respectively. Then we have, 
$$\widehat{HF}(Y) \cong HF(\bm{\gamma} , \bar{h}(\bm{\gamma'}))$$
where $HF(\cdot,\cdot)$ is an appropriately defined version of the Lagrangian intersection Floer homology in $T$. 
\end{theo}

For the percise definition of the Lagrangian intersection Floer homology for collections of curves with local systems, See \cite{Hanselman2016BorderedFH}. We review some of the propeties of this homology which are going to be useful for our application.\\ 

First, note that the Lagrangian Floer homology splits over the connected components i.e. if $\bm{\gamma} = \bm{\gamma_1} \cup \cdots \cup \bm{\gamma_n}$ and $\bar{h}(\bm{\gamma'})=\bm{\gamma'_1} \cup \cdots \cup \bm{\gamma'_{n'}}$ then 
$$HF(\bm{\gamma} , \bar{h}(\bm{\gamma'})) =  \bigoplus_{\substack{1\leq i \leq n \\ 1\leq j\leq n'}}  HF(\bm{\gamma_i}, \bm{\gamma'_j}).$$

Hence we can focus on connected immersed curves decorated with local systems. We can go one step further and assume that these curves are primitive due to the following lemma. 
\begin{lemm}\label{Lemma:primitivereduce}\cite{Hanselman2016BorderedFH}
Let $\bm{\gamma_1}=(\gamma_1,k_1,A_1)$ where $\gamma_1$ is the $n$-fold cover of $\gamma$. Let $\bm{\gamma}=(\gamma, nk, A)$ where $A$ is a block matrix of the form
\begin{equation*}
\label{eq:monodromy}
\left[\begin{array}{c|c|c|c}
0 & \cdots & 0 & A \\  \hline
I_k & \cdots & 0 & 0 \\  \hline
0 & \ddots & 0 & 0 \\  \hline
0 & \cdots & I_k & 0
\end{array} \right] 
\end{equation*}
Then for any decorated curve $\bm{\gamma'}$ we have
$$CF(\bm{\gamma_1},\bm{\gamma'}) = CF(\bm{\gamma},\bm{\gamma'}).$$ 
\end{lemm}

For the primitive immersed curves, we can relate the dimension of the Lagrangian Floer homology to the intersection numbers. Recall that for essential immersed curves $\gamma$ and $\gamma'$, the \emph{intersection number} $i(\gamma, \gamma')$ is the minimal number of geometric intersection between transverse curves $\widehat{\gamma}, \widehat{\gamma}'$ that are homotopic to $\gamma, \gamma'$.\\

We need to define two technical properties before stating Theorem \ref{Theorem:Lagrangiandimension} about the dimension of Lagrangian Floer homology. For a closed immersed curve $\gamma : S^1 \rightarrow T$, we can consider the map $\gamma \circ p : \mathbb{R} \rightarrow T$ where $p : \mathbb{R} \rightarrow S^1$ is the standard covering map. This map lifts to $\widetilde{\gamma}: \mathbb{R} \rightarrow \widetilde{T}$ where $\widetilde{T} = \mathbb{R}^2 \setminus \mathbb{Z}^2$ is the universal cover of $T$. If $\widetilde{\gamma}$ is an embedding, we say that $\gamma$ is \emph{unobstructed}. This definition is equivalent to is equivalent to $\gamma$ having no immersed fishtails (See \cite{Abouzaid2006OnTF}).\\

Loops $\gamma , \gamma' \subset T$ are \emph{commensurable} if there is an immersed loop $\delta \subset T$ and integers $n,n'$ such that $\gamma$ is freely homotopic to $\delta^n$ and $\gamma'$ is freely homotopic to $\delta^{n'}$. 

\begin{theo}\label{Theorem:Lagrangiandimension}\cite{Hanselman2016BorderedFH}
Suppose $\bm{\gamma}=(\gamma,k,A)$ and $\bm{\gamma'}=(\gamma',k',A')$ are primitive unobstructed immersed curves equipped with local systems. If $\gamma$ and $\gamma'$ are incommensurable, then 
$$\dim \  HF(\bm{\gamma}, \bm{\gamma'})= k k' \cdot i(\gamma, \gamma').$$
\end{theo}

Now to compute $i(\gamma, \gamma')$, we rely on choosing nice representatives $\widehat{\gamma}$ and $\widehat{\gamma}'$ for the homotopy classes of $\gamma$ and $\gamma'$ which realize the minimal geometric intersection. Freedman, Hass and Scott \cite{Freedman1982CLOSEDGO} proved that the shortest length geodesic in the free homotopy classes has this property. Based on this Hanselman, Rasmussen and Watson \cite{Hanselman2016BorderedFH} introduced the \emph{pegboard diagrams} and $\epsilon$-\emph{geodesics}.\\

First, note that it is generally easier to represent a curve $\gamma$ in the punctured torus $T$, by drawing an arc in the cover $\widetilde{T} = \mathbb{R}^2 \setminus \mathbb{Z}^2$ which maps to $\gamma$ with the covering map $\widetilde{T} \rightarrow T$.\\

Intuitively, the pegboard diagrams come from imagining $T$ as a flat Euclidean torus with the puncture imagined as a peg. We place a rubber band in the homotopy class determined by $\gamma$, and see what position it settles into. The cover $\widetilde{T}$ can be imagined as Euclidean $\mathbb{R}^2$ with a lattice of pegs i.e. a peg-board.\\

To give a precise definition of pegboard diagrams, we first start with a geometric model for the punctured torus $T$. Let $\widehat{T}$ be $\partial M$ with a flat metric $g$. Consider a manifold $T_{\epsilon}$ with a complete Riemannian metric $g_{\epsilon}$ defined as follows. $T_{\epsilon}$ is the union of two parts. The first part, denoted by $\widehat{T}_{2\epsilon}$, is the complement of the ball of radius $2 \epsilon$ centered at $z \in \widehat{T}$ equipped with the flat metric $g$. The second part is modeled on a surface of revolution in $\mathbb{R}^{3}$, as illustrated in Figure \ref{Figure:geometricmodel}. Note that $T_{\epsilon}$ is homeomorphic to $T$. Furthermore, $T_{\epsilon}$ embeds in $\widehat{T} \times \mathbb{R}$ and hence we have a projection $p : T_{\epsilon} \rightarrow \widehat{T}_{\epsilon} \subset T$. Note that this projection is equal to identity on $\widehat{T}_{2\epsilon}$. A curve $\gamma \subset T_{\epsilon}$ is called an $\epsilon$-geodesic if it is geodesic for the metric $g_{\epsilon}$. Hanselman, Rasmussen and Watson \cite{Hanselman2016BorderedFH} proved that any nontrivial free homotopy class of loops in $T$ can be represented by and $\epsilon$-geodesic. Furthermore, this representative is either unique or entirely contained in the flat part of $T_{\epsilon}$ i.e. is either boundary parallel or a Euclidean geodesic. As a result, any two distinct $\epsilon$-geodesics intersect minimally and transversally.\\

\begin{figure}[h]
\centering
\begin{center}
\includegraphics[scale=0.3]{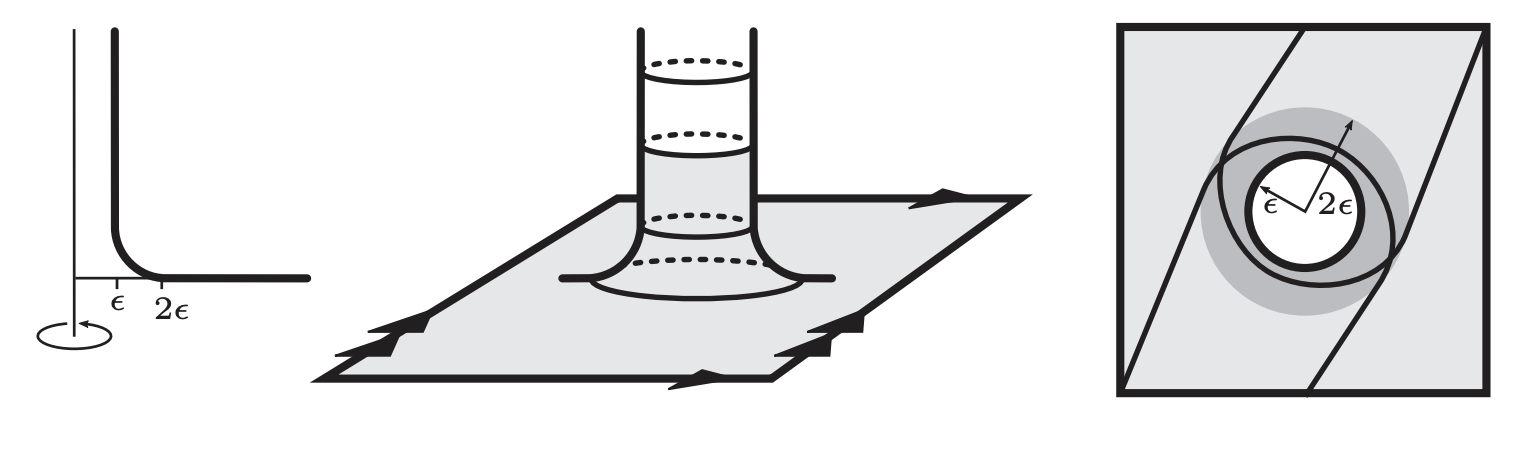}
\end{center}
\caption{The manifold $T_{\epsilon}$ a geometric model for $T$. The figure om the right shows the projection $p:T_{\epsilon} \rightarrow \widehat{T}_{\epsilon}$ \cite{Hanselman2016BorderedFH}}\label{Figure:geometricmodel}
\end{figure}

We can use the $\epsilon$-geodesics to directly compute the intersection numbers, but we can also use them to construct the $\epsilon$-pegboard diagrams. Given an $\epsilon$-geodesic $\gamma \subset T_{\epsilon}$, consider the projection $p(\gamma) \subset T$. This projection consists of two parts. The first part is $p(\gamma) \cap \widehat{T}_{2\epsilon}$ which consists from a family of Euclidean geodesics. The second part is called the \emph{geodesic corners} which describes the behaviour of $p(\gamma)$ around the puncture (i.e. in the Euclidean ball with radius $2\epsilon$ around $z \in \widehat{T}$). We can modify these corners by a homotopy to obtain a curve $\gamma'$ which wraps around a Euclidean circle centered around the puncture, as seen in Figure \ref{Figure:geodesiccorner}. We refer to curve $\gamma'$ as an $\epsilon$-pegboard diagram. For more details see \cite{Hanselman2016BorderedFH}.\\

\begin{figure}[h]
\centering
\begin{center}
\includegraphics[scale=0.25]{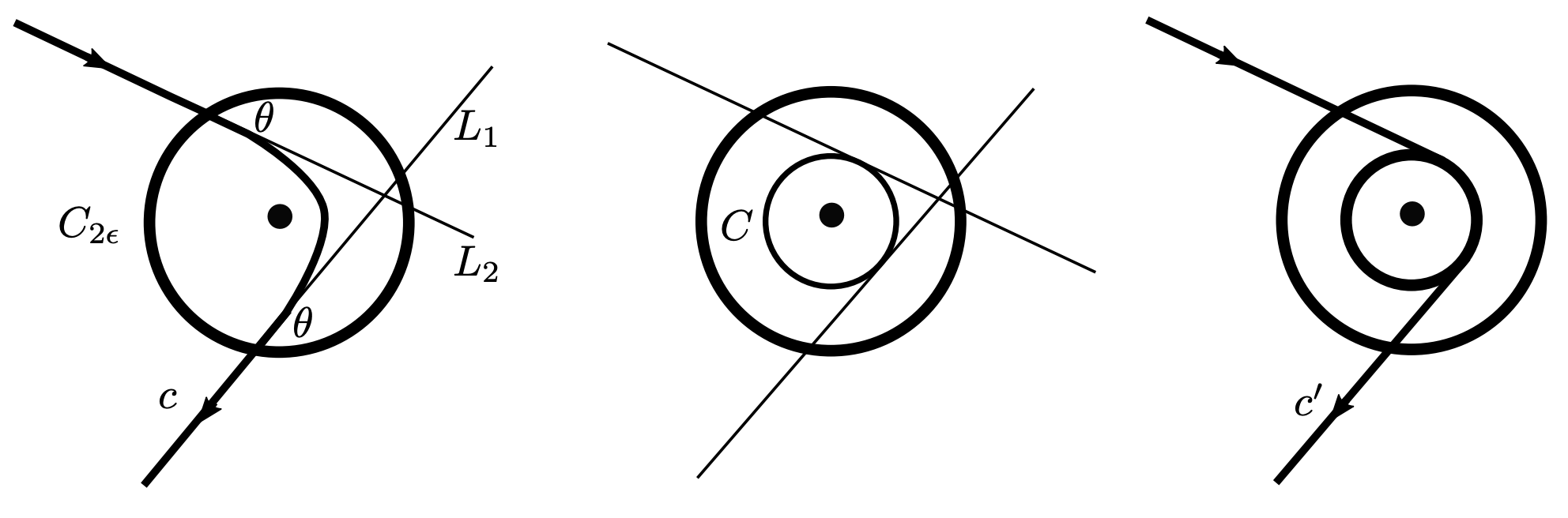}
\end{center}
\caption{Turning a geodesic corner to a pegboard corner \cite{Hanselman2016BorderedFH}}\label{Figure:geodesiccorner}
\end{figure}
 
Hanselman, Rasmussen and Watson \cite{Hanselman2016BorderedFH} also proved that if $\gamma'_1$ and $\gamma'_2$ are $\epsilon$-pegboard diagrams induced by distinct $\epsilon$-geodesics $\gamma_1$ and $\gamma_2$, then they are in minimal position.\\

We can also go one step further and define a \emph{singular pegboard diagram}. The main reference \cite{Hanselman2016BorderedFH} contains a more general definition, but for our purposes, we only need a specific case. First of all, let $\gamma_{\epsilon}$ be the $\epsilon$-pegboard diagram in the homotopy class of $\gamma$. Assume that $\gamma_{\epsilon}$ has at least one corner. We call such curves \emph{tight} and any curve without any corners \emph{loose}. Now we can define $$\gamma_0 := \lim_{\epsilon \rightarrow 0} \gamma_{\epsilon} $$
where the limit is in Hausdorff metric. One can prove that this limit exists. Furthermore, we know that $\gamma_{\epsilon}$ consists of Euclidean geodesics (i.e. straight segments) outside the Euclidean ball $B_{2\epsilon}(z)$. As a result, in the limit $\gamma_0$ will be a collection of Euclidean geodesics in $\widehat{T}$ all passing through $z$. These collection can be recorded using their slopes. Note that $\gamma_0$ loses the information of corners. We can record the homotopy class of the corners e.g. by using the total angle that each corner covers. Let $\bm{c}=(c_1,\cdots,c_n)$ denote the sequence of homotopy classes of the corners. We call $\bar{\gamma} = (\gamma_0 , \bm{c})$ the singular pegboard diagram representing the homotopy class of $\gamma$.\\

Finally we can state a theorem about the computation of the intersection number using the singular pegboard diagrams. 

\begin{theo}\label{Theorem:singularpegboarddiagram}\cite{Hanselman2016BorderedFH}
Let $\bar{\gamma} = (\gamma_0 , \bm{c})$ and $\bar{\gamma}' = (\gamma'_0 , \bm{c'})$ be singular pegboard diagrams representing $\gamma$ and $\gamma'$. Assume that no slope of $\gamma_0$ is also a slope of $\gamma'_0$. Then 
$$i(\gamma, \gamma') = |\gamma_0 \cap \gamma'_0| + \sum_{(c_i,c'_j)} i(c_i,c_j),$$
where $i(c_i,c'_j)$ is a local intersection number determined by the homotopy data at the corners.
\end{theo}

We talked about the immersed curve invariant of a 3-manifold with torus boundary $M$. Similar to the bordered invariants, we also need an immersed curve invariant associated to a knot $K \subset M$. This was first discussed by Chen \cite{Chen2019KnotFH} for $(1,1)$ patterns inside solid torus. Rasmussen introduced the immersed curve invariant $\hfk(P)$ associated to a knot $P \subset S^1 \times D^2$ in his talk in the low-dimensional Workshop in Erdős Center \cite{Rasmusssentalk}. Unfortunately Rasmussen's work on this invariant hasn't been published yet. \\

The construction of $\hfk(P)$ as a collection of arcs and closed curves in ${T = \partial(S^1 \times D^2) \setminus z}$ is an easy extension of the construction of $\widehat{HF}$. As we mentioned, the construction of $\widehat{HF}(M)$ was based on embedding the decorated graph representing $\widehat{\text{CFD}}(M, \alpha_1, \alpha_2)$ in the punctured torus. The type D structures of 3-manifolds has a property called \emph{extendability}. This property is the main reason that the invariant $\widehat{HF}(M)$ consists of a collection of closed curves. $\widehat{\text{CFD}}(M, K)$ is not necessarily extendable, and as a result, after the embedding of the decorated graph we might have components that are immersed segments $I : [0,1] \looparrowright T$. To get to the immersed curve invariant $\hfk(P)$, we turn the immersed segments to immersed arcs $\widetilde{I}: \mathbb{R} \looparrowright Y$ by connecting the endpoints to the puncture in $T$. This process gives us the non-compact components of $\hfk(P)$.\\

Note that similar to the case of $\widehat{HF}$, embedding of a decorated graph generally doesn't give us an immersed $1$-manifold (with boundary), and one needs to pick a specially nice basis.\\

We complete this section with giving a couple of examples of the type D, type A and immersed curve invariants. In these examples, we focus on the genus-one bordered Heegaard diagrams, same as \cite{Chen2019KnotFH}. In this case the count of $J$-holomorphic curves simplifies to a combinatorial count of bigons.\\

We need to first define a genus-one bordered Heegaard diagram. We start with the notations used to define the type A invariant. Cosider a $4$-tuple $(\Sigma, \bm{\alpha}, \beta, z)$ such that:
\begin{itemize}
    \item $\Sigma$ is a compact, oriented surface of genus one with a single boundary component,
    \item  $\bm{\alpha}$ consists of a pair of arcs $(\alpha_1,\alpha_2)$ properly embedded in $\Sigma$, such that $\alpha_1 \cap \alpha_2 = \varnothing$ and the ends of $\alpha_1$ and $\alpha_2$ appears alternatively on $\partial \Sigma$,
    \item $\beta$ is an embedded closed curve in the interior of $\Sigma$ such that $\Sigma \setminus \beta$ is connected, and $\beta$ intersects $\bm{\alpha}$ transversely,
    \item A basepoint $z$ on $\partial \Sigma \setminus \partial \bm{\alpha}$,
    \item Label the arcs on $\partial \Sigma$ so that $(\partial \Sigma, \bm{\alpha},z)$ is as shown in Figure \ref{Figure:typeAtorus}. We use the symbol $I \in \{12,23,123\}$ to denote the arc obtained by concatenation of the arcs labeled by $1,2,3$ accordingly.
\end{itemize}
\begin{figure}[h]
\centering
\begin{center}
\includegraphics[scale=0.25]{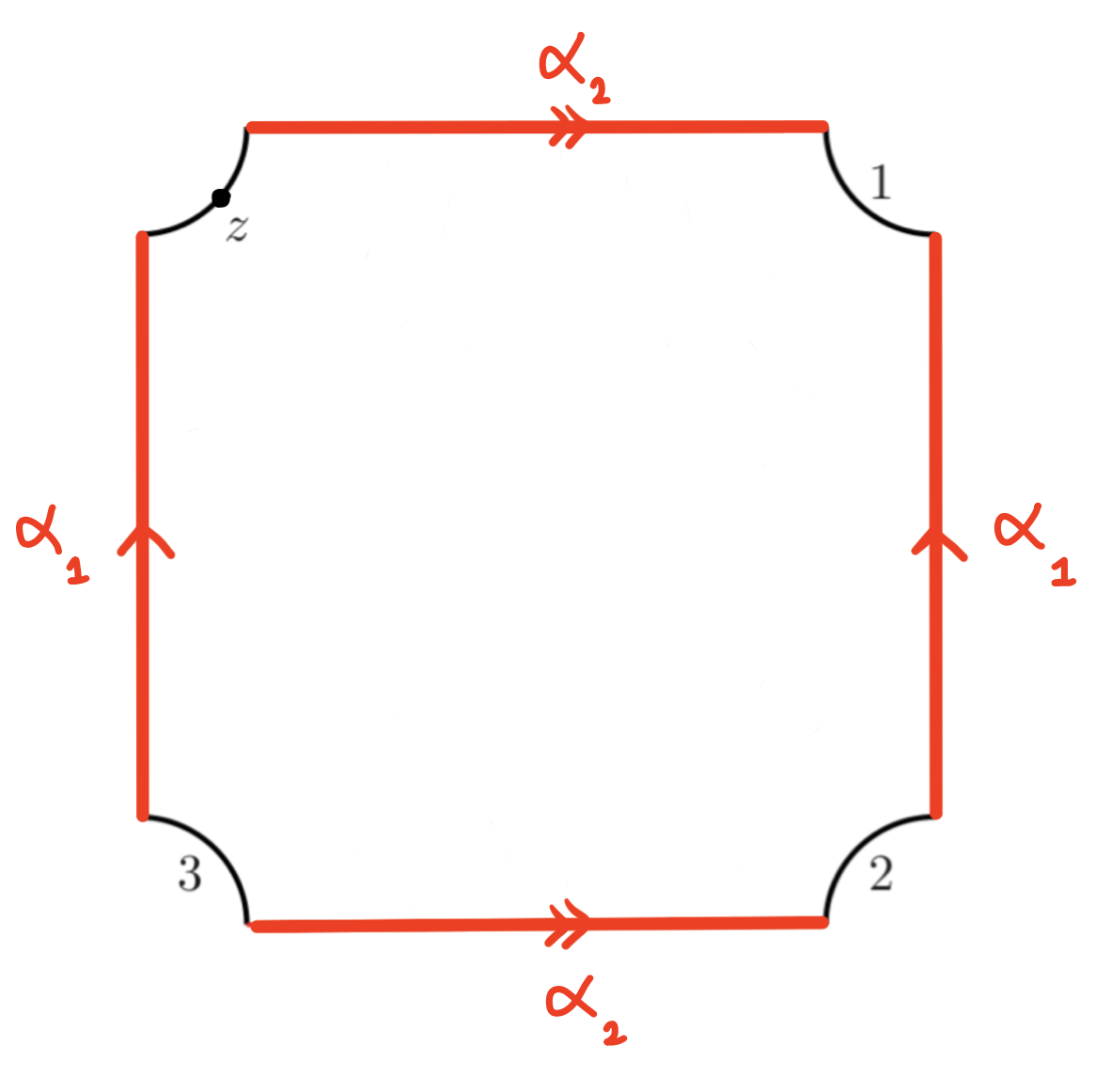}
\end{center}
\caption{Labeling arcs on the boundary of the genus-one bordered Heedaard diagram used to define type A invariants}\label{Figure:typeAtorus}
\end{figure}

Such a $4$-tuple $(\Sigma, \bm{\alpha}, \beta, z)$ specifies a bordered $3$-manifold $(M, \alpha_1, \alpha_2)$. The $3$-manifold $M$ is constructed by attaching a $2$-handle to $\Sigma \times [0,1]$ along $\beta \times \{1\}$. To be more precise, one also needs to collapse the disk bounded by $\partial \Sigma \times \{0\}$ to get the basepoint on the boundary of $M$ and to turn $\alpha_1$ and $\alpha_2$ to closed curves in $\partial M$.\\

We call a $4$-tuple $(\Sigma, \bm{\alpha}, \beta, z)$ with the aforementioned properties a \emph{pointed genus-one bordered Heegaard diagram}. We can also define a \emph{doubly-pointed genus-one bordered Heegaard diagram} by adding another basepoint $w \in \interior{\Sigma} \setminus (\bm{\alpha} \cup \beta)$. The doubly-pointed diagram also specifies a knot $K \subset M$. The knot $K$ can be seen as a union of two arcs. Connect the basepoints $z$ and $w$ by an arc $\gamma \subset \Sigma \setminus \bm{\alpha}$ and an arc $\eta \subset \Sigma \setminus \beta$. Viewing $\Sigma$ as $\Sigma \times \{\frac{1}{2}\} \subset \Sigma \times [0,1]$, let $\gamma'$ be the result of pushing the interior of $\gamma$ slightly into $\Sigma \times [0,\frac{1}{2})$ and let $\eta'$ be the result of pushing the interior of $\eta$ slightly into $\Sigma \times (\frac{1}{2},1]$. We define $K$ as $\gamma' \cup \eta'$, oriented in a way that it intersects $\Sigma \times \{\frac{1}{2}\}$ positively in $z$ and negatively in $w$.\\

\noindent
Now we can express the type A invariants $\widehat{\text{CFA}}(M, \alpha_1, \alpha_2)$ and $\widehat{\text{CFA}}(M, K)$:\\

\begin{enumerate}[leftmargin=*]
\item $\widehat{\text{CFA}}(M, \alpha_1, \alpha_2)$ is generated by the set $\mathcal{G}=\{x| \ x \in \beta \cap \bm{\alpha}\}$ as a $\mathbb{F}_2$ vector space.
\item The $\I$-module structure is as follows: 
$$\widehat{\text{CFA}}(M, \alpha_1, \alpha_2) \cdot \iota_0 = \langle x| \ x \in \alpha_1 \cap \beta \rangle$$
$$\widehat{\text{CFA}}(M, \alpha_1, \alpha_2) \cdot \iota_1 = \langle x| \ x \in \alpha_2 \cap \beta \rangle$$
\item View $\Sigma$ as $(S^1 \times S^1) \setminus \interior{B}$, where $B$ is a disk. Let $\widetilde{\Sigma}$ be the covering space of $\Sigma$ obtained from taking the universal covering $\mathbb{R}^2 \rightarrow (S^1 \times S^1)$ and removing the preimage of $B$ (i.e. $\mathbb{R}^2$ with a lattice of disks removed). The maps $m_{n+1}$ are as follows:
$$m_{n+1}: \widehat{\text{CFA}}(M, \alpha_1, \alpha_2) \otimes \A^{\otimes n} \rightarrow \widehat{\text{CFA}}(M, \alpha_1, \alpha_2)$$
$$m_{n+1}(x, \rho_{i_1}, \cdots, \rho_{i_n}) = \sum_{y \in \mathcal{G}} \# \mathcal{M}(x,y; \rho_{i_1}, \cdots, \rho_{i_n}) \ y,$$
where $i_j \in \{1,2,3,12,23,123\}$ for $j=1, \cdots, n$, and $\# \mathcal{M}(x,y; \rho_{i_1}, \cdots, \rho_{i_n})$ is the (modulo $2$) count of index $1$ embedded disks in $\widetilde{\Sigma}$, such that when we traverse the boundary of such a disk with the induced orientation, we start from a lift of $x$, denoted by $\widetilde{x}$; then we alternate between arcs on (some lift of) $\bm{\alpha}$ and arcs $i_j$ on (some lift of) $\partial B$ for $j=1, \cdots, n$ which end up in $\widetilde{y}$ which is a lift of $y$, and finally we have an arc on some lift of $\beta$ joining $\widetilde{y}$ to the starting point $\widetilde{x}$.
\end{enumerate}
Note that index here is the expected dimension of moduli space $ \mathcal{M}(x,y; \rho_{i_1}, \cdots, \rho_{i_n})$ plus one. $ \mathcal{M}(x,y; \rho_{i_1}, \cdots, \rho_{i_n})$ is the space of holomorphic maps from $D^2$ to $\widetilde{\Sigma}$ with the boundary conditions described above. There is a combinatorial formula in the index described in \cite[Section 5.7]{Lipshitz2008BorderedHF}. We can use the Riemann mapping theorem in certain cases to determine the index as well.\\

Note that by definition the disks contained in the moduli space $ \mathcal{M}(x,y; \rho_{i_1}, \cdots, \rho_{i_n})$ doesn't contain any of the lifts of the base point $z$.\\

To define $\widehat{\text{CFA}}(M, K)$ we only need to change the moduli space $\mathcal{M}(x,y; \rho_{i_1}, \cdots, \rho_{i_n})$ to its subspace $\mathcal{M}_{w}(x,y; \rho_{i_1}, \cdots, \rho_{i_n})$ which only contains disks which doesn't contain any of the lifts of the base point $w$.\\

We defined the type A invariant of the genus-one bordered Heegaard diagrams. Definition of the type D invariant is very similar. First of all, we need to change the labelling of the arcs as done in Figure \ref{Figure:typeDtorus}.\\

The type D invariant $\widehat{\text{CFD}}(M, \alpha_1, \alpha_2)$ is also generated by $\mathcal{G}$ as a $\mathbb{F}_2$ vector space. The $\I$-module structure is defined as follows: 
$$\widehat{\text{CFD}}(M, \alpha_1, \alpha_2) \cdot \iota_0 = \langle x| \ x \in \alpha_2 \cap \beta \rangle,$$
$$\widehat{\text{CFD}}(M, \alpha_1, \alpha_2) \cdot \iota_1 = \langle x| \ x \in \alpha_1 \cap \beta \rangle.$$
The map $\delta$ is as follows:
$$\delta : \widehat{\text{CFD}}(M, \alpha_1, \alpha_2) \rightarrow \A \otimes \widehat{\text{CFD}}(M, \alpha_1, \alpha_2),$$
$$\delta(x) = \sum_{y \in \mathcal{G}} \# \mathcal{M}(x,y; \rho_{i_1}, \cdots, \rho_{i_n})(\rho_{i_1}\cdots\rho_{i_n} \otimes y).$$
\begin{figure}[h]
\centering
\begin{center}
\includegraphics[scale=0.3]{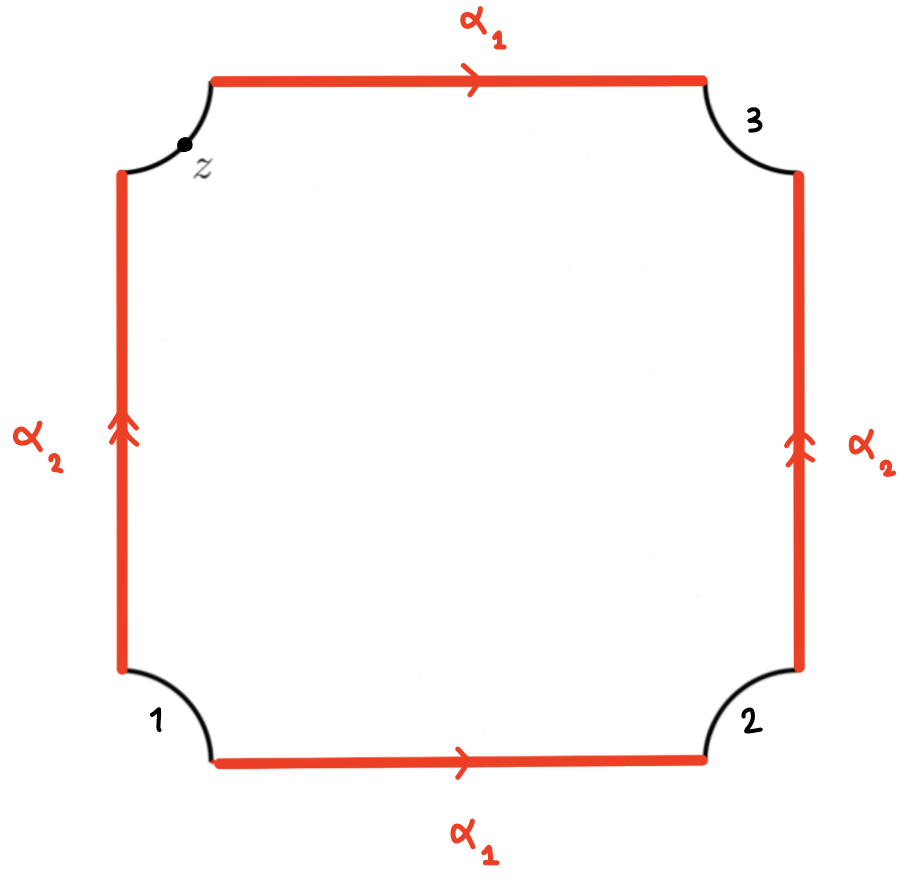}
\end{center}
\caption{Labeling arcs on the boundary of the genus-one bordered Heedaard diagram used to define type D invariants}\label{Figure:typeDtorus}
\end{figure}

Same as the case of type A invariant, we only need to replace $\mathcal{M}(x,y; \rho_{i_1}, \cdots, \rho_{i_n})$ with $\mathcal{M}_{w}(x,y; \rho_{i_1}, \cdots, \rho_{i_n})$ to get the definition of $\widehat{\text{CFD}}(M,K).$\\

Now we are ready to compute some examples.\\

\begin{exem}\label{Example:Mazur}
We start by computing the type A invariant $\widehat{\text{CFA}}(S^1 \times D^2 , Q)$ where $Q$ is the Mazur pattern, as depicted in Figure \ref{Figure:Mazurpattern}. This example was also discussed by Petkova and Wong \cite{Petkova2020TwistedMP}, and Levine \cite{Levine2014NONSURJECTIVESO}. A doubly-pointed genus-one bordered Heegaard diagram $(\mathcal{H}_{Q}, z, w)$ for $(S^1 \times D^2 , Q)$ can be seen in Figure \ref{Figure:MazurHeegaard}.\\

All of the index one disks counted in the definition of the type A invariant can be seen in Figure \ref{Figure:Mazurbigons}. This means that $\widehat{\text{CFA}}(\mathcal{H}_{Q}, z, w)$ can be described as follows : 
$$\widehat{\text{CFA}}(\mathcal{H}_{Q}, z, w) = \langle x_0, \cdots, x_6,y_1, \cdots,y_6 \rangle_{\mathbb{F}_2},$$
$$\widehat{\text{CFA}}(\mathcal{H}_{Q}, z, w) \cdot \iota_{0} = \langle x_0, y_2, y_4, x_4, x_2 \rangle,$$
$$\widehat{\text{CFA}}(\mathcal{H}_{Q}, z, w) \cdot \iota_{1} = \langle x_5, x_6, y_6, y_5, y_1, y_3, x_3, x_1 \rangle.$$
And all the nontrivial values of $m_{n+1}$ are the following: 
$$m_2(x_2 \otimes \rho_1)= x_1 \ , \ m_2(x_4\otimes \rho_1)=x_3 \ , \ m_2(y_4, \otimes \rho_1)=y_3$$
$$m_2(x_1 \otimes \rho_2) = x_0 \ , \ m_2(x_3 \otimes \rho_2)=y_2 \ , \ m_2(x_2 \otimes \rho_{12})=x_0 \ , \ m_2(x_4 \otimes \rho_{12})=y_2$$
$$m_3(y_3 \otimes \rho_2 \otimes \rho_1) = y_1 \ , \ m_4(y_4 \otimes \rho_1 \otimes \rho_2 \otimes \rho_1)=y_1.$$
A graphical representation of $\widehat{\text{CFA}}(\mathcal{H}_{Q}, z, w)$ (with conventions of Hedden and Levine \cite{Hedden2012SplicingKC}) can be seen in Figure \ref{Figure:MazurAgraph}.\\

Now we can turn our attention to the gradings. To determine the gradings, we first need to look at the type $A$ invariant $\widehat{\text{CFA}}(\mathcal{H}_Q,z)$. To do this we need to consider the index one embedded disks which include the lifts of $w$ as well. These disks can be seen in Figure \ref{Figure:MazurUbigons}. By adding these disks to the count, we get the type A invariant described (as a decorated graph) in Figure \ref{Figure:MazurUAgraph}.\\

Note that although the convention of Hedden and Levine \cite{Hedden2012SplicingKC} is only described for reduced decorated graphs, it can be easily extended to non-reduced graphs. For unreduced graphs, the maps $m_{n+1}$ for $n \geq 1$ are still defined by directed paths in the graph which doesn't contain an edge with label $\varnothing$. The map $m_1$ is defined using the edges with label $\varnothing$.\\

\begin{figure}[h]
\centering
\begin{center}
\includegraphics[scale=0.3]{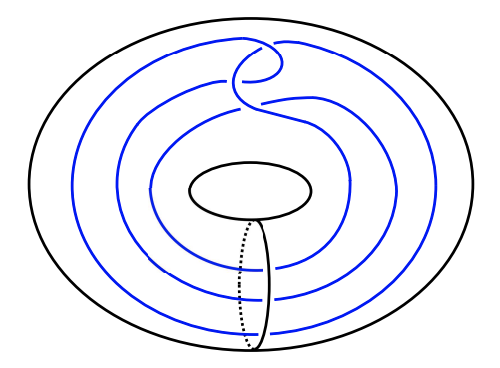}
\end{center}
\caption{Mazur pattern in solid torus \cite{Petkova2020TwistedMP}}\label{Figure:Mazurpattern}
\end{figure}

\begin{figure}[h]
\centering
\begin{center}
\includegraphics[scale=0.25]{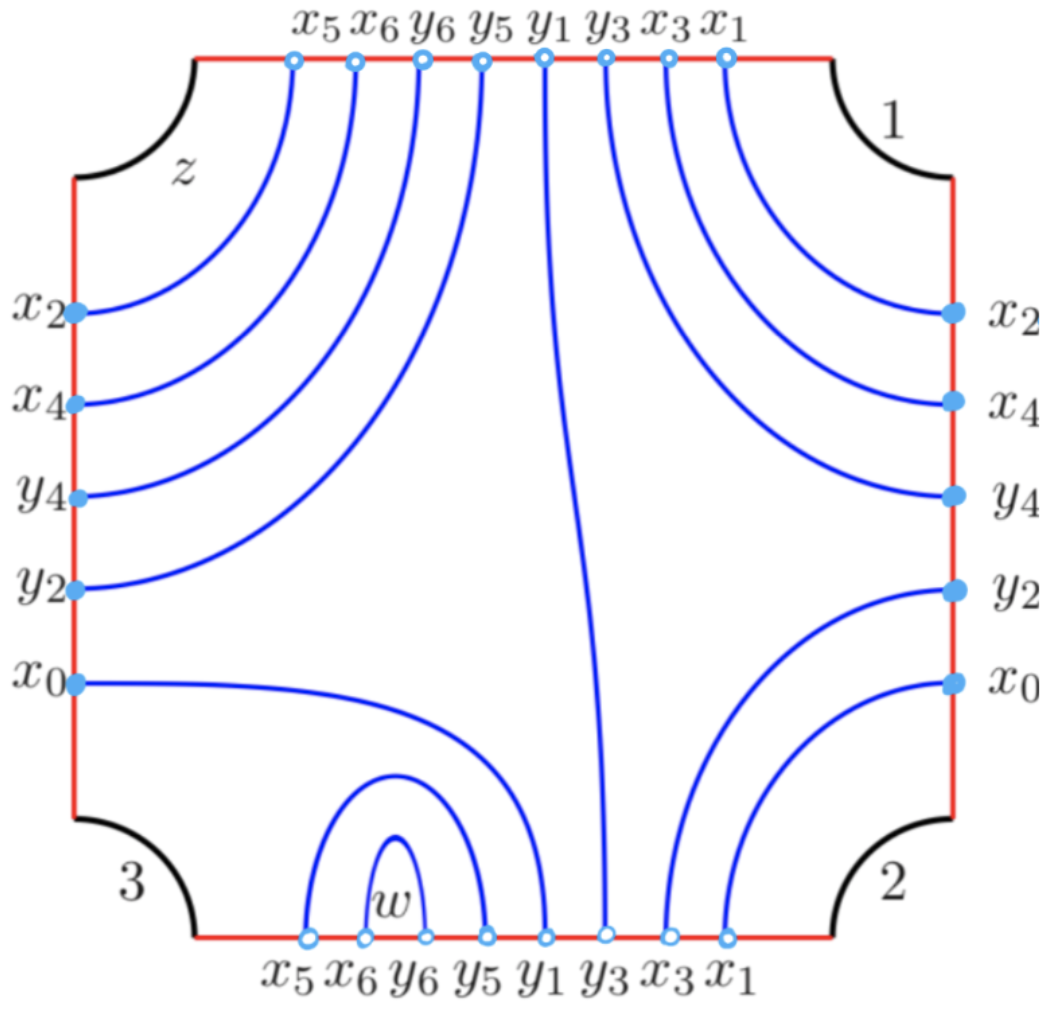}
\end{center}
\caption{Doubly-pointed genus-one bordered Heegaard diagram associated to the Mazur pattern}\label{Figure:MazurHeegaard}
\end{figure}

\begin{figure}[h]
\centering
\begin{center}
\includegraphics[scale=0.33]{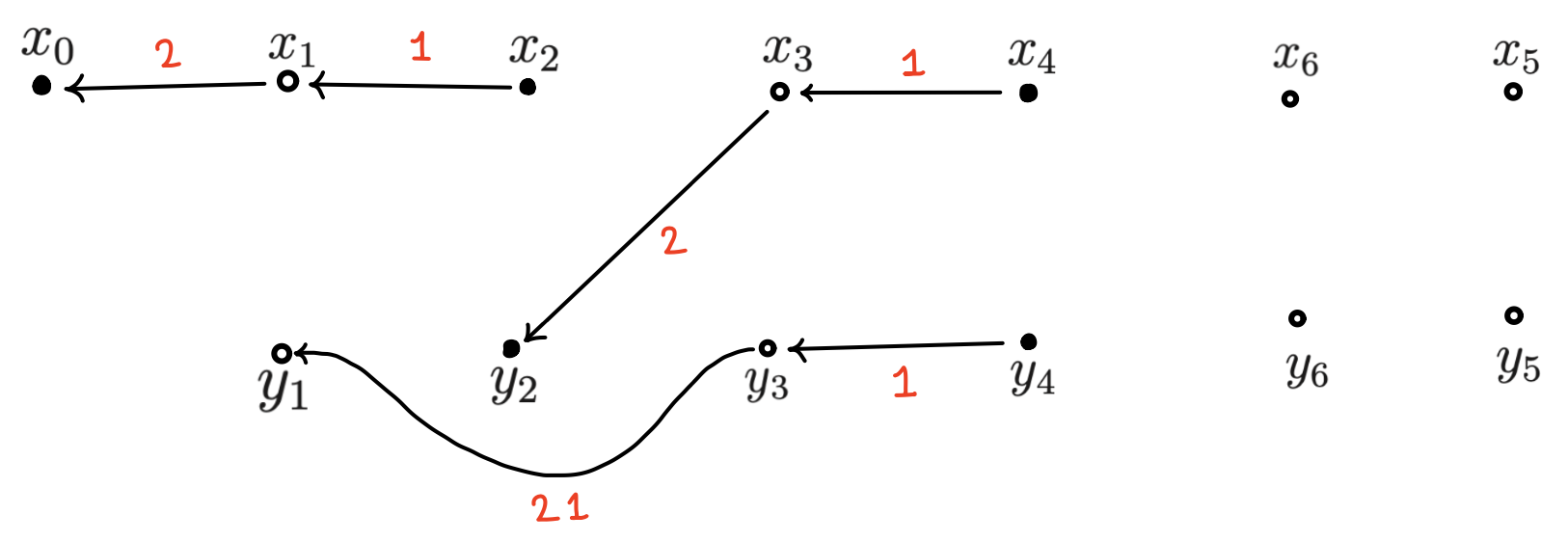}
\end{center}
\caption{Graphical representation of $\widehat{\text{CFA}}(\mathcal{H}_{Q}, z, w)$}\label{Figure:MazurAgraph}
\end{figure}

\begin{figure}[h]
\centering
\begin{center}
\includegraphics[scale=0.33]{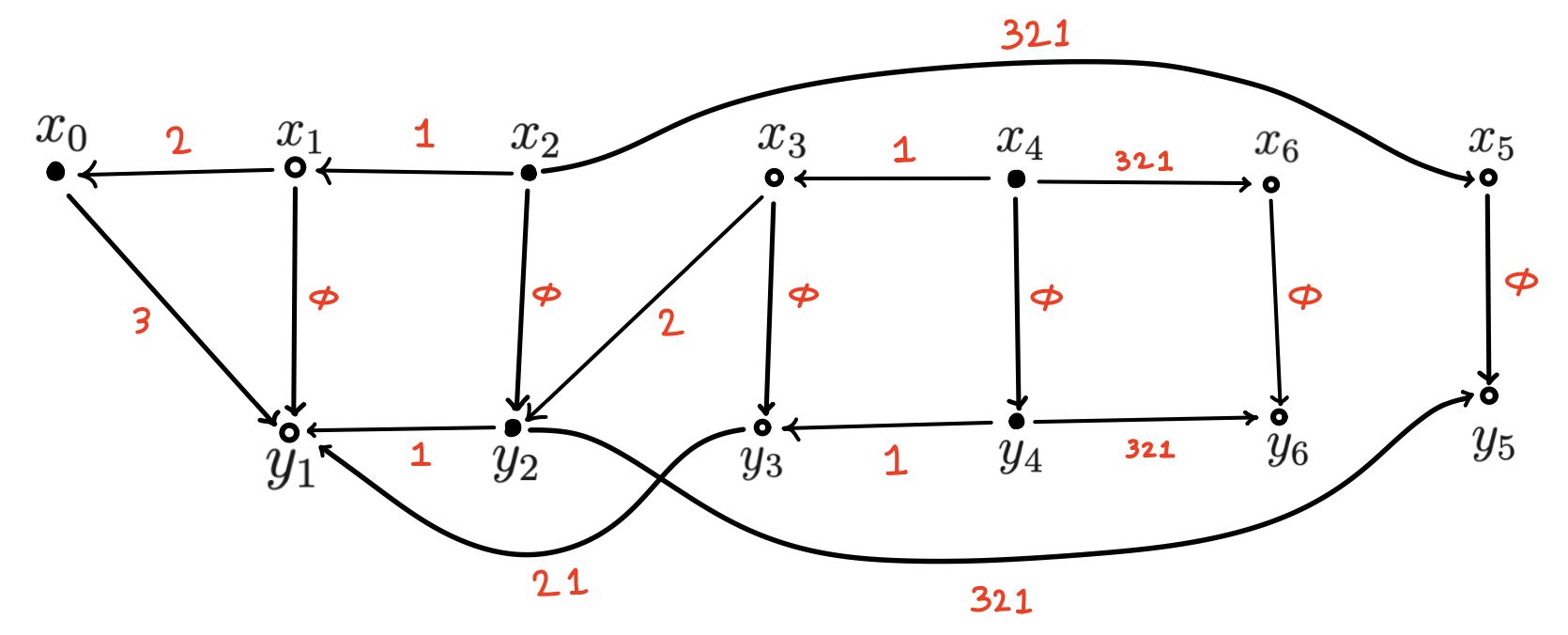}
\end{center}
\caption{Graphical representation of $\widehat{\text{CFA}}(\mathcal{H}_{Q}, z)$}\label{Figure:MazurUAgraph}
\end{figure}

\begin{figure}[h]
\centering
\begin{center}
\includegraphics[width=\textwidth]{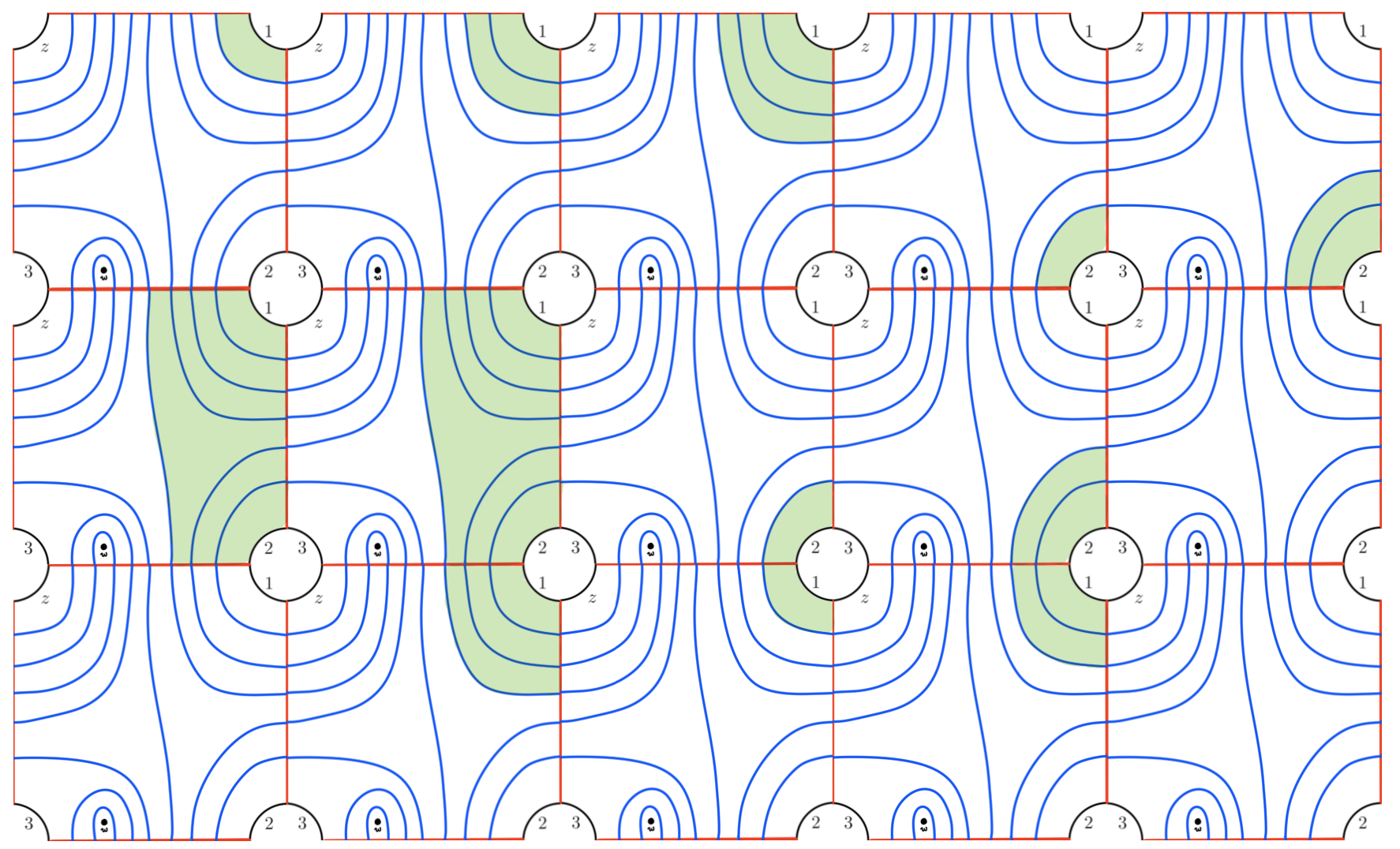}
\end{center}
\caption{Index one embedded disks in $\widetilde{\Sigma}$ which doesn't include lifts of $z$ and $w$}\label{Figure:Mazurbigons}
\end{figure}

Now in order to find the grading structure of $\widehat{\text{CFA}}(\mathcal{H}_{Q}, z, w)$, we first need to pick an arbitrary reference generator. Set $gr(y_4)$ to be $(0;0,0;0)$. The grading map will be : 
$$gr : \widehat{\text{CFA}}(\mathcal{H}_{Q}, z, w) \rightarrow P(y_4) \backslash \widetilde{G}$$
Using the rules described in Figure \ref{Figure:typeAgrading} we have: 
$$gr(y_4) = gr(y_3) \cdot (\frac{1}{2};-\frac{1}{2},\frac{1}{2};0) \Rightarrow gr(y_3)= (-\frac{1}{2};\frac{1}{2},-\frac{1}{2};0),$$
$$gr(y_3)=gr(y_1) \cdot (\frac{1}{2};-1,0;0) \Rightarrow gr(y_1)= (-\frac{1}{2}; \frac{3}{2},-\frac{1}{2};0),$$
$$gr(x_0)=gr(y_1) \cdot (\frac{1}{2};\frac{1}{2},-\frac{1}{2};1) \Rightarrow gr(x_0) =(- \frac{1}{2}; 2,-1;1),$$
$$gr(x_1) = gr(x_0) \cdot (\frac{1}{2}; -\frac{1}{2}, \frac{1}{2}; 0) \Rightarrow gr(x_1)=(-\frac{3}{2}; \frac{3}{2}, -\frac{3}{2}; 1),$$
$$gr(x_2) = gr(x_1) \cdot (\frac{1}{2}; -\frac{1}{2}, \frac{1}{2};0) \Rightarrow gr(x_2) = (-1;1,-1;1),$$
$$gr(x_2)=gr(x_5) \cdot (\frac{1}{2}; -\frac{1}{2}, -\frac{1}{2}; 1) \Rightarrow gr(x_5)=(-\frac{1}{2}; \frac{3}{2}, -\frac{1}{2}; 0),$$
$$gr(y_5)=\lambda^{-1} \mu^{-1} gr(x_5) \Rightarrow gr(y_5)=(-\frac{3}{2};\frac{3}{2},-\frac{1}{2}; -1),$$
$$gr(y_2)=gr(y_5) \cdot (\frac{1}{2}; -\frac{1}{2}, -\frac{1}{2}; 1) \Rightarrow gr(y_2) = (-2; 1,-1;0),$$
$$gr(x_3)=gr(y_2) \cdot (\frac{1}{2}; - \frac{1}{2}, - \frac{1}{2};0) \Rightarrow gr(x_3) = (- \frac{5}{2}; \frac{1}{2}; - \frac{3}{2}; 0),$$
$$gr(x_4)= gr(x_3) \cdot (\frac{1}{2},-\frac{1}{2}, \frac{1}{2};0) \Rightarrow gr(x_4) = (-\frac{5}{2}; 0,-1;0),$$
$$gr(x_4)=gr(x_6) \cdot (\frac{1}{2}; -\frac{1}{2}, -\frac{1}{2};1) \Rightarrow gr(x_6)=(-\frac{5}{2}; \frac{1}{2}, -\frac{1}{2}; -1),$$
$$gr(y_6)=\lambda^{-1} \mu^{-1} gr(x_6) \Rightarrow gr(y_6) = (-\frac{7}{2}; \frac{1}{2}; -\frac{1}{2}; -2).$$

\begin{figure}[h]
\centering
\begin{center}
\includegraphics[width=\textwidth]{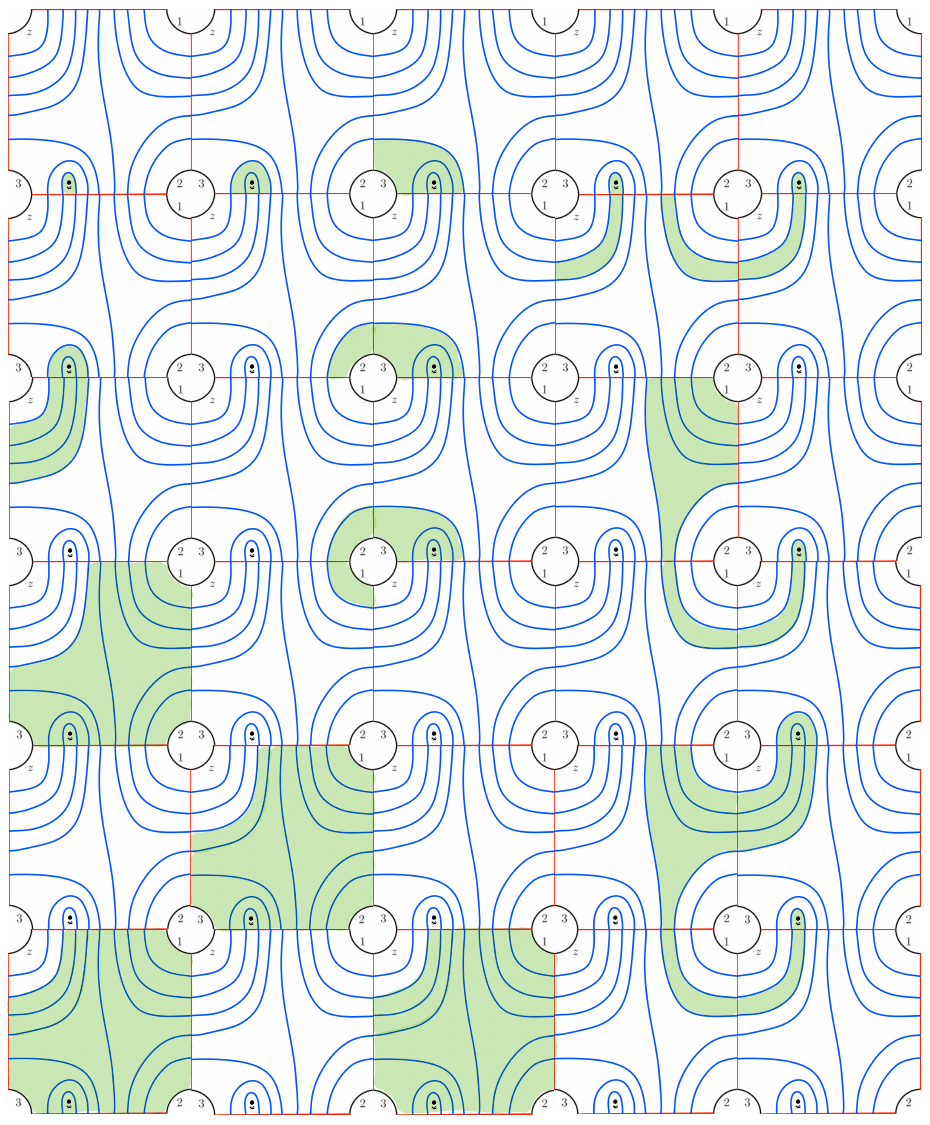}
\end{center}
\caption{Index one embedded disks in $\widetilde{\Sigma}$ which doesn't include lifts of $z$, but include a lift of $w$}\label{Figure:MazurUbigons}
\end{figure}

As we mentioned $P(y_4)$ comes from the periodic domains in the diagram $\mathcal{H}_{Q}$. In our case, since the bordered 3-manifold described by the diagram is a solid torus, the subgroup $P(y_4)$ is cyclic. Now again using the rules of Figure \ref{Figure:typeAgrading}, we have: 
$$gr(y_4)=\lambda^{-1}\mu^{-1}gr(x_4) \Rightarrow gr(y_4)=(-\frac{7}{2}; 0, -1; -1).$$
As a result, $P(y_4)$ is generated by $(-\frac{7}{2}; 0, -1; -1)$. This completes the computation of the grading map.\\

Now we finally discuss the computation of the immersed curve invariant $\hfk(Q)$. Note that a graphical representation of $\widehat{\text{CFD}}(\mathcal{H}_{Q},z,w)$ can be constructed by exchanging $1$ and $3$ labels in Figure \ref{Figure:MazurAgraph}. We need to embed this decorated graph in punctured torus $T$ based on rules described in Figure \ref{Figure:typeDimmeresed}.
The result of this embedding can be seen in Figure \ref{Figure:Mazurimmeresedgraph}. This gives us a family of segments $I_1, \cdots, I_7 \subset T$. Note that $I_4, \cdots,I_7$ are constant segments with image $x_6, x_6, y_5, y_6$. To construct $\hfk(Q)$, we then need to connect the endpoints of $I_1, \cdots, I_7$ to the puncture to turn them to immersed arcs. The result which is the $\hfk(Q)$ can be seen in Figure \ref{Figure:Mazurimmeresedcurve}. We can also consider the lift of $\hfk(Q)$ to $\widetilde{T}$. However, one must keep in mind that in this lift, the relative positions of the lifted components carry essential grading data. A peg-board diagram illustrating the components of $\hfk(Q)$ in $\widetilde{T}$ shown in Figure \ref{Figure:Mazurpegboard}. This diagram should be interpreted with care, as it omits the relative positioning of components and presents them only up to translation.\\
\end{exem}

\begin{rema}
    Note that the pointed Heegaard diagram $(\mathcal{H}_{Q},z)$ descibes a solid torus. This can be seen by using an isotopy on the $\beta$ curve which turns the diagram to the one described in Figure \ref{Figure:Mazurisotoped}, denoted by $(\mathcal{H}_{\infty},z)$. One can then easily compute $\widehat{\text{CFA}}(\mathcal{H}_{\infty},z)$ as follows: 
    $$\widehat{\text{CFA}}(\mathcal{H}_{\infty},z) = \widehat{\text{CFA}}(\mathcal{H}_{\infty},z) \cdot \iota_{0} = \langle x_0 \rangle$$
    $$\forall n \geq 0 \  \ m_{n+3}(x_0 \otimes\rho_{3} \otimes\underbrace{\rho_{23} \otimes \cdots \otimes \rho_{23}}_{n} \otimes \rho_{2})=x_0.$$
    The grading structure is also as follows: 
    $$gr : \widehat{\text{CFA}}(\mathcal{H}_{\infty},z) \rightarrow \langle (-\frac{1}{2};0,1) \rangle \backslash G$$
    $$gr(x_0)=(0; 0,0).$$
    Based on Theorem \ref{Theorem:typeAinvariance}, there exists a $\A_{\infty}$ homotopy equivalence $$f: \widehat{\text{CFA}}(\mathcal{H}_{\infty},z) \rightarrow \widehat{\text{CFA}}(\mathcal{H}_{Q},z).$$
\end{rema}

\begin{figure}[h]
\centering
\begin{center}
\includegraphics[scale=0.37 ]{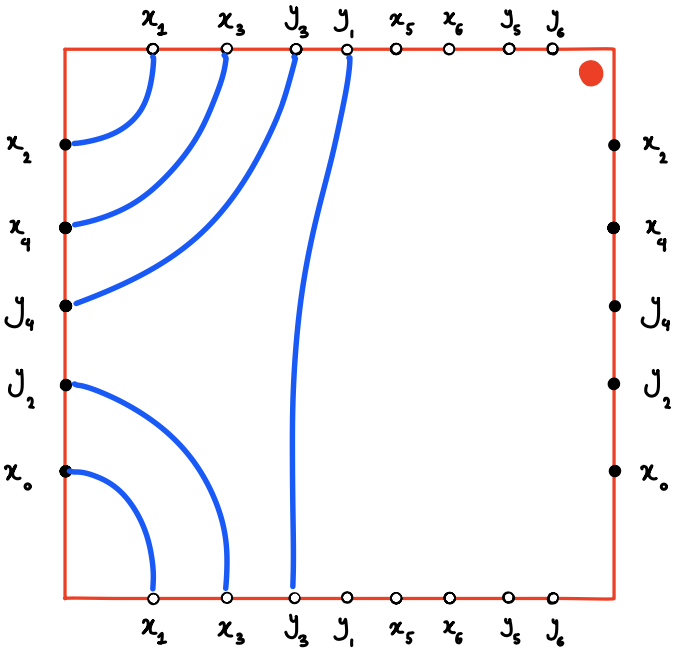}
\end{center}
\caption{The embedding of a decorated graph representing $\widehat{\text{CFD}}(\mathcal{H}_{Q}, z,w)$ in puntured torus $T$}\label{Figure:Mazurimmeresedgraph}
\end{figure}

\begin{figure}[h]
\centering
\begin{center}
\includegraphics[scale=0.3]{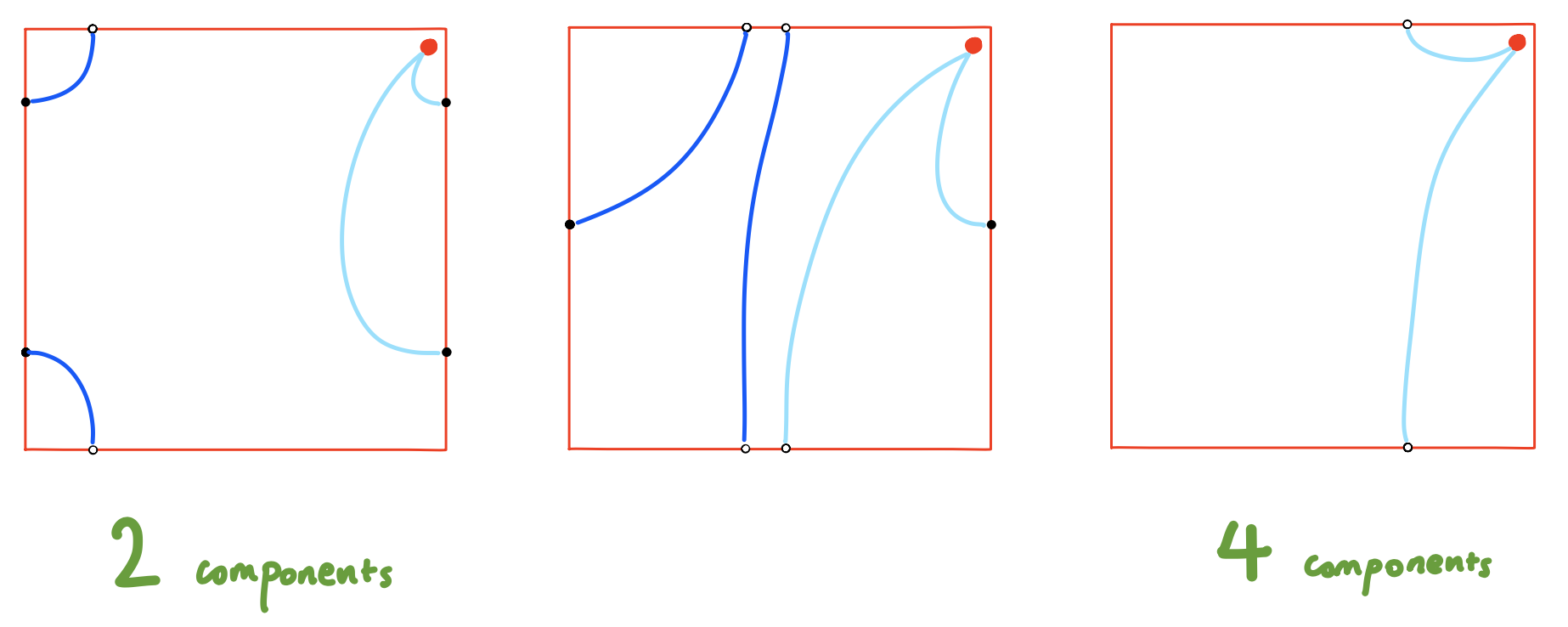}
\end{center}
\caption{Immersed curve invariant of Mazur pattern}\label{Figure:Mazurimmeresedcurve}
\end{figure}

\begin{figure}[h]
\centering
\begin{center}
\includegraphics[scale=0.35]{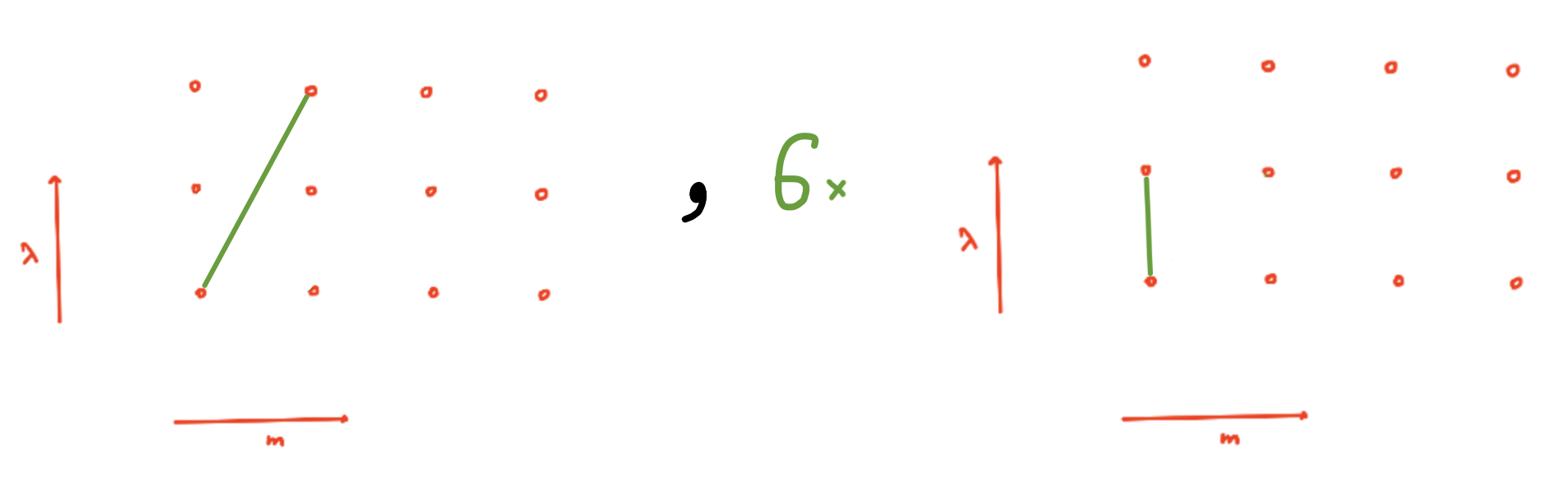}
\end{center}
\caption{A peg-board diagram of components of $\hfk(Q)$} (up to translation)\label{Figure:Mazurpegboard}
\end{figure}

\begin{figure}[h]
\centering
\begin{center}
\includegraphics[scale=0.25]{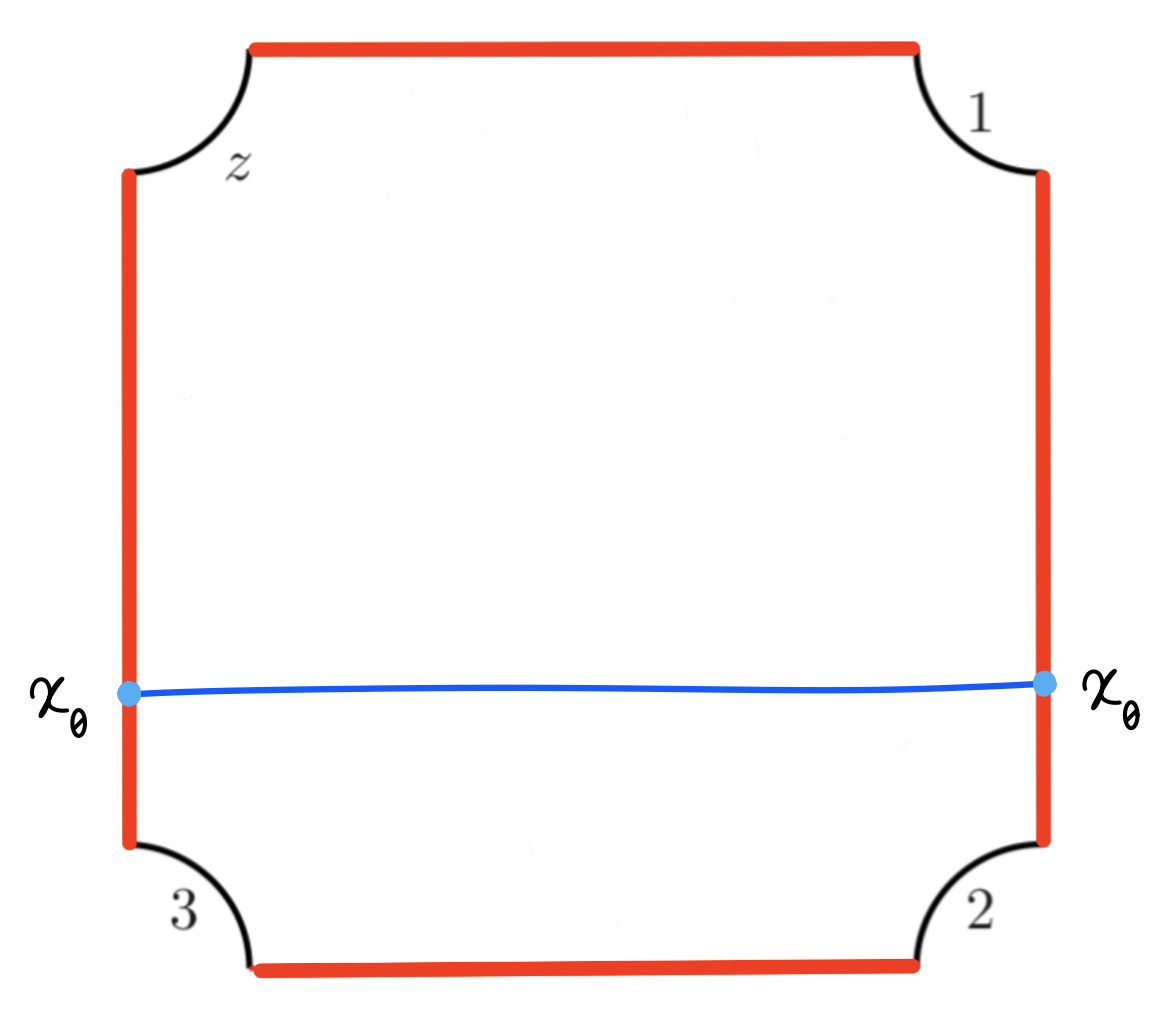}
\end{center}
\caption{$(\mathcal{H}_{\infty},z)$ which is the result of an isotopy on the $\beta$ curve in $(\mathcal{H}_Q,z)$ }\label{Figure:Mazurisotoped}
\end{figure}

\begin{exem}\label{Example:twistedmeridian}
We compute the type $D$ invariant of a solid torus 
$${\widehat{\text{CFD}}(H'_{\frac{1}{ m}}, \alpha_1, \alpha_2)}$$ which is a solid torus with the meridian homologous to $[\alpha_1] - m [\alpha_2]$ in the first homology of the boundary.\\ 

Let's denote this type $D$ module as $\widehat{\text{CFD}}(\mathcal{H}'_{\frac{1}{m}} ,z')$ for $m \in \mathbb{Z}^{+}$. We start by computing $\widehat{\text{CFD}}(\mathcal{H}'_{\frac{1}{3}} , z')$. A bordered Heegaard diagram for this bordered manifold can be seen in Figure \ref{Figure:1mHeeggard}. All the index one disks counted in the definition of $\widehat{\text{CFD}}(\mathcal{H}'_{\frac{1}{3}}, z')$ can be seen in Figure \ref{Figure:1mbigons}. As a result one can see that: 
$$\widehat{\text{CFD}}(\mathcal{H}'_{\frac{1}{3}}, z') = \langle \eta, \xi_1, \xi_2, \xi_3 \rangle_{\mathbb{F}_2},$$
$$\widehat{\text{CFD}}(\mathcal{H}'_{\frac{1}{3}}, z') \cdot \iota_{0} = \langle  \eta \rangle,$$
$$\widehat{\text{CFD}}(\mathcal{H}'_{\frac{1}{3}}, z') \cdot \iota_{1} = \langle  \xi_1, \xi_2, \xi_3 \rangle.$$
And the $\delta$ map is as follows: 
$$\delta(\eta)= \rho_3 \otimes \xi_1 + \rho_1 \otimes \xi_3,$$
$$\delta(\xi_i)= \rho_{23} \otimes \xi_{i+1} \ \ i=1,2.$$
As a result $\widehat{\text{CFD}}(\mathcal{H}'_{\frac{1}{3}}, z')$ can be described as the decorated graph seen in Figure \ref{Figure:1:3Dgraph}.\\

\begin{figure}[h]
\centering
\begin{center}
\includegraphics[scale=0.3]{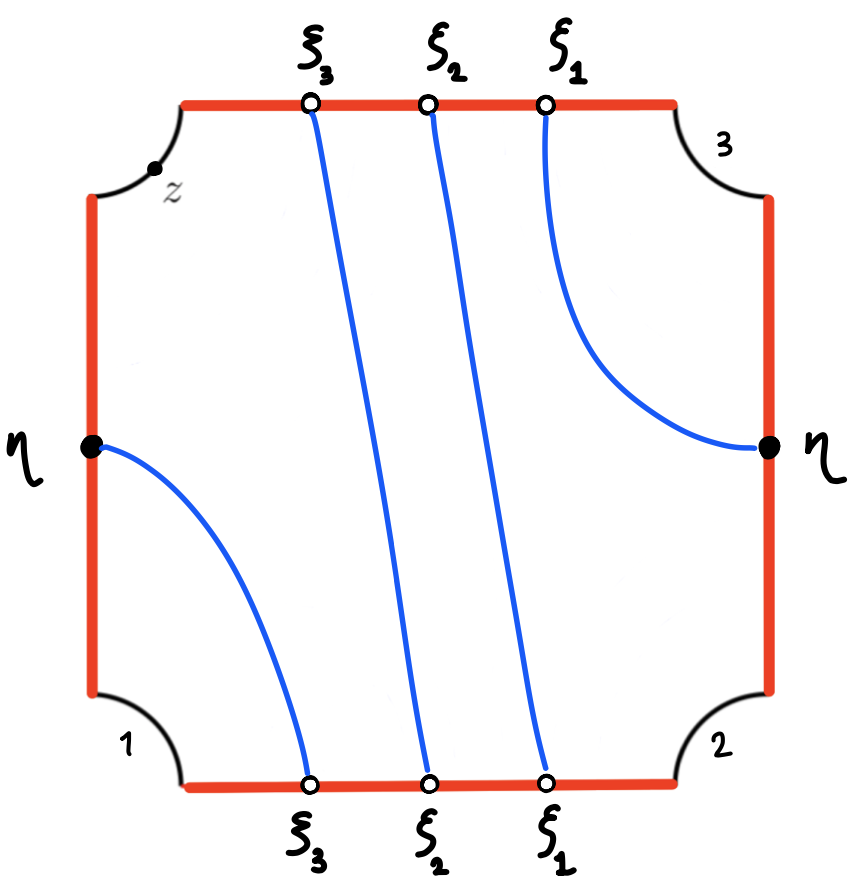}
\end{center}
\caption{Genus-one bordered Heegaard diagram $(\mathcal{H}'_{\frac{1}{3}},z')$ for $H'_{\frac{1}{3}}$}\label{Figure:1mHeeggard}
\end{figure}

\begin{figure}[h]
\centering
\begin{center}
\includegraphics[scale=0.33]{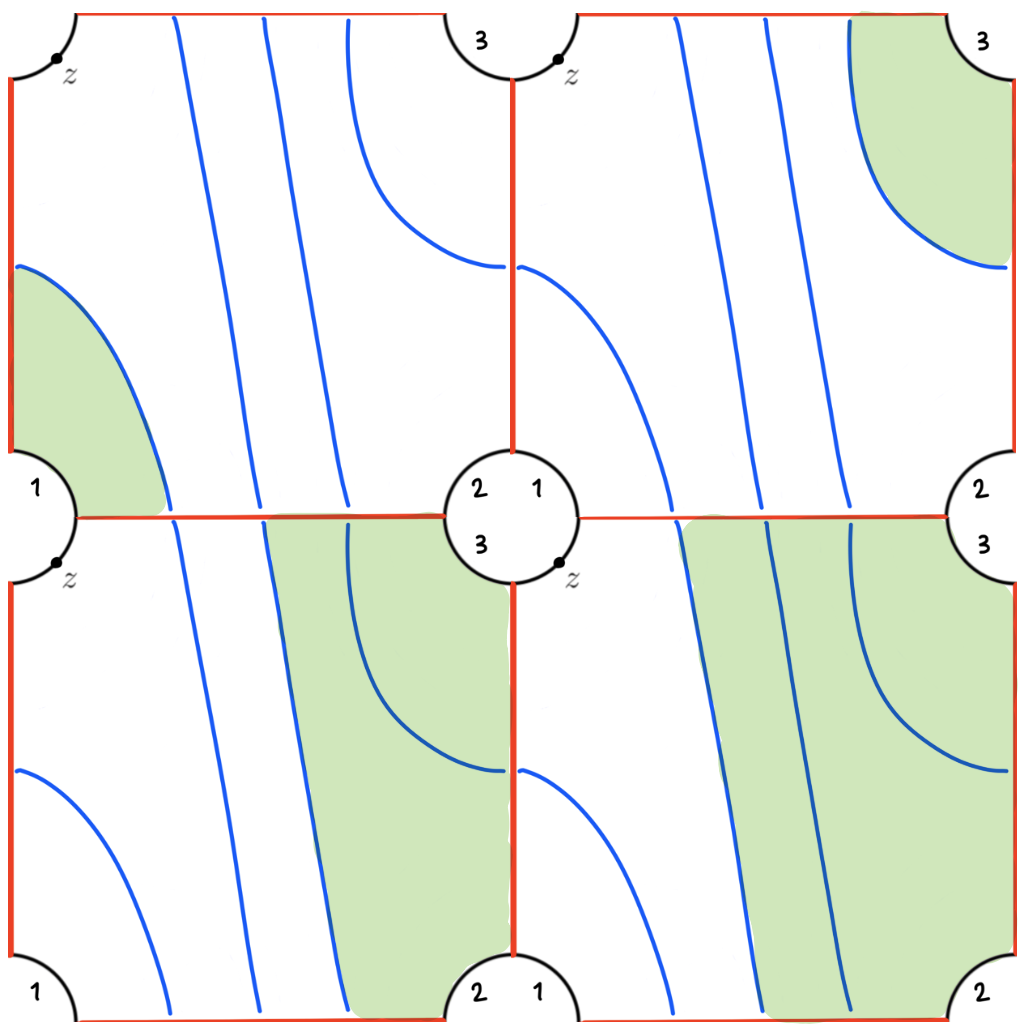}
\end{center}
\caption{Index one embedded disks in $\widetilde{\Sigma}$ counted in definition of $\widehat{\text{CFD}}(\mathcal{H}'_{\frac{1}{3}} ,z')$ }\label{Figure:1mbigons}
\end{figure}

\begin{figure}[h]
\centering
\begin{center}
\includegraphics[scale=0.3]{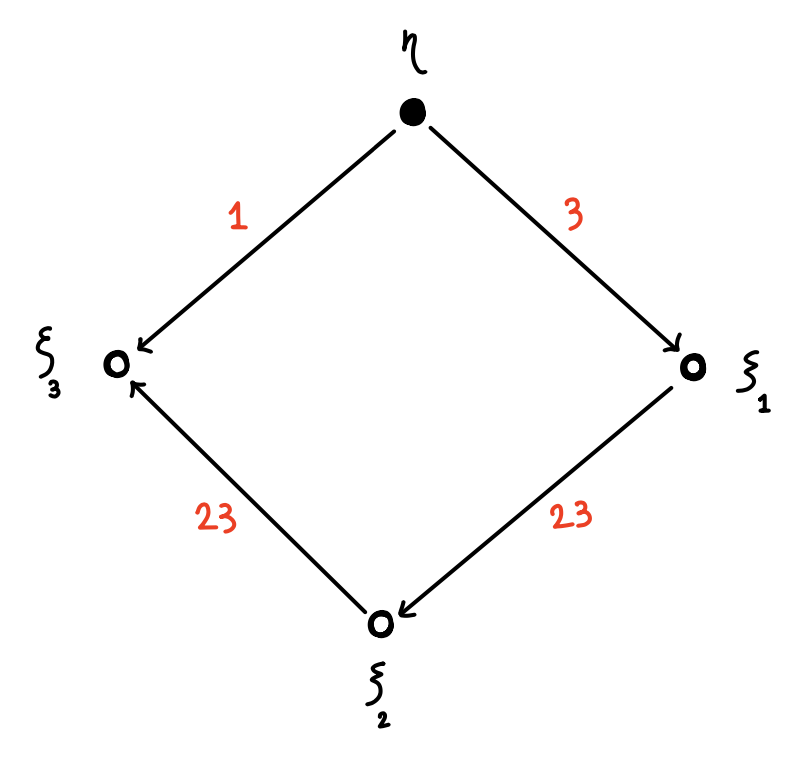}
\end{center}
\caption{Decorated graph representing $\widehat{\text{CFD}}(\mathcal{H}'_{\frac{1}{3}}, z')$}\label{Figure:1:3Dgraph}
\end{figure}

We can now easily generalize this computation to $\widehat{\text{CFD}}(\mathcal{H}'_{\frac{1}{m}} ,z')$. The $\beta$ curve in $\mathcal{H}'_{\frac{1}{m}}$ intersects $\alpha_2$ in one point and $\alpha_1$ in $m$ points and hence: 
$$\widehat{\text{CFD}}(\mathcal{H}'_{\frac{1}{m}}, z') = \langle \eta, \xi_1, \cdots , \xi_m \rangle_{\mathbb{F}_2},$$
$$\widehat{\text{CFD}}(\mathcal{H}'_{\frac{1}{m}}, z') \cdot \iota_{0} = \langle  \eta \rangle,$$
$$\widehat{\text{CFD}}(\mathcal{H}'_{\frac{1}{m}}, z') \cdot \iota_{1} = \langle  \xi_1, \cdots , \xi_m \rangle.$$
The count of index one disks used in definition of map $\delta$ is also very similar 
$$\delta(\eta)= \rho_3 \otimes \xi_1 + \rho_1 \otimes \xi_m,$$
$$\delta(\xi_i)= \rho_{23} \otimes \xi_{i+1} \ \ i=1,\cdots,m-1.$$
A decorated graph representing $\widehat{\text{CFD}}(\mathcal{H}'_{\frac{1}{3}}, z')$ can be seen in Figure \ref{Figure:1:mDgraph}.\\

\begin{figure}[h]
\centering
\begin{center}
\includegraphics[scale=0.3]{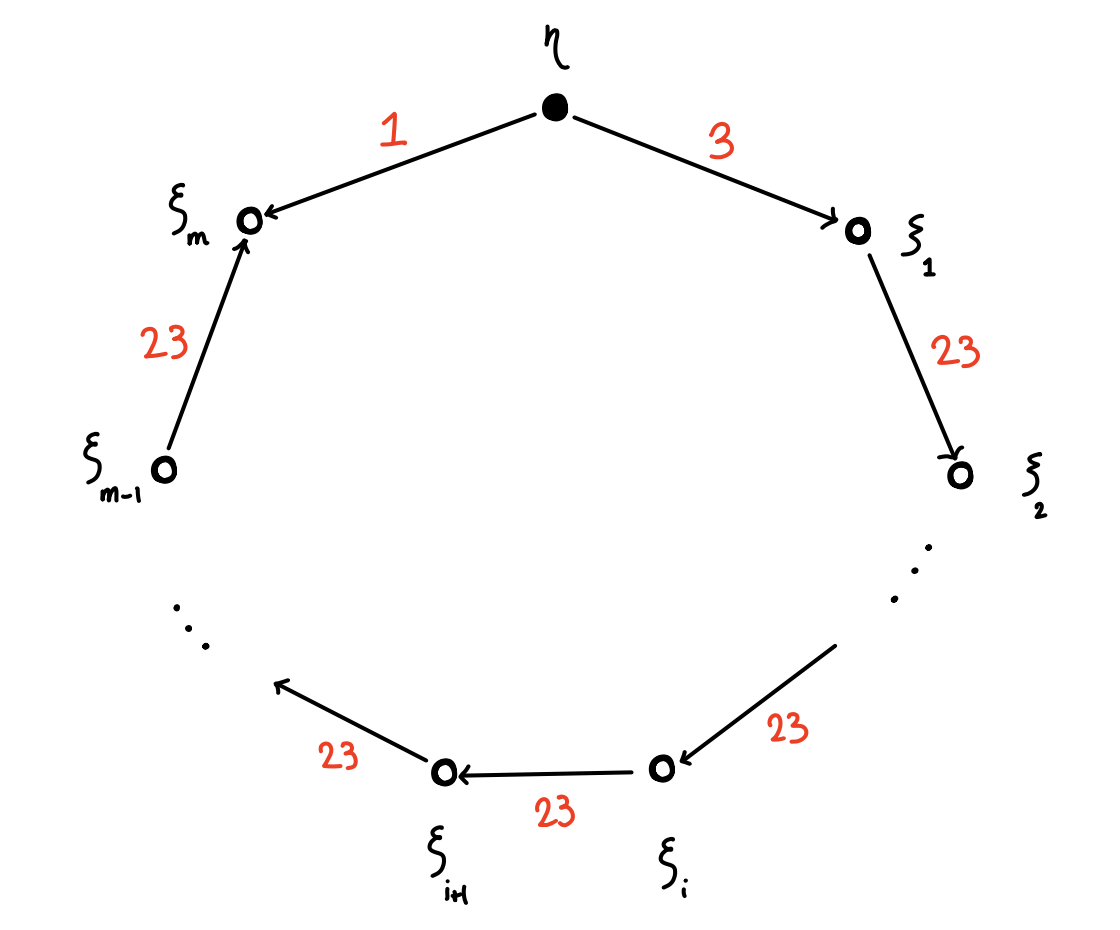}
\end{center}
\caption{Decorated graph representing $\widehat{\text{CFD}}(\mathcal{H}'_{\frac{1}{m}}, z')$}\label{Figure:1:mDgraph}
\end{figure}

We can also compute the grading structure on $\widehat{\text{CFD}}(\mathcal{H}'_{\frac{1}{m}}, z')$ as follows. We denote the grading function by $gr_{m}$. We pick $\eta$ as the reference generator. Then using the rules of Figure \ref{Figure:typeDgrading} we have: 
$$gr_{m}(\eta) = (\frac{1}{2};-\frac{1}{2}, \frac{1}{2}) \cdot gr_{m}(\xi_1) \Rightarrow gr_{m}(\xi_1) = (-\frac{1}{2};\frac{1}{2}, -\frac{1}{2}),$$
and for $i = 1, \cdots, m$: 
$$gr_{m}(\xi_{i}) = (\frac{1}{2};0, 1) \cdot gr_{m}(\xi_{i+1}) \Rightarrow gr_{m}(\xi_{i}) = (-\frac{1}{2}; \frac{1}{2}, - \frac{2i-1}{2}).$$
Also we can compute $P(\eta)$ as follows: 
$$gr_{m}(\eta) = (\frac{1}{2}; \frac{1}{2}, -\frac{1}{2}) \cdot gr_{m}(\xi_m) \Rightarrow gr_{m}(\eta)=(- \frac{m-1}{2}; 1, -m).$$
This means that $P(\eta) = \langle (- \frac{m-1}{2}; 1, -m) \rangle$ and 
$$gr_{m}: \widehat{\text{CFD}}(\mathcal{H}'_{\frac{1}{m}}, z') \rightarrow G / \langle (- \frac{m-1}{2}; 1, -m) \rangle.$$
Note that our computation of $gr_m$ comes from moving clockwise on the decorated graph represented in Figure \ref{Figure:1:mDgraph} starting from $\eta$. We can also compute $gr_m$ by moving in the counterclockwise direction, which will give us the following: 

$$gr(\eta) = (\frac{1}{2}; \frac{1}{2}, -\frac{1}{2}) \cdot gr_{m}(\xi_m) \Rightarrow gr_{m}(\xi_m)  = (-\frac{1}{2}; -\frac{1}{2}, \frac{1}{2}),$$
and for $i=1, \cdots m-1$: 
$$gr_{m}(\xi_{m-i}) = (\frac{1}{2};0, 1) \cdot gr_{m}(\xi_{m-i+1}) \Rightarrow gr_{m}(\xi_{m-i}) = (\frac{2i-1}{2};-\frac{1}{2},\frac{2i+1}{2}).$$
Note that the first and second computations are just different choices of representatives of the same cosets in $G / \langle (- \frac{m-1}{2}; 1, -m) \rangle$. Having these two different choices will be useful, see Corollary \ref{Corollary:doublesidedgrading}.\\

Now we can also look at the immersed curve invariant $\widehat{HF}$.  Figure \ref{Figure:1:3immeresedgraph} shows the result of embedding the decorated graph of $\widehat{\text{CFD}}(\mathcal{H}'_{\frac{1}{3}}, z')$ in $T$. One can see that the immersed curve invariant has a loose representative in the form of a Eulidean geodesic whilch lifts to a line with slope $3$ in $\widetilde{T}$. In the general case of $\widehat{\text{CFD}}(\mathcal{H}'_{\frac{1}{m}}, z')$, the immersed curve invariant will be an Euclidean geodesic which lifts to a line with slope $m$ in $\widetilde{T}$.\\\\

\begin{figure}[h]
\centering
\begin{center}
\includegraphics[scale=0.3]{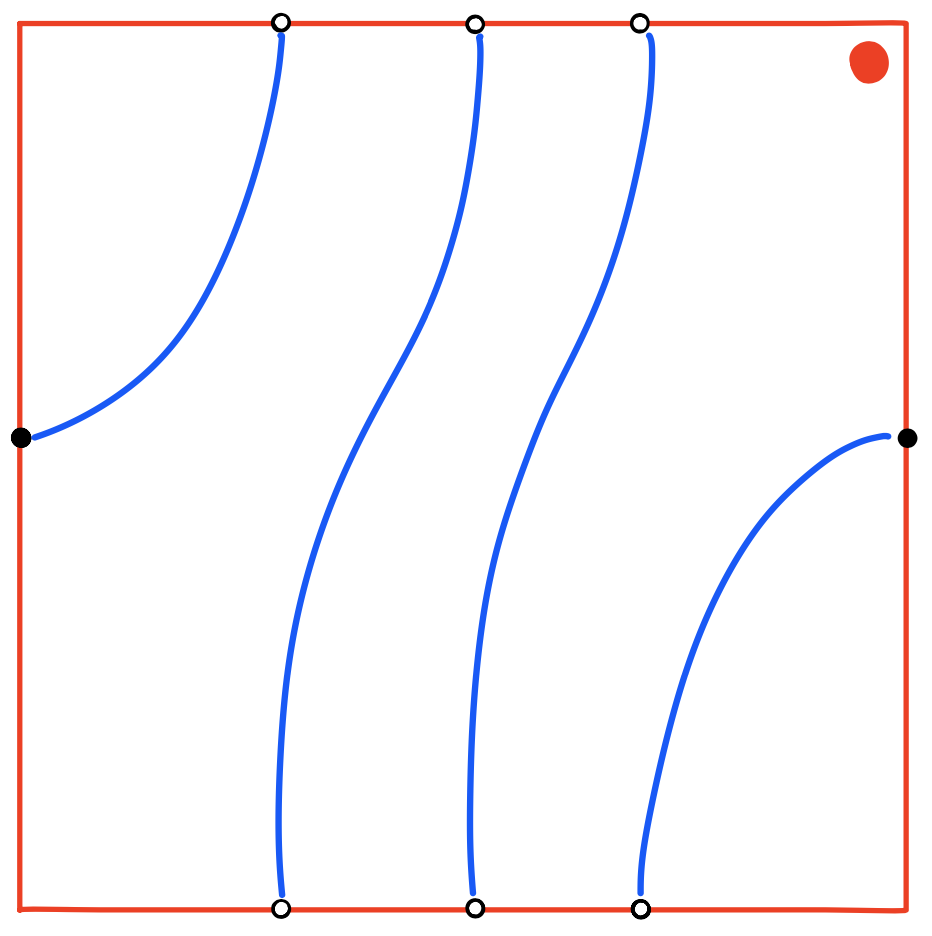}
\end{center}
\caption{The embedding of a decorated graph representing $\widehat{\text{CFD}}(\mathcal{H}'_{\frac{1}{3}}, z')$ in puntured torus $T$}\label{Figure:1:3immeresedgraph}
\end{figure}
\end{exem}

\begin{rema}
In Section \ref{Section:Boxtensor}, we examine $\widehat{\text{CFD}}(\mathcal{H}'_{\frac{1}{m}}, z')$ when $m \rightarrow \infty$. More specifically we will compare the complexes $\widehat{\text{CFD}}(\mathcal{H}'_{\frac{1}{m}}, z')$ and $\widehat{\text{CFD}}(\mathcal{H}'_{\frac{1}{m+1}}, z')$. Going on, abusing the notation, we use the same notations to refer to the generators of these two different type $D$ structures, i.e. 
$$\widehat{\text{CFD}}(\mathcal{H}'_{\frac{1}{m}}, z') = \langle \eta, \xi_1, \cdots , \xi_{m} \rangle, \ \text{and} \ $$
$$\widehat{\text{CFD}}(\mathcal{H}'_{\frac{1}{m+1}}, z') = \langle \eta, \xi_1, \cdots , \xi_{m+1} \rangle.$$
Note that we used the notations $gr_{m}$ and $gr_{m+1}$ to differentiate between the gradings of the mentioned type $D$ modules.\\ 
\end{rema}

\begin{rema}\label{Remark:Twogradings}
We described two distinct methods for selecting representatives of the grading cosets of the generators of $$\widehat{\text{CFD}}(\mathcal{H}'_{\frac{1}{m}}, z'),$$ based on the direction of traversal in the decorated graph shown in Figure \ref{Figure:1:mDgraph}. We refer to these as the \emph{clockwise} and \emph{counterclockwise} representatives of $gr_{m}$, respectively.\\    
\end{rema}

By combining the clockwise and counterclockwise representatives, we can have Corollary \ref{Corollary:doublesidedgrading}. 

\begin{coro}\label{Corollary:doublesidedgrading}
For any $k \in \mathbb{N}$ and any integer $m > 2k$, we can choose the following representatives for the coset values of $gr_{m}$: 
$$gr_{m}(\xi_{i+1}) = (-\frac{1}{2}; \frac{1}{2}, - \frac{2i+1}{2})$$
$$gr_{m}(\xi_{m-i}) = (\frac{2i-1}{2};-\frac{1}{2},\frac{2i+1}{2})$$
for $i=0, \cdots, k$. This means that we can pick the fixed representatives for the gradings of generators $\xi_{m-k}, \cdots , \xi_{m}, \eta, \xi_{1}, \cdots, \xi_{k+1}$ as $m \rightarrow \infty$.
\end{coro}

\section{Using Immersed curve invariants}\label{Section:Immeresedcurve}

In this section, we are going to use the immersed curve invarinats to prove Theorem \ref{Theorem:totaldim}. First, we need a pairing theorem for the immersed curve invariant of a knot in solid torus (similar to Theorem \ref{Theorem:immeresedpairing}). We explain the setting of this theorem in the following.\\

Consider a knot $Q \subset S^1 \times D^2$. Let $M$ be a 3-manifolds with torus boundary. Let $h: \partial(S^1 \times D^2) \rightarrow \partial M$ be an orientation reversing homeomorphism, and let $\bar{h}$ denote the composition of $h$ with the elliptic involution of $\partial M$. Let $Y := M \cup_{h} (S^1 \times D^2)$.\\

\begin{theo}\label{Theorem:immeresedpairingknots}
Let $\bm{\gamma}$ and $\bm{\gamma'}$ denote $\widehat{HF}(M)$ and $\hfk(Q)$ respectively. Then we have 
$$\widehat{HFK}(Y , Q) \cong HF(\bm{\gamma} , \bar{h}(\bm{\gamma'})).$$    
\end{theo}

As we explain in the following, this setting is compatible with the construction of twist families. Recall from Section \ref{Section:Introduction}, that a knot $K_{m}$ in a twist family is the result of performing $(-\frac{1}{m})$ surgery on an unknot $c \subset S^3 \setminus N(K)$. We can think of $K$ as a knot in the solid torus $ H := S^3 \setminus N(c)$. Let $\lambda_{c} \subset \partial H$ be a Seifert longitude of $c$ which is a meridian of solid torus $H$. Also let $\mu_{c} \subset \partial H$ be a meridian of $c$, which can be seen as a longitude of the torus $\partial H$. We pick $z = \lambda_{c} \cap \mu_{c}$ as the basepoint. In this setting $(H, \mu_{c}, \lambda_{c})$ is a bordered 3-manifold.\\

Similar to before, let $T$ be a punctured torus as in Figure \ref{Figure:typeDimmeresed}. As we discussed, we have an immersed curve invariant $\hfk(K)$ which is a collection of curves and arcs in $T$ equipped with local systems.\\

We now can also consider another solid torus $H' := N(c)$. Note that $\mu_{c}$ is a meridian of $H'$ and $-\lambda_{c}$ can be seen as the longitude of $\partial H'$. Note that the negative sign is due to the fact that $H$ and $H'$ induce different orientations on the boundary torus.\\ 

Performing $(-\frac{1}{m})$ surgery on $c$ is equivalent to replacing this solid torus with $H'_{\frac{1}{m}}$ which is a solid torus with meridian homologous to $[\mu_{c}] - m \cdot [\lambda_{c}]$. We computed the immersed curve invariant of $H'_{\frac{1}{m}}$ in Example $\ref{Example:twistedmeridian}$, and the result was an Euclidean geodesic which lifts to a line with slope $m$ in $\widetilde{T}$.\\

Now we are ready to present the first proof of Theorem \ref{Theorem:totaldim}.

\begin{proof}
 We are going to use Theorem \ref{Theorem:immeresedpairingknots}. We are gluing $(H,K)$ and $H'_{\frac{1}{m}}$ to get $(S^3 , K_m)$. Let $\bm{l_m} \subset T$ be the immersed curve invariant of $H'_{\frac{1}{m}}$, and let $\bm{\gamma}$ be $\hfk(K)$. We know that 
 $$\widehat{HFK}(S^3 , K_m) \cong HF(\bm{l_m} , \bar{h}(\bm{\gamma}))$$
 Now let $\bm{\gamma_1}, \cdots, \bm{\gamma_{n}}$ be the connected components of $\bar{h}(\bm{\gamma})$ then we have
$$HF(\bm{l_m} , \bar{h}(\bm{\gamma'})) =  \bigoplus_{1\leq j\leq n'}  HF(\bm{l_m}, \bm{\gamma_j})$$ 
Recall that $\bm{\gamma_j}$ is an immersed curve decorated with a local system. Let 
$$\bm{\gamma_j}=(\gamma_j , k_j, A_j) \ \ \text{for} \ \ j=1,\cdots,n.$$
Using Lemma \ref{Lemma:primitivereduce}, we can assume that all $\gamma_j$ are primitive. Note that $\bm{l_m}$ consists of a single primitive curve $l_m$ with trivial local systems. Now we can use Theorem \ref{Theorem:Lagrangiandimension} to get the following
$$\dim (HF(\bm{l_m}, \bm{\gamma_j})) = k_j \cdot i(l_m, \gamma_j) \ \ \text{for} \ \ j=1,\cdots,n'.$$
As a result, we only need to see how $i(l_m, \gamma_j)$ changes as $m \rightarrow \infty$.\\

We use  $(\widehat{\gamma}_j, \bm{c^j})$ to denote the singular pegboard diagram of the immersed curve $\gamma_j$. Now we are going to use Theorem \ref{Theorem:singularpegboarddiagram}. Note that $\widehat{\gamma}_j$ has finitely many slopes and as a result for $m \gg0$, none of the slopes of $\widehat{\gamma}_j$ will be equal to slope of $l_m$. Furthermore, the immersed curve $l_m$ is loose (i.e. doesn't have any corners) hence
$$i(l_m, \gamma_j) = | l_m \cap \widehat{\gamma}_j |.$$
Recall that the punctured torus $T$ is defined as $\widehat{T} \setminus z$. As we mentioned before $\widehat{\gamma}_j$ can be seen as a collection of Euclidean geodesics in $\widehat{T}$ passing through $z$. We denote these collection of Euclidean geodesics as $\{L^{j}_1, \cdots,L^{j}_{t_j}\}$. Clearly 
$$| l_m \cap \widehat{\gamma}_j | = \sum_{1 \leq j \leq n'} \sum_{1 \leq i \leq t_j} | l_m \cap L^{j}_{i} |.$$
Assume that $L^{j}_{i}$ lifts to a line with slope $\frac{p^{j}_{i}}{q^{j}_{i}}$ in the universal cover of $\widehat{T}$ (which is $\mathbb{R}^2$). Recall that $\gamma_m $ lifts to a line with slope $\frac{m}{1}$. As a result 
$$| l_m \cap L^{j}_{i} | = \Biggr| \text{det} \begin{vmatrix}
    1       & m \\
    q^{j}_{i}       & p^{j}_{i}
\end{vmatrix}\Biggr| = |p^{j}_{i} - m \cdot q^{j}_{i}|.$$
As a result, for $m \gg0$, we will have: 
$$ | l_m \cap L^{j}_{i} | = m \cdot q^{j}_{i} - p^{j}_{i}.$$
Now we can put everything together as follows. For $m \gg0$ we have 
$$\dim (\widehat{HFK}(S^3 , K_m)) = \dim (HF(\bm{l_m} , \bar{h}(\bm{\gamma})) = \sum_{1 \leq j \leq n'} \dim(HF(\bm{l_m} , \bm{\gamma_j}))$$
$$= \sum_{1 \leq j \leq n'} k_j \cdot i(l_m, \gamma_j) = \sum_{1 \leq j \leq n'} k_{j} \cdot (\sum_{1 \leq i \leq t_{j}} m \cdot q^{j}_{i} - p^{j}_{i})$$
$$= m \cdot \underbrace{(\sum_{1\leq j \leq n'} \sum_{1\leq i\leq t_j} k_{j} q^{j}_{i})}_{D} - \underbrace{(\sum_{1\leq j \leq n'} \sum_{1\leq i\leq t_j} k_{j} p^{j}_{i})}_{d}.$$
This finishes the proof of Theorem \ref{Theorem:totaldim}.
\end{proof}
In the rest of this section we investigate the effect of twisting on the immersed curve invariant. As mentioned before if we imagine $K = K_0$ as a knot in solid torus $S^1 \times D^2$, we can construct the twist family $\{K_m\}$ by repeatedly applying the meridional Dehn twist $\tau_{\partial D^2}$ to $K_0$. Rasmussen \cite{Rasmusssentalk} introduced a Dehn twist formula for $\hfk$ as follows.

\begin{theo}\cite{Rasmusssentalk}\label{Dehntwistformula} For a knot $P \subset S^1 \times D^2$ we have 
 $$\hfk(\tau_{\partial D^{2}}(P)) = \tau_{\partial D^{2}} (\hfk(P)),$$ 
 where we use $\tau_{\partial D^{2}}$ to denote the meridional Dehn twist both as a homemorphism of the boundary torus and its extension to a homemorphism of the solid torus. 
\end{theo}

Based on this formula we can compute $\hfk(K_m)$ from $\hfk(K)$ for a twist family $\{K_m\}$ as follows 
$$\hfk(K_{m}) = \tau^{m}_{\partial D^{2}}(\hfk(K)).$$

We start by analyzing the affect of $\tau_{\partial D^{2}}$ on a pegboard diagram. We consider the pegboard diagrams in $\widetilde{T}$ which is seen as $\mathbb{R}^2$ with a lattice of pegs. We denote the horizontal and vertical coordinates of $\widetilde{T}$ by $m$ and $\lambda$ respectively, as previously seen in Figure \ref{Figure:Mazurpegboard}. We also use $\overrightarrow{m}$ and $\overrightarrow{\lambda}$ to denote unit vectors in horizontal and vertical directions. Let the horizontal coordinate i.e. $m$, be the lift of the meridian of solid torus i.e. $\partial D^{2}$.\\

Let $\gamma$ be a pegboeard diagram in $\widetilde{T}$. The curve $\gamma$ consists of oriented linear segments $L_1, \cdots, L_t$ and corners $c_1, \cdots, c_{t-1}$, where $c_i$ is the corner sitting between $L_{i}$ and $L_{i+1}$. Let $L_i$ be a linear segment with slope $\frac{p_i}{q_i}$ i.e. 
$$L_i = q_i \cdot \overrightarrow{m} + p_i \cdot \overrightarrow{\lambda},$$
as a vector.\\

Now we describe the pegboard diagram of $\tau_{\partial D^2}(\gamma)$ in Proposition \ref{Proposition:pegboardDehntwist}.

\begin{prop}\label{Proposition:pegboardDehntwist}
    The pegboard diagram of $\tau_{\partial D^{2}}(\gamma)$ consists of linear segments $L'_1,\cdots,L'_t$ and corners $c'_t,\cdots,c'_{t-1}$ such that
    $$L'_i=(q_i+p_i) \cdot \overrightarrow{m} + p_i \cdot \overrightarrow{\lambda}$$
    This means that the linear parts of $\tau_{\partial D^{2}}(\gamma)$ comes from applying Dehn twist on the linear parts of $\tau_{\partial D^{2}}(\gamma)$, while none of its corners get unwrapped and no extra corners will be formed.  
\end{prop}

\begin{proof}
    We consider all possible lifts of a loose representative of the meridian. We can take these lifts to be $M_n$ for all $n \in \mathbf{Z}$, where $M_n$ is the line $\lambda = n+2\epsilon$. This can be seen in Figure \ref{Figure:hfkex2}.\\

    Now we take all the intersections of $\gamma$ with $M_n$. Moving on the curve $\gamma$, as we reach an intersection, we shift the rest of the curve one unit (i.e. $\pm \overrightarrow{m}$) in the horizontal direction and replace the intersection with a lift a horizontal segment. The sign of the shift comes from the sign of the intersection as follows. The intersections of $\gamma$ and $M_n$ split as follows
    $$M_n \cap \gamma = \bigsqcup_{1 \leq i \leq t} M_n \cap L_i$$
    On an intersection in $M_n \cap L_i$ the sign of the shift is equal to the sign of $p_i$ i.e. if the linear segment is moving upward the shift is equal to $+\overrightarrow{m}$ and if the linear segment is moving downward the shift is equal to $-\overrightarrow{m}$.\\

    An example of this can be seen in the following. We call this construction the \emph{meridian shifting}. This construction gives us (a lift of) a curve in the homotopy class of $\tau_{\partial D^2}(\gamma)$ as seen in Figure \ref{Figure:hfkex3}. We are going to explain how we can homotope this to a pegboard diagram of $\tau_{\partial D^2}(\gamma)$ (i.e. pull it to be tight) as seen in Figure \ref{Figure:hfkex4}.\\

    \begin{figure}[h]
\centering
\begin{minipage}{0.5\textwidth}
  \centering
  \includegraphics[width=.9\linewidth]{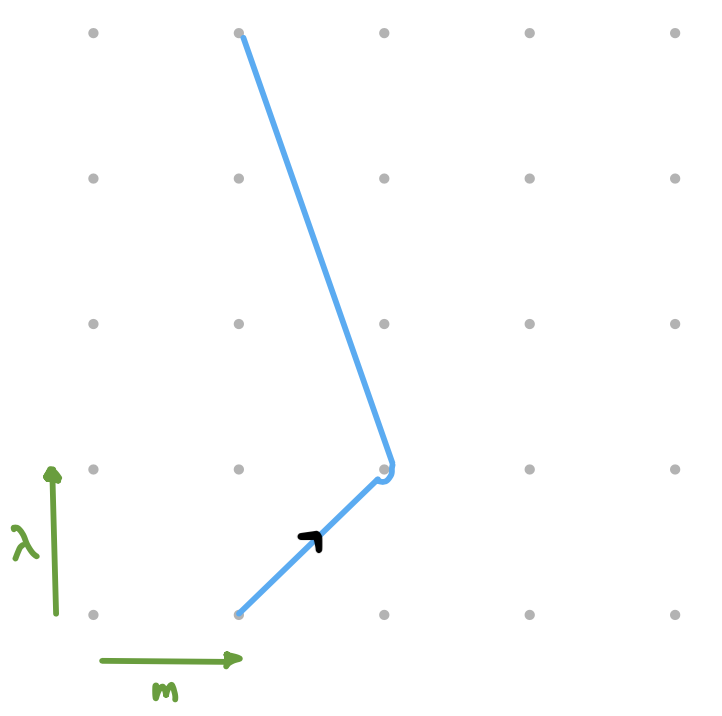}
  \captionof{figure}{Pegboard diagram $\gamma$}
  \label{Figure:hfkex}
\end{minipage}%
\begin{minipage}{.5\textwidth}
  \centering
  \includegraphics[width=.9\linewidth]{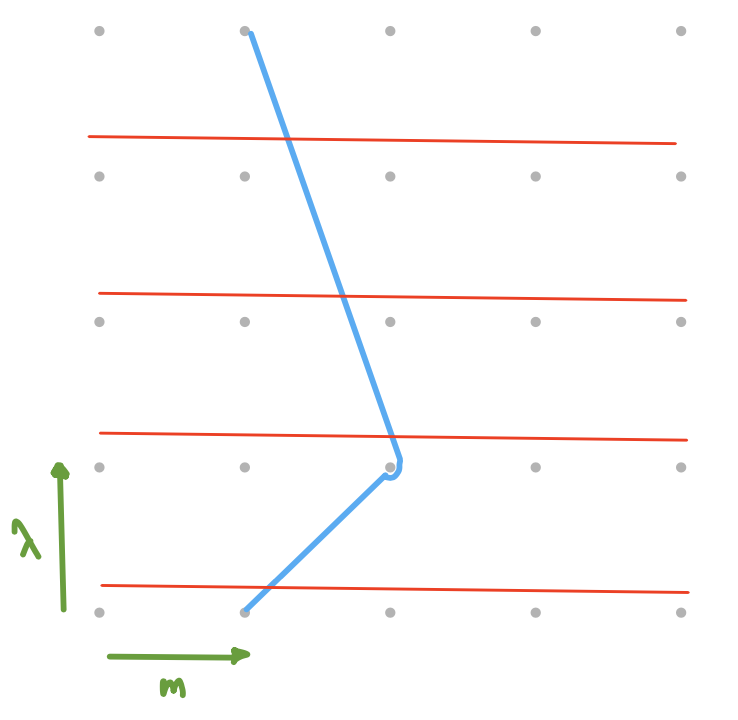}
  \captionof{figure}{$\gamma$ and meridional lifts $M_n$}
  \label{Figure:hfkex2}
\end{minipage}
\end{figure}

\begin{figure}[h]
\centering
\begin{minipage}{0.5\textwidth}
  \centering
  \includegraphics[width=.9\linewidth]{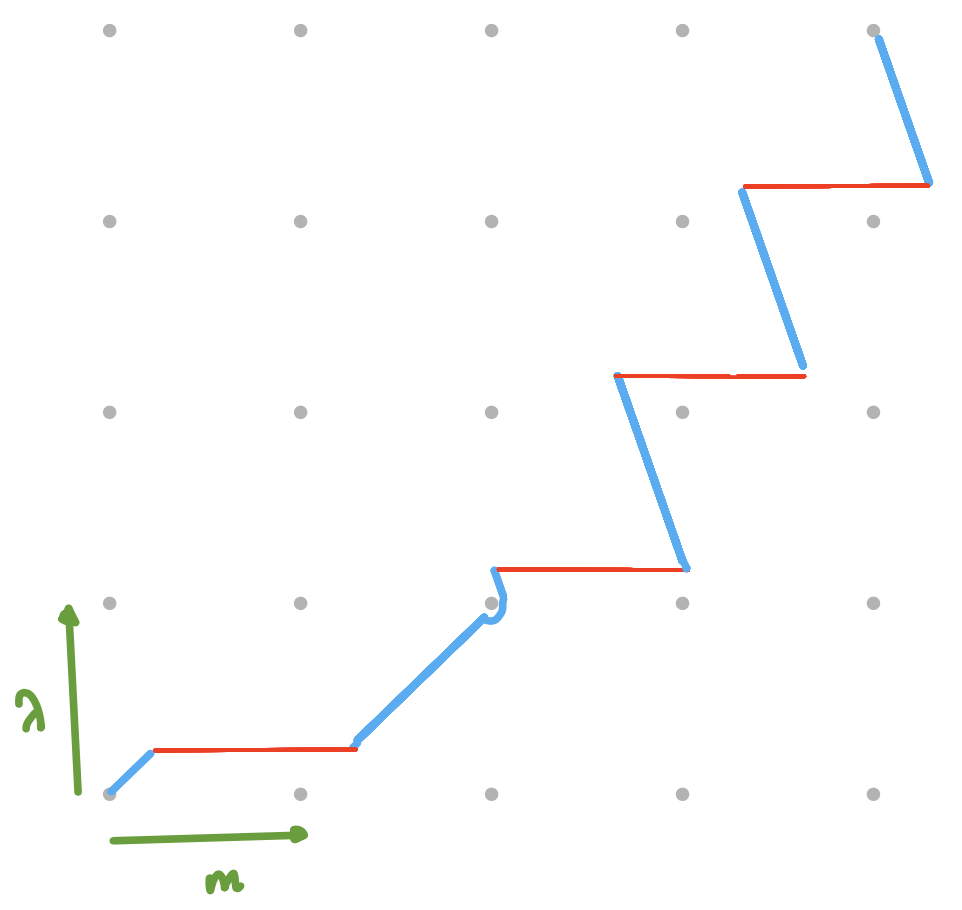}
  \captionof{figure}{Meridian shifting construction}
  \label{Figure:hfkex3}
\end{minipage}%
\begin{minipage}{0.5\textwidth}
  \centering
  \includegraphics[width=.9\linewidth]{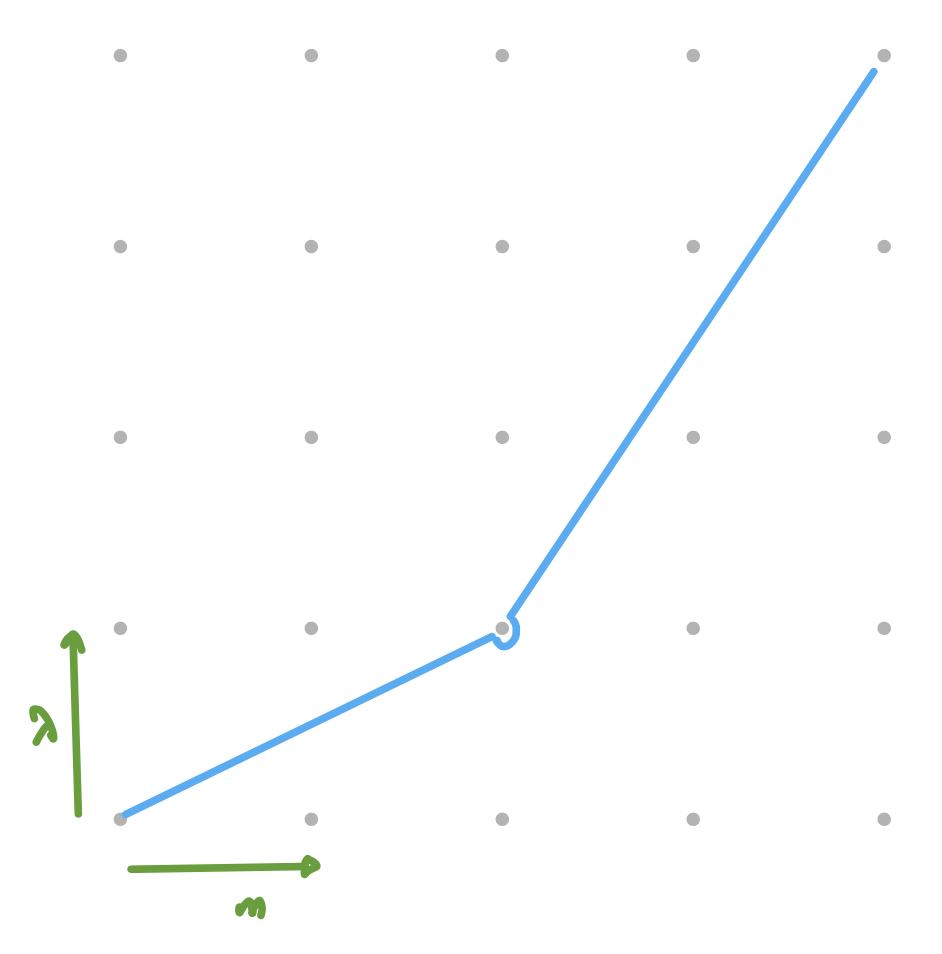}
  \captionof{figure}{Pegboard diagram of $\tau_{\partial D^{2}}(\gamma)$}
  \label{Figure:hfkex4}
\end{minipage}
\end{figure}

    We can see two more complicated examples of meridian shifting in the Figures \ref{Figure:meridianshitingex} and \ref{Figure:meridianshitingex2}. In these examples we perturbed the horizontal segments a bit to keep the diagrams immersed.\\

We can first show that that if we apply the meridian shifting on a linear segment $L_i=q_i \cdot \overrightarrow{m} + p_i \cdot \overrightarrow{\lambda}$, the result, which we denote by $\tau_{\partial D^2}(L_i)$,  will be homotopic to a linear segment $L'_i=(q_i+p_i) \cdot \overrightarrow{m} + p_i \cdot \overrightarrow{\lambda}$.\\

\begin{figure}[h]
\centering
\begin{center}
\includegraphics[scale=0.5]{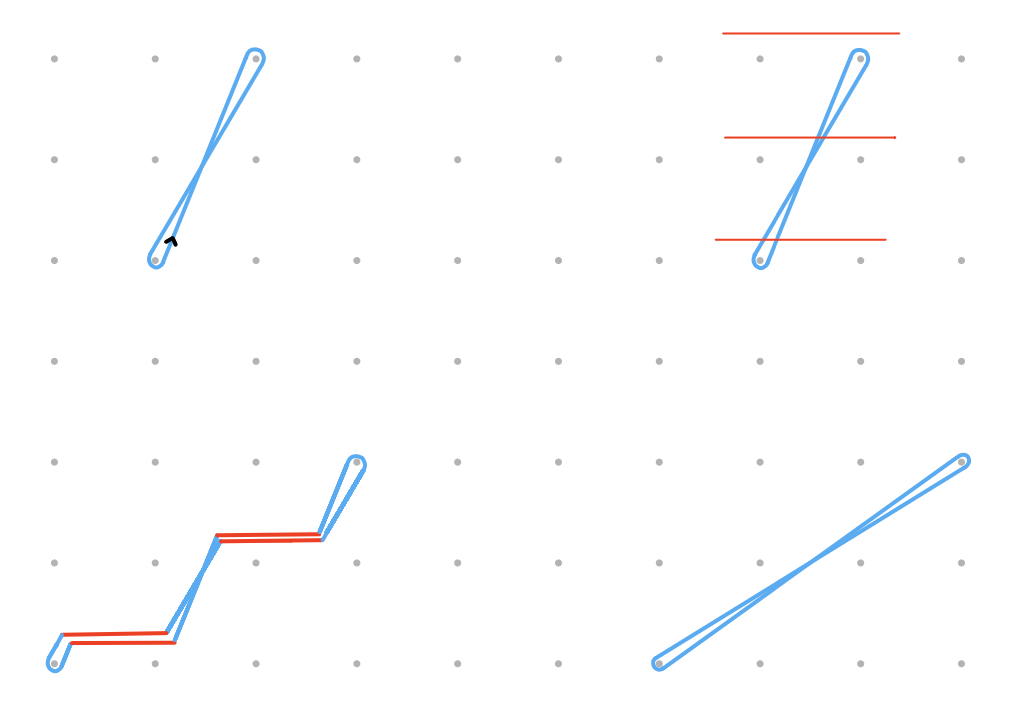}
\end{center}
\caption{Another example of meridian shifting (left to right)}\label{Figure:meridianshitingex}
\end{figure}

\begin{figure}[h]
\centering
\begin{center}
\includegraphics[scale=0.5]{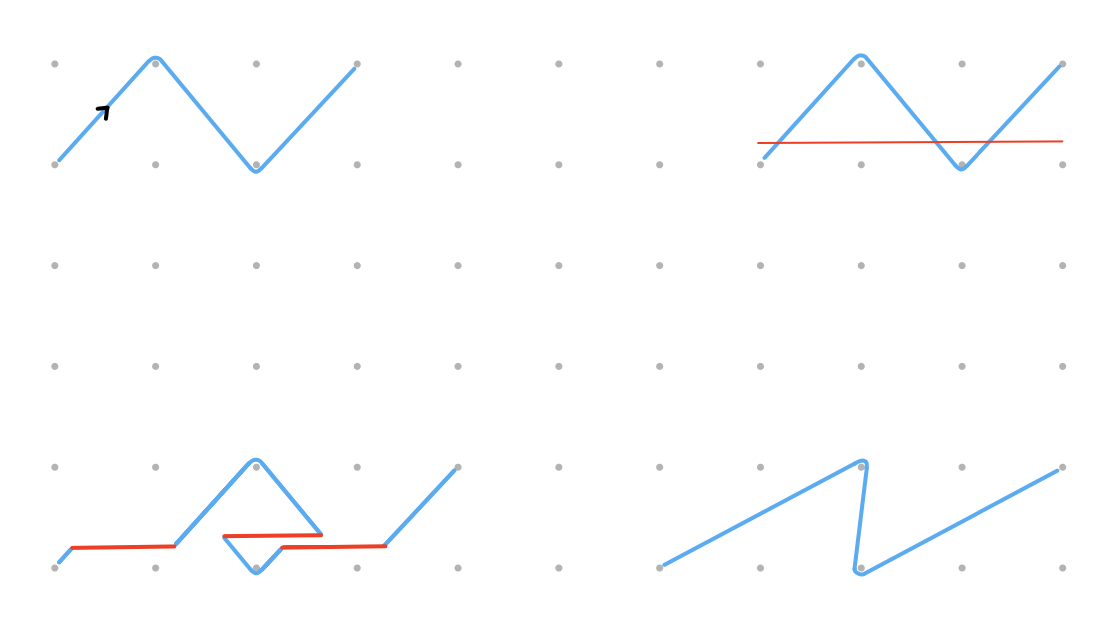}
\end{center}
\caption{Another example of meridian shifting (left to right)}\label{Figure:meridianshitingex2}
\end{figure}

The homotopy is through the bigons in plane bound by $L'_i$ and $\tau_{\partial D^2}(L_i)$.  We only need to show there is no pegs inside these bigons. This follows from a simple geometric argument sketched in Figures \ref{Figure:bigons} and \ref{Figure:geomarg}.\\

 \begin{figure}[h]
\centering
\begin{minipage}{.5\textwidth}
  \centering
  \includegraphics[width=.9\linewidth]{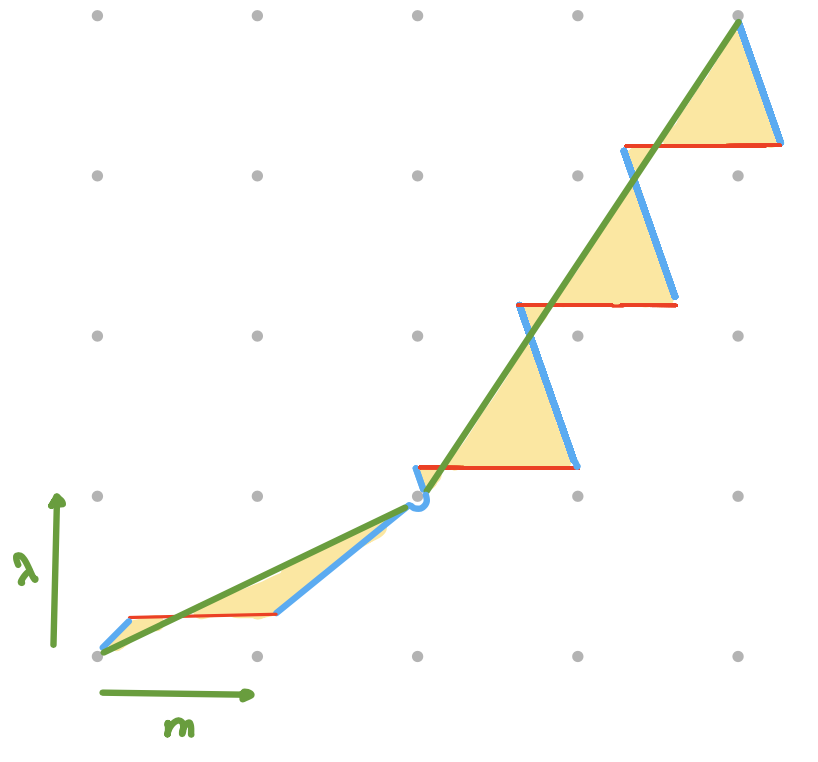}
  \captionof{figure}{Homotopy bigons in example of Figure \ref{Figure:hfkex3}}
  \label{Figure:bigons}
\end{minipage}%
\begin{minipage}{.5\textwidth}
  \centering
  \includegraphics[width=.9\linewidth]{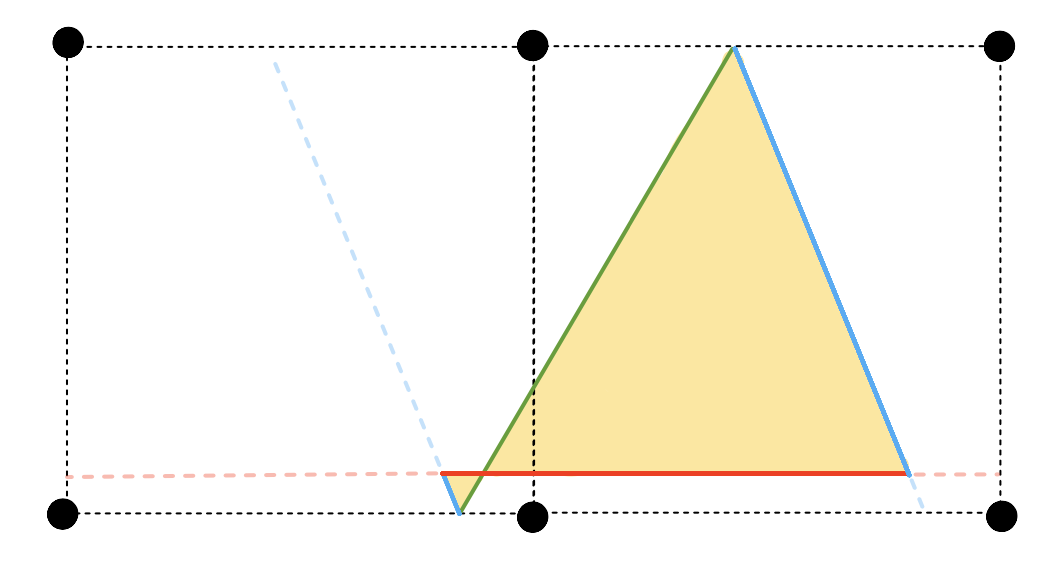}
  \captionof{figure}{Bigons don't contain the punctures}
  \label{Figure:geomarg}
\end{minipage}
\end{figure}

Now we only need to show that the corners doesn't get unwrapped. For a corner $c_i$ to get unwrapped, the curve has to change direction by at most an angle of $\pi$. Assume that the corner $c_i$ wraps around a peg with coordinates $(x_i,y_i)$. Consider the horizontal line $\lambda = y_i$ oriented in the positive direction. We call this line $H_{y_{i}}$.\\

If $L_i$ and $L_{i+1}$ are one the same side of $H_{y_{i}}$, the corner can't be unwrapped. Without loss of generality assume they are both above $H_{y_{i}}$. Note that $L'_{i}$ and $L'_{i+1}$ remain above $H_{y_{i}}$, while for an unwrapping to happen, $L'_{i+1}$ needs to go to the below $H_{y_{i}}$ as seen in Figure \ref{Figure:onlyunwrap}.\\

 \begin{figure}[h]
\centering
  \includegraphics[width=.9\linewidth]{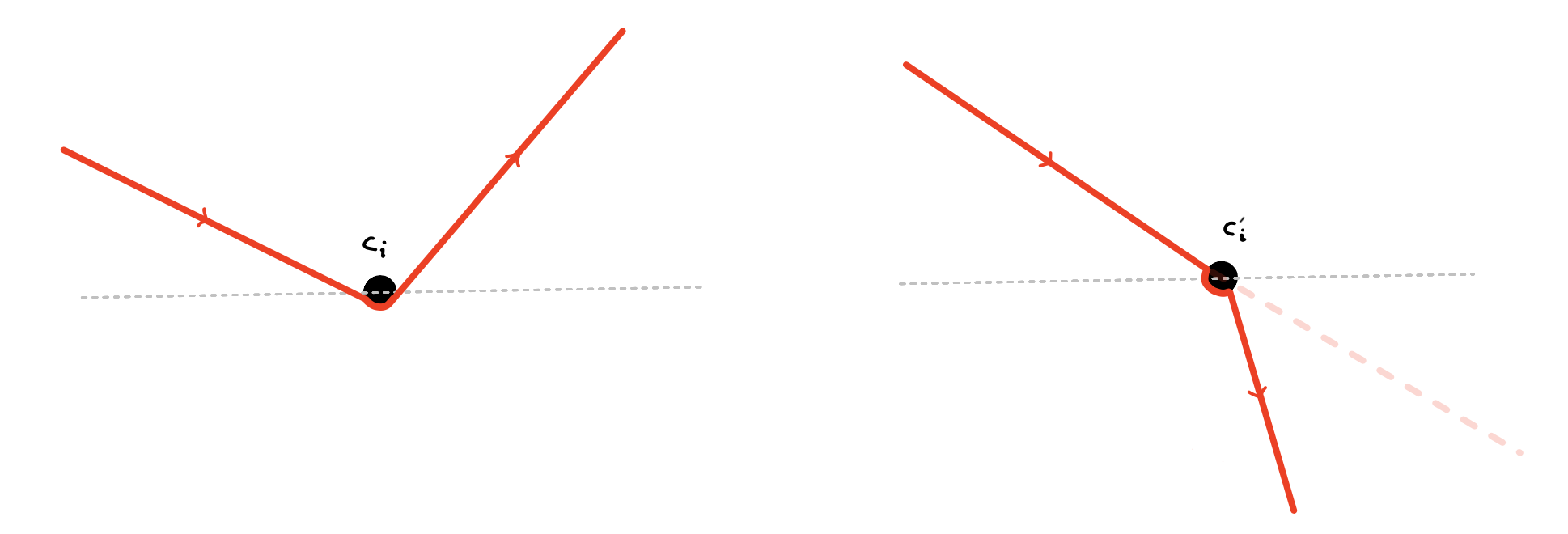}
  \captionof{figure}{Only configuration for an unwrapping when $L_i$ and $L_{i+1}$ are both above $H_{y_{i}}$. This can't happen in meridian shifting.}
  \label{Figure:onlyunwrap}
  \end{figure}

If $L_i$ and $L_{i+1}$ sit on different sides of $H_{y_{i}}$, we have two cases. Without loss of generality, assume that $L_i=q_i \cdot \overrightarrow{m} + p_i \cdot \overrightarrow{\lambda}$ where $q_i,p_i \geq 0$.  Let $\theta_i \in [0,\pi]$ be the angle between $H_{y_{i}}$ and $L_i$. Note that $\text{cot}(\theta)=\frac{q_i}{p_i}$. Similarly define $\theta_{i+1}$ for $L_{i+1}$, and also $\theta'_{i}, \theta'_{i+1}$ for $L'_{i}, L'_{i+1}$. Now, we can define the two cases.\\

\textbf{Case 1:} When $\theta_{i+1} > \theta_i$:\\

Note that unwrapping can only happen when $\theta'_{i+1} < \theta'_{i}$ as seen  in Figure \ref{Figure:onlyunwrap2}. Recall that 
$$L'_i = (q_i+p_i) \cdot \overrightarrow{m} + p_i \cdot \overrightarrow{\lambda} \ , \ L'_{i+1} = (q_{i+1}+p_{i+1}) \cdot \overrightarrow{m} + p_{i+1} \cdot \overrightarrow{\lambda}$$
Note that $p_i,p_{i+1} \geq 0$. Now we have that 
$$\theta'_{i+1} \leq \theta'_{i} \Leftrightarrow \text{cot}(\theta'_{i+1}) \geq \text{cot}(\theta'_{i}) \Leftrightarrow \frac{q_{i+1}+p_{i+1}}{p_{i+1}} \geq \frac{q_i+p_i}{p_i} \Leftrightarrow \frac{q_{i+1}}{p_{i+1}} \geq \frac{q_i}{p_i} $$
$$\Leftrightarrow \text{cot}(\theta_{i+1}) \geq \text{cot}(\theta_{i})  \Leftrightarrow \theta_{i+1} \leq \theta_{i}.$$
This means that an unwrapping can't happen in this case.\\ 

 \begin{figure}[H]
\centering
  \includegraphics[width=.9\linewidth]{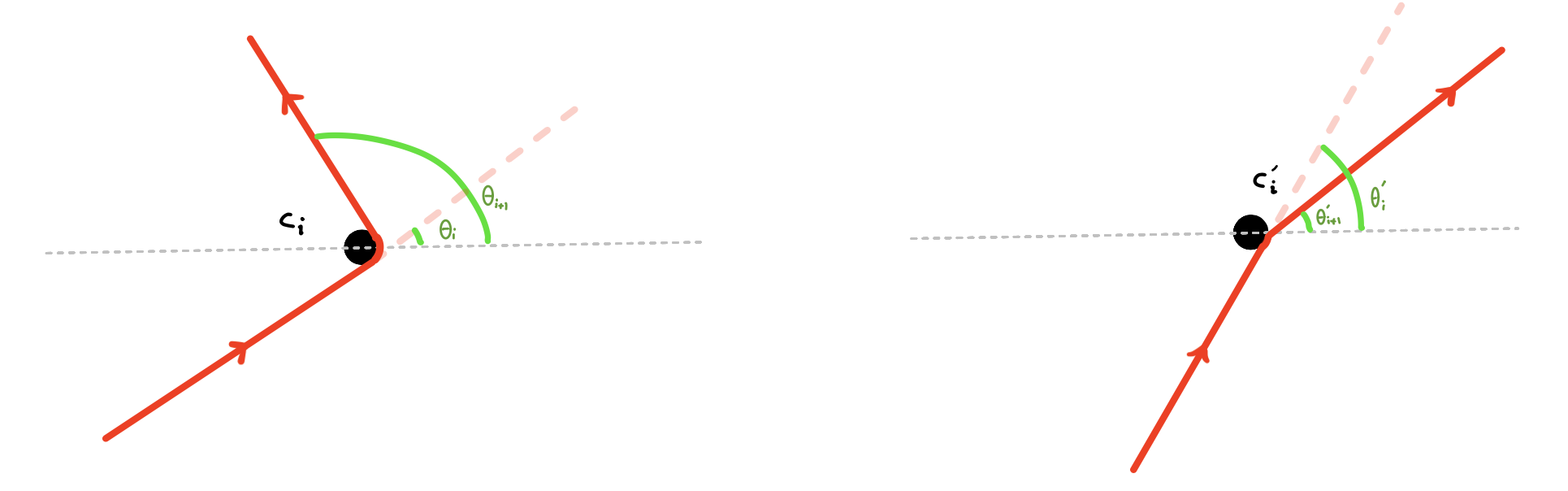}
  \captionof{figure}{Only configuration for an unwrapping in Case 1. As discussed above, this can't happen in meridian shifting. }
  \label{Figure:onlyunwrap2}
  \end{figure}

\textbf{Case 2:} When $\theta_{i+1} < \theta_i$:\\

This case can only happen when the curve wraps around the peg and self-intersects near $c_i$. Otherwise the corner $c_i$ can be discarded with a homotopy of $\gamma$ (i.e. the wrapping isn't tight) which won't happen in a pegvoard diagram. For an unwrapping to happen $L'_i$ needs to rotate with an angle more than $\pi$ under the meridian shifting construction which can't happen. The effect of meridian shifting in this case can be seen in Figures \ref{Figure:Case2-1} and \ref{Figure:Case2-2}.\\

 \begin{figure}[H]
\centering
\begin{minipage}{.5\textwidth}
  \centering
  \includegraphics[width=.6\linewidth]{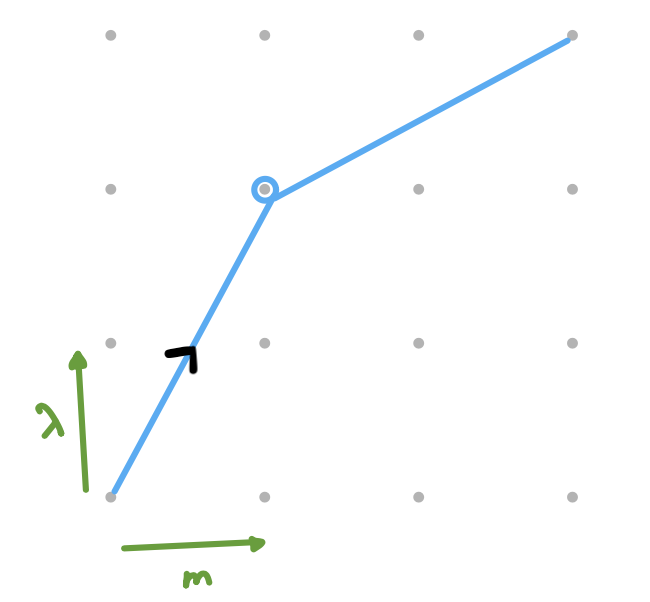}
  \captionof{figure}{$L_i,L_{i+1}$ in Case 2}
  \label{Figure:Case2-1}
\end{minipage}%
\begin{minipage}{.5\textwidth}
  \centering
  \includegraphics[width=.9\linewidth]{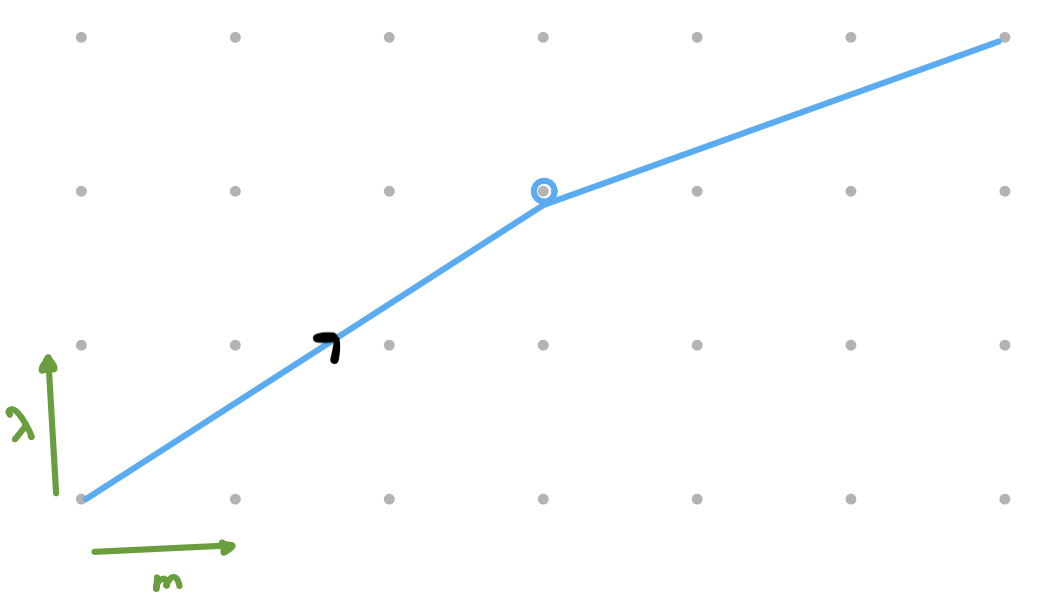}
  \captionof{figure}{$L'_i,L'_{i+1}$ in Case 2}
  \label{Figure:Case2-2}
\end{minipage}
\end{figure}
\end{proof}

Iterated application of Proposition \ref{Proposition:pegboardDehntwist} gives us the following corollary about the behaviour of the immersed curve invariant under twisting.

\begin{coro}\label{Corollary:hfkundertwisting}
There exist $N \in \mathbb{N}$ such that for $m\geq N$, the slopes of (linear segments in a singular pegboard diagram of) $\hfk(K_m)$ are all positive. Furthermore the slopes of $\hfk(K_m)$ converge to zero as $m \rightarrow \infty$. 
\end{coro}
\section{Properties of the box tensor product}\label{Section:Boxtensor}
\subsection{General structure of $\widehat{\text{CFA}}(\mathcal{H}_{K}, z, w) \boxtimes \widehat{\text{CFD}}(\mathcal{H}'_{\frac{1}{m}} , z')$}\label{Subsection:GenstructureofBoxtensor}\hfill\\

This section applies the box tensor product and Theorems \ref{Theorem:pairing} and \ref{Theorem:pairingknots} to examine the twisting problem. As discussed in Section \ref{Section:Immeresedcurve}, the construction of a twist family $\{K_m\}$ is constructed by gluing two bordered $3$-manifolds, $H$ and $H'_{\frac{1}{m}}$. Here $H$ is the solid torus $S^3 \setminus N(c)$ containing $K$. The boundary $\partial H$ is parameterized by two curves: $\mu_{c}$ (the meridian of $c$) and $\lambda_{c}$ (the Seifert longitude of $c$). The second $3$-manifold, $H'_{\frac{1}{m}}$, is similarly a solid torus, but its meridian is homologous to $[\mu_{c}] - m \cdot [\lambda_{c}]$.\\

Consider a doubly-pointed bordered Heegaard diagram $(\mathcal{H}_{K}, z, w)$ representing $(H, K)$. Also note that we computed the type $D$ invariant of $H'_{\frac{1}{m}}$ in Example \ref{Example:twistedmeridian} using the pointed bordered Heegaard diagram $(\mathcal{H}'_{\frac{1}{m}} , z')$. Using Theorems \ref{Theorem:pairing} and \ref{Theorem:pairingknots} we have : 
$$\widehat{CFK}(S^3,K_m) \simeq \widehat{\text{CFA}}(\mathcal{H}_{K}, z, w) \boxtimes \widehat{\text{CFD}}(\mathcal{H}'_{\frac{1}{m}} , z'),$$
$$\widehat{CF}(S^3) \simeq \widehat{\text{CFA}}(\mathcal{H}_{K}, z) \boxtimes \widehat{\text{CFD}}(\mathcal{H}'_{\frac{1}{m}} , z').$$
Let $B_K$ be a basis of $\widehat{\text{CFA}}(\mathcal{H}_{K}, z)$ (and $\widehat{\text{CFA}}(\mathcal{H}_{K}, z, w)$) formed from the disjoint union of $B^{0}_K$ and $B^{1}_K$, where $B^{i}_K$ is a basis for $\widehat{\text{CFA}}(\mathcal{H}_{K}, z) \cdot \iota_{i}$ over $\mathbb{F}_2$. As before, we use $B'_m=\{\eta, \xi_1, \cdots, \xi_m\}$ as the basis of $\widehat{\text{CFD}}(\mathcal{H}'_{\frac{1}{m}} , z')$.\\

We can construct a basis for the box tensor products as the disjoint union $\bm{C^{m}}= C^{\bu} \sqcup C^{\circ}_{1} \sqcup \cdots \sqcup C^{\circ}_{m}$ where: 
$$C^{\bu}= \{x^{0} \otimes \eta \ | \ x^{0} \in B^{0}_K\}, $$
$$C^{\circ}_{i}= \{x^{1} \otimes \xi_i \ | \ x^{1} \in B^{1}_K\} \ \text{for} \ i=1, \cdots, m.$$
In our figures, we arrange the elements of $\bm{C^{m}}$ around a circle as follows. Considering the vertices of the decorated graph in Figure \ref{Figure:1:mDgraph}, we turn each vertex to a box. The box on top, associated to vertex $\eta$, contains all of the elements of $C^{\bu}$, labeled by elements of $B^{0}_K$. The box associated to vertex $\xi_i$, for $i=1,\cdots,m$, contains the elements of $C^{\circ}_{i}$, labeled by elements of $B^{1}_K$. Abusing the notation, we use $C^{\bu}, C^{\circ}_{1}, \cdots C^{\circ}_{m}$ to refer to the associated boxes as well. We refer to $C^{\bu}$ as the \emph{black box}, and to $C^{\circ}_{1}, \cdots C^{\circ}_{m}$ as the \emph{white boxes}. A general schematic of this arrangement can be seen in Figure \ref{Figure:Boxtensorscheme}.\\

\begin{figure}[h]
\centering
\begin{center}
\includegraphics[scale=0.5]{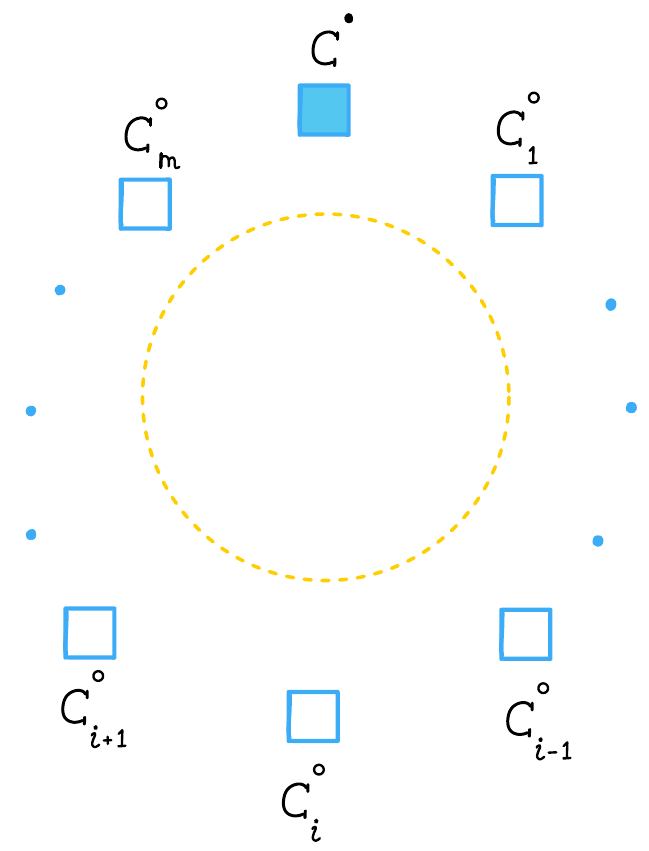}
\end{center}
\caption{The circular arrangement of the boxes in $\bm{C^{m}}$}\label{Figure:Boxtensorscheme}
\end{figure}

Note that both of the chain complexes have the same base vector space and only differ in their differential. Since we are working over $\mathbb{F}_2$, the base set is in one-to-one correspondence with $2^{\bm{C^{m}}}$ as follows:
$$I \subseteq \bm{C^{m}} \longleftrightarrow S_I=\sum_{\nu \in I} \nu \in \widehat{\text{CFA}}(\mathcal{H}_{K}, z, w) \boxtimes \widehat{\text{CFD}}(\mathcal{H}'_{\frac{1}{m}} , z')$$

We give the black and white boxes structure of a cyclic graph by letting the boxes to be the vertices and connecting neighbouring boxes in the arrangement in Figure \ref{Figure:Boxtensorscheme}. This is essentially the underlying undirected graph in Figure \ref{Figure:1:mDgraph}. We denote this graph by $\bm{GC^m}$.\\

As described before the differential of the box tensor product in general is defined by Equation \ref{Equation:Boxtensordifferential}. Similar to our graph-theoretic descriptions of type A and type D structures (but much simpler), we can use directed graphs to describe chain complexes. In our examples, for any elements $\nu , \nu' \in \bm{C^{m}}$ we draw a directed from $\nu$ to $\nu'$, if and only if, (a summand of) $\nu'$ appears in $\partial \nu$.\\

An explicit example is illustrated in Figure \ref{Figure:BoxtensorMazur} and \ref{Figure:BoxtensorMazurU} for $m=3$ and the Mazur pattern $Q$ (using computations of $\widehat{\text{CFA}}(\mathcal{H}_{Q}, z, w)$ and $\widehat{\text{CFA}}(\mathcal{H}_{Q}, z)$ in Example \ref{Example:Mazur}). The differentials can be seen as green arrows in Figures \ref{Figure:BoxtensorMazur} and \ref{Figure:BoxtensorMazurU}.\\

We can also go on and compute the homology. The homology of $$\widehat{\text{CFA}}(\mathcal{H}_{K}, z) \boxtimes \widehat{\text{CFD}}(\mathcal{H}'_{\frac{1}{m}} , z')$$ is just $\widehat{HF}(S^3) = \mathbb{F}_2$ and is generated by $$x_0 \otimes \eta + x_1 \otimes \xi_1 + x_1 \otimes \xi_2 + x_1 \otimes \xi_3.$$

We can also observe that $m \rightarrow \infty$, the differentials of $$\widehat{\text{CFA}}(\mathcal{H}_{K}, z, w) \boxtimes \widehat{\text{CFD}}(\mathcal{H}'_{\frac{1}{m}} , z')$$
eventually become constant; in other words, they \emph{stabilize}. To make such observations precise, we introduce the notions of \emph{inclusion maps} and \emph{nested sequence structures}, which will allow us to formalize the idea of stabilization.\\ 

We define two natural inclusions of the basis sets $\bm{C^{m}} \hookrightarrow \bm{C^{m+1}}$ which will endow the family $\{\bm{C^{m}}\}_{m \in \mathbb{N}}$ with two different nested sequence structures. While this may be a slight abuse of notation, we adopt this convention throughout the paper to streamline the presentation and simplify subsequent statements.\\

The first inclusion is the natural inclusion of ${\Phi'_{m} : \bm{C^{m}} \hookrightarrow \bm{C^{m+1}}}$ for all $m \in \mathbb{N}$, defined as follows: 
\begin{equation}\label{Equation:naturalinclusion}
\begin{cases}
\Phi'_{m} (x^0 \otimes \eta) := (x^0 \otimes \eta) \ \text{for all} \ x^0 \otimes \eta \in C^{\bu}\\
\Phi'_{m} (x^1 \otimes \xi_i) := (x^1 \otimes \xi_i) \ \text{for all} \ x^1 \otimes \xi_i \in C^{\circ}_{i}\ , \ i \in \{1,\cdots,m\}
\end{cases}
\end{equation}

For some of our purposes,, we need another inclusion $\Phi_{m} : \bm{C^{m}} \hookrightarrow \bm{C^{m+1}}$ which is defined as follows:
\begin{equation}\label{Equation:inclusion}
    \begin{cases}
      \Phi_{m} (x^0 \otimes \eta) := (x^0 \otimes \eta) \ \text{for all} \ x^0 \otimes \eta \in C^{\bu}\\
      \Phi_{m} (x^1 \otimes \xi_i) := (x^1 \otimes \xi_i) \ \text{for all} \ x^1 \otimes \xi_i \in C^{\circ}_{i}\ , \ i \leq \frac{m}{2}\\
      \Phi_{m} (x^1 \otimes \xi_i) := (x^1 \otimes \xi_{i+1}) \ \text{for all} \ x^1 \otimes \xi_i \in C^{\circ}_{i}\ , \ i > \frac{m}{2}
    \end{cases}       
\end{equation}
Intuitively, the map $\Phi'_{m}$ comes constructing $\bm{C^{m+1}}$ by adding a new white box $C^{\bu}_{m+1}$ to $\bm{C^{m}}$ after all the other white boxes $C^{\bu}_{1}, \cdots, C^{\bu}_{m+1}$. On the other hand, the map $\Phi_{m}$ comes constructing $\bm{C^{m+1}}$ by adding a new white box to $\bm{C^{m}}$ (exactly) in the middle of its white boxes. We can see a pictorial description of these two maps in Figures \ref{Figure:Inclusion1} and \ref{Figure:Inclusion2} (Also see Figure \ref{Figure:Nested}).\\

\begin{figure}[H]
\centering
\begin{center}
\includegraphics[width=\textwidth]{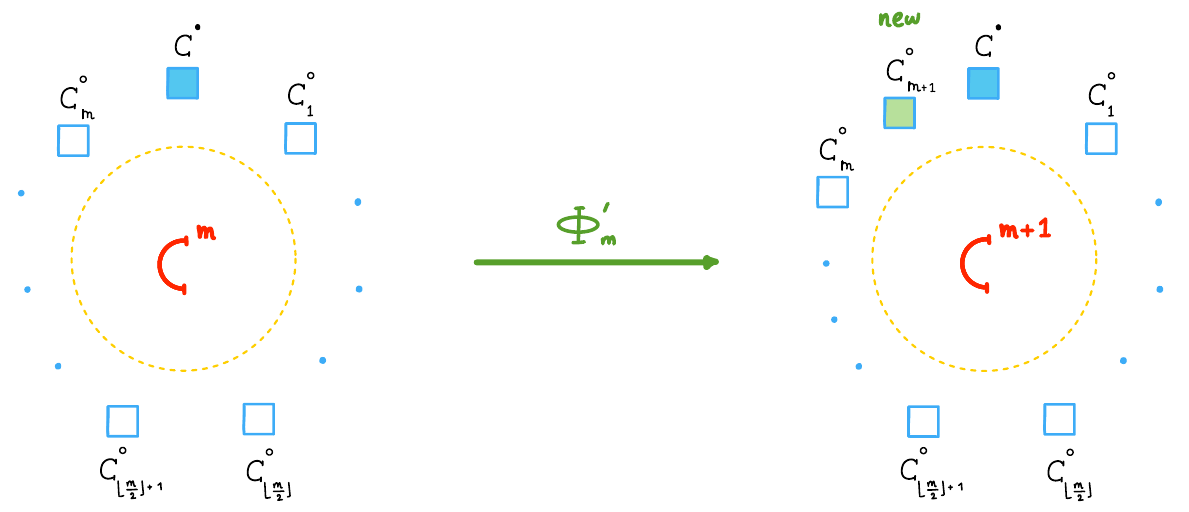}
\end{center}
\caption{The natural inclusion $\Phi'_{m}$ }\label{Figure:Inclusion1}
\end{figure}

\begin{figure}[H]
\centering
\begin{center}
\includegraphics[width=\textwidth]{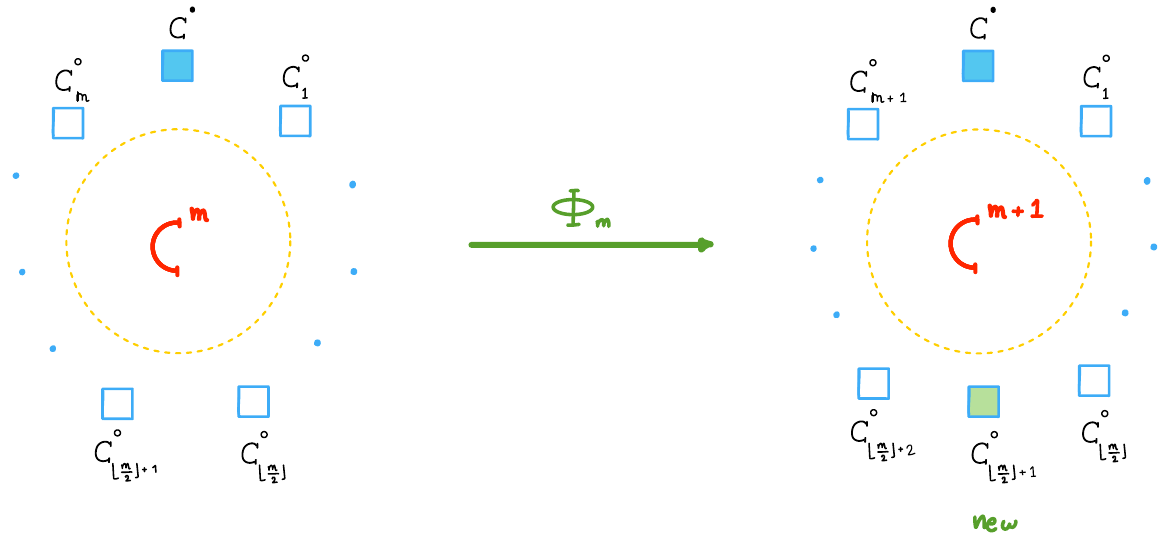}
\end{center}
\caption{The inclusion $\Phi_{m}$}\label{Figure:Inclusion2}
\end{figure}

We now start analyzing the properties of these complexes. We start by examining the differentials in Proposition \ref{Proposition:boxtensordifferential}. We use the graph-theoretic description described earlier.

\begin{prop}\label{Proposition:boxtensordifferential}
There are four types of directed edges in the graph-theoretic description of the chain complexes  $$\widehat{\text{CFA}}(\mathcal{H}_{K}, z) \boxtimes \widehat{\text{CFD}}(\mathcal{H}'_{\frac{1}{m}} , z') \ \text{and} \ \widehat{\text{CFA}}(\mathcal{H}_{K}, z, w) \boxtimes \widehat{\text{CFD}}(\mathcal{H}'_{\frac{1}{m}} , z')$$ as described with respect to the basis $\bm{C^{m}}$:
\begin{enumerate}
    \item The directed edges inside $C^{\bu}$.  For any $x^0 , y^0 \in B^{0}_{K}$ such that $y^0$ appears in $m_1(x^0)$, we will have a directed edge 
    $$x^0 \otimes \eta \longrightarrow y^0 \otimes \eta$$
    \item The directed edges originating from $C^{\bu}$ moving in clockwise direction. For any $x^0 \in B^{0}_{K}$ and $x^1 \in B^{1}_{K}$ such that $x^1$ appears in $$m_{n+2}(x^0 \otimes \rho_{3} \otimes \underbrace{\rho_{23} \otimes \cdots \otimes \rho_{23}}_{n})$$ for some $n \geq 0$, we will have a directed edge 
    $$ x^0 \otimes \eta \longrightarrow x^1 \otimes \xi_{n+1}$$
    \item The directed edges originating from $C^{\bu}$ moving in counterclockwise direction. For any $x^0 \in B^{0}_{K}$ and $x^1 \in B^{1}_{K}$ such that $x^1$ appears in $m_2(x^0 \otimes \rho_1)$, we will have a directed edge 
    $$x^0 \otimes \eta \longrightarrow x^1 \otimes \xi_m$$
    \item The directed edges between white boxes (possibly) moving in clockwise direction. For any $x^1 , y^1 \in B^{1}_{K}$such that $y^1$ appears in $$m_{n+1}(x^1 \otimes \underbrace{\rho_{23} \otimes \cdots \otimes \rho_{23}}_{n})$$ for some $n \geq 0$, we will have a family of directed edges 
    $$ x^1 \otimes \xi_i \longrightarrow y^1 \otimes \xi_{i+n}$$
    for $1 \leq i \leq m-n$. 
\end{enumerate}
\end{prop}
\begin{proof}
    This directly follows from the definition of the differential on the box tensor product and the description of the type $D$ structure $\widehat{\text{CFD}}(\mathcal{H}'_{\frac{1}{m}} , z')$. 
\end{proof}
Figure \ref{Figure:Edgedecomposition} depicts Proposition \ref{Proposition:boxtensordifferential}. Note that the type $(4)$ edges are invariant under circular shift maps (see Subsection \ref{Subsection:ballsandshiftmaps} for a more detailed and more precise explanation).\\

Using Proposition \ref{Proposition:boxtensordifferential}, we can define some notations that will be useful in the future. Let $\nu, \nu'$ be elements of the basis $\bm{C^{m}}$ viewed as vertices in the graph-theoretic description. We use the notation $\nu \xrightarrow{[i]} \nu'$ for an edge of the $i$-th type going from $\nu$ to $\nu'$.\\

For $i\in \{1,2,3,4\}$, let $E_{i}(\nu,\nu')$ be the set of the edges of type $(i)$ which start in $\nu$ and end in $\nu'$. Furthermore, extending this notation, let $E_{i}(\nu,\cdot)$ (resp.~$E_{i}(\cdot ,\nu)$) be the set of the edges of type $(i)$ which start (resp.~end) in $\nu$.\\

Now we can write 
$$\partial (\nu) = \sum_{1\leq i \leq 4} \ \sum_{\nu' \in \bm{C^{m}}}|E_{i}(\nu,\nu')|\cdot \nu' \mod 2  $$
We can also define 
$$\partial_{i}(\nu) : = \sum_{\nu' \in \bm{C^{m}}}|E_{i}(\nu,\nu')|\cdot \nu' \mod 2,$$ and hence we will have 
$\partial = \partial_1+\partial_2+\partial_3+\partial_4$. However, note that the maps $\partial_{i}$ are not necessarily differentials. 

\begin{figure}[ht]
\centering
\begin{center}
\includegraphics[width=0.8\textwidth]{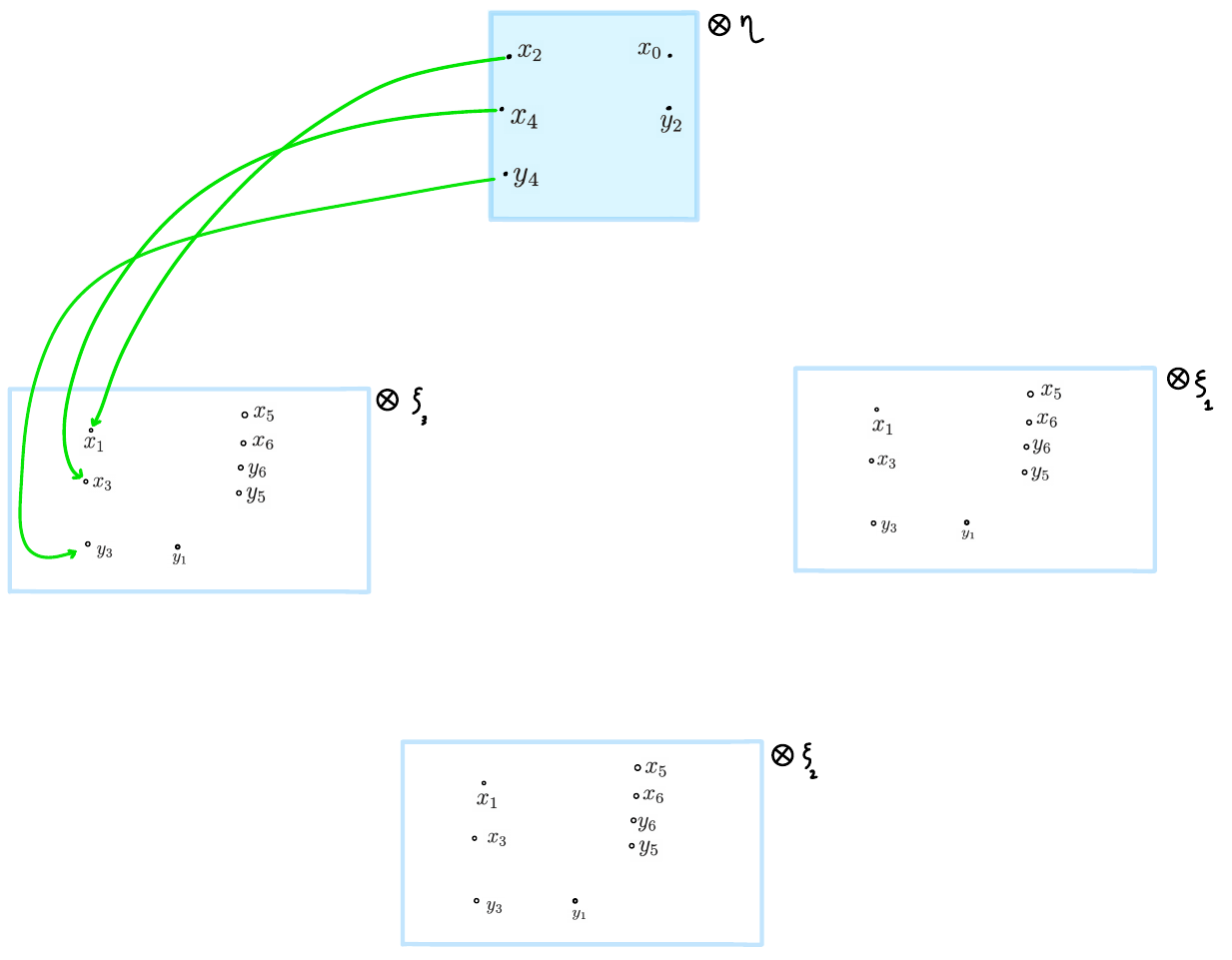}
\end{center}
\caption{The complex $\widehat{\text{CFA}}(\mathcal{H}_{Q}, z, w) \boxtimes \widehat{\text{CFD}}(\mathcal{H}'_{\frac{1}{3}} , z')$}\label{Figure:BoxtensorMazur}
\end{figure}

\begin{figure}[ht]
\centering
\begin{center}
\includegraphics[width=0.8\textwidth]{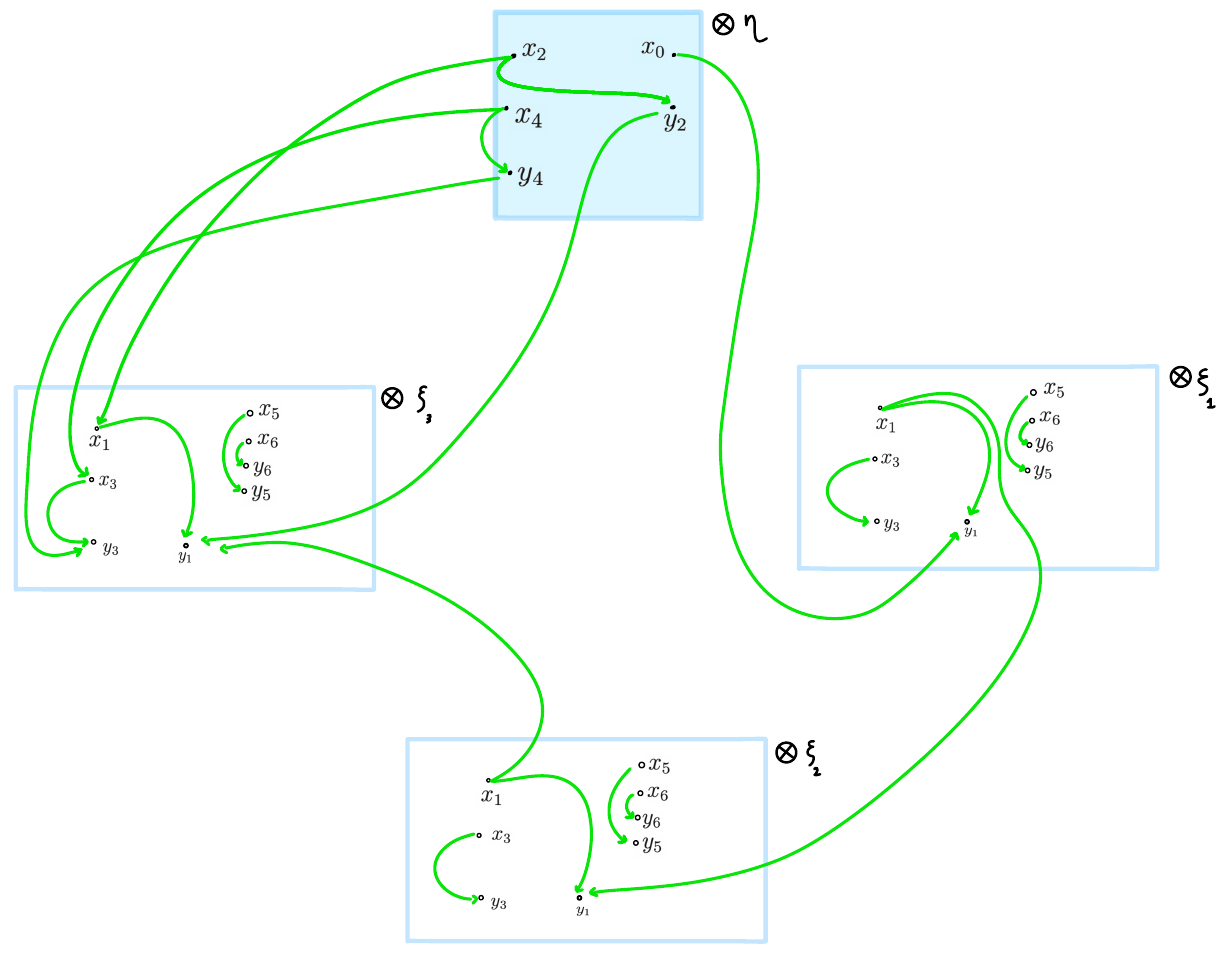}
\end{center}
\caption{The complex $\widehat{\text{CFA}}(\mathcal{H}_{Q}, z) \boxtimes \widehat{\text{CFD}}(\mathcal{H}'_{\frac{1}{3}} , z')$}\label{Figure:BoxtensorMazurU}
\end{figure}

\begin{figure}[ht]
\centering
\begin{center}
\includegraphics[width=0.6\textwidth]{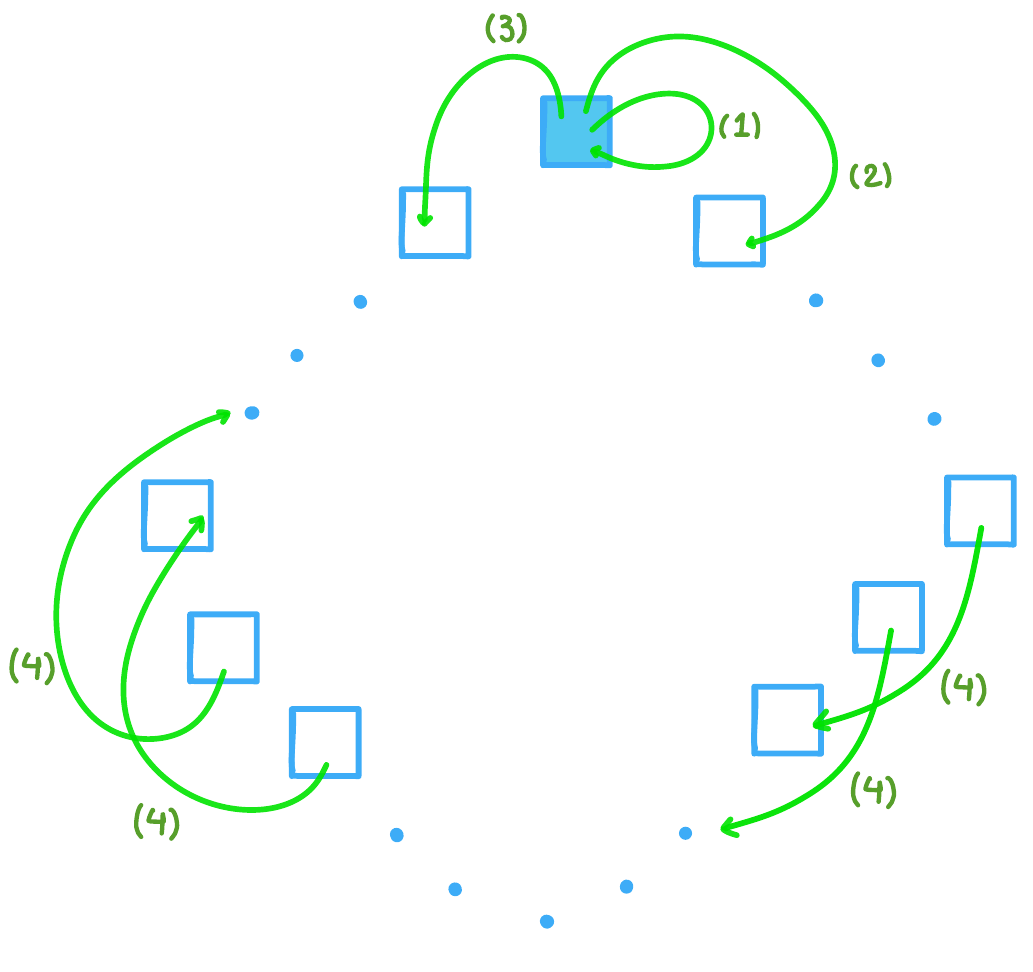}
\end{center}
\caption{The decomposition of edges to four types described in Proposition \ref{Proposition:boxtensordifferential}}\label{Figure:Edgedecomposition}
\end{figure}

\subsection{Gradings in $\widehat{\text{CFA}}(\mathcal{H}_{K}, z, w) \boxtimes \widehat{\text{CFD}}(\mathcal{H}'_{\frac{1}{m}} , z')$}\label{Subsection:Gradingsinboxtensor}\hfill\\

We need to start by examining the gradings. We start by Lemma \ref{Lemma:Gradingsubgroup} about the grading of $\widehat{\text{CFA}}(\mathcal{H}_{K}, z, w)$.\\

Recall that the solid torus $H$ is defined as $S^3 \setminus N(c)$ for unknot $c$. We parameterized $\partial H$ with  the pair of curves $(\mu_{c} , \lambda_{c})$ where $\lambda_{c}$ is a Seifert longitude of $c$ and a meridian of $H$. 

\begin{lemm}\label{Lemma:Gradingsubgroup}
For any choice of reference generator $\bm{z_0} \in \widehat{\text{CFA}}(\mathcal{H}_{K}, z, w)$, the associated subgroup $P(\bm{z_0})$ is generated by an element $(M; 0,1 ; \omega) \in \widetilde{G}$, where $M \in \frac{\mathbb{Z}}{2}$ and $\omega = - \lk(K,c)$.
\end{lemm}
\begin{proof}
As mentioned in \cite{Hanselman2018HeegaardFH}, the $\text{\emph{Spin}}^{\mathbb{C}}$ component of the generator of $P(\bm{z_0})$ is determined by the topology of the bordered $3$-manifold. This means that the $\text{\emph{Spin}}^{\mathbb{C}}$ component will be equal to the one computed for the special case of Mazur pattern in Example $\ref{Example:Mazur}$, and hence is $(0, \pm 1)$.\\

We need to elaborate this a bit further. As we mentioned before $P(\bm{z_0})$ comes from the grading of periodic domains in $\mathcal{H}_{K}$. As proved in \cite[Lemma~4.20]{Lipshitz2008BorderedHF}, the periodic domains are in one-to-one correspondence with $H_2(H, \partial H)$. In fact, for any periodic domain $B$, one can construct a surface $T_B \subset H$ with $\partial T_B \subset \partial H$. The $\text{\emph{Spin}}^{\mathbb{C}}$ of elements of $P(\bm{z_0})$ come from the expansion of $[\partial T_B] \in H_1(\partial H)$ in terms of the parametrizing curves. In our case, $H_2(H, \partial H) \cong \mathbb{Z}$, and it is generated by the meridian disk. The boundary of the meridian disk in the first homology is clearly $0 \cdot [\mu_{c}] + 1 \cdot [\lambda_{c}]$, and as a result the $\text{\emph{Spin}}^{\mathbb{C}}$ component of the generator will be $(0,1)$.\\

We now need to show that the Alexander component of the generator is equal to $-\lk(K,c)$. The Alexander component of grading of a periodic domain $B$ is defined as $n_w(B)-n_z(B)$. Based on the construction of $T_B$, we have 
$$n_w(B)-n_z(B) = -(K \cdot T_B) = - \lk(K,c).$$
\end{proof}

As we discussed before the grading functions 
$$gr_{K} : \widehat{\text{CFA}}(\mathcal{H}_{K}, z, w) \rightarrow P(\bm{z_0}) \backslash \widetilde{G}$$ 
$$gr_{m}: \widehat{\text{CFD}}(\mathcal{H}'_{\frac{1}{m}}, z') \rightarrow G / P(\eta) $$
induce a grading function on the box tensor product: 
$$gr_{K_m} : \widehat{\text{CFA}}(\mathcal{H}_{K}, z, w) \boxtimes \widehat{\text{CFD}}(\mathcal{H}'_{\frac{1}{m}}, z') \rightarrow P(\bm{z_0}) \backslash \widetilde{G} / (P(\eta);0).$$
We examine the behaviour of $gr_{K_m}$ and the relative Masolv and Alexander gradings as $m \rightarrow \infty$ in Lemmas \ref{Lemma:Gradingchangebym} and \ref{Lemma:Gradingchangebyshift}. Let $h_m$ and $a_m$ be the Maslov and Alexander gradings of  $\widehat{\text{CFA}}(\mathcal{H}_{K}, z, w) \boxtimes \widehat{\text{CFD}}(\mathcal{H}'_{\frac{1}{m}} , z')$.

\begin{lemm}\label{Lemma:Gradingchangebym}
    Let $n \in \mathbb{N}$ be an integer. Consider two elements ${\nu_1, \nu_2 \in \bm{C^{n}}}$ such that for each $m \geq n$, the double cosets $gr_{K_m}(\nu_1)$ and $gr_{K_m}(\nu_2)$ has representatives which doesn't depend on $m$. In other words, there are elements 
    $$g_{i}=(\alpha_i; \beta_i, \gamma_i ; \delta_i) \in \widetilde{G} \ \  \text{for} \ \  i \in \{1,2\},$$ such that for all $m \geq n$, the double coset $gr_{K_m}(\nu_i)$ has representative $g_{i}$. Then we have: 
    \begin{gather*}
     h_m(\nu_1) - h_m(\nu_2) = ((\beta_2-\beta_1)(M-\frac{1}{2}-\beta_1-\beta_2)) \cdot m + \\
     (\gamma_2-\gamma_1)M + (\alpha_1 - \alpha_2) + \frac{1}{2}(\beta_2-\beta_1) - (\beta_2 - \beta_1) \gamma_1 - (\gamma_2 - \gamma_1)\beta_2,
    \end{gather*}
    and 
    \begin{gather*}
        a_m(\nu_1) - a_m(\nu_2) = ((\beta_2-\beta_1)\omega) \cdot m + (\delta_1 - \delta_2) - (\gamma_1 - \gamma_2)\omega. 
    \end{gather*}
\end{lemm}
\begin{proof}
    Let $p_K$ and $p_m$ respectively deonte the generators of $P(\bm{z_0})$ and $P(\eta)$. To compute the relative Maslov and Alexander grading we need to find a representative of the double coset $[g_1]$ with $\text{\emph{Spin}}^{\mathbb{C}}$ component equal to $(\beta_2, \gamma_2)$. This representative is clearly 
    $$p_{K}^{\gamma_2-\gamma_1+(\beta_2-\beta_1)m} \cdot g_1 \cdot p_{m}^{\beta_2-\beta_1}$$
Computing the Maslov and Alexander component of this representative verifies the stated formulas.
\end{proof}
\begin{lemm}\label{Lemma:Gradingchangebyshift}
Let $x^1$ be an element of $B^{1}_{K}$ with $(a;b,c;d) \in \widetilde{G}$ as a representative of the grading coset $gr_K(x^1)$. Let $i,j$ be positive integers. Then 
$$x^1 \otimes \xi_i , x^1 \otimes \xi_j  \in \bm{C^{m}},$$
for any $m \geq \max \{i,j\}$, and their relative gradings are as follows :
$$h_m(x^1 \otimes \xi_j) - h_m(x^1 \otimes \xi_i) = (j-i)(M-2b-\frac{1}{2}),$$
$$a_m(x^1 \otimes \xi_j) - a_m(x^1 \otimes \xi_i) = (j-i)\omega.$$
\end{lemm}
\begin{proof}
Based on the definition of the induced grading on box tensor product we have:
\begin{gather*}
gr_{K_m}(x^1 \otimes \xi_i) = [gr_K(x^1) \cdot gr_m(\xi_i)]=[(a;b,c;d) \cdot (-\frac{1}{2} ; \frac{1}{2} , - \frac{(2i-1)}{2}; 0)]\\
 = [(a-\frac{1}{2} - \frac{(2i-1)}{2} b - \frac{c}{2} ; b + \frac{1}{2}, c - \frac{(2i-1)}{2} ;d )].
\end{gather*}
Similarly we have: 
$$gr_{K_m}(x^1 \otimes \xi_j) = [(a-\frac{1}{2} - \frac{(2j-1)}{2} b - \frac{c}{2} ; b + \frac{1}{2}, c - \frac{(2j-1)}{2} ;d )]$$
Here the $\text{\emph{Spin}}^{\mathbb{C}}$ component of the two representatives have the same first element, any we can find representatives with equal $\text{\emph{Spin}}^{\mathbb{C}}$ components as follows: 
\begin{gather*}
gr_{K_m}(x^1 \otimes \xi_i)  = [p_K^{(i-j)} \cdot (a-\frac{1}{2} - \frac{(2i-1)}{2} b - \frac{c}{2} ; b + \frac{1}{2}, c - \frac{(2i-1)}{2} ;d )]
\end{gather*}
Computing the Maslov and Alexander component of this representative verifies the stated formulas.
\end{proof}

As it can be seen in Lemma \ref{Lemma:Gradingchangebyshift}, the second component of a representative of the grading coset $gr_K(x^1)$ plays a role in the computation of relative Maslov gradings. Lemma \ref{Lemma:secondcomponentmod2} presents a useful fact about this component.\\

\begin{lemm}\label{Lemma:secondcomponentmod2}
   For $j\in \{1,2\}$, let $x_j$ be an element of $B^{1}_{K}$. Also let ${(a_j;b_j,c_j;d_j) \in \widetilde{G}}$ be a representative of the grading coset $gr_K(x_j)$. Then we have $b_1 -b_2 \in \mathbb{Z}$.
\end{lemm}
\begin{proof}
    We use the graph-theoretic description of type A invariants to prove this fact. Let $\Gamma_A$ be the decorated graph representing $\widehat{\text{CFA}}(\mathcal{H}_{K}, z, w)$.\\
     
     For any (not necessarily directed) path $P$ in $\Gamma_A$ connecting $x_1$ and $x_2$, let $P_{l}$ be the subset of edges of $P$ with label $l$ for $l \in \{\varnothing, 1,2,3,32,21,321\}$. Based on the grading rules described in Figure \ref{Figure:typeAgrading}, we only need to show that 
     $$\frac{|P_{1}|+|P_{2}|+|P_{3}|+|P_{321}|}{2} \in \mathbb{Z}.$$
     Now note that $P_{2}$ is the subset of edges in $P$ which go from 
 a $\circ$-labeled vertex to a $\bu$-labeled one. On the other hand, $P_{1} \cup P_{2} \cup P_{3}$ is the subset of edges in $P$ which go from 
 a $\bu$-labeled vertex to a $\circ$-labeled one. As a result, we have:
    $$|P_{2}| = |P_{1}|+|P_{3}|+|P_{321}|.$$
This comes from the fact that $x_1$ and $x_2$ are in $B^{1}_{K}$, and hence both are $\circ$-labeled. Finally, we can deduce that 
$$\frac{|P_{1}|+|P_{2}|+|P_{3}|+|P_{321}|}{2} = \frac{2|P_{2}|}{2} = |P_{2}| \in \mathbb{Z},$$
which finishes our proof.
\end{proof}
A similar argument gives us Corollary \ref{Corollary:secondcomponentmod2-two}
\begin{coro}\label{Corollary:secondcomponentmod2-two}
Let $x_1 \in B^{1}_{K}$ and $x_0 \in B^{0}_{K}$ be elements of $B_{K}$. Assume that 
$$gr_K(x_1) = [(a_1;b_1,c_1;d_1)] \ \  \text{and} \ \ gr_K(x_0) = [(a_0;b_0,c_0;d_0)].$$
Then we have $b_1 -b_0 + \frac{1}{2} \in \mathbb{Z}$.
\end{coro}
\subsection{Minimality and homogeneity of cycles }\label{Subsection:Minimalandhomogenous}\hfill\\

Now we examine the homology of the chain complex 
$$ \widehat{\text{CFA}}(\mathcal{H}_{K}, z, w) \boxtimes \widehat{\text{CFD}}(\mathcal{H}'_{\frac{1}{m}} , z')).$$ 
In the rest of this chapter we assume $\lk(K,c) = -\omega \neq 0$, which is a condition repeated in all of the results we are trying to prove.\\

Recall that the elements of the chain complex are of the form $S_I$ for $I \subseteq \bm{C^m}$. We call $S_I$ \emph{homogeneous} if all of the elements $\nu \in I$ have the same Maslov and Alexander grading (or equivalently same $gr_{K_m}$).\\

We also define a minimality property for cycles in the chain complex as follows. The inclusion partial order on $2^{\bm{C^m}}$ induces a partial order on elements of the chain complex. We call a cycle element $S_I$ \emph{minimal}, if it is a minimal element among cycles with respect to this partial order i.e. if for all non-empty $I' \subset I$, we have $\partial S_{I'} \neq 0$. It is easy to see that a minimal cycle needs to be homogeneous.\\

Similar to the cycles, we can define the homogeneity and minimality properties for the homology classes. We call a homology class \emph{minimal} (resp.~\emph{homogeneous}) if it has a minimal (resp.~homogeneous) representative.\\

We extend these properties to a (linear) basis of the homology as well. Abusing notation slightly, we refer to a collection of representative cycles
$$\bm{S^m}=\{S_{I_1}, \cdots, S_{I_{d_m}}\}$$ 
as a basis for the homology if the corresponding homology classes form a basis of the homology vector space.\\

Let $\bm{S^m}=\{S_{I_1}, \cdots, S_{I_{d_m}}\}$ be a basis for the homology. We call $\bm{S^m}$ \emph{homogeneous}, if $S_{I_j}$ is homogeneous for $ j = 1, \cdots, d_m$.\\

We call $\bm{S^m}$ \emph{minimal} if it is minimal among all of the bases of the homology i.e. if for all $j \in \{1, \cdots, d_m\}$ and all subsets $I'_j \subset I_j$, the set of chain complex elements 
$\bm{S^m} \setminus \{S_{I_j}\} \cup \{S_{I'_j}\}$ is not a basis for homology.\\

We can always assume our basis to be homogenous and minimal. Lemma \ref{Lemma:minimalbases} states a useful fact about minimal bases.\\

\begin{lemm}\label{Lemma:minimalbases}
Let $\bm{S^m}=\{S_{I_1}, \cdots, S_{I_{d_m}}\}$ be a minimal basis for homology. Then for all  $j \in \{1, \cdots, d_m\}$, the chain complex element $S_{I_j}$ is minimal. As a result, a minimal basis is also homogeneous. 
\end{lemm}
\begin{proof}
We prove this statement by contradiction. Assume there exist $k$, and subset $I'_k \subset I_k$ such that $S_{I'_k}$ is a cycle. Let $I''_k$ be $I_k \setminus I'_k$. Note that 
$$S_{I_k} = S_{I'_{k}} + S_{I''_{k}} \Rightarrow \partial S_{I_k} = \partial S_{I''_{k}} = 0.$$
The cycle $S_{I'_k}$ defines a homology class $[S_{I'_k}]$ which can be written as a sum of elements of $\bm{S^m}$ i.e. there exists $J \subseteq \{1, \cdots, d_m\}$ such that: 
$$S_{I'_k} = \sum_{j \in J} S_{I_j}.$$
We have two cases now: 
\begin{enumerate}
  \item If $k \notin J$, then we can claim that $\bm{S^m} \setminus \{S_{I_k}\} \cup \{S_{I''_k}\}$ will be a basis for homology. This is due to the fact that 
  $$[S_{I_k}] - [S_{I''_k}] =[S_{I'_k}] = \sum_{j \in J} [S_{I_j}].$$
  This contradicts minimality of $\bm{S^m}$. 
  \item If $k \in J$, then we have 
  $$[S_{I'_k}] - [S_{I_k}] = \sum_{j \in J \setminus \{j\}} [S_{I_j}].$$
  As a result, $\bm{S^m} \setminus \{S_{I_k}\} \cup \{S_{I'_k}\}$ will be a basis for homology which contradicts minimality of $\bm{S^m}$. 
\end{enumerate}
\end{proof}

\subsection{Length and size of minimal cycles}\label{Subsection:lengthsizecycles}\hfill\\

Now we turn our attention to properties of minimal cycles. Let $N_A := |B_K|$ be the dimension of $\widehat{\text{CFA}}(\mathcal{H}_{K}, z, w)$ over $\mathbb{F}_2$. Also for now assume that $\widehat{\text{CFA}}(\mathcal{H}_{K}, z, w)$ is bounded. As mentioned before, this is not a restrictive condition. The boundedness assumption means that there exist an integer $L_A$ such that the maps
$$m_{n+1} : \widehat{\text{CFA}}(\mathcal{H}_{K}, z, w) \otimes \A^{n} \rightarrow \widehat{\text{CFA}}(\mathcal{H}_{K}, z, w)$$
are zero for $n \geq L_A$.\\

Now we define a \emph{length} property for elements $S_{I} \in \widehat{\text{CFA}}(\mathcal{H}_{K}, z, w)$. The length of an elements $S_{I}$ is defined as the length of the smallest path of boxes in the graph $\bm{GC^m}$ which contains $I$. We give an equivalent definition in the following paragraphs.\\

First, we can define a distance function on the boxes using the metric of the graph $\bm{GC^m}$. This induces a distance function, denoted by $\dist$, on elements of $\bm{C^m}$ where the distance between two elements of $\bm{C^m}$ is defined to be the distance of the boxes they belong to. We can also define a Hausdorff distance, denoted by $\dist_{H}$, for subsets $I \subseteq \bm{C^m}$ in the standard way. \\

Second, we can order the black and white boxes $C^{\bu}, C^{\circ}_{1}, \cdots C^{\circ}_{m}$ in the clockwise order starting from the black box. This induces a partial order on the elements of $\bm{C^{m}}$. We call this partial order, the \emph{clockwise order} and denote it by $\prec$.\\

Let $\nu_0,\cdots,\nu_{n-1}$ be the elements of $I$ ordered with respect to this partial order. In other words, for each $0 \leq i \leq j \leq n-1$, the element $\nu_j$ doesn't belong to a box coming before the box containing element $\nu_i$. In the rest of this discussion, the indices of elements $\nu_i$ are taken modulo $n$.\\

Finally, we can express the length as follows: 
\begin{equation}\label{Equation:length}
 \len(S_I) = \min \Bigl\{ \sum_{i=0}^{n-2} \dist(\nu_{t+i}, \nu_{t+i+1}) \ | \ t \in \{0,\cdots,n-1\} \Bigr\}.   
\end{equation}
Note that the diameter of $I$ is clearly a lower bound for $\len(S_I)$.\\


In a series of lemmas, we are going to explore properties of chain complex and its minimal cycles. We frequently use the graph-theoretic description of the complex.\\

Lemma \ref{Lemma:sizehomogenous} restricts the size of a homogeneous element. Lemma \ref{Lemma:differentialorder} talks about the relationship of the differentials and the clockwise order and distance function.  Lemma \ref{Lemma:gapsinminimal} gives a restriction on the gaps between the elements of a minimal cycle. Lemma \ref{Lemma:lengthminimal} gives an upper bound on the length of minimal cycles.

\begin{lemm}\label{Lemma:sizehomogenous}
 Assume $S_{I}$ is a homogeneous element. Then we have $|I| \leq N_A$.
\end{lemm}
\begin{proof}
    We prove this by contradiction. Assume that $|I| > N_A$. As a result
    $$|I \cap \ (\underset{1 \leq i \leq m}{\bigcup}C^{\circ}_{i})| > N_A - |C^{\bu}| = |B^{1}_{K}|.$$
    In other words, $I$ has more than $N_A$ elements in the white boxes, which in turn means there exist $x^1 \in B^{1}_{K}$ and $1 \leq j_1 < j_2 \leq m$ such that 
    $$x^1 \otimes \xi_{j_1}, x^1 \otimes \xi_{j_2} \in I. $$
    Now using Lemma \ref{Lemma:Gradingchangebyshift}, we have $$a_m(x^1 \otimes \xi_{j_2})-a_m(x^1 \otimes \xi_{j_1}) = (j_2-j_1)\omega \neq 0$$
    This constradicts the homogeniety condition. 
\end{proof}

\begin{lemm}\label{Lemma:differentialorder}
Assume that there is a directed edge $e=(\nu, \nu')$ in the graph-theoretic description of the box tensor product. Then we have $\nu \preceq \nu'$ and 
$$\dist(\nu,\nu') < L_A.$$
We call the distance $\dist(\nu,\nu')$, the length of the edge $e$.
\end{lemm}
\begin{proof}
    This follows from Proposition \ref{Proposition:boxtensordifferential}. All of the edges that doesn't originate from $C^{\bu}$ move in the clockwise direction and they don't lead to or pass over the $C^{\bu}$. This gives us that $\nu \preceq \nu'$. Furthermore due to the boundedness condition, the directed edge $(\nu, \nu')$ comes from a map $m_n$ for $n \leq L_A$ which means that $\dist(\nu,\nu') \leq L_A-1$.  
\end{proof}

\begin{lemm}\label{Lemma:gapsinminimal}
     Assume $S_{I}$ is a minimal cycle. Order the elements of $I$ with respect to the  clockwise order as we described before and let the result be $\nu_0, \cdots, \nu_{n-1}$. Then there exist at most one $i\in \{0, \cdots, n-1\}$  s.t. 
$$\dist(\nu_i, \nu_{i+1}) > L_A.$$

\end{lemm}

\begin{proof}
 We prove this by contradiction. Assume that there exist ${0 \leq j_1 < j_2 \leq n-1}$ such that $\dist(\nu_{j_{l}}, \nu_{j_{l+1}}) > L_A$ for $l=1,2$.\\

 Now consider the set $I' = \{\nu_{j_{1}+1},\cdots, \nu_{j_2}\}$. This is clearly an interval in the ordered set $I$. Note that based on the definition of the clockwise order and assumption on $\dist(\nu_{j_1}, \nu_{j_1+1})$, we have that $\nu_{j_1+1} \in C^{\circ}_{l}$ for some $l > L_A$. As a result $I' \cap C^{\bu} = \varnothing$. Let $I''$ denote the complement of $I'$ in $I$.\\
 
Consider elements $S_{I'}$ and $S_{I''}$. We have that
$$I = I' \sqcup I'' \Rightarrow S_{I} = S_{I'} + S_{I''} \Rightarrow \partial S_{I} = \partial S_{I'} + \partial S_{I''}$$
We know that there exist subsets $J' , J'' \subseteq \bm{C^m}$ such that 
$$\partial S_{I'} = S_{J'} \ \text{and} \ \partial S_{I''} = S_{J''},$$
and as a result $$S_{J'} + S_{J''} = \partial S_{I} = 0.$$
We are going to prove that $J' \cap J'' = \varnothing$, which means that there won't be any cancellations in the sum $S_{J'} + S_{J''}$ i.e. 
$$S_{J'} + S_{J'' } = S_{J' \sqcup J''}.$$
But this can only happen when $J' = J'' = \varnothing$, meaning that both $S_{I'}$ and $S_{I''}$ are cycles. This contradicts the minimality of $S_I$.\\

To prove that $J \cap J' = \varnothing$, we will prove that the edges originating from $I'$ and $I''$ can not have the same ends. We start by decomposing $I''$ as $I''_1 \sqcup I''_2$. Let $I''_1$ consist of all of the elements of $I''$ which come before $I'$ in the clockwise order. Let $I''_2 := I'' \setminus I$ which are the elements of $I'' $ coming after $I'$. Note that $I''_1$ and $I''_2$ are both intervals.\\

We start by considering $I''_2$. Assume that we have edges $e'$ and $e''$ originating from $I'$ and $I''_2$, which have the same end. As discussed in Lemma \ref{Lemma:differentialorder}, both of the edges move in the clockwise direction. As a result, the edge $e'$ must pass over the gap between the last element of $I'$, i.e. $\nu_{j_2}$, and the first element of $I''$, i.e. $\nu_{j_2+1}$. This contradicts Lemma \ref{Lemma:differentialorder}, as in this case length of $e'$ will be greater than $\dist(\nu_{j_2},\nu_{j_2+1})$.\\

The same argument works for clockwise edges originating from $I''_1$. We only need to consider the possible counterclockwise edges originating from $I''_1 \cap C^{\bu}$. Let $e''$ be such as edge. As we discussed in Proposition \ref{Proposition:boxtensordifferential}, the edge $e''$ ends in $C^{\circ}_{m}$ and has length one. As a result, the size of the gap between the last element of $I'$ and the end of $e''$ is at least $\dist(\nu_{j_2}, \nu_{j_2+1})-1$ which is still bigger than the bound on the length of edges presented in Lemma \ref{Lemma:differentialorder}. This completes our proof.\\

This proof is illustrated in Figure \ref{Figure:gapslemma}. 
\end{proof}

\begin{lemm}\label{Lemma:lengthminimal}
 Assume $S_{I}$ is a minimal cycle. Then we have $\len(S_{I}) \leq N_{A}L_{A}$.
\end{lemm}
\begin{proof}
As before we order the elements of $I$ and assume that the resulting ordered set is $\nu_0, \cdots, \nu_{n-1}$. Choose $t \in \{0, \cdots, n-1\}$ such that $\dist(\nu_{t},\nu_{t+1})$ is equal to $$\max \{\dist(\nu_j,\nu_{j+1}) \ | \ j \in \{0, \cdots, n-1\}\}. $$ 
Now we have that 
$$\len(S_I) \leq \sum_{i=0}^{n-2} \dist(\nu_{t+1+i},\nu_{t+2+i}) \leq N_{A}L_{A},$$
where the final inequality follows from Lemma \ref{Lemma:sizehomogenous} and Lemma \ref{Lemma:gapsinminimal}.
\end{proof}

We explored some of essential properties of cycles in chain complex $$\widehat{\text{CFA}}(\mathcal{H}_{K}, z, w) \boxtimes \widehat{\text{CFD}}(\mathcal{H}'_{\frac{1}{m}} , z').$$ 

\begin{figure}[h]
    \centering
    \includegraphics[width=0.7\linewidth]{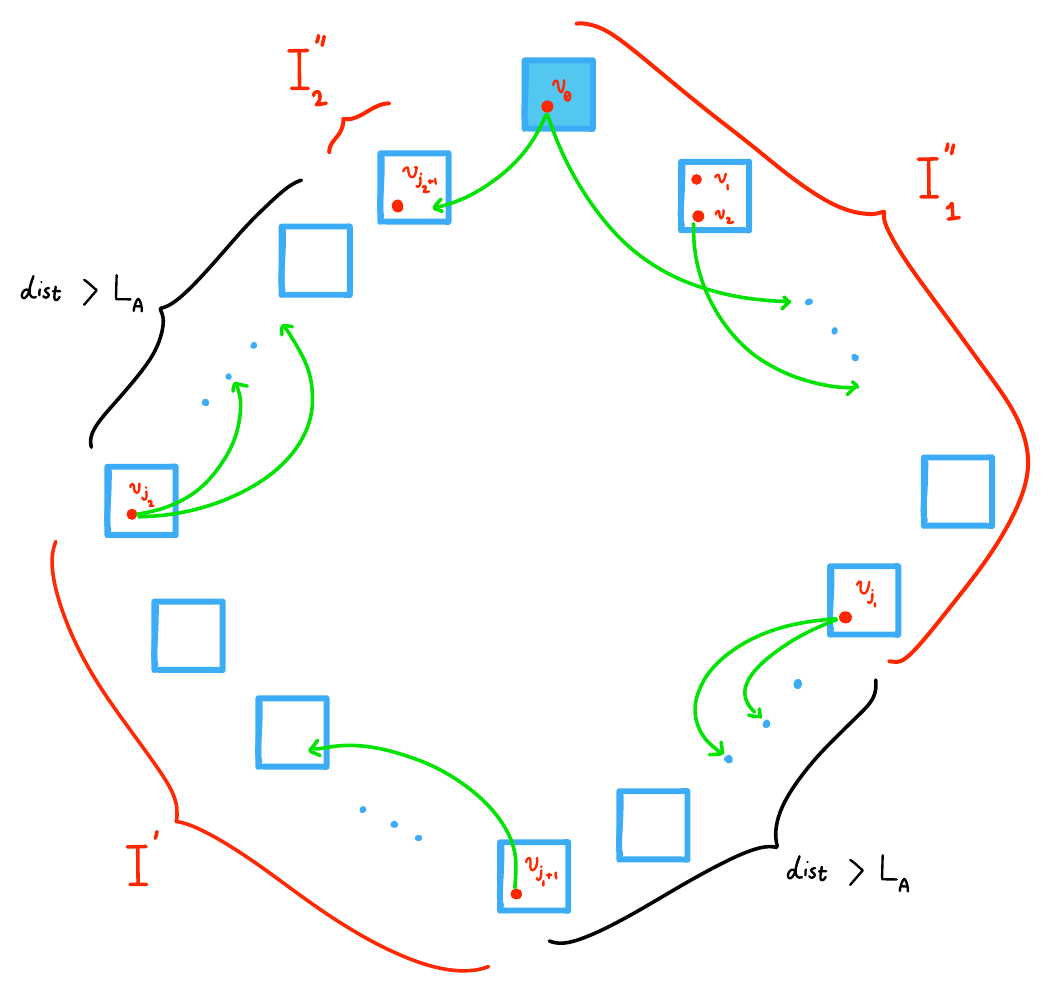}
    \caption{Illustration of the proof of Lemma \ref{Lemma:gapsinminimal}}
    \label{Figure:gapslemma}
\end{figure}

\subsection{Length and size of minimal relations}\label{Subsection:lengthsizerelations}\hfill\\

To understand the homology, we also need to examine some of the properties of boundaries. To this end, we define the space of \emph{relations}.

\begin{defi}
Let $I_1, \cdots, I_r \subseteq \bm{C^m}$, and assume that elements $S_{I_1}, \cdots, S_{I_r}$ are minimal cycles. A \emph{relation} between these cycles is a subset $R \subseteq \bm{C^m}$ and the equality 
$$\partial S_{R} = S_{I_1} + \cdots + S_{I_r}.$$
The set of all elements $S_{R}$ arising from relations forms a linear space referred to as the \emph{space of relations}.
\end{defi}

We call such a relation \emph{homogeneous} if all elements in ${R \cup I_1 \cup \cdots \cup I_r}$ have the same grading. We call a relation \emph{minimal} if there are no non-trivial subsets $R' \subset R$ and $T \subseteq \{1, \cdots, r\}$ such that 
$$\partial S_{R'} = \sum_{t \in T} S_{I_t}.$$
It is easy to see that a minimal relation needs to be homogeneous.\\

We are going to prove that length of a minimal relation is also bounded. We start with a bound on the gaps between elements of $R$. 

\begin{lemm}\label{Lemma:gapsinminimalrelation}
     Assume we have a minimal relation
     $$\partial S_{R} = S_{I_1} + \cdots + S_{I_r}$$
     Order the elements of $R$ with respect to the  clockwise order as we described before and let the result be $\nu_0, \cdots, \nu_{n-1}$. Then there exist at most one $i\in \{0, \cdots, n-1\}$  s.t. 
$$\dist(\nu_i, \nu_{i+1}) > (N_A+1)L_A.$$
\end{lemm}

\begin{proof}
The proof is similar to Lemma \ref{Lemma:gapsinminimal}. We use similar notations and refer to the facts stated in that proof.\\

Assume there exist ${0 \leq j_1 < j_2 \leq n-1}$ such that ${\dist(\nu_{j_{l}}, \nu_{j_{l+1}}) > (N_A+1)L_A}$ for $l=1,2$. Consider the subsets 
$$R'=\{\nu_{j_{1}+1},\cdots, \nu_{j_2}\} \ \text{and} \ R''= R \setminus R'.$$
There exist subsets $J', J'' \subset \bm{C^m}$ such that 
$$\partial S_{R'} = S_{J'} \ \text{and} \ \partial S_{R''} = S_{J''}.$$
As we discussed in the proof of Lemma $\ref{Lemma:gapsinminimal}$, we have $J' \cap J'' = \varnothing$. In fact based on the ideas presented in that proof we can see that : 
$$\dist_{H}(J', J'') \geq N_AL_A.$$
Now we consider the gaps between $J'$ and $J''$. There are two such gaps, the first one comes before all elements of $J'$, and the second one comes after. We prove the lemma for the case when none of the gaps include $C^{\bu}$. The same argument works for the other case, and in fact since none of the directed edges (i.e. differentials) lead to or pass over $C^{\bu}$, the bound can be made stronger.\\

Let $E=\{\kappa_0, \kappa_1 , \cdots, \kappa_l\}$ (resp.~$E'=\{\kappa'_0, \kappa'_1, \cdots, \kappa'_{l'} \}$) be the ordered set of elements of $I_1 \cup \cdots \cup I_r$ which are inside or in one of the ends of the first (resp.~second) gap. Note that $J'$ and $J''$ are subsets of $I_1 \cup \cdots \cup I_r$, as a result $E$ and $E'$ are not empty.\\ 

We are going to prove that at least one of $l$ and $l'$ is strictly greater than $N_A$. This will be in contradiction with our homogeneity assumption, due to Lemma \ref{Lemma:sizehomogenous}.\\

If $l$ and $l'$ are both less than $N_A$, then there exist $0 \leq o \leq l-1$ and $0 \leq o' \leq l'-1$ such that 
$$\dist(\kappa_o,\kappa_{o+1}) > L_A \ \text{and} \ \dist(\kappa'_{o'},\kappa'_{o'+1}) > L_A.$$
Now for all $1 \leq t \leq r$, all of the elements of subset $I_t$ either come before the gap $(\kappa_o,\kappa_{o+1})$ or come after it (in the clockwise order). This is a result of Lemma \ref{Lemma:gapsinminimal}. The same fact holds for the gap $(\kappa'_o,\kappa'_{o+1})$. We can define the subset $T' \subseteq \{1, \cdots, n\}$ as follows 
$$T' = \{ t' \ | \ \text{All elements of $I_{t'}$ come after $(\kappa_o,\kappa_{o+1})$ and before $(\kappa'_o,\kappa'_{o+1})$} \}.$$
Also let $T''$ be the complement of $T'$. Now we can easily see the following: 
$$(\bigcup_{t' \in T'} I_{t'}) \cap J'' = \varnothing \ \text{and}\  (\bigcup_{t'' \in T''} I_{t''}) \cap J' = \varnothing.$$
As a result we have 
$$\partial S_{R'} = S_{J'} = \sum_{t' \in T'} S_{I_{t'}} \ \text{and} \ \partial S_{R''} = S_{J''} = \sum_{t'' \in T''} S_{I_{t''}}.$$
This contradicts the minimality of the relation and completes our proof.
\end{proof}
Same as before, we can use Lemma \ref{Lemma:gapsinminimalrelation} to get a bound on the length of a relation.

\begin{coro}\label{Corollary:lengthminimalrelation}
    There exist a fixed integer $L_B$ such that for any minimal relation such as 
    $$\partial S_{R} = S_{I_1} + \cdots + S_{I_r},$$
    we have $\len(S_{R}) \leq L_B.$
\end{coro}
\begin{proof}
    The proof is completly similar to Lemma \ref{Lemma:lengthminimal}. Setting $$L_B = N_A(N_A+1)L_A,$$ and using Lemma \ref{Lemma:gapsinminimalrelation} gives us the result. 
\end{proof}
\begin{rema}\label{Remark:notsharpbound}
    We can find a much better upper bound $L_B$. Examining the proof of Lemma \ref{Lemma:gapsinminimalrelation}, shows that if there exist ${0 \leq j_1 < j_2 \leq n-1}$ such that ${\dist(\nu_{j_{l}}, \nu_{j_{l+1}}) > L_A}$ for $l=0,1$, then for at least one $i \in \{0,1\}$ we have 
    $$|(\nu_{j_{i}}, \nu_{j_{i+1}}) \ \cap \ (\bigcup_{1 \leq t \leq r} I_{t})| \geq \lfloor \frac{\dist(\nu_{j_{l}}, \nu_{j_{l+1}})}{L_A} \rfloor -1. $$
    The homogenity condition gives us that $|\bigcup_{1 \leq t \leq r} I_{t})| \leq N_A$ and as a result we can see that we can set 
    $$L_B = 3N_AL_A$$
    in Corollary \ref{Corollary:lengthminimalrelation}.\\
    
    Although we have this significantly sharper upper bound,  we continue to assume that $L_{B}=N_A(N_A+1)L_A$ as a safe choice. This assumption simplifies later arguments where the larger bound may be needed, although the sharper bound might still suffice in those cases.\\
\end{rema}

We have already established a bound for $\len(S_{R})$, from which the bound on $\diam(R)$ follows immediately. However, the relation also involves the subsets $I_1,\dots,I_{r}$. Using similar arguments, we can derive an analogous result for the union $\underset{1 \leq t \leq r}{\bigcup} I_{t}$, as stated in Corollary \ref{Corollary:diamofminimalrelation}. 

\begin{coro}\label{Corollary:diamofminimalrelation}
Consider a minimal relation $\partial S_{R} = S_{I_1} + \cdots + S_{I_r}.$ Then we have 
$$\diam(\underset{1 \leq t \leq r}{\bigcup} I_{t}) \leq N_{A}(N_{A}+1)L_{A} = L_{B}.$$
As a consequence, we can also write
$$\diam(R \ \cup \ \underset{1 \leq t \leq r}{\bigcup} I_{t}) \leq 2L_{B}+L_{A}.$$
\end{coro}

We do not provide a full proof of Corollary \ref{Corollary:diamofminimalrelation}. Note that the second statement follows from the first using Lemma \ref{Lemma:differentialorder} and Corollary \ref{Corollary:lengthminimalrelation}, since there is at least one edge connecting $R$ to $\underset{1 \leq t \leq r}{\bigcup} I_{t}$.\\

Moreover, we emphasize that can replicate the exact same method used in the proof of Corollary \ref{Corollary:lengthminimalrelation}. The argument proceeds by:
\begin{itemize}
    \item Analyzing the large (i.e. bigger than $L_{A}$) gaps between consecutive elements of $\underset{1 \leq t \leq r}{\bigcup} I_{t}$,
    \item Identifying elements of $R$ that fall into those gaps—which must exist due to the minimality assumption,
    \item Observing (via Lemma \ref{Lemma:gapsinminimalrelation}) that the gaps between such elements cannot be too large, together with the fact that $|R| \leq N_{A}$
\end{itemize}
Using this approach, we obtain the bound
$$\diam(\underset{1 \leq t \leq r}{\bigcup} I_{t}) \leq N_{A}^{2}(N_{A}+1)L_{A} = N_{A}L_{B}.$$
However, similar to what we saw in Remark \ref{Remark:notsharpbound}, this bound can be further refined.

\begin{rema}
It is clear that the homology is fully determined by the space of cycles and the space of relations. Moreover, it suffices to identify a basis for each of these spaces. By focusing on minimal cycles and minimal relations, we construct these spaces from well-controlled and structured elements. 
\end{rema}

\subsection{Homology of $\widehat{\text{CFA}}(\mathcal{H}_{K}, z) \boxtimes \widehat{\text{CFD}}(\mathcal{H}'_{\frac{1}{m}} , z')$ : Algebra}\label{Subsection:HomologyofboxtensorS^3:Algebra}\hfill\\

In the final part of this section, we focus on the properties of 
$$\widehat{\text{CFA}}(\mathcal{H}_{K}, z) \boxtimes \widehat{\text{CFD}}(\mathcal{H}'_{\frac{1}{m}} , z').$$
Note that some of our previous proofs rely on the Alexander grading, and as a result, can not be extended to $\widehat{\text{CFA}}(\mathcal{H}_{K}, z) \boxtimes \widehat{\text{CFD}}(\mathcal{H}'_{\frac{1}{m}} , z')$. We use a different technique to gather information about the homology of this latter complex. First, recall that there is a chain homotopy 
$$\CFh(S^3) \simeq \widehat{\text{CFA}}(\mathcal{H}_{K}, z) \boxtimes \widehat{\text{CFD}}(\mathcal{H}'_{\frac{1}{m}} , z').$$
This means that we only need to find the generator of the homology 
$$H_{*}(\widehat{\text{CFA}}(\mathcal{H}_{K}, z) \boxtimes \widehat{\text{CFD}}(\mathcal{H}'_{\frac{1}{m}} , z')).$$
In particular, we want to expand this generator in the basis $\bm{C^m}$.\\

Our main tool here is Theorem \ref{Theorem:typeAinvariance} on invariance of type $A$ invariants. Recall the definition of $\widehat{\text{CFA}}(\mathcal{H}_{\infty} , z)$, which was type $A$ invariant associated to the decorated graph of Figure \ref{Figure:HinfinitytypeA}. We also need to recall the grading of this type $A$ invariant. The subgroup $P(x_0)$ is generated by $(-\frac{1}{2}; 0,1)$, and the coset $gr(x_0)$ has representative $(0;0,0)$.\\

Let $\bm{z_0}$ be the reference element for grading of $\widehat{\text{CFA}}(\mathcal{H}_{K},z)$. Let 
$$\phi : P(x_0)\backslash G \rightarrow P(\bm{z_0}) \backslash G, \ \text{and}$$
$$\widehat{f} : \widehat{\text{CFA}}(\mathcal{H}_{\infty} , z) \rightarrow \widehat{\text{CFA}}(\mathcal{H}_{K},z)$$
respectively be the isomorphism of $G$-sets and the $\A_{\infty}$ homotopy equivalence given by Theorem \ref{Theorem:typeAinvariance}. As mentioned, this $\A_{\infty}$ homotopy equivalence induces a chain homotopy equivalence on the box tensor products denoted by 
$$\widehat{f} \boxtimes \id_{m} : \widehat{\text{CFA}}(\mathcal{H}_{\infty}, z) \boxtimes \widehat{\text{CFD}}(\mathcal{H}'_{\frac{1}{m}} , z') \longrightarrow \widehat{\text{CFA}}(\mathcal{H}_{K}, z) \boxtimes \widehat{\text{CFD}}(\mathcal{H}'_{\frac{1}{m}} , z'),$$
where $\id_{m}$ denotes the identiny map of the type $D$ invariant $\widehat{\text{CFD}}(\mathcal{H}'_{\frac{1}{m}} , z')$.\\

Note that the chain complex 
$$\widehat{\text{CFA}}(\mathcal{H}_{\infty}, z) \boxtimes \widehat{\text{CFD}}(\mathcal{H}'_{\frac{1}{m}} , z')$$
is generated by a single generator $x_0 \otimes \eta$ and has trivial differentials. As a result $\widehat{f} \boxtimes \id_m (x_0 \otimes \eta)$, is a generator for the homology $$H_{*}(\widehat{\text{CFA}}(\mathcal{H}_{K}, z) \boxtimes \widehat{\text{CFD}}(\mathcal{H}'_{\frac{1}{m}} , z')).$$ 
Lemma \ref{Lemma:homotopyequivgenerator} examines the properties of this generator which will be useful later. We use $gr_S$ to denote the grading of the type $A$ invariant $\widehat{\text{CFA}}(\mathcal{H}_{K}, z)$, and $gr_{S_m}$ to denote the induced grading on 
$$\widehat{\text{CFA}}(\mathcal{H}_{\infty}, z) \boxtimes \widehat{\text{CFD}}(\mathcal{H}'_{\frac{1}{m}} , z').$$
Note that dropping the Alexander component of $gr_K$ (resp.~$gr_{K_m}$), gives us $gr_S$ (resp.~$gr_{S_m}$).\\

Recall that $P(\bm{z_0})$ is generated by the coset $[(M;0,1)]$. This comes from Lemma \ref{Lemma:Gradingsubgroup}, after dropping the Alexander component. 


\begin{figure}[h]
    \centering
    \includegraphics[width=0.3\linewidth]{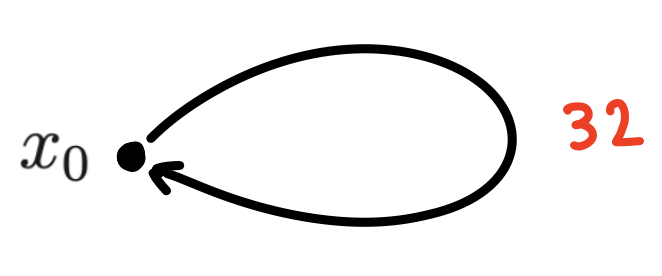}
    \caption{Decorated graph representing $\widehat{\text{CFA}}(\mathcal{H}_{\infty} , z)$ }
    \label{Figure:HinfinitytypeA}
\end{figure}

\begin{lemm}\label{Lemma:homotopyequivgenerator}
    Let $\widehat{f}$ be the special $\A_{\infty}$ homotopy equivalence defined above. 
    Consider a subset $\widehat{I}_m \subseteq \bm{C^m}$ such that 
    $$\widehat{f} \boxtimes \id_m (x_0 \otimes \eta) = S_{\widehat{I}_m}.$$ Then the cycle $S_{\widehat{I}_m}$ is homogeneous with respect to $gr_{S_m}$. Furthermore, the grading double coset of elements of $\widehat{I}_m$ has a representative $(\alpha; \beta, \gamma)$ such that $$2\beta= M + \frac{1}{2}.$$
    Finally there exist fixed integer $L_H$ and a fixed element $\widehat{\nu}_0$ satisfying 
    $$\dist_{H}(\{\widehat{\nu}_0\} , C^{\bu}) \leq L_H,$$
    such that for all $m > L_H$ we have $\widehat{\nu}_0 \in \widehat{I}_m \subseteq \bm{C^m}$. 
\end{lemm}
\begin{proof} 
Note that
\begin{equation}\label{Equation:boxtensorhomotopyequivformula}
\widehat{f} \boxtimes \id_m (x_0 \otimes \eta) = \sum_{k=0}^{\infty} (f_{k+1} \otimes \id_m) \circ (x_0 \otimes \delta_{k}(\eta)).    
\end{equation}

If an element $y \otimes \sigma \in \bm{C^m}$ appears in $\widehat{f} \boxtimes \id_m $ then there exist $k \geq 0$ such that $y$ is a summand in $\widehat{f}_{k+1}(x_0 \otimes a_1 \otimes \cdots \otimes a_k)$, and $a_1 \otimes \cdots \otimes a_k \otimes \sigma$ is a summand in $\delta_{k}(\eta)$. Based on Theorem \ref{Theorem:typeAinvariance}, we have:
$$gr_S(y)= gr_S(\widehat{f}_{k+1}(x_0 \otimes a_1 \otimes \cdots \otimes a_k)) = \phi(gr(x_0))gr(a_1)\cdots gr(a_k) \lambda^{k}.$$
Using the properties of gradings of type $D$ invariants we have:
$$gr(a_1) \cdots gr(a_k)gr_m(\sigma) = gr_m(a_1 \otimes \cdots \otimes a_k \otimes \sigma) = gr(\delta_{k}(\eta)) = \lambda^{-k} gr_m(\eta)$$
$$ \Rightarrow gr_m(\sigma) = gr(a_k)^{-1} \cdots gr(a_1)^{-1} \lambda^{-k}gr_m(\eta).$$
Now we can compute the grading of the element $y \otimes \sigma$ as follows
$$gr_{S_m}(y \otimes \sigma) = gr_S(y) gr_m(\sigma) = $$
$$\phi(gr(x_0))gr(a_1)\cdots gr(a_k) \lambda^{k}gr(a_k)^{-1} \cdots gr(a_1)^{-1} \lambda^{-k}gr_m(\eta) \Rightarrow$$
$$gr_{S_m}(y \otimes \sigma)=\phi(gr(x_0))gr_m(\eta).$$
This gives us the homogeneity of $S_{\widehat{I}_m}$.\\

To get the second result, we need to examine $\phi(gr(x_0))$. As we mentioned, $\phi$ is an isomorphism of $G$-sets. As a result we have 
$$\phi(gr(x_0) \cdot (-\frac{1}{2} ; 0,1)) = \phi(gr(x_0)) \cdot (-\frac{1}{2} ; 0,1)$$
But we also know that 
$$gr(x_0) \cdot (-\frac{1}{2} ; 0,1) = [(0 ; 0,0)] \cdot (-\frac{1}{2} ; 0,1) = [(-\frac{1}{2} ; 0,1)] = [(0 ; 0,0)],$$
where the last equality follows from the fact that $(-\frac{1}{2} ; 0,1) \in P(x_0)$. This means that the following equality holds in $P(y_0) \backslash G$:
$$\phi(gr(x_0)) \cdot (-\frac{1}{2} ; 0,1) = \phi(gr(x_0)) \Rightarrow$$
$$\phi(gr(x_0)) \cdot (-\frac{1}{2} ; 0,1) \cdot \phi(gr(x_0))^{-1} = [(0;0,0)]$$
Assume that $\phi(gr(x_0)) =[(\alpha; \beta, \gamma)]$, then we can rewrite the top equation: 
$$(\alpha; \beta, \gamma) \cdot (-\frac{1}{2} ; 0,1) \cdot (-\alpha; -\beta, -\gamma) \in P(y_0) \Rightarrow (2\beta-\frac{1}{2};0,1) \in P(y_0)$$
As we mentioned $P(y_0)$ is a cyclic subgroup generated by $(M;0,1)$, and this gives us that $$2\beta-\frac{1}{2}=M.$$ 
Note that 
$$gr_{m}(\eta)=(0;0,0)$$
which means that the the grading double coset of elements of $\widehat{I}_m$ is equal to 
$$\phi(gr(x_0)) gr_m(\eta) = [(\alpha; \beta, \gamma) \cdot (0;0,0)] =[(\alpha; \beta, \gamma)].$$
This finishes the proof of the second part of the lemma.\\

For the final part of the lemma, define $\widehat{J}_k \subseteq B_K$ for all integer $k \geq 0$ as follows:
$$\widehat{f}_{k+1}(x_0 \otimes \underbrace{\rho_3 \otimes \rho_{23} \otimes \cdots \otimes \rho_{23}}_{k}) = \sum_{y \in \widehat{J}_k} y.$$
Also consider $\widehat{J}' \subseteq B_K$ such that
$$\widehat{f}_{2}(x_0 \otimes \rho_1) = \sum_{y \in \widehat{J}'} y$$
Then for any $m \in \mathbb{N}$ we have: 
$$\widehat{f} \boxtimes \id_m (x_0 \otimes \eta) = \sum_{y \in \widehat{J}_0} y \otimes \eta + \sum_{k=1}^{m} \ \sum_{y \in \widehat{J}_k} y \otimes \xi_{k} + \sum_{y \in \widehat{J}'} y \otimes \xi_{m}.$$
From this expansion, the third result of the lemma is clear. We can take $L_H$ to be the smallest number such that $\widehat{J}_{L_H} \neq \varnothing$. Picking an arbitrary $y_0 \in \widehat{J}_{L_H}$ we can set 
\begin{equation}
  \widehat{\nu}_0 =
    \begin{cases}
      y_0 \otimes \eta \in C^{\bu} & \text{if $L_H=0$}\\
      y_0 \otimes \xi_{L_H} \in C^{\circ}_{L_H} & \text{if $0<L_H<+\infty$}\\
    \end{cases}       
\end{equation}

If $\widehat{J}_k = \varnothing$ for all $k \geq 0$, then $\widehat{J}'$ must be non-empty. We define $L_H$ to be one, and pick an arbitrary $y_0 \in \widehat{J}'$. The element $\widehat{\nu}_0 = y_0 \otimes \xi_{m} \in C^{\circ}_{m}$ has our desired properties.
\end{proof}

\begin{rema}
    The homogeneity statement in Lemma \ref{Lemma:homotopyequivgenerator} is a more general statement (as apparent by the proof) and is part of Theorem 10.43. in \cite{Lipshitz2008BorderedHF}, but since the relation of the statements might not be apparent in the first sight, we decided to rewrite the proof. 
\end{rema}

Note that we didn't use all the properties of an $\A_{\infty}$ homotopy equivalence in the proof of Lemma \ref{Lemma:homotopyequivgenerator}. Lemma \ref{Lemma:homotopyequivgenerator} was intentionally phrased in a more general setting. In fact, in our case we can prove $L_{H}=0$, as we can see in Lemma \ref{Lemma:homotopyequivgenerator2}. 

\begin{lemm}\label{Lemma:homotopyequivgenerator2}
Consider an arbitrary $\A_{\infty}$ homotopy equivalence
$$f : \widehat{\text{CFA}}(\mathcal{H}_{\infty} , z) \rightarrow \widehat{\text{CFA}}(\mathcal{H}_{K},z).$$
Consider a subset $I_m \subseteq \bm{C^m}$ such that 
    $$f \boxtimes \id_m (x_0 \otimes \eta) = S_{I_m}.$$
Then there is a fixed element $\nu_0 \in C^{\bu} \cap I_m$ for all $m \in \mathbb{Z}_{\geq 0}$.\\
\end{lemm}
\begin{proof}
    Based on proof of Lemma \ref{Lemma:homotopyequivgenerator2}, we only need to prove that $J_0 \neq \varnothing$ i.e. $f_{1}(x_{0}) \neq 0$. We use proof by contradiction, so assume that $f_{1}(x_{0}) = 0$.\\
    
    Based on the definition of $\A_{\infty}$ homotopy equivalence, there exists an $\A_{\infty}$ homomorphism 
    $$g: \widehat{\text{CFA}}(\mathcal{H}_{K},z) \rightarrow \widehat{\text{CFA}}(\mathcal{H}_{\infty} , z),$$
    such that $f \circ g$ and $g \circ f$ are homotopic to identity. This means that 
    $$g \circ f - \id : \widehat{\text{CFA}}(\mathcal{H}_{\infty} , z) \rightarrow \widehat{\text{CFA}}(\mathcal{H}_{\infty} , z)$$
    is a null homotopic $\A_{\infty}$ homomorphism. Based on part of the definition of a null homotopic $\A_{\infty}$ homomorphism, there exists a map 
    $$H_{1} : \widehat{\text{CFA}}(\mathcal{H}_{\infty} , z) \rightarrow \widehat{\text{CFA}}(\mathcal{H}_{\infty} , z),$$
    such that we have  
    $$(g \circ f)_{1}(x_0) - \id_{1}(x_0) = m_{1}(H_{1}(x_0)) + H_{1}(m_1(x_0)).$$
    Note that the map $m_1$ vanishes on $(\mathcal{H}_{\infty} , z)$, hence the right hand side of the equality is zero. However the left hand side will be
    $$(g \circ f)_{1}(x_0) - \id_{1}(x_0) = g_{1}(f_{1}(x_0)) - x_0 = - x_0.$$
    This is a clear contradiction, and hence the proof is complete.\\
\end{proof}

Up to this point in this subsection we used $\A_{\infty}$ homotopy equivalences of $\widehat{\text{CFA}}(\mathcal{H}_{\infty} , z)$ and $\widehat{\text{CFA}}(\mathcal{H}_{K},z)$ to find representatives for the unique non-trivial class in 
    $$H_{*}(\widehat{\text{CFA}}(\mathcal{H}_{K}, z) \boxtimes \widehat{\text{CFD}}(\mathcal{H}'_{\frac{1}{m}} , z')).$$
One might ask if all these representatives can be generated using this method. The answer is positive. We suspect that this can be concluded from a more general theorem in the bordered Floer theory, however since we haven't found an specific reference yet, we sketch a proof of this fact in our special setting. This is done in Proposition \ref{Proposition:generatorsofHFS^3}.

\begin{prop}\label{Proposition:generatorsofHFS^3}
    Let $m_0$ be a fixed positive integer. Let $S_I$ be a cycle in 
    $$\widehat{\text{CFA}}(\mathcal{H}_{K}, z) \boxtimes \widehat{\text{CFD}}(\mathcal{H}'_{\frac{1}{m_0}} , z'),$$ such that $[S_I]$ is non-trivial in the homology. Then there exists an $\A_{\infty}$ homotopy equivalence
$$f_{S_I} : \widehat{\text{CFA}}(\mathcal{H}_{\infty} , z) \rightarrow \widehat{\text{CFA}}(\mathcal{H}_{K},z),$$
such that 
$$f_{S_I} \boxtimes \id_{m_0} (x_0 \otimes \eta) = S_{I}.$$
\end{prop}
\begin{proof}
    Consider an arbitrary fixed $\A_{\infty}$ homotopy equivalence
$$f' : \widehat{\text{CFA}}(\mathcal{H}_{\infty} , z) \rightarrow \widehat{\text{CFA}}(\mathcal{H}_{K},z).$$
Also consider a subset $I' \subseteq \bm{C^m}$ such that 
    $$f' \boxtimes \id_{m_0} (x_0 \otimes \eta) = S_{I'}.$$
 
We are going to construct a null homotopic $\A_{\infty}$ homomorphism 
$$g: \widehat{\text{CFA}}(\mathcal{H}_{\infty} , z) \rightarrow \widehat{\text{CFA}}(\mathcal{H}_{K},z),$$
such that 
$$g \boxtimes \id_{m_0} (x_0 \otimes \eta) = S_{I} - S_{I'}.$$
Then by setting $f_{S_I} = f'+g$ we will have the desired result.\\

To construct $g$, we first note that $[S_{I'}]$ is a generator of the homology and hence $[S_{I'}] = [S_{I}]$. This means there exist a relation 
$$\partial S_{R} = S_{I} - S_{I'}.$$
Based on the definition of an $\A_{\infty}$ null homotopy, we need to define the maps 
$$H_{i} : \widehat{\text{CFA}}(\mathcal{H}_{\infty} , z) \otimes \A^{\otimes(i-1)} \rightarrow \widehat{\text{CFA}}(\mathcal{H}_{K},z).$$
These collection of maps is defined as follows. We set the maps to be zero apart from two following cases: 
$$H_{1}(x_0) = S_{R \ \cap \  C^{\bu}} \ \text{and} \ H_{n+1}(x_0 \otimes \underbrace{\rho_{3} \otimes \rho_{23} \otimes \cdots \otimes \rho_{23}}_{n}) = S_{R \  \cap \  C^{\circ}_{n}}.$$
One can see that if we define an $\A_{\infty}$ null homotopy $g$ using this collection of maps, it will satisfy the mentioned relations.
\end{proof}

\begin{rema}\label{Remark:familygrowth}
Proposition \ref{Proposition:generatorsofHFS^3} has a very interesting corollary. Note that once we construct the $\A_{\infty}$ homotopy equivalence $f_{S_I}$, then we can consider a cycle $S_{I_m}$ in chain complex 
$$\widehat{\text{CFA}}(\mathcal{H}_{K}, z) \boxtimes \widehat{\text{CFD}}(\mathcal{H}'_{\frac{1}{m}} , z') \ \text{for all} \ m \in \mathbb{Z}_{>0},$$
such that it satisfies 
$$f_{S_I}\boxtimes \id_{m} (x_0 \otimes \eta) = S_{I_{m}}.$$\\
In particular, we have $S_{I_{m_0}} = S_{I}$. Intuitively, this means any cycle which generates the homology of 
$$\widehat{\text{CFA}}(\mathcal{H}_{K}, z) \boxtimes \widehat{\text{CFD}}(\mathcal{H}'_{\frac{1}{m}} , z') \ \text{for}\  m=m_0,$$
can grow to a family of generating cycles for all $m \geq m_0$. We can relate the cycles in this family as follows.\\ 

In the rest of this discussion, let $f=f_{S_{I}}$. Similar to the proof of Lemma \ref{Lemma:homotopyequivgenerator}, we define subsets $J_{k} \subseteq B_K$ and $J' \subseteq B_K$ as follows: 
$$f_{k+1}(x_0 \otimes \underbrace{\rho_3 \otimes \rho_{23} \otimes \cdots \otimes \rho_{23}}_{k}) = \sum_{y \in J_k} y \ \  \text{for} \ k \in \mathbb{Z}_{\geq 0} \ ,$$
$$\text{and} \ \ f_{2}(x_0 \otimes \rho_{1}) = \sum_{y \in J'} y.$$
Using the notations, we can write:
\begin{equation}\label{Equation:generatorexpand}
 f \boxtimes \id_m (x_0 \otimes \eta) = S_{I_m} = \sum_{y \in J_0} y \otimes \eta + \sum_{k=1}^{m} \sum_{y \in J_k} y \otimes \xi_{k} + \sum_{y \in J'} y \otimes \xi_{m}.   
\end{equation}
We can use this expansion of $S_{I_{m}}$ to compare it with $S_{I_{m+1}}$ as follows. Define the following elements of the chain complex:
\begin{equation}\label{Equation:generatordecompose}
S_{I''_m} : = \sum_{y \in J_0} y \otimes \eta + \sum_{k=1}^{m} \sum_{y \in J_k} y \otimes \xi_{k} \ \text{and} \ S_{I'_m} = \sum_{y \in J'} y \otimes \xi_{m}.
\end{equation}
It is clear that $S_{I_m} = S_{I''_m} + S_{I'_m}$. Now we can write
$$S_{I_{m+1}} = \sum_{y \in J_0} y \otimes \eta + \sum_{k=1}^{m+1} \sum_{y \in J_k} y \otimes \xi_{k} + \sum_{y \in J'} y \otimes \xi_{m}$$
$$ = S_{\Phi'_{m}(I''_{m})} + \sum_{y \in J_{m+1}} y \otimes \xi_{m+1} + S_{\Phi_{m}(I'_{m})}.$$
Furthermore, we can write
\begin{equation}\label{Equation:Twostabilizationformula}
S_{I''_{m+1}} = S_{\Phi'_{m}(I''_{m})} + \sum_{y \in J_{m+1}} y \otimes \xi_{m+1} \  \ \text{and}  \ \ S_{I'_{m+1}} = S_{\Phi_{m}(I'_{m})}.
\end{equation}
Recall that $I_{m+1}$ is a subset of $\bm{C^{m+1}}$, and here we have used both of the inclusions 
$$\Phi_{m} \ , \ \Phi'_{m} : \bm{C^{m}} \longhookrightarrow \bm{C^{m+1}},$$
defined in Subsection \ref{Subsection:GenstructureofBoxtensor}, to relate $I_{m+1} \subseteq \bm{C^{m+1}}$ to $I_{m} \subseteq \bm{C^{m}}$. \\
\end{rema}

Now, we prove an algebraic fact examining the properties of $S_{I'_m}$ in Proposition \ref{Proposition:=type(1)and(3)edges}. The relevance of this fact will become apparent in the next section. 

\begin{prop}\label{Proposition:=type(1)and(3)edges}
  Consider the element $S_{I^{\bu}}$, defined as follows
    $$S_{I^{\bu}} := S_{I_{m} \cap C^{\bu}} = \sum_{y \in J_{0}} y \otimes \eta.$$
 Then we will have 
 $$\partial_{3} (S_{I^{\bu}}) + \partial(S_{I'_{m}}) = 0.$$
\end{prop}
\begin{proof}
   Based on Proposition \ref{Proposition:boxtensordifferential}, we can rephrase the desired result as follows 
   \begin{equation}\label{Equation:Proposition:=type(1)and(3)edgesI}
    \sum_{y \in J_0} m_{2}(y \otimes \rho_1)\otimes \xi_{m} + \sum_{y \in J'} m_1(y)\otimes \xi_{m} = 0.
   \end{equation}
   Now recall that $J_0$ and $J'$ were defined as subsets of $B_{K}$ given by
   \begin{equation}\label{Equation:Proposition:=type(1)and(3)edgesII}
    f_{1}(x_0)=\sum_{y \in J_0} y \  \ \text{and} \ \  f_{2}(x_0 \otimes \rho_1)= \sum_{y \in J'} y.
   \end{equation}
  Using the compatibility condition in the definition of an $\A_{\infty}$ homomorphism, we can write: 
  $$m_2(f_1(x_0)\otimes\rho_1) +m_{1}(f_2(x_{0} \otimes \rho_1))+ f_{2}(m_1(x_0) \otimes \rho_1) + f_{1}(m_2(x_0 \otimes \rho_1)) = 0.$$
  Based on the description of $\widehat{\text{CFA}}(\mathcal{H}_{\infty},z)$ (appearing in Figure \ref{Figure:HinfinitytypeA}), we have $m_1(x_0)=0$ and $m_2(x_0 \otimes \rho_1) = 0$. Hence, we can conclude 
  $$m_2(f_1(x_0)\otimes\rho_1) +m_{1}(f_2(x_{0} \otimes \rho_1)) = 0.$$
  Combining this with Equation \ref{Equation:Proposition:=type(1)and(3)edgesII} gives us Equation \ref{Equation:Proposition:=type(1)and(3)edgesI} and completes the proof. 
\end{proof}

In the last result of this subsection, we combine Lemma \ref{Lemma:homotopyequivgenerator2} with Lemma \ref{Lemma:homotopyequivgenerator} to also deduce a result about the grading of element $\nu_{0}$. This is stated in Corollary \ref{Corollary:gradingofhomotopyequivgenerator}. \\

\begin{coro}\label{Corollary:gradingofhomotopyequivgenerator}
Using notations of Lemma \ref{Lemma:homotopyequivgenerator2}, we have that
$$gr_{S_m}(\nu_{0}) = [(\alpha_0;\beta_0,\gamma_0)] \ \text{for} \ m \in \mathbb{Z}_{\geq 0},$$
such that
$$2\beta_0 = M + \frac{1}{2}.$$
\end{coro}
\begin{proof}
    We also use the notations of Lemma \ref{Lemma:homotopyequivgenerator}. If
    $$I_m \ \cap \  \widehat{I}_{m} \ \cap \ C^{\bu} \neq \varnothing,$$
    then we are done. So assume that the mentioned intersection is empty. As a result the cycle $S_{I_{m}} - S_{\widehat{I}_m}$ has a non-trivial intersection with $C^{\bu}$.\\
    
    We know that
    $$[S_{\widehat{I}_m}] = [S_{I_m}] \in H_{*}(\widehat{\text{CFA}}(\mathcal{H}_{K}, z) \boxtimes \widehat{\text{CFD}}(\mathcal{H}'_{\frac{1}{m}} , z')).$$
    This means there exist a relation 
    $$\partial S_{R_m} = S_{I_{m}} - S_{\widehat{I}_{m}}.$$
Using graph-theoretic description of $$\widehat{\text{CFA}}(\mathcal{H}_{K}, z) \boxtimes \widehat{\text{CFD}}(\mathcal{H}'_{\frac{1}{m}} , z'),$$ 
we can show that there exist elements 
    $$\nu_{R_m} \in R_m \ \cap\  C^{\bu}, \ \nu_{0} \in I_m \ \cap \ C^{\bu} \ \text{and} \ \widehat{\nu}_{0} \in \widehat{I}_m \ \cap \ C^{\bu} $$
    and edges 
    $$\nu_{R_m} \longrightarrow \nu_{0} \ \text{and} \ \nu_{R_m} \longrightarrow \widehat{\nu}_{0}.$$
    This must be true since otherwise we can consider a set $R'_m$ as follows: 
    $$R'_m : = \{ \nu \ | \ \nu \in R_{m} \ \cap \ C^{\bu}, \ \exists \ \nu' \in I_{m} \ \cap \ C^{\bu} \ \text{and an edge} \ \nu \rightarrow \nu'\}.$$
    Then we have 
    $$\partial S_{R'_m} = S_{I_m \ \cap \ C^{\bu}} + S_{R''_m} \ \text{for} \ R''_m \subseteq \bm{C^m} \setminus C^{\bu}.$$
    And as a result, we have the cycle 
    $$S_{I_m} - \partial S_{R'_m} = S_{I_m \setminus C^{\bu}} -  S_{R''_m},$$ 
    which sits outside the black box $C^{\bu}$. However, this cycle is a representative of the generator of the homology as 
    $$[S_{I_m} - \partial S_{R'_m}] = [S_{I_m}].$$
    This contradicts the combination of Lemma
    \ref{Lemma:homotopyequivgenerator2} and Proposition \ref{Proposition:generatorsofHFS^3}.\\

    Now that we have the edges 
    $$\nu_{R_m} \longrightarrow \nu_{0} \ \text{and} \ \nu_{R_m} \longrightarrow \widehat{\nu}_{0},$$
    we are going to prove that the edges between elements of $C^{\bu}$ preserve the the $\spin^{\mathbb{C}}$ component of the grading i.e. for any edge 
    $$\nu_1 \longrightarrow v_2 \ \ \text{for} \ \nu_1,\nu_2 \in C^{\bu},$$
    the grading double cosets $gr_{S_{m}}(\nu_1)$ and $gr_{S_{m}}(\nu_2)$ has representatives with equal $\spin^{\mathbb{C}}$ components. This will finish our proof.\\

    Based on Proposition \ref{Proposition:boxtensordifferential}, we know that any such edge must be of the first type. This means that there exists $x_1,x_2 \in B^{0}_K$ such that
    $$\nu_1 = x_1 \otimes \eta, \ \nu_2= x_2 \otimes \eta, \ \text{and} \ x_{2} \ \text{appears in} \ m_1(x_{1}).$$
    Now using Equation \ref{Equation:typeAgrading}, we have 
    $gr_{S_m}(x_2) = \lambda^{-1} gr_{S_m}(x_1),$
which gives us the equality of the $\spin^{\mathbb{C}}$ components
\end{proof}
\subsection{Homology of $\widehat{\text{CFA}}(\mathcal{H}_{K}, z) \boxtimes \widehat{\text{CFD}}(\mathcal{H}'_{\frac{1}{m}} , z')$ : Combinatorics}\label{Subsection:HomologyofboxtensorS^3:Combinatorics}\hfill\\

In Subsection \ref{Subsection:HomologyofboxtensorS^3:Algebra}, we studied a lot of properties of the homology 
$$H_{*}(\widehat{\text{CFA}}(\mathcal{H}_{K}, z) \boxtimes \widehat{\text{CFD}}(\mathcal{H}'_{\frac{1}{m}} , z')),$$ 
using the $\A_{\infty}$ homotopy equivalences. In theory, all of the information about this homology can be retrieved from this algebraic perspective, but working with $\A_{\infty}$ homotopy equivalences is not always easy. In this subsection, we use the combinatorial structure of the chain complex to study some of the other properties of the homology which are going to be useful later.\\

In the rest of this subsection we use the graph-theoretic description of the chain complex 
$$\widehat{\text{CFA}}(\mathcal{H}_{K}, z) \boxtimes \widehat{\text{CFD}}(\mathcal{H}'_{\frac{1}{m}} , z').$$
We start by defining an \emph{alternating path} between elements of a subset ${I \subseteq \bm{C^{m}}.}$
\begin{defi}
    An alternating path between elements $\nu, \nu' \in I$ consists of a sequence $\nu_0,\cdots,\nu_{k}$ of elements of $I$ which satisfies the following conditions:
    \begin{itemize}[leftmargin=*]
   \item We have $\nu_0 = \nu$ and $\nu_k = \nu'$.
   \item For any $i \in \{1, \cdots, k\}$, there exists an element $\nu'_{i} \in \bm{C^{m}}$ and two edges 
   $$\nu_{i} \longrightarrow \nu'_{i} \ \text{and} \ \nu_{i-1} \longrightarrow \nu'_{i}.$$
   \end{itemize}
   An alternating path is called \emph{good} if none of the edges mentioned above are of the third type (based on terminology introduced in Proposition \ref{Proposition:boxtensordifferential}).
\end{defi}
A schematic picture of an alternating path can be seen in Figure \ref{Figure:alternatingpath}.

\begin{figure}[h]
    \centering
    \includegraphics[width=0.7\linewidth]{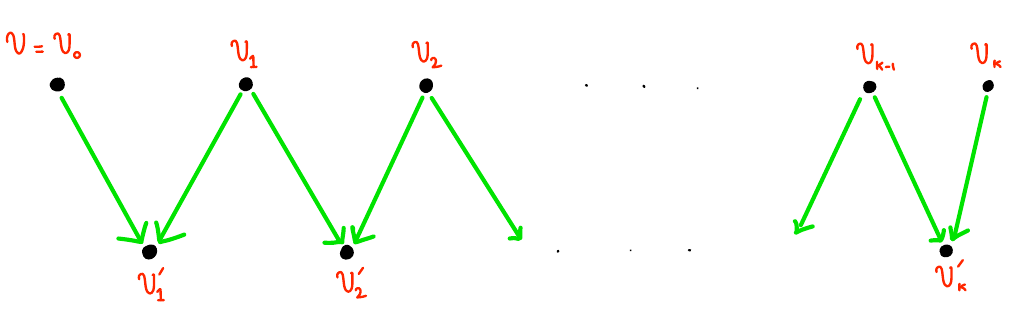}
    \caption{A schematic picture of an alternating path}
    \label{Figure:alternatingpath}
\end{figure}

We can use alternating paths to define an equivalence relation on the elements of $I$, where 
$$\nu \sim \nu' \Longleftrightarrow \ \text{there exist a good alternating path between} \ \nu \ \text{and} \ \nu'.$$
The importance of this equivalence relation is due to the following simple property.
\begin{lemm}\label{Lemma:alternatingpathgrading}
   For any two elements $\nu, \nu' \in I$, we have: 
   $$\nu \sim \nu' \Longrightarrow h_{m}(\nu) = h_{m}(\nu').$$
\end{lemm}
\begin{proof}
    Note that the differential in the chain complex 
    $$\widehat{\text{CFA}}(\mathcal{H}_{K}, z) \boxtimes \widehat{\text{CFD}}(\mathcal{H}'_{\frac{1}{m}} , z')$$
    decreases the Masolv grading by one. The statement then easily follows by a induction on the length of the alternating path between $\nu$ and $\nu'$.
\end{proof}

Proposition \ref{Proposition:equivalenceclassblack} states an important fact about the equivalence classes of this relation which will be useful later. First we remind some of the facts from the previous subsection.\\

Let $m_0$ be a positive integer. Let $S_I$ be a cycle such that $[S_I]$ is the generator of the homology 
    $$H_{*}(\widehat{\text{CFA}}(\mathcal{H}_{K}, z) \boxtimes \widehat{\text{CFD}}(\mathcal{H}'_{\frac{1}{m_0}} , z')).$$
As mentioned in Remark \ref{Remark:familygrowth}, we can use Proposition \ref{Proposition:generatorsofHFS^3}, to grow $S_I$, to a family of generators $S_{I_m}$ defined as:
$$S_{I_m}  = f_{S_{I}} \boxtimes \id_{m}(x_{0} \otimes \eta) \in \widehat{\text{CFA}}(\mathcal{H}_{K}, z) \boxtimes \widehat{\text{CFD}}(\mathcal{H}'_{\frac{1}{m}} , z') \ \text{for} \ m \in \mathbb{Z}_{>0}.$$
Note that $S_{I} = S_{I_{m_0}}$. Now we can use Equations \ref{Equation:generatordecompose} and \ref{Equation:Twostabilizationformula}. We can consider elements 
$$S_{I''_{m}} , S_{I'_{m}} \in \widehat{\text{CFA}}(\mathcal{H}_{K}, z) \boxtimes \widehat{\text{CFD}}(\mathcal{H}'_{\frac{1}{m}} , z') \ \text{for} \ m \in \mathbb{Z}_{>0}.$$
Now we are ready to state and prove Proposition \ref{Proposition:equivalenceclassblack}.
\begin{prop}\label{Proposition:equivalenceclassblack}
    Consider the equivalence relation $\sim$ on the subset $I''_{m} \subseteq \bm{C^{m}}$. Take an element $\nu \in I''_{m}$ and let $T=[\nu]_{\sim} \subseteq I''_{m}$. Then at least one of the following is true: 
    \begin{itemize}
        \item The subsets $T$ and $C^{\bu}$ have non-empty intersection
        \item The chain complex element $S_T$ is a nullhomologous cycle
    \end{itemize}
   Furthermore, if $S_{I_{m}}$ is a minimal cycle, then the latter can only happen when $T \cap I'_{m} \neq \varnothing$.  
\end{prop}
\begin{proof}
First, from Proposition \ref{Proposition:=type(1)and(3)edges} we can conclude that
$$(\partial_1 + \partial_2 + \partial_4) (S_{I''_{m}})= 0.$$
Now, note that if $\nu,\nu' \in I''_{m}$ are two elements such that $[\nu]_{\sim} \neq [\nu']_{\sim}$ then 
$$(\bigcup_{i=1,2,4} E_{i}(\nu,\cdot)) \  \bigcap \ (\bigcup_{i=1,2,4} E_{i}(\nu',\cdot)) = \varnothing.$$
This follows from the definition of the relation $\sim$ and the good alternating paths. As a result, any element of $\bm{C^{m}}$ appearing in $(\partial_1 + \partial_2 + \partial_4)(S_{T})$ can not be cancelled by the differentials of other elements in $I''_{m} \setminus T$ and hence $(\partial_1 + \partial_2 + \partial_4)(S_{T}) = 0$. Now, if we have that $T \cap C^{\bu} = \varnothing$, we can conclude 
$\partial_{3}(S_T) = 0 \Rightarrow \partial(S_T)= (\partial_1 + \partial_2 + \partial_4)(S_{T}) = 0.$
This follows from the definition of type $(3)$ edges in Proposition \ref{Proposition:boxtensordifferential} (in fact, by the same reasoning, it is clear that in this case we also have $\partial_1(S_T) = \partial_2(S_T) =0$). Finally, since $S_T$ is a cycle which is disjoint from the black box $C^{\bu}$, we can use Lemma \ref{Lemma:homotopyequivgenerator2} and Proposition \ref{Proposition:generatorsofHFS^3} to deduce that $S_T$ is nullhomologous.\\

To prove the second statement of the Proposition \ref{Proposition:equivalenceclassblack}, note that if ${T \cap I'_m = \varnothing}$, then none of the elements of the $S_T$ will be cancelled in the summation $S_{I''_m}+S_{I'_m} = S_{I_m}$. Hence, we will have $T \subseteq I$. This contradicts the minimality assumption. 
\end{proof}

\section{Stabilization of extremal knot Floer homologies}\label{Section:ExtermalknotFloerstabilization}
\subsection{Closed balls $B^{\bu}_{t}$ and shift maps $R_{\pm}$}\label{Subsection:ballsandshiftmaps}\hfill\\

In this section we use the properties of the box tensor product studied in Section \ref{Section:Boxtensor}, to prove Theorems \ref{Theorem:extremalhfk}, \ref{Theorem:extremalAlexander}, and \ref{Theorem:jumpsAlexander}. We start by defining some notation and stating some basic facts.\\

We start by considering the closed ball of radius $t$ centered at $C^{\bu}$. We denote this ball by $B^{\bu}_{t}$. It is clear that 
$$B^{\bu}_{t} = \{C^{\circ}_{m-t+1}, \cdots, C^{\circ}_{m} , C^{\bu}, C^{\circ}_{1}, \cdots,C^{\circ}_{t}\}.$$
Abusing the notation, we also use $B^{\bu}_{t}$ to denote all of the elements in this closed ball. Figure \ref{Figure:Blackboxball} depicts the ball $B^{\bu}_{t}$.\\

\begin{figure}[h]
    \centering
    \includegraphics[width=0.5\linewidth]{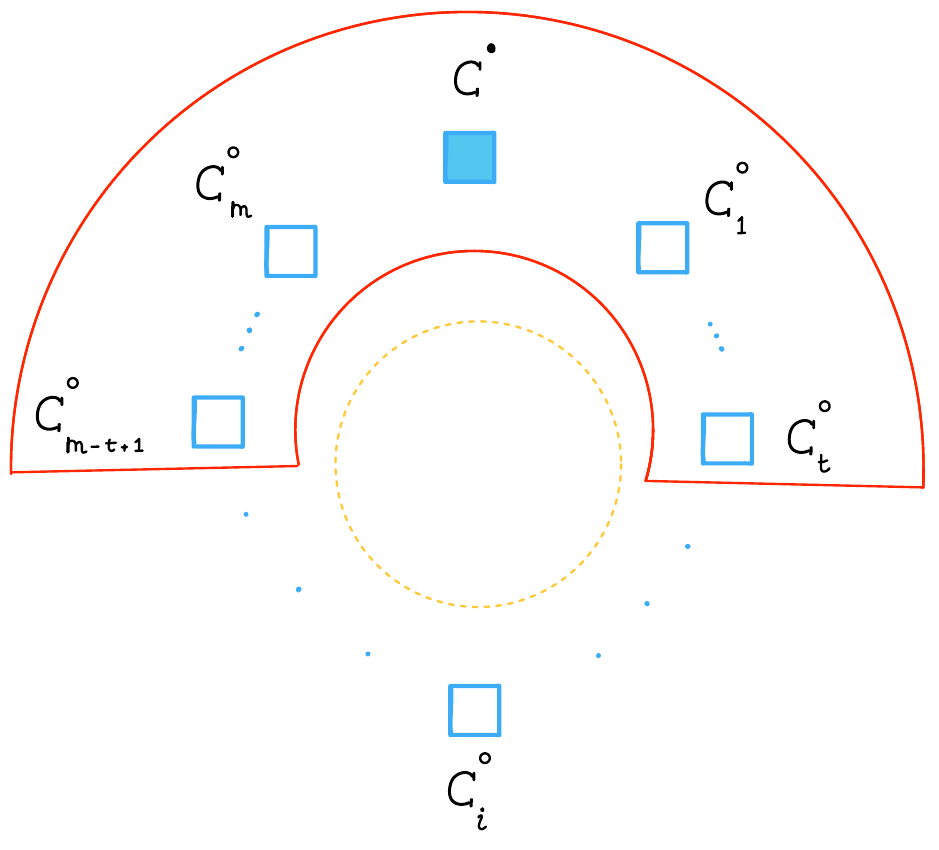}
    \caption{The closed metric ball $B^{\bu}_{t}$}
    \label{Figure:Blackboxball}
\end{figure}

It is also useful to define \emph{positive and negative half-balls} as follows: 
$$B^{\bu}_{[0,t]} := C^{\bu} \sqcup C^{\circ}_{1} \sqcup \cdots \sqcup  C^{\circ}_{t} \subseteq B^{\bu}_{t},$$
$$B^{\bu}_{[-t,0]} := C^{\circ}_{m-t+1} \sqcup \cdots \sqcup C^{\circ}_{m} \sqcup C^{\bu} \subseteq B^{\bu}_{t}.$$

Next we define two \emph{shift maps} $R_{\pm}$ as follows: 
$$R_{+} : \bm{C^m} \setminus \{C^{\bu}, C^{\circ}_{m}\} \rightarrow \bm{C^m} \ \text{and} \ R_{-} : \bm{C^m} \setminus \{C^{\bu}, C^{\circ}_{1}\} \rightarrow \bm{C^m},$$
$$R_{+}(x^1 \otimes \xi_i) : = x^1 \otimes \xi_{i + 1} \ \text{and} \ R_{-}(x^1 \otimes \xi_i) : = x^1 \otimes \xi_{i - 1}.$$
Note that the shift maps are injective and hence we can extend them to maps on power sets. Proposition \ref{Proposition:shiftbounddary} explains how shift maps preserve the differentials away from the black box. 

\begin{prop}\label{Proposition:shiftbounddary}
    For any subsets $I \subseteq \bm{C^m} \setminus B^{\bu}_{L_A+1}$, we have 
    $$\partial S_{I} = S_{J} \Rightarrow \partial S_{R_{\pm}(I)} = S_{R_{\pm}(J)}.$$
\end{prop}
\begin{proof}
    First, note that due to Lemma \ref{Lemma:differentialorder}, we have that 
    $$J \subseteq \bm{C^m} \setminus B^{\bu}_{1}.$$
    As a result, $R_{\pm}(J)$ is well-defined.\\
    
    The proof follows from Proposition \ref{Proposition:boxtensordifferential}. Since we are away from the black box, all of the differentials are of the fourth type. These subset of differentials is clearly invariant under shift maps by definition, and this gives us our result.
\end{proof}

We can use this proposition to see that shift maps preserve the non-trivial homology classes. This is explained in Proposition \ref{Proposition:shifttrivialhomology}.

\begin{prop}\label{Proposition:shifttrivialhomology}
Consider a subsets $I \subseteq \bm{C^m} \setminus B^{\bu}_{L_B+2L_A+1}$ such that $S_I$ is a minimal cycle and $[S_I] \neq 0$. Then we have $$[S_{R_{\pm}(I)}] \neq 0$$ 
\end{prop}
\begin{proof}
    We prove the proposition by contradiction. We also prove it for $S_{R_{+}(I)}$ and the argument for the $S_{R_{-}(I)}$ is similar.\\ 
    
    First, note that based on Proposition \ref{Proposition:shiftbounddary}, $S_{R_{\pm}(I)}$ is clearly a cycle and hence $[S_{R_{+}(I)}]$ is well-defined. Using Proposition \ref{Proposition:shiftbounddary} and minimality of $S_{I}$, we can easily prove that $S_{R_{+}(I)}$ is also minimal.\\

    Assume there exist a minimal relation such that  
$$\partial S_T = S_{R_{+}(I)},$$
then we use Corollary \ref{Corollary:lengthminimalrelation} to conclude that 
$$\diam(T) \leq \len(S_T) \leq L_B. $$
Using Lemma \ref{Lemma:differentialorder}, we also know that 
$$\dist_{H}(T,R_{+}(I))<L_A.$$
Combining these inequalities we have
$$\dist_{H}(T,C^{\bu}) \geq \dist_{H}(R_{+}(I),C^{\bu}) -L_A - L_B \Rightarrow T \subseteq \bm{C^m} \setminus B^{\bu}_{L_A+1}.$$
This means we can consider the subset $R_{-}(T)$ and use Proposition \ref{Proposition:shiftbounddary} to conclude
$$\partial S_{R_{-}(T)} = S_{R_{-}(R_{+}(I))} =S_{I}.$$
This means that $[S_{I}]$ is trivial in homology which contradicts our assumption.\\
\end{proof}

\subsection{Extremal knot Floer homologies fall in a closed ball}\label{Subsection:extremahhfkfallinball}\hfill\\

We start studying the extremal knot Floer homologies in order to prove Theorem \ref{Theorem:extremalhfk}. Without loss of generality, we prove the theorem for the case when $\lk(K,c) > 0$.\\

As we have said before, as a result of the pairing theorem, we have: 
$$\widehat{CFK}(S^3,K_m) \simeq \widehat{\text{CFA}}(\mathcal{H}_{K}, z, w) \boxtimes \widehat{\text{CFD}}(\mathcal{H}'_{\frac{1}{m}} , z').$$
Lemma \ref{Lemma:extremalnearblackbox} states an important property of representatives of extremal knot Floer homologies in this chain complex. Before stating the lemma, we define an invariant $L_S$ as follows
$$L_S := L_B + N_AL_A +2L_A+1.$$
Note that $L_S$ doesn't depend on $m$.
\begin{lemm}\label{Lemma:extremalnearblackbox}
Fix an integer $k \in \mathbb{N}$. Let $S_I$ be a minimal cycle such that
\begin{equation}\label{Equation:extremalassumption}
[S_I] \in \bigcup_{0 \leq j \leq k-1} \widehat{HFK}(K_m,-g(K_m)+j) \ \text{and} \ [S_I] \neq 0.
\end{equation}
Then $S_I$ can't be very far from the black box $C^{\bu}$. To be precise, we have 
$$I \subseteq B^{\bu}_{L_S+k-1}.$$
\end{lemm}

\begin{proof}
We prove this lemma by contradiction. Assume that there is an element $\nu \in I$ such that $\dist_{H}(\{\nu\}, C^{\bu}) \geq L_S+k$. Based on Lemma \ref{Lemma:lengthminimal}, we have:
$$\diam(I) \leq \len(S_I) \leq N_AL_A \Rightarrow \dist_{H}(I, C^{\bu}) \geq L_S+k -N_AL_A$$
$$\Rightarrow I \subseteq \bm{C^m} \setminus B^{\bu}_{L_B+ 2L_A+k+1}$$
As a result, we can define the following family of $k$ subsets by iterated application of the shift map $R_+$:
$$R_{+}(I), R^2_{+}(I), \cdots, R^{k}_{+}(I) \subseteq \bm{C^m} \setminus B^{\bu}_{L_B+2L_A+1}.$$
Using  Proposition \ref{Proposition:shiftbounddary} and \ref{Proposition:shifttrivialhomology} we know that homology classes 
$$[S_{R_{+}(I)}], [S_{R^2_{+}(I)}], \cdots, [S_{R^{k}_{+}(I)}]$$
are all non-trivial.\\

Using Lemma \ref{Lemma:Gradingchangebyshift}, we can compute the Alexander grading of these classes as follows : 
$$a_m(S_{R^{i}_{+}(I)}) = a_m(S_{I}) + i \omega \ \text{for} \ 1 \leq i \leq k.$$
Combining this with Equation \ref{Equation:extremalassumption} we have: 
$$a_m(S_{R^{k}_{+}(I)}) \leq -g(K_m) + k-1 + k \omega < -g(K_m),$$
where the final inequality comes from the fact that $\omega = - \lk(K,c) \leq -1$. This is a clear contradiction as we know that knot Floer homology vanishes in the Alexander gradings below $-g(K_m)$.
\end{proof}

Now we start analyzing stabilization phenomena. We need to look at how the complex
$$\widehat{\text{CFA}}(\mathcal{H}_{K}, z, w) \boxtimes \widehat{\text{CFD}}(\mathcal{H}'_{\frac{1}{m}} , z')$$
changes as $m \rightarrow \infty$.\\

Recall Equation \ref{Equation:inclusion}, where we defined inclusion maps: 
$$\Phi_{m} : \bm{C^m} \hookrightarrow \bm{C^{m+1}}$$
As we mentioned there, it is easier to imagine $\bm{C^m}$ as a nested sequence, where $\bm{C^{m+1}}$ can be constructed from $\bm{C^m}$ by adding a white box. This is also depicted in Figure \ref{Figure:Nested}.\\

As it can be seen in Figure \ref{Figure:Nested}, it is clear that for any integer $L$, the elements of the closed ball 
$$B^{\bu}_L \subseteq \bm{C^m}$$ 
stabilize for $m > 2L$ i.e. $\Phi_m$ restricts to a bijection on the closed metric balls. Imagining the family $\bm{C^m}$ as a nested sequence, we can see the close ball as a subset of both $\bm{C^m}$ and $\bm{C^{m+1}}$ for high enough $m$.\\ 

\begin{figure}[h]
    \centering
    \includegraphics[width=\linewidth]{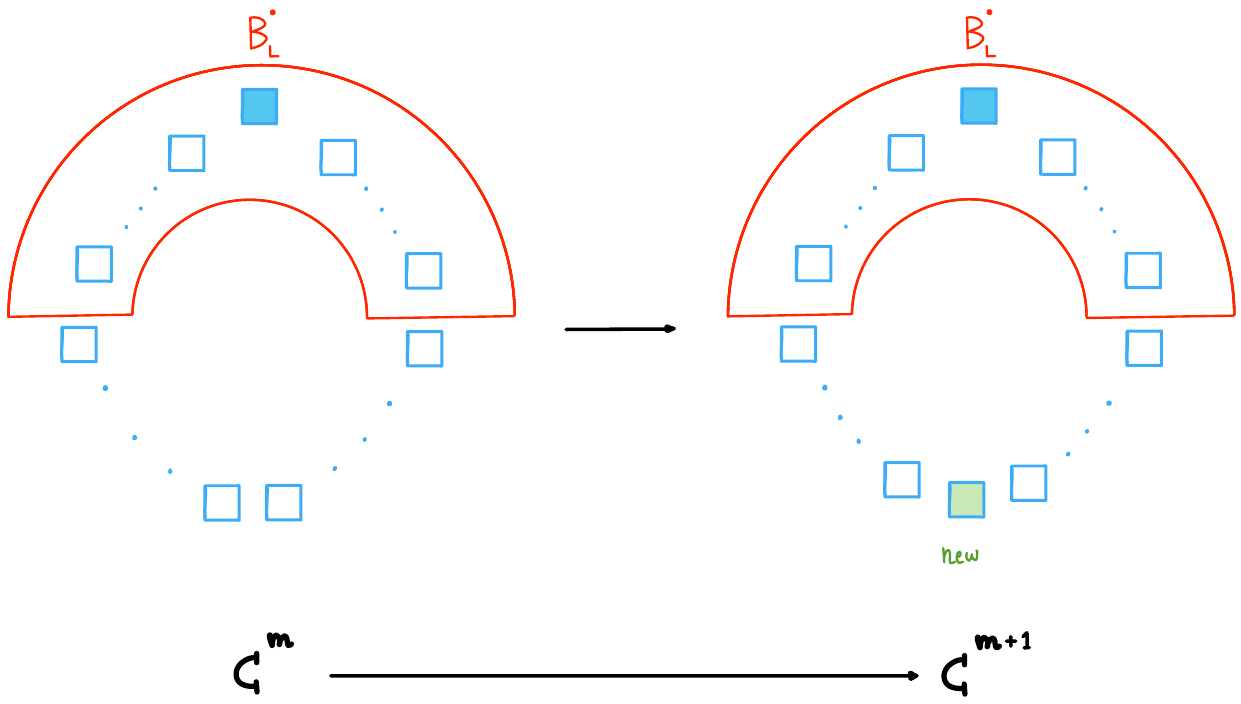}
    \caption{Constructing $\bm{C^{m+1}}$ from $\bm{C^m}$ by adding a new white box. For $m>2L$, this process doesn't affect $B^{\bu}_{L}$}
    \label{Figure:Nested}
\end{figure}

\begin{figure}[h]
    \centering
    \includegraphics[width=\linewidth]{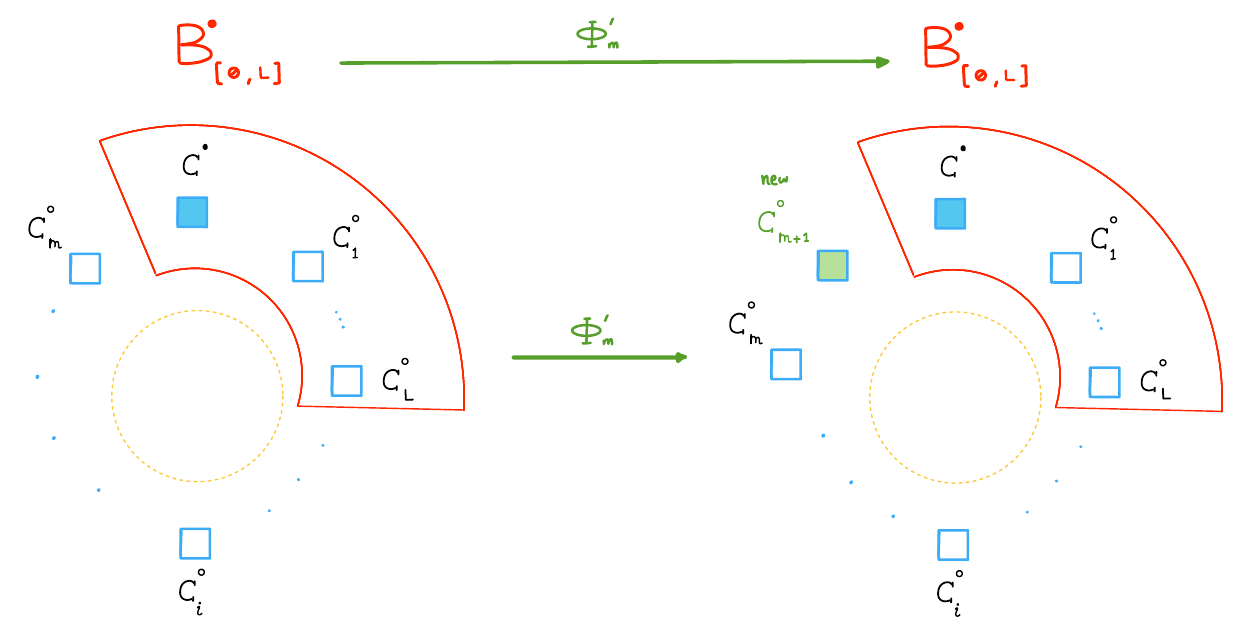}
    \caption{The half-ball $B^{\bu}_{[0,L]}$ and its stabilization under natural inclusion $\Phi'_{m}$}
    \label{Figure:Halfball}
\end{figure}

Clearly, this statement applies to the half-balls $B^{\bu}_{[0,L]}$ and $B^{\bu}_{[-L,0]}$ as well. Furthermore, the elements of the half-ball $B^{\bu}_{[0,L]}$ stabilize under the other nested sequence structure induced by inclusions $\Phi'_m$ (for $m > t$). This is illustrated in Figure \ref{Figure:Halfball}.\\

\subsection{Proof of Theorem \ref{Theorem:extremalhfk}}\label{Subsection:ProofofTheoremextremalhfk}\hfill\\

Now we are ready to prove Theorem \ref{Theorem:extremalhfk}.\\

\begin{proof}[Proof of Theorem \ref{Theorem:extremalhfk}]
 Assume that $m$ is high enough. To be precise assume that $m >  2L$ where 
 $$L:=(L_S+k-1) + L_B+2L_A.$$
 As we mentioned we can consider the stabilized closed metric balls 
 $$B^{\bu}_{L_S+k-1} \subset B^{\bu}_{L} \subset \bm{C^m} \subset \bm{C^{m+1}}$$
 Using Lemma \ref{Lemma:differentialorder}, we know that the differentials of elements of $B^{\bu}_{L_S+k-1}$ remain in the closed ball 
 $$B^{\bu}_{L_S+k-1+L_A} \subseteq B^{\bu}_{L}.$$ Furthermore, based on Proposition \ref{Proposition:boxtensordifferential}, these differentials doesn't depend on $m$. As a result, the differentials of the elements of $B^{\bu}_{L_S+k-1}$ also stabilize. In fact, the same argument shows that the differentials of elements in 
 $$B^{\bu}_{(L_S+k-1)+L_B+L_A} \subseteq B^{\bu}_{L-L_A}$$
 also remain in the ball $B^{\bu}_{L}$, and hence will be stabilized for $m>2L$.\\
 
 We will prove that this stabilized data, i.e elements of $B^{\bu}_{L}$ and differentials of the elements of $B^{\bu}_{L-L_A}$, includes the extremal knot Floer homologies 
 $$\bigcup_{0 \leq j \leq k-1} \widehat{HFK}(K_m,-g(K_m)+j).$$
 
 First, note that the space of cycles $S_I$ such that $I \subseteq B^{\bu}_{L_S+k-1}$ will be stabilized. We denote this set of cycles by 
$$\mathcal{C}_{L_S+k-1} \subseteq \widehat{\text{CFA}}(\mathcal{H}_{K}, z, w) \boxtimes \widehat{\text{CFD}}(\mathcal{H}'_{\frac{1}{m}} , z') \ \text{for all} \ m>2L.$$
Furthermore, for any family of minimal cycles 
$$S_{I_1}, \cdots, S_{I_r} \in \mathcal{C}_{L_S+k-1}$$
and a minimal relation 
$$\partial S_{R} =S_{I_1}+ \cdots + S_{I_r},$$
we can use Corollary \ref{Corollary:lengthminimalrelation} and Lemma \ref{Lemma:differentialorder} (similar to the proof of Proposition \ref{Proposition:shifttrivialhomology}) to deduce that 
$$R \subseteq B^{\bu}_{(L_S+k-1)+L_B+L_A} \subseteq B^{\bu}_{L-L_A}.$$
This means that the set of all such relations stabilize for $m>2L$. As a result we will have a stabilized subspace
$$\mathcal{H}_{L_S+k-1} \subseteq H_{*}(\widehat{\text{CFA}}(\mathcal{H}_{K}, z, w) \boxtimes \widehat{\text{CFD}}(\mathcal{H}'_{\frac{1}{m}} , z')) \ \text{for all} \ m>2L,$$
where $\mathcal{H}_{L_S+k-1}$ is the set of all homology classes of cycles $\mathcal{C}_{L_S+k-1}$.\\

Now we consider the relative Alexander grading of homogeneous classes in $\mathcal{H}_{L_S+k-1}$. As a result of Corollary \ref{Corollary:doublesidedgrading}, the grading double cosets of any element in $B^{\bu}_{L}$ has representatives which doesn't depend on $m$. We can then use Lemma \ref{Lemma:Gradingchangebym} to compute the relative Alexander gradings. It is clear that the exact amount of relative Alexander gradings might not stabilize, but since all the  relative Alexander gradings are linear functions of $m$, their signs will. As a result for $m \gg0$, we can pick a fixed homogeneous  class 
$$H_{\min} \in \mathcal{H}_{L_S+k-1}$$
such that for any other homogeneous class $H \in \mathcal{H}_{L_S+k-1} $ we have 
$$a_m(H) - a_m(H_{\min}) \geq 0.$$
Again, using Lemma \ref{Lemma:Gradingchangebym}, we can consider the following stabilized subspace:
$$\mathcal{H}^{k}_{L_S+k-1} \subseteq H_{*}(\widehat{\text{CFA}}(\mathcal{H}_{K}, z, w) \boxtimes \widehat{\text{CFD}}(\mathcal{H}'_{\frac{1}{m}} , z')) \ \text{for} \ m\gg0,$$
$$\mathcal{H}^{k}_{L_S+k-1} := \langle \{ H \in \mathcal{H}_{L_S+k-1} \ | \ a_m(H) - a_m(H_{\min}) \leq k-1 \} \rangle$$
Finally, based on Lemma \ref{Lemma:extremalnearblackbox} we have that:
$$a_m(H_{\min}) = -g(K_{m}) \ \text{and} \ $$
$$\bigcup_{0 \leq j \leq k-1} \widehat{HFK}(K_m,-g(K_m)+j) \simeq \mathcal{H}^{k}_{L_S+k-1}.$$
This means that we are almost done. We only need to examine how the absolute Maslov grading changes as $m \rightarrow +\infty$. First, assume that $$gr_{K_m}(H_{\min}) = [(\alpha_{\min}; \beta_{\min}, \gamma_{\min}; \delta_{\min})] \ \text{for} \ m\gg0.$$
Note that for any other homology class $H \in \mathcal{H}^{k}_{L_S+k-1}$ with 
$$gr_{K_m}(H) = [(\alpha_{H}; \beta_{H}, \gamma_{H}; \delta_{H})] \ \text{for} \ m\gg0,$$
we have  
\begin{equation}\label{Equation:betaofextremal}
   \beta_{H} = \beta_{\min}. 
\end{equation}
This follows from Lemma \ref{Lemma:Gradingchangebym}, and the fact that 
$a_m(H) - a_m(H_{\min})$ has bounded absolute value (due to definition of $\mathcal{H}^{k}_{L_S+k-1}$), and hence must be a linear function with zero slope.\\ 

To examine the absolute (and not relative) Maslov grading, we need a reference element. We can use Lemma \ref{Lemma:homotopyequivgenerator} and deduce that there is a fixed element 
$$\nu_0 \in B^{\bu}_{L} \subset \bm{C^m} \ \text{for} \ m\gg0,$$
with absolute Maslov grading zero. Furthermore, combining Lemma \ref{Lemma:homotopyequivgenerator} and Corollary \ref{Corollary:doublesidedgrading}, we have a fixed element $(\alpha_{0}; \beta_{0}, \gamma_{0}; \delta_{0}) \in \widetilde{G}$ such that
$$2 \beta_0 = M+\frac{1}{2} \ \text{and}$$
$$gr_{K_m}(\nu_0) = [(\alpha_{0}; \beta_{0}, \gamma_{0}; \delta_{0})] \ \text{for} \ m \gg0.$$
Now using Lemma \ref{Lemma:Gradingchangebym}, we can compute $h_{m}(H_{\min})$ as a function of $m$:
$$h_{m}(H_{\min}) = h_{m}(H_{\min}) - h_{m}(\nu_0) = ((\beta_{0}-\beta_{\min})(M-\frac{1}{2}-\beta_{\min}-\beta_{0})) \cdot m + \text{cte} $$
We are going to show that 
$$((\beta_{0}-\beta_{\min})(M-\frac{1}{2}-\beta_{\min}-\beta_{0})) = F_K.$$
Based on Equation \ref{Equation:betaofextremal}, this will finish our proof.\\

First, let's define the function $f(x)$ as follows:
$$f(x) : = (\beta_{0}-x)(M-\frac{1}{2}-x-\beta_{0}).$$
Second, note that all of our arguments up to here can be repeated for $k$ extremal knot Floer homologies with highest Alexander gradings i.e. 
$$\bigcup_{0 \leq j \leq k-1} \widehat{HFK}(K_m,g(K_m)-j).$$
This means we can also consider a class with the highest relative Alexander grading in $\mathcal{H}_{L_S+k-1}$, denoted by $H_{\max}$, and also assume that
$$gr_{K_m}(H_{\max}) = [(\alpha_{\max}; \beta_{\max}, \gamma_{\max}; \delta_{\max})] \ \text{for} \ m\gg0.$$
We will similarly have that 
$$a_{m}(H_{\max}) = g(K_m) \ \text{and} \ h_{m}(H_{\max}) = f(\beta_{\max}) \cdot m + \text{cte}.$$
This means that we have 
$$\widehat{HFK}(K_m,-g(K_m))\cong \widehat{HFK}(K_{m+1},-g(K_{m+1})) \ [f(\beta_{\min})], \ \text{and},$$
$$\widehat{HFK}(K_m,g(K_m))\cong \widehat{HFK}(K_{m+1},g(K_{m+1})) \ [f(\beta_{\max})].$$
However, we also can use the symmetry of the knot Floer homology:
$$\widehat{HFK}(K_m,g(K_m))[2g(K_m)]\cong \widehat{HFK}(K_m,-g(K_m)), \text{and}$$
$$\widehat{HFK}(K_{m+1},g(K_{m+1}))[2g(K_{m+1})]\cong \widehat{HFK}(K_m,-g(K_{m+1})).$$
We can summarize these four isomorphisms in the following square: 
\[
\begin{tikzcd}
\widehat{HFK}(K_{m+1},g(K_{m+1})) \arrow{r}{[f(\beta_{\max})]} \arrow[swap]{d}{[2g(K_{m+1})]} & \widehat{HFK}(K_{m},g(K_{m})) \arrow{d}{[2g(K_m)]} \\
\widehat{HFK}(K_{m+1},-g(K_{m+1})) \arrow{r}{[f(\beta{\min})]} & \widehat{HFK}(K_{m},-g(K_{m}))
\end{tikzcd}
\]
Combining all these we have:
$$-f(\beta_{\max})-2g(K_m)+f(\beta_{\min})=-2g(K_{m+1}) \Rightarrow $$
$$f(\beta_{\max})-f(\beta_{\min}) = 2(g(K_{m+1})-g(K_{m}))$$
We can rewrite this final equation as:
\begin{equation}\label{Equation:F1}
  (\beta_{\min}-\beta_{\max})(M-\frac{1}{2}-\beta_{\max}-\beta_{\min}) = 2(g(K_{m+1})-g(K_{m}))  
\end{equation}
Also note that we have 
$$a_m(H_{\max})-a_m(H_{\min}) = 2g(K_m)$$
Using Lemma \ref{Lemma:Gradingchangebym}, we have that 
\begin{equation}\label{Equation:F2}
 2g(K_m) = ((\beta_{\min}-\beta_{\max})\omega)\cdot m + \text{cte}.    
\end{equation}
Now combining Equations \ref{Equation:F1}, \ref{Equation:F2} and Theorem \ref{Theroem:BakerTaylor}, we can conclude:
\begin{equation*}
    \begin{cases}
      \beta_{\max}+ \beta_{\min} = M -\omega -\frac{1}{2}\\
       \beta_{\min}-\beta_{\max}=-x([\hat{D}])
    \end{cases}       
\end{equation*}
$$\Rightarrow 2\beta_{\min} = M - \omega - \frac{1}{2}-x([\hat{D}])$$
This finishes our proof. 
\end{proof}

\section{Stabilization of Alexander polynomial}\label{Section:Alexanderpolynomial}
\subsection{Defining the Alexander jump sequence $d_{m,i}$}\hfill\\

We now can continue with the proof of Theorem \ref{Theorem:extremalAlexander} and Theorem \ref{Theorem:jumpsAlexander}. As mentioned before, we use $\{\alpha_{m,i}\}_{i \in \mathbb{Z}}$ represent the series of coefficients of the Alexander polynomial of $K_m$, i.e.,
     $$\Delta_{K_m}(t)= \sum_{i \in \mathbb{Z}} \alpha_{m,i} \cdot t^i.$$
Clearly, we have $\alpha_{m,i} = 0$ when $|i|> \text{deg}(\Delta_{K_m})$.\\

We define the \emph{Alexander jump sequence} $\{d_{m,i}\}_{i \in \mathbb{Z}}$ as follows
$$d_{m,i} := \alpha_{m,i} - \alpha_{m,i+\omega} \ \text{for all} \ i \in \mathbb{Z} \ , \ r \in \{0,1\}.$$
We are going to prove an stabilization result about the sequence $\{d_{m,i}\}$ which will give us Theorem \ref{Theorem:extremalAlexander} and Theorem \ref{Theorem:jumpsAlexander} as corollaries. This stabilization result is phrased in Proposition \ref{Proposition:Alexanderjumpsequence}. 
\begin{prop}\label{Proposition:Alexanderjumpsequence}
There exist a sequence $\{d_{\infty,i}\}_{i \in \mathbb{Z}}$ such that for any $k \in \mathbb{N}$ and sufficiently large $m$, we have  
$$d_{m, -\text{deg}(\Delta_{K_m}) + i} = d_{\infty, i} \ \text{for all} \  0\leq i \leq k-1.$$ 
In other words, the first $k$ non-trivial terms of the sequence $\{d_{m,i}\}$ stabilize as $m \rightarrow \infty$. Furthermore, the total number of non-zero elements in the sequence $\{d_{m,i}\}$ stabilizes as $m \rightarrow \infty$. 
\end{prop}

\subsection{Far decomposition of $\HFKh(K_m)$}\label{Subsection:Fardecomposition}\hfill\\

To prove Proposition \ref{Proposition:Alexanderjumpsequence}, we first introduce some notations. First we define an invariant $L_{\delta}$ as follows:
$$L_{\delta} := L_{B}+2L_{A}+1.$$ We have seen the importance of the invariant $L_{\delta}$ in Proposition \ref{Proposition:shifttrivialhomology}.\\

We consider the set of elements in $\bm{C^{m}}$ which are \emph{far} from the black box $C^{\bu}$. To be more specific, we define $D'_m$ and $D''_m$ as follows:
$$D'_{m} : = \bm{C^m} \setminus B^{\bu}_{L_{\delta}} \ \text{and} \ D''_{m}:=R_{-}(D'_{m}).$$
We can see these subsets in Figure \ref{Figure:Alexanderjumpsdefects}.\\
\begin{figure}[h]
    \centering
    \includegraphics[width=\linewidth]{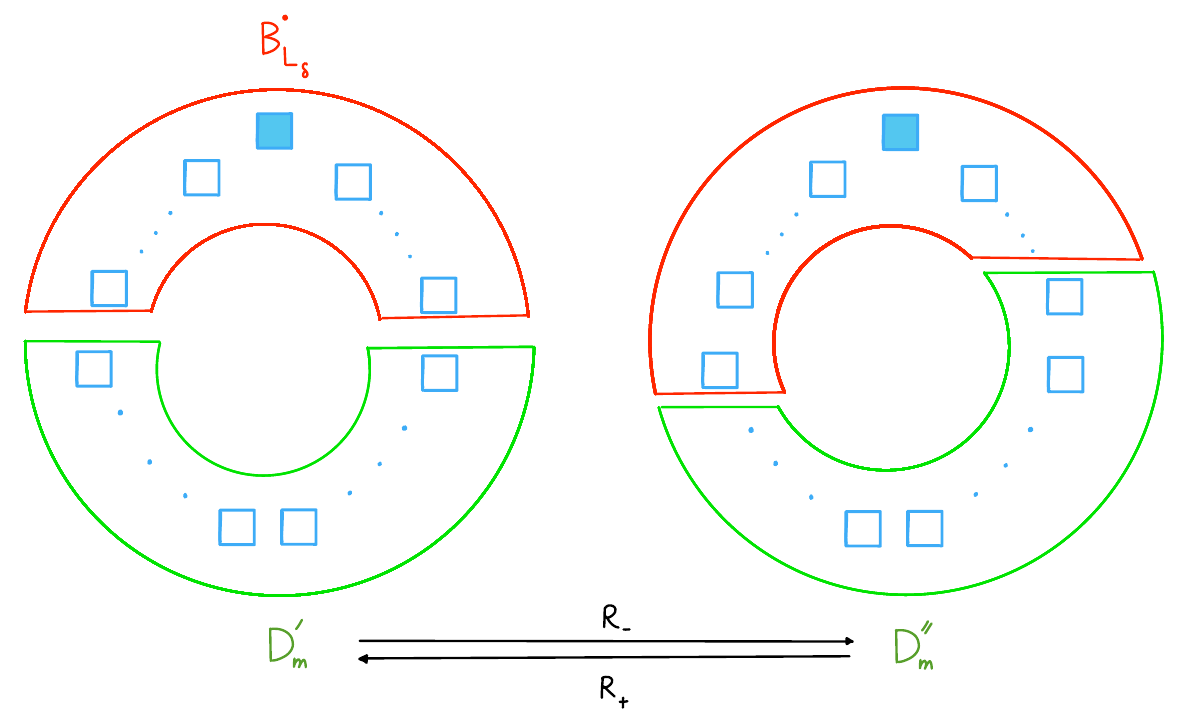}
    \caption{The subsets $D'_m$ and $D''_m$}
    \label{Figure:Alexanderjumpsdefects}
\end{figure}

Now recall that 
$$ \HFKh(K_m) = H_{*}(\widehat{\text{CFA}}(\mathcal{H}_{K}, z, w) \boxtimes \widehat{\text{CFD}}(\mathcal{H}'_{\frac{1}{m}} , z')).$$\
For any integer $i \in \mathbb{Z}$, we consider the subspace consisting of all homology classes in Alexander grading $i$ with a representative sitting in $D'_m$ i.e.
$$\mathcal{H}^{i}_{D'_m} := \{ H \in \HFKh(K_m , i) \ | \ H = [S_I]  \ \text{for some} \ I \subseteq D'_m  \}.$$
It is clear from the definition that $\mathcal{H}^{i}_{D'_m}$ is a vector space over $\mathbb{F}_2$. We consider the direct sum of all of these subspaces as follows: 
$$\mathcal{H}_{D'_m} := \bigoplus_{i \in \mathbb{Z}}\mathcal{H}^{i}_{D'_m}.$$
This is the subspace of homology classes that are \emph{far} from the the black box.\\

Now we can formally introduce the \emph{far decomposition} of $\HFKh(K_m)$ in Definition \ref{Defintion:Fardecomposition}. 

\begin{defi}\label{Defintion:Fardecomposition}
    Let $\overline{\mathcal{H}^{i}_{D'_m}}$ be the complement of $\mathcal{H}^{i}_{D'_m}$ as a subspace of $\HFKh(K_m,i)$. This means that we have 
    $$\HFKh(K_m , i) = \mathcal{H}^{i}_{D'_m} \oplus \overline{\mathcal{H}^{i}_{D'_m}}.$$
   We call this the \emph{far decomposition} of $\HFKh(K_m , i)$. We can also write 
   $$ \HFKh(K_m) = \mathcal{H}_{D'_m} \oplus \overline{\mathcal{H}_{D'_m}} \  \ \text{where} \ \ \overline{\mathcal{H}_{D'_m}} := \bigoplus_{i \in \mathbb{Z}}\overline{\mathcal{H}^{i}_{D'_m}}.$$
   We call this the \emph{far decomposition} of $\HFKh(K_m)$.
\end{defi}

Similarly we can define subspace $\mathcal{H}^{i}_{D''_m} \subseteq \HFKh(K_m , i)$ and its complement subspace $\overline{\mathcal{H}^{i}_{D''_m}}$, and then by taking direct sum over all $i \in \mathbb{Z}$, we have the subspaces $\mathcal{H}_{D''_m}$ and $\overline{\mathcal{H}_{D''_m}}$. It is clear that we also have the following direct sum decompositions of the knot Floer homology: 

$$ \HFKh(K_m,i) = \mathcal{H}^{i}_{D''_m} \oplus \overline{\mathcal{H}^{i}_{D''_m}} \ \  \text{and } \ \HFKh(K_m) = \mathcal{H}_{D''_m} \oplus \overline{\mathcal{H}_{D''_m}}$$

\subsection{Shift isomorphisms on far homologies}\label{Subsection:Shiftsonfar}\hfill\\

As we mentioned before the map $R_{-}: D'_{m} \rightarrow D''_{m}$ induces a map on the power sets. We claim that this map induces an isomorphism on the defined subspaces. This fact is phrased in Lemma \ref{Lemma:Alexanderjumpsdefects}. 
\begin{lemm}\label{Lemma:Alexanderjumpsdefects}
The map $R_{-} : \mathcal{H}^{i+\omega}_{D'_m} \rightarrow \mathcal{H}^{i}_{D''_m}$ defined by $$R_{-}([S_{I}]) := [S_{R_{-}(I)}],$$
is a well defined isomorphism of vector spaces. 
\end{lemm}
\begin{proof}
    We first need to show that the map is well defined. First consider the spaces of homogeneous cycles defined as follows: 
    $$\mathcal{C}^{i+\omega}_{D'_m}= \{ S_{I} \ | \ I \subseteq D'_m \ , \  \partial S_{I} = 0 \ , \ a_{m}(S_{I}) = i + \omega \} \ \text{and} \ $$
    $$\mathcal{C}^{i}_{D''_m}= \{ S_{I} \ | \ I \subseteq D''_m \ , \  \partial S_{I} = 0 \ , \ a_{m}(S_{I}) = i\}.$$
    Let $S_I \in \mathcal{C}^{i+\omega}_{D'_m}$ be a cycle. Since $L_{\delta} > L_{A}+1$, we have  
    $$ I \subseteq D'_m = \bm{C^m} \setminus B^{\bu}_{L_{\delta}} \subset \bm{C^m} \setminus B^{\bu}_{L_{A}+1}.$$
    As a result, we can use Proposition \ref{Proposition:shiftbounddary} to deduce that $S_{R_{-}(I)}$ is a cycle. Furthermore, we clearly have that 
    $$I \subseteq D'_m \Rightarrow R_{-}(I) \subseteq R_{-}(D'_m) = D''_m.$$
    Based on Lemma \ref{Lemma:Gradingchangebyshift}, the cycle $S_{R_{-}(I)}$ is homogenous and its Alexander grading can be computed as follows 
    $$a_{m}(S_{R_{-}(I)}) = a_m(S_{I}) - \omega = i.$$ Combination of these facts gives us that $S_{R_{-}(I)} \in \mathcal{C}^{i}_{D''_m}$. Hence the induced map 
    $$R_{-} : \mathcal{C}^{i+\omega}_{D'_m} \longrightarrow \mathcal{C}^{i}_{D''_m}$$ is a well-defined linear map. We can easily see that it is also an isomorphism as it has an inverse induced by $R_{+}$. \\

    We proved that $R_{-}$ induces a bijection on the cycles. It suffices to show that it also induces a bijection on the space of minimal relations. Consider a minimal relation 
    $$\partial S_{T} = S_{I_1} + \cdots + S_{I_r}$$ 
    where $S_{I_1}, \cdots, S_{I_r}$ are minimal cycles in $\mathcal{C}^{i+\omega}_{D'_m}$.\\
    
    Similar to the proof of Proposition \ref{Proposition:shifttrivialhomology}, we can use Corollary \ref{Corollary:lengthminimalrelation} and Lemma \ref{Lemma:differentialorder} to deduce that
    \begin{align}\label{Eqation:sidesofrelation}
    \dist_{H}(T, C^{\bu}) \geq \dist_{H}(\bigcup_{1 \leq j \leq r} I_j, C^{\bu}) - \dist_{H}(T,\bigcup_{1 \leq j \leq r} I_j) - \diam(T)\\
    \Rightarrow \dist_{H}(T, C^{\bu}) > L_{\delta} - L_A - L_B \geq L_A+1 \ \Rightarrow \  T \subseteq \bm{C^m} \setminus B^{\bu}_{L_A+1}.
    \end{align}
    Now we can again use use Proposition \ref{Proposition:shiftbounddary} to deduce that the following 
    $$\partial S_{R_{-}(T)} = S_{R_{-}(I_1)} + \cdots + S_{R_{-}(I_r)}$$ 
    is a minimal relation between cycles in $\mathcal{C}^{i}_{D''_m}$. This means that $R_{-}$ induces a map between the spaces of minimal relations, which is also bijective since it has an inverse induced by $R_{+}$. This compeletes our proof.
\end{proof}
Next, we will examine how the isomorphism 
$$R_{-} : \mathcal{H}^{i+\omega}_{D'_m} \rightarrow \mathcal{H}^{i}_{D''_m}$$ changes the Maslov grading modulo 2. This is stated in Lemma \ref{Lemma:Alexanderjumpsdefects-Maslovchange}.   
\begin{lemm}\label{Lemma:Alexanderjumpsdefects-Maslovchange} For any minimal homology class $H \in \mathcal{H}^{i+\omega}_{D'_m}$, the homology class $R_{-}(H)$ is also minimal and hence homogeneous. Furthermore,
$$h_{m}(H) \equiv  h_m(R_{-}(H))  \mod 2.$$
\end{lemm}
\begin{proof}
    Similar to the proof of Proposition \ref{Proposition:shifttrivialhomology}, we can use the minimality of $H$ and Proposition \ref{Proposition:shiftbounddary}, to prove that $R_{-}(H)$ is also minimal.\\
    
    To prove the second part of the statement, let $H \in \mathcal{H}^{i+\omega}_{D'_m}$ be any minimal homology class. Based on the definition, there exists subsets $I$ such that $H = [S_{I}]$. Consider the element $x_1 \in B^{1}_{K}$ and $i\in \mathbb{Z}$ such that $x \otimes \xi_{i} \in I$. Also assume that we have $gr_{K}(x_1) = [(a_1 ; b_1,c_1;d_1)]$. We can then use Lemma \ref{Lemma:Gradingchangebyshift} to write:
    $$h_{m}(H) - h_m(R_{-}(H)) = M-2b_1-\frac{1}{2}.$$
    Hence we only need to prove  
    $$M-2b_1-\frac{1}{2} \equiv 0  \mod 2.$$
    Consider the element $\nu_0$ coming from Lemma \ref{Lemma:homotopyequivgenerator2}. There exist an element $x_0 \in B^{0}_{K}$ such that $\nu_0=x_0 \otimes \eta$. Assume that $gr_{K}=[(a_0;b_0,c_0;d_0)]$. Then Corollary \ref{Corollary:gradingofhomotopyequivgenerator} gives us that $2b_0 = 2\beta_0 = M + \frac{1}{2}.$
    On the other hand, we can use Corollary \ref{Corollary:secondcomponentmod2-two} to write 
    $$2b_0 - 2b_1 \equiv 1 \mod 2.$$
    As a result, we have 
    $$M- 2b_1 - \frac{1}{2} \equiv M -2b_0 - \frac{1}{2} - 1\equiv 0 \mod 2.$$
\end{proof}

\subsection{Stabilization of the complement subspaces}\label{Subsection:Stabilizationofcomplement}\hfill\\

We prove some stabilization results about the complement subspaces $\overline{\mathcal{H}^{i}_{D'_m}}$ and $\overline{\mathcal{H}^{i}_{D''_m}}$, stated in Lemma \ref{Lemma:Alexanderjumpsdefects-stabilization} and Lemma \ref{Lemma:Alexanderjumpsdefects-relativeset}.\\

We state and prove the stabilization result about $\overline{\mathcal{H}_{D'_m}}$, and a similar result holds for $\overline{\mathcal{H}_{D''_m}}$.

\begin{lemm}\label{Lemma:Alexanderjumpsdefects-stabilization}
As $m \rightarrow \infty$, the subspace $\overline{\mathcal{H}_{D'_m}}$ will be stabilized. In other words, there exist a (ungraded) vector space $\overline{\mathcal{H}'_{\infty}}$ such that:
$$\overline{\mathcal{H}_{D'_m}} \cong \overline{\mathcal{H}'_{\infty}}.$$
\end{lemm}
\begin{proof}
     Proof of this stabilization result follows the same strategy as the previous proofs. We will show that there exist an integer $L_{\Delta}$, such that $\overline{\mathcal{H}_{D'_m}}$ can be computed from the \emph{data} of $B^{\bu}_{L_{\Delta}} \subset \bm{C^m}$ for $m\gg0$. The term \emph{data} here means the elements of $B^{\bu}_{L_{\Delta}}$ and their differentials. Based on Lemma \ref{Lemma:differentialorder}, the differentials of elements of $B^{\bu}_{L_{\Delta}}$ is included in $B^{\bu}_{L_{\Delta}+L_A}$. As a result, this data will be stabilized for ${m > 2(L_{\Delta}+L_A)}$.\\

    First, similar to before we work with a minimal basis for $\overline{\mathcal{H}_{D'_m}}$. Let's assume $H$ is a homology class in this basis, and pick a minimal cycle $S_I$ such that $H=[S_I]$. Based on the definition of $\overline{\mathcal{H}_{D'_m}}$, we have that $I\not\subset D'_{m}$. This means that the intersection $ I \cap B^{\bu}_{L_{\delta}}$ must be non-empty. On the other hand, using Lemma \ref{Lemma:lengthminimal}, we have:
    $$\diam(I) \leq \len(S_I) \leq N_AL_A \Rightarrow I \subseteq B^{\bu}_{L_{\delta}+N_AL_A}.$$
    Similar to the proof of Theorem \ref{Theorem:extremalhfk}, we can consider the space of cycles with representatives lying in $B^{\bu}_{L_{\delta}+N_AL_A}$, denoted by $\mathcal{C}_{L_{\delta}+N_AL_A}$. Based on what we have proven so far, we know that this subspace contains representatives for minimal classes in $\overline{\mathcal{H}_{D'_m}}$.\\

    Now we consider a minimal basis $\mathcal{B}$ for $\mathcal{C}_{L_{\delta}+N_AL_A}$. The next step would be to see how we can determine if a minimal cycle $S_I \in \mathcal{B}$ represents a class in $\overline{\mathcal{H}_{D'_m}}$. More generally, we can analyze the minimal relations 
    $$\partial S_{T} = (S_{I_1} + \cdots + S_{I_r}) + (S_{I'_1} + \cdots + S_{I'_{r'}}),$$
    where we have divided the cycles on the right hand side of the relation into two groups such that:
    \begin{itemize}
    \setlength\itemsep{1em}
        \item $S_{I_j} \in \mathcal{B} \ \text{and} \ I_{j} \not\subset D'_m  \  \text{for} \ j\in \{1,\cdots,r\}$ where $r \in \mathbb{Z}_{>0}$.
        
        \item $I'_{j'} \subseteq D'_m \ \text{for} \ j' \in \{1,\cdots, r'\}$ where $r' \in \mathbb{Z}_{\geq 0}$.
    \end{itemize}
     Note that the assumption $r>0$ is necessary, as we only care about relations which include at least one cycle in $\mathcal{C}_{L_{\delta}+N_AL_A}$ which can potentially represent a class in $\overline{\mathcal{H}_{D'_m}}$. These minimal relations would be enough to determine the subspace $\overline{\mathcal{H}_{D'_m}}$, as it is the complement of $\mathcal{H}_{D'_m}$.\\
    
    Now we can use Lemma \ref{Lemma:differentialorder} to deduce that
    $$\dist_{H}(T, C^{\bu}) \leq \dist_{H}(T,\bigcup_{1\leq j \leq r} I_{j} ) + \dist_{H}(\bigcup_{1\leq j \leq r} I_{j},C^{\bu}) \leq L_{A} + L_{\delta}+N_AL_A$$
    Combining this with Corollary \ref{Corollary:lengthminimalrelation}, we have that
    $$T \subseteq B^{\bu}_{L_{\delta}+N_AL_A+L_{A}+L_{B}}.$$
    This means we can set 
    $$L_{\Delta} = L_{\delta}+N_AL_A+L_{A}+L_{B}$$
    and the proof will be complete, as $B^{\bu}_{L_{\Delta}}$ includes both a basis for the space of all cycles representing classes in $\overline{\mathcal{H}_{D'_m}}$, and a basis for the space of all relations between them.   
\end{proof}
As we mentioned, a similar argument works for the subspace $\overline{\mathcal{H}_{D''_m}}$. Let the vector space $\overline{\mathcal{H}''_{\infty}}$ denote the result of stabilization for these subspaces.\\

Based on the proof of Lemma \ref{Lemma:Alexanderjumpsdefects-stabilization}, we can also state Corollary \ref{Corollary:Alexanderjumpsdefects-stabilization}.

\begin{coro}\label{Corollary:Alexanderjumpsdefects-stabilization}
    There exist an integer $L_{\Delta}$, such that any homology class in $\overline{\mathcal{H}'_{\infty}}$  or $\overline{\mathcal{H}''_{\infty}}$ has a fixed representative lying in $B^{\bu}_{L_{\Delta}}$. 
\end{coro}

For $m\gg0$, we can use the isomorphisms 
$$\overline{\mathcal{H}_{D'_m}} \cong \overline{\mathcal{H}'_{\infty}} \ \text{and} \ \overline{\mathcal{H}_{D''_m}} \cong \overline{\mathcal{H}''_{\infty}},$$ 
to consider the Alexander grading map $a_m$ as  grading maps 
$$a_m : \overline{\mathcal{H}'_{\infty}} \rightarrow \mathbb{Z} \ \text{and} \ a_m : \overline{\mathcal{H}''_{\infty}} \rightarrow \mathbb{Z}.$$
We can say the same thing about the Maslov grading map $h_m$ as well.\\

Similar to the proof of Theorem \ref{Theorem:extremalhfk}, the main missing element from Lemma \ref{Lemma:Alexanderjumpsdefects-stabilization} is that our stabilization technique can't directly address these gradings. Again similar to proof of Theorem \ref{Theorem:extremalhfk}, we can use the relative gradings and some other tricks to bypass this issue. Lemma \ref{Lemma:Alexanderjumpsdefects-relativeset} will be useful in this process.

  \begin{lemm}\label{Lemma:Alexanderjumpsdefects-relativeset}
      Consider a minimal basis $\mathcal{B}'_{\infty}$ for $\overline{\mathcal{H}'_{\infty}}$, and a minimal basis $\mathcal{B}''_{\infty}$ for $\overline{\mathcal{H}''_{\infty}}$. For any sufficiently large integer $m$, we define the set $\mathcal{PH}_{m}$ as follows: 
    $$\mathcal{PH}_{m} := \{(H',H'') \ | \ a_{m}(H') - a_{m}(H'') = \omega\} \subset \mathcal{B}'_{\infty} \times \mathcal{B}''_{\infty}.$$
    Then the sets $\mathcal{PH}_{m}$ will be stabilized to a set $\mathcal{PH}_{\infty}$, as $m \rightarrow \infty$.
  \end{lemm}
  \begin{proof}
    As a result of Corollary \ref{Corollary:doublesidedgrading}, the grading double cosets of any element in $B^{\bu}_{L_{\Delta}}$ has representatives which doesn't depend on $m$. We then use Lemma \ref{Lemma:Gradingchangebym} and Corollary \ref{Corollary:Alexanderjumpsdefects-stabilization} to deduce that for any pair of homology classes 
    $$H' \in \mathcal{B}'_{\infty}  \ \text{and} \ H'' \in  \mathcal{B}''_{\infty},$$
    the relative Alexander grading $a_{m}(H') - a_{m}(H'')$ is a linear function of $m$.\\

    Let $m_0$ be a sufficiently large integer and consider    
    $$ \mathcal{P}_{\max}= \max_{\substack{H' \in \overline{\mathcal{B}'_{\infty}} \\ H'' \in \overline{\mathcal{B}''_{\infty}}}}  \big| \big(a_{m_0}(H') - a_{m_0}(H'')\big)-\omega \big|.$$
    This is well-defined since both $\mathcal{B}'_{\infty}$ and $\overline{\mathcal{B}''_{\infty}}$ are finite sets.\\
    
     Now since $|a_{m}(H') - a_{m}(H'')|$ linearly increases for $m \gg0$, we can deduce that for $m > m_0+\mathcal{P}_{\max}$, the elements  $ \mathcal{PH}_{m}$ are exactly the pairs $(H',H'')$ satisfying  
    $$a_{m}(H') - a_{m}(H'') = 0 \cdot m + \omega,$$ 
    i.e. $a_{m}(H') - a_{m}(H'')$ is a constant function (linear with slope zero) with value $\omega$. This means that for $m > m_0+ \mathcal{P}_{\max}$, the set $\mathcal{PH}_{m}$ will be stabilized. 
  \end{proof}
We can go one step further and show that not only $\mathcal{PH}_{m}$ stabilize as sets, but they also stabilize as ordered sets. To make this precise, we define the subset $\mathcal{PH}^{i}_{m} \subseteq \mathcal{PH}_{m}$ as:
$$\mathcal{PH}^{i}_{m} := \{(H',H'') \in \mathcal{PH}_{m} \ | \ a_{m}(H'')=i \}.$$
Similarly we can define the subset $\mathcal{PH}^{i}_{\infty} \subseteq \mathcal{PH}_{\infty}$.\\

Now note that the set $\mathcal{PH}_{m}$ can be decomposed as 
$$\mathcal{PH}_{m} = \mathcal{PH}^{g_{m,1}}_{m} \sqcup \cdots \sqcup \mathcal{PH}^{g_{m,s_m}}_{m},$$
for a sequence of ordered integers $g_{m,1} < g_{m,2} < \cdots < g_{m,s_m}$.
Corollary \ref{Corollary:Alexanderjumpsdefects-relativeset-ordered} states that this decomposition also stabilizes as $m \rightarrow \infty$. 

\begin{coro}\label{Corollary:Alexanderjumpsdefects-relativeset-ordered}
    There exist an integer $s_{\infty}$ such that for $m\gg0$ we have $$s_m = s_{\infty}.$$ Furthermore there exist a fixed decomposition 
    $$\mathcal{PH}_{\infty} = \mathcal{PH}_{\infty,1} \sqcup \mathcal{PH}_{\infty,2} \sqcup \cdots \sqcup \mathcal{PH}_{\infty,s_{\infty}},$$
    such that 
    $$\mathcal{PH}_{\infty,i} = \mathcal{PH}^{g_{m,i}}_{\infty} = \mathcal{PH}^{g_{m,i}}_{m} \ \text{for all} \  1 \leq i \leq s_{\infty} \ \text{and} \ m \gg0.$$
    Also for any pair of integers $1 \leq i_1 \leq i_2 \leq s_{\infty}$, we have
    $$g_{m,i_1} - g_{m,i_2}$$
    stabilizes to a linear function of $m$. 
\end{coro}
\begin{proof}
    Consider the order induced on $\mathcal{PH}_{\infty}$ by the map 
    $$a_{m}(\text{pr}_2(\cdot)) : \mathcal{PH}_{\infty} \rightarrow \mathbb{Z}.$$
    To prove the statement of Corollary \ref{Corollary:Alexanderjumpsdefects-relativeset-ordered}, we only need to show that this order will be stabilized as $m \rightarrow \infty$. Similar to the proof of Lemma \ref{Lemma:Alexanderjumpsdefects-relativeset}, for any pairs of elements 
    $$(H'_1,H''_1) \ \text{and} \ (H'_2,H''_2) \  \text{in} \  \mathcal{PH}_{\infty},$$ 
    we have that $a_{m}(H''_1) - a_{m}(H''_2)$ is a linear function of $m$. As a result, for $m\gg0$, the sign of $a_{m}(H''_1) - a_{m}(H''_2)$ stabilizes. This means that the relative order of the elements 
    $$(H'_1,H''_1) \ \text{and} \ (H'_2,H''_2)$$ stabilizes as $m \rightarrow \infty$. Since $\mathcal{PH}_{\infty}$ is a finite set, we have that for sufficiently large $m$, all of the relative orders stabilize.\\
    
    The final statement about is obvious based on the aforementioned linearity of $a_{m}(H''_1) - a_{m}(H''_2)$. This finishes our proof.  
\end{proof}
Finally, we can also derive Corollary \ref{Corollary:Alexanderjumpsdefects-relativeset-maslov} about the stabilization of Maslov gradings modulo 2.
\begin{coro}\label{Corollary:Alexanderjumpsdefects-relativeset-maslov}
     For any element $(H',H'') \in \mathcal{PH}_{\infty}$, the gradins $h_{m}(H')$ and $h_{m}(H'')$ stabilize modulo 2, as $m \rightarrow \infty$. In other words, for any two integers $m_1,m_2\gg0$, we have 
     $$h_{m_1}(H') \equiv h_{m_2}(H') \ \  (\text{mod} \ 2) \ \text{and} \ h_{m_1}(H'') \equiv h_{m_2}(H'') \ \  (\text{mod} \ 2).$$
\end{coro}
\begin{proof}
To compute the absolute Maslov gradings, we use an argument similar to one in Proof of Theorem \ref{Theorem:extremalhfk}.\\

We can use Lemma \ref{Lemma:homotopyequivgenerator}, and deduce that there is a fixed element 
$$\nu_0 \in B^{\bu}_{L_{\Delta}} \subset \bm{C^m} \ \text{for} \ m\gg0,$$
with absolute Maslov grading zero. Then we can use Lemma \ref{Lemma:homotopyequivgenerator}, Corollary \ref{Corollary:doublesidedgrading} and Lemma \ref{Lemma:Gradingchangebym} to show that $$h_{m}(H') = h_{m}(H')-h_{m}(\nu_0)$$
is a linear function of $m$, with slope 
$$(\beta_{H'}-\beta_{0})(M - \frac{1}{2} - \beta_{H'} - \beta_{0}),$$
for $\beta_{H'},\beta_{0}$ elements of $\frac{1}{2}\mathbb{Z}$. We are going to show that this slope vanishes modulo $2$. Based on Lemma \ref{Lemma:homotopyequivgenerator}, we know that $$2\beta_{0} = M + \frac{1}{2}.$$
Using this we have 
$$(\beta_{H'}-\beta_{0})(M - \frac{1}{2} - \beta_{H'} - \beta_{0}) = (\beta_{H'}-\beta_{0})(\beta_{0} - \beta_{H'} -1). $$
As a result $\beta_{H'}-\beta_{0}$ must be an integer (in $\frac{1}{2}\mathbb{Z}$), as otherwise the slope won't be an integer. Using this fact we can easily see that the slope is an even integer and our proof is complete.The same proof works for stabilization of $h_{m}(H'')$ modulo 2. 
\end{proof}

\subsection{Stabilization of the Alexander jump sequence}\label{Subsection:stabilizationofAlexanderjump}\hfill\\

We are now finally ready to prove Proposition \ref{Proposition:Alexanderjumpsequence}.

\begin{proof}
    Based on basic properties of knot Floer homology, we have:
    $$\alpha_{m,i} = \chi(\HFKh(K_m,i)) = \chi(\mathcal{H}^{i}_{D'_m}) + \chi(\overline{\mathcal{H}^{i}_{D'_m}}).$$
    Now we can write that
    $$d_{m,i} = \alpha_{m,i} - \alpha_{m,i+\omega} = \chi(\HFKh(K_m,i)) - \chi(\HFKh(K_m,i+\omega))$$
    $$ = \Big( \chi(\mathcal{H}^{i}_{D''_m}) + \chi(\overline{\mathcal{H}^{i}_{D''_m}}) \Big) - \Big( \chi(\mathcal{H}^{i+\omega}_{D'_m}) + \chi(\overline{\mathcal{H}^{i+\omega}_{D'_m}})\Big)$$
    $$= \Big( \chi(\mathcal{H}^{i}_{D''_m}) - \chi(\mathcal{H}^{i+\omega}_{D'_m})  \Big) + \Big( \chi(\overline{\mathcal{H}^{i}_{D''_m}}) -\chi(\overline{\mathcal{H}^{i+\omega}_{D'_m}}) \Big).$$
    Using Lemmas \ref{Lemma:Alexanderjumpsdefects} and \ref{Lemma:Alexanderjumpsdefects-Maslovchange}, we can write that :
    $$\chi(\mathcal{H}^{i}_{D''_m}) -\chi(\mathcal{H}^{i+\omega}_{D'_m}) = 0.$$
    As a result, we have :
    \begin{equation*}\label{Equation:Alexanderjumps-chi}
        d_{m,i} = \chi(\overline{\mathcal{H}^{i}_{D''_m}}) - \chi(\overline{\mathcal{H}^{i+\omega}_{D'_m}}). 
    \end{equation*}
    Now we can use Lemma \ref{Lemma:Alexanderjumpsdefects-stabilization} and Corollary \ref{Corollary:Alexanderjumpsdefects-stabilization}. As a result, for $m \gg 0$ we can write:
    \begin{equation*}
        d_{m,i} = \begin{cases}
            \chi(p_2(\mathcal{PH}_{\infty,j})) - \chi(p_1(\mathcal{PH}_{\infty,j})), & \text{if $i=g_{m,j}$ for $1 \leq j \leq s_{\infty}$.}\\
            0, & \text{otherwise}.
        \end{cases}
    \end{equation*}
    Note that based on Corollary \ref{Corollary:Alexanderjumpsdefects-relativeset-maslov}, the terms $\chi(p_2(\mathcal{PH}_{\infty,j}))$ and $\chi(p_1(\mathcal{PH}_{\infty,j}))$ are well-defined.\\ 

   The above formula for values of $d_{m,i}$ is enough to prove the second statement of Proposition \ref{Proposition:Alexanderjumpsequence}. For sufficiently large $m$, we can write  
   $$\Big| \{ i \in \mathbb{Z} \ | \ d_{m,i}\neq 0 \} \Big| = $$
   $$\Big| \{ j \in \{1,\cdots,s_{\infty}\} \ | \  \chi(p_2(\mathcal{PH}_{\infty,j})) - \chi(p_1(\mathcal{PH}_{\infty,j})) \neq 0 \}\Big|.$$
   The right-hand side of this equation clearly doesn't depent on $m$, and hence is stabilized.\\

   Now we turn our focus to the first statement of Proposition \ref{Proposition:Alexanderjumpsequence} i.e. stabilization of extremal values of the sequence $(d_{m,i})_{i \in \mathbb{Z}}$. First let $j_0$ be the smallest element in $\{1, \cdots, s_{\infty}\}$ such that 
   $$d_{m,g_{m,j_0}} = \chi(p_2(\mathcal{PH}_{\infty,j_0})) - \chi(p_1(\mathcal{PH}_{\infty,j_0})) \neq 0. $$
   This means that $d_{m,g_{m,j_0}}$ will be the first non-zero term in the sequence $(d_{m,i})_{i \in \mathbb{Z}}$. As a result, we have : 
   $$g_{m,j_0} = -\deg(\Delta_{K_m}) \ \text{and}$$
   $$ \chi(p_2(\mathcal{PH}_{\infty,j_0})) - \chi(p_1(\mathcal{PH}_{\infty,j_0})) = d_{m,-\deg(\Delta_{K_m})}$$
   From the above equation, it is clear that $d_{m,-\deg(\Delta_{K_m})}$ stabilizes as ${m \rightarrow \infty}$.\\ 
   
   Now we need to extend this to the rest of extremal values of the sequence. Consider the intersection 
   $$NT_{m,k} := \{0, \cdots, k-1\} \ \cap \ \{g_{m,j} - g_{m,j_0}  \ | \ j_0 \leq j \leq s_{\infty}\}.$$
   Note that we can write 
  \begin{equation*}
        d_{m,-\deg(\Delta_{K_m})+i} = \begin{cases}
            \chi(p_2(\mathcal{PH}_{\infty,j})) - \chi(p_1(\mathcal{PH}_{\infty,j})), & \text{if $i=g_{m,j} - g_{m,j_0} \in NT_{m,k}$.}\\
            0, & \text{otherwise}.
        \end{cases}
    \end{equation*}
   
 Now we just need to prove that the sets $NT_{m,k}$ stabilize to a set $NT_{\infty,k}$, as $m \rightarrow \infty$. Based on the last statement of Corollary \ref{Corollary:Alexanderjumpsdefects-relativeset-ordered}, we can write :
   $$g_{m,j} - g_{m,j_0} = sg_j \cdot m + cg_j \ \text{for} \ m \gg0.$$
   Hence for sufficiently large $m$, the set $NT_{m,k}$ can only consist of elements $g_{m,j} - g_{m,j_0}$ which have zero slope (as a linear function of $m$) i.e.
   $$NT_{m,k}= NT_{\infty,k} = \{cg_{j} \ | \ j_0\leq j \leq s_{\infty}\ , \  sg_{i} = 0 \ , \ 0 \leq cg_{i_0} \leq k-1\} \ \text{for} \ m \gg0.$$
   Our proof is complete.  
\end{proof}

\subsection{Proofs of Theorems \ref{Theorem:extremalAlexander} and \ref{Theorem:jumpsAlexander}}\label{Subsection:proofsofalexanderstabilization}\hfill\\

We now derive Theorems \ref{Theorem:extremalAlexander} and \ref{Theorem:jumpsAlexander} from Proposition \ref{Proposition:Alexanderjumpsequence}. We start with Theorem \ref{Theorem:jumpsAlexander}.

\begin{proof}
Based on definitions, it is clear that $j_{m}$ is equal to the total number of non-zero elements in the sequence $\{d_{m,i}\}_{i \in \mathbb{Z}}$. As a result, the statement of Theorem \ref{Theorem:jumpsAlexander} is part of Proposition \ref{Proposition:Alexanderjumpsequence}. 
\end{proof}

Now we address Theorem \ref{Theorem:extremalAlexander}.

\begin{proof}
The first $k$ extremal coefficients of the Alexander polynomial i.e. 
$$\alpha_{m,-\deg(\Delta_{K_m}) + i} \ \text{for} \ i \in \{0,\cdots,k-1\},$$
can be computed from the first $k$ non-trivial terms of the sequence $\{d_{\infty,i}\}$. As a result the stabilization of these coefficients follow from Proposition \ref{Proposition:Alexanderjumpsequence}.\\

We only need to prove that $\deg(K_m)$ is a linear function of $m$ for sufficiently large $m$, and that the slope of this linear function is equal to 
$$\frac{l}{2}\wind_{c}(K) \ \text{for an integer} \ l \in \{0, \cdots, x([\hat{D}])\} . $$
Based on proofs of Theorem \ref{Theorem:extremalhfk} and Proposition \ref{Proposition:Alexanderjumpsequence}, and using Lemma \ref{Lemma:Gradingchangebym} we have: 
$$-g(K_m) - (-\deg(\Delta_{K_m})) = a_{m}(H_{\min}) - g_{m,j_0} $$
$$ = ((\beta - \beta_{\min})\omega) \cdot m + \text{cte} \ \text{for} \ m\gg0.$$
for some fixed $\beta \in \frac{1}{2}\mathbb{Z}$. Note that $\omega = - \wind_{c}(K)$.\\ 

Using Theorem \ref{Theroem:BakerTaylor}, we have: 
$$\deg(\Delta_{K_m}) = \Big( \frac{1}{2}(-2\beta + 2\beta_{\min} + x([\hat{D}]))\wind_{c}(K) \Big) \cdot m + \text{cte} \ \text{for} \ m\gg0.$$
Let $l=-2\beta + 2\beta_{\min} + x([\hat{D}])$. Finally we have 
$$0 \leq \deg(\Delta_{K_m}) \leq g(K_m) \ \Rightarrow \ 0 \leq l \leq x([\hat{D}]).$$
This finishes our proof.
\end{proof}
\section{Stabilization of the tau invariant}\label{Section:tau}
\subsection{Definition of $\tau(K_{m})$}\hfill\\

There are different equivalent definitions of the $\tau(K)$ in the Heegaard Floer literature. We use a definition given by Juhasz \cite{Juhász_2023} in the following.\\

As we mentioned, we work with the following two chain complexes:  
$$\widehat{CFK}(S^3,K_m) \simeq \widehat{\text{CFA}}(\mathcal{H}_{K}, z, w) \boxtimes \widehat{\text{CFD}}(\mathcal{H}'_{\frac{1}{m}} , z'),$$
$$\widehat{CF}(S^3) \simeq \widehat{\text{CFA}}(\mathcal{H}_{K}, z) \boxtimes \widehat{\text{CFD}}(\mathcal{H}'_{\frac{1}{m}} , z').$$
As we discussed, the two chain complexes have the same base vector space, i.e. a vector space generated by $\bm{C^m}$, and they only differ in the differential. In terms of the graph-theoretic descriptions, one can show that the second chain complex includes all of the edges of the first chain complex, but it can include additional edges.\\

Since the chain complexes share the same base space, we can consider the Alxander grading map, also as a map 
$$a_m : \widehat{\text{CFA}}(\mathcal{H}_{K}, z) \boxtimes \widehat{\text{CFD}}(\mathcal{H}'_{\frac{1}{m}} , z') \longrightarrow \mathbb{Z}.$$
The differential in $\widehat{CF}(S^3)$ might not preserve the Alexander grading, and hence this is not a grading map. However, one can show that this differential always decreases the Alexander grading and hence we can define a filtration as follows. The subspace 
$$\mathcal{F}(K,i) \subseteq \widehat{\text{CFA}}(\mathcal{H}_{K}, z) \boxtimes \widehat{\text{CFD}}(\mathcal{H}'_{\frac{1}{m}} , z')$$ 
is defined as the subspace generated by $\{\nu \in \bm{C^m} \ | \ a_{m}(\nu) \leq i \}.$\\

As we mentioned, the differential in $\widehat{CF}(S^3)$ can be restricted to a differential on $\mathcal{F}(K,i)$ and we can define the homology $H_{*}(\mathcal{F}(K,i))$. Furthermore, the inclusion induces a map 
$$ H_{*}(\mathcal{F}(K,i)) \longrightarrow H_{*}(\widehat{\text{CFA}}(\mathcal{H}_{K}, z) \boxtimes \widehat{\text{CFD}}(\mathcal{H}'_{\frac{1}{m}} , z')) = \HF(S^3).$$
This map then can be used to define tau invariant, as seen in Definition \ref{Definition:tau}.

\begin{defi}\label{Definition:tau}
    Let $\{K_m\}$ be a twist family of knots. The tau invariant $\tau(K_m)$ is defined as
    $$\tau(K_m) = \min \{i \in \mathbb{Z} \ : \ H_{*}(\mathcal{F}(K,i)) \rightarrow \HF(S^3) \ \text{is nontrivial}\}.$$
\end{defi}

For our purposes, we need to rephrase this definition. As mentioned before $\HF(S^3)$ has a single generator. Consider the set of cycles 
$$ \mathcal{HS}_{m} \ \subset \ \widehat{\text{CFA}}(\mathcal{H}_{K}, z) \boxtimes \widehat{\text{CFD}}(\mathcal{H}'_{\frac{1}{m}} , z'),$$
defined as 
$$\mathcal{HS}_{m} := \{ S_I \ | \ I \subseteq \bm{C^{m}}\ , \  \partial S_I = 0\ ,\  [S_I] \ \text{is a generator of} \  \HF(S^3) \}.$$
Now we can rephrase Definition \ref{Definition:tau}, in Proposition \ref{Proposition:taudef}.
\begin{prop}\label{Proposition:taudef}
   Let $\{K_m\}$ be a twist family of knots. Then we have 
   $$\tau(K_m) = \underset{S_{I}\in \mathcal{HS}_{m}}{\min} \Big( \underset{\nu \in I}{\max} \ a_{m}(\nu) \Big).$$
\end{prop}
\begin{proof}
    Using Definition \ref{Definition:tau} and definition of $\mathcal{HS}$, we have 
    $$\tau(K_m) = \min \{i \in \mathbb{Z} \ : \ \mathcal{F}(K,i) \cap \mathcal{HS}_{m} \neq \varnothing \}.$$
    Furthermore, it is clear that for any element $S_I \in \mathcal{HS}_{m}$, we have 
    $$\min \{i \in \mathbb{Z} \ : \ S_I \in \mathcal{F}(K,i) \} = \underset{\nu \in I}{\max} \ a_{m}(\nu).$$
    The proposition follows from these two facts.
\end{proof}
We can further restrict the domain of the minimisation as follows. For any $S_{I} \in \mathcal{HS}_{m}$, there exist a subset $T \subseteq I$ such that $S_{T}$ is a homogeneous cycle with $[S_{T}] = [S_{I}]$.  We denote the subset $T$ with the notation $MH(I)$. Note that both $S_I$ and $S_{MH(I)}$ are elements of $\mathcal{HS}_{m}$, however since $MH(I) \subseteq I$, we have 
$$\underset{\nu \in MH(I)}{\max} \ a_{m}(\nu) \leq \underset{\nu \in I}{\max} \ a_{m}(\nu).$$
Hence, we can discard all the non-homogeneous cycles from $\mathcal{HS}_{m}$ and still get the same result in the min-max function in Proposition \ref{Proposition:taudef}. Equivalently, if we define 
$$\mathcal{HZ}_{m} := \{ S_I \ | \ S_{I} \in \mathcal{HS}_{m} \  , \  S_{I} \ \text{homogeneous}\},$$
we can write:   
\begin{equation}\label{Equation:taudef}
\tau(K_m) = \underset{S_{I}\in \mathcal{HZ}_{m}}{\min} \Big( \underset{\nu \in I}{\max} \ a_{m}(\nu) \Big).\end{equation}
We will use this approach to the tau invariant in the next subsection.\\

For any subset $I \subseteq \bm{C^{m}}$, let $V_{\max}(I)$ be the smallest element (in the clockwise order) such that $$a_{m}(V_{\max}(I)) = \underset{\nu \in I}{\max} \ a_{m}(\nu).$$
For any cycle $S_I \in \mathcal{HS}_{m}$, we define $V_{\max}(S_I)$ to be $V_{\max}(I)$.\\

In other words, $V_{\max}(S_I)$ is the maximizer of the Alexander grading on the subset $I$. The main strategy for the proof of Theorem \ref{Theorem:tau} is to analyze the stabilization of these maximizers. To be more precise, define the set $\mathcal{P}_m \subset \bm{C^{m}}$ as follows:
$$\mathcal{P}_m := \{ V_{\max}(S_{I}) \ | \ S_{I} \in \mathcal{HZ}_{m} \}.$$
We can rewrite Equation \ref{Equation:taudef}, using the set $\mathcal{P}_m$ as follows: 
\begin{equation}\label{Equation:tauandPm}
\tau(K_{m}) = \underset{\nu \in \mathcal{P}_{m}}{\min}  \ a_{m}(\nu).  
\end{equation}
Later in this section, we prove Lemma \ref{Lemma:taupotentials}, which is the main gateway into proving Theorem \ref{Theorem:tau}.
This result states that for $m \gg0$, the set $\mathcal{P}_m$ stabilises under the nested sequence structure induced by the inclusions $\Phi_{m}$.
\begin{lemm}\label{Lemma:taupotentials}
    For $m\gg0$ we will have that 
    $$\Phi_{m}(\mathcal{P}_{m}) = \mathcal{P}_{m+1} \subset \bm{C^{m+1}}.$$
    Furthermore, there exist an integer $L_{\tau}$ such that $\mathcal{P}_m$ is a fixed subset of the ball $B^{\bu}_{L_{\tau}}$ for high enough $m$. 
\end{lemm}
Prior to this, we need a thorough analysis of elements with Maslov grading zero in $\bm{C^{m}}$. This is the main topic of Subsection \ref{Subsection:Maslovgradingzeroelements}.

\subsection{The set of elements with Maslov grading zero}\label{Subsection:Maslovgradingzeroelements}\hfill\\

We are going to analyze the set of elements with Maslov grading zero in this subsection. We use the notation $GZ_{m}$ to denote this set. In other words, $GZ_{m} \subseteq \mathbf{C^{m}}$ is defined as follows: 
$$GZ_{m} = \{ \nu \in \bm{C^{m}} \ | \ h_m(\nu) =0 \}. $$

We start by some simple facts about stabilization of relative gradings, stated in Corollaries \ref{Corollary:Gradingsinaball} and \ref{Corollary:Masolvgradingzeroinball}. Some of these facts were used in the arguments of Section \ref{Section:ExtermalknotFloerstabilization}, but wasn't explicitly stated before. 

\begin{coro}\label{Corollary:Gradingsinaball}
 Let $t$ be a positive integer and consider the closed ball $B^{\bu}_{t} \subseteq \bm{C^{m}}$ for $m \gg 0$. Consider two elements $\nu_1,\nu_2 \in B^{\bu}_{t}$, and consider fixed representatives for the grading double coset of elements as follows: 
 $$gr_{K_m}(\nu_i) = [(\alpha_i; \beta_i, \gamma_i ; \delta_i)] \ \text{for} \ i \in \{1,2\}.$$
 Then all the following statements are true for all $m \gg0:$
 \begin{equation}\label{Equation:Alexandergradingslope}
  a_{m}(\nu_1) - a_{m}(\nu_{2}) = \text{cte.} \Leftrightarrow  \beta_{2} - \beta_{1} = 0\ ,
\end{equation}
 \begin{equation}\label{Equation:Alexandergradingslope2}
 \beta_{2} > \beta_{1}  \Longrightarrow a_{m}(\nu_2) > a_{m}(\nu_1),  \ \text{and}\end{equation}
\begin{equation}\label{Equation:Gradingsinaball-Maslov}
((\beta_2-\beta_1)(M-\frac{1}{2}-\beta_1-\beta_2))>0 \Rightarrow h_m(\nu_1) >  h_m(\nu_2)
\end{equation}
\end{coro}
\begin{proof}
    This follows from Lemma \ref{Lemma:Gradingchangebym}. We know that relative Alexander and Maslov grading of elements in $B^{\bu}_{t}$ grow linearly. This means that the sign of these relative gradings will stabilize to the sign of their slope as $m \rightarrow \infty$, if the slope is non-zero. When the slope is zero, and only in this case, the relative gradings will be constant (as a function of $m$).\\
    
    For high enough $m$, this stabilization of signs will happen for all pairs of elements in $B^{\bu}_{t}$. The statements follow from the formulas in Lemma \ref{Lemma:Gradingchangebym}. Note that we are working under the assumption that $\omega < 0$. 
\end{proof}

With a similar argument, we get a result regarding the elements with Maslov grading zero in $B^{\bu}_{t}$. This result is stated in Corollary \ref{Corollary:Masolvgradingzeroinball}. 

\begin{coro}\label{Corollary:Masolvgradingzeroinball}
    Consider an element $\nu \in B^{\bu}_{t}$ with a fixed representative $(\alpha; \beta, \gamma ; \delta)$ for the grading double coset $gr_{K_{m}}(\nu)$ for $m \gg 0$ (coming from Corollary \ref{Corollary:doublesidedgrading}). Let $\beta_{0}$ be $\frac{M}{2}+\frac{1}{4}$. Then we have 
    $$h_{m}(\nu) = 0 \Longrightarrow \beta = \beta_0 \ \text{or} \ \beta=\beta_0-1.$$
    Furthermore, the set $GZ_{m} \cap B^{\bu}_{t}$ (i.e elements with absolute Maslov grading zero inside the ball $B^{\bu}_{t}$) will be stabilized for high enough $m$.
\end{coro}
\begin{proof}
    We use Lemma \ref{Lemma:homotopyequivgenerator2} and Corollary \ref{Corollary:gradingofhomotopyequivgenerator}. We know that there is an element $\nu_0 \in C^{\bu} \subset B^{\bu}_{t}$ such that $h_{m}(\nu_0) = 0$ for all positive integers $m$. Furthermore, we know that 
    $$gr_{K_m}(\nu_{0}) = [(\alpha_0; \beta_0, \gamma_0; \delta_0)].$$
    We use the same argument of Corollary \ref{Corollary:Gradingsinaball}. Based on Lemma \ref{Lemma:Gradingchangebym}, for any element $\nu \in B^{\bu}_{t}$, the absolute Maslov grading $h_{m}(\nu) = h_m(\nu) - h_{m}(\nu_0)$ is a fixed linear function of $m$. This means that the whole set of of Maslov gradings of elements of $B^{\bu}_{t}$ is given by a fixed set of linear functions. As a result, for high enough $m$, the element $\nu$ will have Maslove grading if and only if $h_{m}(\nu)$ is the constant zero function. Hence, we have:
    $$h_m(\nu) = 0 \Leftrightarrow (\beta_0-\beta)(M-\frac{1}{2}-\beta-\beta_0) =0 \ \ \text{and}$$
     $$(\gamma_0-\gamma)M + (\alpha - \alpha_0) + \frac{1}{2}(\beta_0-\beta) - (\beta_0 - \beta) \gamma - (\gamma_0 - \gamma)\beta_0 = 0.$$
     This gives us the stabilization statement as the conditions mentioned above doesn't depend on $m$. Now to prove the statement about $\beta$ we can write 
     $$(\beta_0-\beta)(M-\frac{1}{2}-\beta-\beta_0) =0 \Rightarrow \beta = \beta_0 \ \text{or} \ \beta = M-\frac{1}{2}-\beta_0  = \beta_0-1.$$ 
\end{proof}

We are going to extend the result of Corollary \ref{Corollary:Masolvgradingzeroinball} from $GZ_{m} \cap B^{\bu}_{t}$ to the whole set $GZ_{m}$, as it can be seen in Proposition \ref{Proposition:MasolvgradingzeroAll}.\\

First, recall that we have two grading $gr_{K_m}$ comes from two gradings $gr_K$ and $gr_{m}$ on $\widehat{\text{CFA}}(\mathcal{H}_{K}, z, w)$ and $\widehat{\text{CFD}}(\mathcal{H}'_{\frac{1}{m}} , z')$ respectively. We can always choose fixed representatives for $gr_{K}$. We introduced two different methods for choosing the representatives for $gr_m$ on $\widehat{\text{CFD}}(\mathcal{H}'_{\frac{1}{m}} , z')$, as described in Remark \ref{Remark:Twogradings}.\\

Now we can state and proof Proposition \ref{Proposition:MasolvgradingzeroAll}. 

\begin{prop}\label{Proposition:MasolvgradingzeroAll}
    Assume that $m \gg0$. Then for any element $\nu \in GZ_{m}$, Let $(\alpha,\beta,\gamma,\delta)$ be the representative for grading double coset $gr_{K_m}(\nu)$ coming from the counterclockwise representatives of $gr_m$. Then we have:
    $$\beta \in \{\beta_0,\beta_0 -1,\beta_0-2\}$$ 
\end{prop}
\begin{proof}
It is useful to focus on the nested sequence structure induced from the inclusion $\Phi_m$ in this proof.\\

The statement follows from Corollary \ref{Corollary:Masolvgradingzeroinball} if $\nu \in C^{\bu}$, hence we can assume there exists $x \in B_{K}$ and $i \in \{1,\dots, m\}$ such that $\nu = x \otimes \xi_i$.\\

Consider the element $\nu' = x \otimes \xi_{m} $ in $C^{\circ}_{m}$. You can think of this element as a fixed element of $B^{\bu}_{1}$, stabilized under the inclusion $\Phi_m$. Let $(\alpha';\beta',\gamma';\delta')$ be the representative for $gr_{K_m}(\nu')$ coming from the counterclockwise representatives of $gr_m$. In particular, this means that
$$\text{if} \  \ gr_{K}(x) = [(a;b,c;d)] \ \  \text{then} \ \  \beta' = b - \frac{1}{2}.$$ 
Note that this representative is fixed and does not depend on $m$. Furthermore, note that we have $\beta = \beta'$, coming from the description of the counterclockwise representatives of $gr_m$. \\

Similar to the proof of Corollary \ref{Corollary:Masolvgradingzeroinball}, let $\nu_0 \in C^{\bu}$ be the element coming from Lemma \ref{Lemma:homotopyequivgenerator2}. For all $m \gg0$, as a result of Lemma \ref{Lemma:Gradingchangebym}, we have: 
\begin{equation}\label{Equation:Maslovgrading-firstline}
h_{m}(\nu') = h_{m}(\nu') - h_{m}(\nu_0) = (\beta_{0}- \beta')((\beta_0 -1) - \beta') \cdot m + F,
\end{equation}
where $F$ is a fixed integer.\\

Now we can use Lemma \ref{Lemma:Gradingchangebyshift} and write:
\begin{equation*}
h_{m}(x \otimes \xi_{j}) - h_{m}(\nu') = h_{m}(x \otimes \xi_{j}) - h_{m}(x \otimes \xi_{m}) = (j-m)(M-2b-\frac{1}{2})    
\end{equation*}
\begin{equation}\label{Equation:Maslovgrading-secondline}
 = (j-m)(2\beta_{0}-2\beta'-2) \ \ \text{for} \ j\in \{1,\dots,m\}.   
\end{equation}
Based on Equations \ref{Equation:Maslovgrading-firstline} and \ref{Equation:Maslovgrading-secondline}, we can see that $h_{m}(\nu')$ is a linear function of $m$ with slope $(\beta_{0}- \beta')((\beta_0 -1) - \beta')$, while the set of numbers 
$$\{h_{m}(x \otimes \xi_j) - h_{m}(\nu') \ | \ 1\leq j \leq m\} $$ forms an arithmetic progression with size $m$ and slope $(2\beta_0-2\beta'-2)$ starting from zero.\\

Remember that we are working with the assumption that $m \gg0$. Hence, to have an integer $i \in \{1,\dots,m\}$ such that 
$$h_{m}(x \otimes \xi_i) = 0 \  \  \text{or equivalently} \  \ h_{m}(x \otimes \xi_i) - h_{m}(\nu') = - h_{m}(\nu'),$$
we must have: 
$$|(2\beta_0-2\beta'-2)| \geq |(\beta_{0}- \beta')((\beta_0 -1) - \beta')|.$$
This can only happen when $|\beta_0-\beta'| \leq 2$.\\

If $\beta' \in \{\beta_0+1, \beta_0+2\}$, we can use Corollary \ref{Corollary:Gradingsinaball} (especially Equation \ref{Equation:Gradingsinaball-Maslov}) to deduce that $h_{m}(\nu') > 0$. Furthermore, we have 
$$h_{m}(x \otimes \xi_j) - h_{m}(\nu') = -4(j-m) \ \text{or}  \ -6(j-m) \Longrightarrow h_{m}(x \otimes \xi_j) > h_{m}(\nu') > 0.$$
This gives us the desired result. 
\end{proof}

As a result of Proposition \ref{Proposition:MasolvgradingzeroAll}, we can consider the the decomposition 
$$GZ_m = GZ_{m}^{0}  \ \sqcup \  GZ_{m}^{1} \  \sqcup \  GZ_{m}^{2},$$
where $GZ_{m}^{i}$ is defined to be the set of elements $\nu \in GZ_{m}$, such that $(\alpha;\beta_0-i,\gamma;\delta)$ is the representative for grading double coset $gr_{K_m}(\nu)$ coming from the counterclockwise representatives of $gr_m$.\\

Extending the argument in the proof of Proposition \ref{Proposition:MasolvgradingzeroAll}, we can further analyze the set $GZ_{m}$ and its decomposition, as done in Corollary \ref{Corollary:GZdecomposition}.

\begin{coro}\label{Corollary:GZdecomposition}
Assume that $m \gg0$. There exist a positive integer $L_{GZ}$ such that 
$GZ_{m}^{0}$ and $GZ_{m}^{2}$ are stabilized subsets of $B^{\bu}_{L_{GZ}}$. To be more precise we will have: 
$$GZ_{m}^{0} \subseteq B^{\bu}_{[-L_{GZ}, 0]} \ \  \text{and} \ \ \Phi_{m}(GZ_{m}^{0}) = GZ_{m+1}^{0}, $$
$$GZ_{m}^{2} \subseteq B^{\bu}_{[0,L_{GZ}]} \ \  \text{and} \ \ \Phi_{m}(GZ_{m}^{2}) = GZ_{m+1}^{2}.$$
Furthermore, there exist a set $ZB \subseteq B_{K}$ such that 
$$GZ_{m}^{1} \ \setminus \ C^{\bu} = \{x \otimes \xi_i \ | \ x \in ZB \ , \  i \in \{1,\dots,m\}\}.$$
\end{coro}
\begin{proof}
    Note that based on Corollary \ref{Corollary:Masolvgradingzeroinball}, the statement holds for $GZ_m \cap C^{\bu}$.\\ 
    
    To prove the first statement, similar to the proof of Proposition \ref{Proposition:MasolvgradingzeroAll}, Consider an element $\nu \in GZ^{0}_m \setminus C^{\bu}$. Assume that $\nu = x \otimes \xi_i$. Consider $\nu' = x \otimes \xi_m$ with $(\alpha';\beta',\gamma';\delta')$ being the representative for $gr_{K_m}(\nu')$ coming from the counterclockwise representatives of $gr_m$.
    It folllows from the assumption $\nu \in GZ^{0}_m$ that $\beta' = \beta_0$. Rewriting Equations \ref{Equation:Maslovgrading-firstline} and \ref{Equation:Maslovgrading-secondline}, we have
    $$h_{m}(\nu') = F \  \ \text{and}  \ \ h_{m}(x \otimes \xi_j)-h_{m}(\nu') = -2(j-m) \ \ \text{for} \ j\in \{1,\dots,m\}.$$
    As a result, we have 
    $$h_{m}(\nu) = h_{m}(x \otimes \xi_i) = 0 \Longleftrightarrow m-\frac{F}{2} \in \mathbb{Z_{+}} \ \  \text{and} \ \  i = m-\frac{F}{2} \Rightarrow \nu \in B^{\bu}_{[-\frac{F}{2},0]} $$
   Running this argument over all the choices for $\nu' \in C^{\circ}_{m}$ gives us the first desired statement. A similar argument gives us the statement about $GZ^{2}_m$.\\

To prove the final statement about $GZ_{m}^{1}$, we can proceed the same way, but this time $\beta' = \beta_0-1$. Then we have 
$$h_{m}(x \otimes \xi_j)-h_{m}(\nu') = 0 \ \ \text{for} \ j\in \{1,\dots,m\}.$$
Hence, if $h_{m}(x \otimes \xi_i) = 0$ then we have $h_{m}(x \otimes \xi_j) = 0$ for all $j\in \{1,\dots,m\}$. This completes the proof.
\end{proof}

\begin{rema}
The results in Corollary \ref{Corollary:GZdecomposition} and  Corollary \ref{Corollary:Masolvgradingzeroinball} might seem contradictory at first. However, one can notice that in Corollary \ref{Corollary:Masolvgradingzeroinball} we use the fixed grading representatives for elements of the closed ball, coming from Corollary \ref{Corollary:doublesidedgrading}. This coincides with the counterclockwise representatives used in Corollary \ref{Corollary:GZdecomposition} on the negative half-ball, but they are different for the elements of positive half-ball. In particular, when we change from the counterclockwise representatives to the clockwise representatives on the positive half-ball, the second component changes by $+ 1$ (apart from the elements of $C^{\bu}$). This can also be seen as all the elements of $GZ_{m}^{2}$ sit in the positive half-ball, and hence they all have fixed representatives for their grading double coset where the second component is equal to $$(\beta_0-2)+1=\beta_0-1.$$
\end{rema}

\subsection{Stabilization of maximizer of the Alexander grading}\label{Subsection:Stabilizationofmaximizers}\hfill\\

In this subsection, we aim to prove Lemma \ref{Lemma:taupotentials}. We start by showing that the distance of $\mathcal{P}_{m}$ and the black box $C^{\bu}$ is bounded from above. This is the content of Proposition \ref{Proposition:taupotentials-bounded}.

\begin{prop}\label{Proposition:taupotentials-bounded}
There exist a fixed integer $L_{\tau}$ such that 
 $$\mathcal{P}_{m} \subseteq B^{\bu}_{L_{\tau}} \ \ \text{for} \ \ m \gg0.$$
\end{prop}
\begin{proof}
Let $S_{I} \in \mathcal{HZ}_{m}$ be a generator for the homology. We need to look at the maximizer of the Alexander grading in $I$. Note that since $S_{I}$ is homogeneous, we have $I \subseteq  GZ_{m}$. In Corollary \ref{Corollary:GZdecomposition}, we proved that $GZ^{0}_{m}$ and $GZ^{2}_{m}$ have bounded distance from $C^{\bu}$. This means we can pick $L_{\tau}$ to be bigger than $L_{GZ}$, and focus on $I \cap GZ^{1}_{m}$.\\

The last part of argument is very similar to the proofs of Proposition \ref{Proposition:MasolvgradingzeroAll} and Corollary \ref{Corollary:GZdecomposition}. Using notations of Corollary \ref{Corollary:GZdecomposition},  
$$\text{consider an element} \ \  x \in ZB \ \  \text{and} \ \   \nu' = x \otimes \xi_{m} \in GZ_{m}^{1}.$$
Also, consider element $\nu_0 \in I$ coming from Lemma \ref{Lemma:homotopyequivgenerator2} and Corollary \ref{Corollary:gradingofhomotopyequivgenerator}. We can use Lemmas \ref{Lemma:Gradingchangebym} and \ref{Lemma:Gradingchangebyshift} to write the following equations:
\begin{equation}\label{Equation:taupotentials-1}
a_{m}(\nu_0) - a_{m}(\nu') = -\omega \cdot m + F'_{x}
\end{equation}
\begin{equation}\label{Equation:taupotentials-2}
a_{m}(x\otimes\xi_{j}) -a_{m}(x\otimes \xi_{m}) = (j-m)\omega  \ \ \text{for} \ j\in\{1,\dots,m\}
\end{equation}
where $F'_{x}$ is an integer which only depends on choice of $x \in ZB$. Combining Equations \ref{Equation:taupotentials-1} and \ref{Equation:taupotentials-2}, we can write 
$$a_{m}(x \otimes \xi_{j}) - a_{m}(\nu_0) = j \cdot \omega - F'_{x}  \ \ \text{for} \ j\in\{1,\dots,m\}.$$
Recall that $\omega < 0$. As a result, we have 
\begin{equation}\label{Equation:GZ1isinpositivehalfball}
j \geq |\frac{F'_{x}}{\omega}| \Longrightarrow a_{m}(x \otimes \xi_{j}) \leq a_{m}(\nu_0) \Longrightarrow x \otimes \xi_{j} \neq V_{\max}(S_{I}).  
\end{equation}
We can run this argument for all the choices of $x \in ZB$ and $\nu_0$. Hence, we restrict the elements in $GZ_{m}^{1}$ that can appear as $V_{\max}(S_{I})$. As the conclusion we have that if we asuume 
$$L_{\tau} \geq  \underset{x \in ZB}{\max} \  |\frac{F'_{x}}{\omega}| \ \  \text{and} \ \ L_{\tau} \geq L_{GZ},$$
then we can be sure that $V_{\max}(S_{I}) \in B^{\bu}_{L_{\tau}}$ for high enough $m$. This gives us the desired result. 
\end{proof}

Corollary \ref{Corollary:placeofmaximizer} follows from the proof of Proposition \ref{Proposition:taupotentials-bounded} and Corollary \ref{Corollary:GZdecomposition}.

\begin{coro}\label{Corollary:placeofmaximizer}
 Assume $m \gg0$. Let $S_{I} \in \mathcal{HZ}_{m}$ be a homogeneous generator of homology, then we have 
 $$V_{\max}(S_{I}) \in (I \cap B^{\bu}_{[0,L_{\tau}]}) \ \cup \ (I \cap GZ_{m}^{0}).$$
\end{coro}

Now we have to prove that the set $\mathcal{P}_{m}$ stabilizes. We will prove something stronger in the form of Lemmas \ref{Lemma:generatorinpositivehalfballstablizes} and \ref{Lemma:generatorinnegativehalfballstablizes}. These lemmas are stabilization result for the collection of subsets of a closed ball $B^{\bu}_{L}$ that appear in a homogeneous generator of homology (i.e. an element of $\mathcal{HZ}_{m}$). Although getting a general result for the closed ball might not be possible, we can prove weaker stabilization results on each of the two half-balls which will be sufficient for the proof of Lemma \ref{Lemma:taupotentials}. 

\begin{lemm}\label{Lemma:generatorinpositivehalfballstablizes}
Let $L$ be a positive integer. Define $\mathcal{HI}^{+}_{m,L}$ as follows:
$$\mathcal{HI}^{+}_{m,L} := \{ I \cap B^{\bu}_{[0,L]} \ | \ S_{I} \in \mathcal{HZ}_{m} \}.$$
For $m \gg0$, the family $\mathcal{HI}^{+}_{m,L}$ is stabilized under inclusions $\Phi_{m}$. 
\end{lemm}

\begin{lemm}\label{Lemma:generatorinnegativehalfballstablizes}
Define $\mathcal{HI}^{-}_{m}$ as follows: 
$$\mathcal{HI}^{-}_{m} := \{ I \cap GZ_{m}^{0} \ | \ S_{I} \in \mathcal{HZ}_{m} \}.$$
For $m \gg0$, the family $\mathcal{HI}^{-}_{m}$ is stabilized under inclusions $\Phi_{m}$. 
\end{lemm}

\begin{rema}\label{Remark:generatorinnegativehalfballstablizes}
    Note that for any integer $L \geq L_{GZ}$ and $S_{I} \in \mathcal{HZ}_{m}$ we have 
    $$( I \cap B^{\bu}_{[-L,0]} ) \setminus GZ^{1}_{m} = I \cap GZ_{m}^{0} \ \ \text{for} \ m \gg0. $$
    This follows from Corollary \ref{Corollary:GZdecomposition}. Based on this, the choice of notation in Lemma \ref{Lemma:generatorinnegativehalfballstablizes} is justified.\\ 
\end{rema}
We start with the proof of Lemma \ref{Lemma:generatorinpositivehalfballstablizes}, using the algebraic properties of the complex studied in Subsection \ref{Subsection:HomologyofboxtensorS^3:Algebra}.

\begin{proof}
   First of all, note that we can write:
   \begin{equation}\label{Equation:HI+rewritten}
    \mathcal{HI}^{+}_{m,L} = \{ MH(I) \cap B^{\bu}_{[0,L]} = I \cap B^{\bu}_{[0,L]} \cap GZ_{m} \ | \ S_{I} \in \mathcal{HS}_{m} \}.   
   \end{equation}
   
   Based on Proposition \ref{Proposition:generatorsofHFS^3}, all of the elements of $\mathcal{HS}_{m}$ come from $\A_{\infty}$ homotopy equivalences between $\widehat{\text{CFA}}(\mathcal{H}_{\infty} , z)$ and $\widehat{\text{CFA}}(\mathcal{H}_{K},z)$. Consider an arbitrary $\A_{\infty}$ homotopy equivalence $f$ and the subset $I_m \subseteq \bm{C^m}$ which satisfies the following:
    $$f \boxtimes \id_m (x_0 \otimes \eta) = S_{I_m}.$$
As discussed in Equation \ref{Equation:generatordecompose}, we have elements $S_{I''_m}$ and $S_{I'_m}$ such that 
$$S_{I_m} = S_{I''_m} + S_{I'_m} \ \  \text{,} \ \ I'_{m} \subseteq C^{\circ}_{m} \ \ \text{and}$$ 
\begin{equation}\label{Equation:Twostabilizationformula-repeated}
I''_{m+1} = \Phi'_{m}(I''_{m}) \ \sqcup \  NI''_{m+1} \  \ \text{and}  \ \ I'_{m+1} = \Phi_{m}(I'_{m}). 
\end{equation}
Equation \ref{Equation:Twostabilizationformula-repeated} is a rephrasing of Equation \ref{Equation:Twostabilizationformula}. The subset $NI''_{m+1} \subseteq C^{\circ}_{m+1}$ is defined as follows:  
$$S_{NI''_{m+1}} := f_{m+2}(x_0 \otimes \underbrace{\rho_3 \otimes \rho_{23} \otimes \cdots \otimes \rho_{23}}_{m+1}) \otimes \xi_{m+1}.$$
We can conclude that for $m \gg0$, the subset $I_m  \cap B^{\bu}_{[0,L]}$ will be stabilized, as we can write:
$$ \Phi_{m}(I_m  \cap B^{\bu}_{[0,L]}) = \Phi_{m}(I''_{m}  \cap B^{\bu}_{[0,L]}) = \Phi'_{m}(I''_{m}) \cap B^{\bu}_{[0,L]}$$
$$= I''_{m+1} \cap B^{\bu}_{[0,L]} = I_{m+1} \cap B^{\bu}_{[0,L]}.$$
In the first line we use the fact the for $m \gg0$, two inclusions $\Phi_{m}$ and $\Phi_{m'}$ restrict to the same map on $B^{\bu}_{[0,L]}$. The first and last equality use the fact that for $m \gg0$, all the subsets $I'_{m}$, $I'_{m+1}$ and $NI''_{m+1}$ are outside the positive half-ball $B^{\bu}_{[0,L]}$.\\

In other words, for $m \geq L+1$ we have that $I_{m} \cap B^{\bu}_{[0,L]}$ is determined by the set of maps $\{f_1,\dots,f_{L+1}\}$, and hence stabilizes.\\ 

Running this argument over the set of $\A_{\infty}$ homotopy equivalences of $\widehat{\text{CFA}}(\mathcal{H}_{\infty} , z)$ and $\widehat{\text{CFA}}(\mathcal{H}_{K},z)$, we can conclude that the family 
$${\{ I \cap B^{\bu}_{[0,L]} \ | \ S_{I} \in \mathcal{HS}_{m} \}}$$
will be stabilized for $m \gg0$. Since $GZ_{m} \cap B^{\bu}_{[0,L]}$ is also a stabilized subset for $m \gg0$, we can combine this fact with Equation \ref{Equation:HI+rewritten} to get the desired result. 
\end{proof}

Before we prove Lemma \ref{Lemma:generatorinnegativehalfballstablizes}, we need to prove a fact relating the materials of Subsection \ref{Subsection:HomologyofboxtensorS^3:Combinatorics} and Proposition \ref{Proposition:MasolvgradingzeroAll}. This is stated in Proposition \ref{Proposition:betaperservation}.

\begin{prop}\label{Proposition:betaperservation}
Consider a homogeneous generator of homology ${S_{I} \in \mathcal{HZ}_{m}}$, and the equivalence relation $\sim$ on the subset $I \subseteq \bm{C^{m}}$ (defined in Subsection \ref{Subsection:HomologyofboxtensorS^3:Combinatorics}). Let $\nu,\nu'$ be two elements of $I \setminus C^{\bu}$ with 
$$[\nu]_{{\sim}} = [\nu']_{{\sim}}.$$
Assume that $(\alpha;\beta,\gamma;\delta)$ and $(\alpha';\beta',\gamma';\delta')$ be the representative for grading double cosets $gr_{K_m}(\nu)$ and $gr_{K_m}(\nu')$, coming from the counterclockwise representatives of $gr_m$. Then we have 
$$\beta = \beta'$$
\end{prop}
\begin{proof}
 We know that there exist a good alternating path $\nu_0, \dots, \nu_{k}$, with $\nu=\nu_0$ and $\nu'=\nu_{k}$. Consider two consecutive elements $\nu_i$ and $\nu_{i-1}$ in this path. Based on the definition, there exists $\nu'_{i}$ and edges 
   $$\nu_{i} \longrightarrow \nu'_{i} \ \text{and} \ \nu_{i-1} \longrightarrow \nu'_{i} \ ,$$
   which are not of the third type. Also let 
   $$(\alpha_i;\beta_i,\gamma_i;\delta_i) \  \ \text{and} \  \  (\alpha'_{i};\beta'_{i},\gamma'_{i};\delta'_{i})$$ be the representatives for grading double cosets $gr_{K_m}(\nu_{i})$ and $gr_{K_m}(\nu'_{i})$ coming from the counterclockwise representatives of $gr_m$.\\
   
   First, we assume that the edge $\nu_{i} \longrightarrow \nu'_{i}$ is of the fourth type. This means there exist elements $x_i ,x'_i \in B_{K}$ and integer $j, j' \in \{1,\dots,m\}$ such that: 
   $$\nu_{i} = x_{i} \otimes \xi_j \ , \ \nu'_{i} = x'_{i} \otimes \xi_{j'}  \ \  \text{and}$$
   $$x'_{i} \  \ \text{appears in } \ m_{j'-j+1}(x_{i} \otimes \underbrace{\rho_{23} \otimes \cdots \otimes \rho_{23}}_{j'-j}).$$
   Using Equation \ref{Equation:typeAgrading}, we can write 
   $$gr_{K}(x'_i)=\lambda^{j'-j+1} \ gr_{K}(x_i)\cdot \underbrace{gr(\rho_{23}) \cdots gr(\rho_{23})}_{j'-j}.$$
   This means that if 
   $$gr_{K}(x_i) = [(a_i;,b_i,c_i;d_i)] \ \ \text{and} \ \  gr_{K}(x'_i) = [(a'_{i};,b'_{i},c'_{i};d'_{i})]$$
   then we have $b_{i}=b'_{i}$.\\
   
   Now, using the description of the counterclockwise representatives of $gr_{m}$, we can write 
   $$\beta_{i} = b_{i} -\frac{1}{2} \ \ \text{and} \ \ \beta'_{i} = b'_{i} -\frac{1}{2} \Longrightarrow \beta_{i} = \beta'_{i}.$$

 The exact argument works for edges of the first type. If the edge is of the second type (in this case we must have $\nu_i = x_i \otimes \eta$), the same argument gives us $b_i = b'_{i} + \frac{1}{2}$. Then we can write 
 $$\beta_i = b_{i} \ \ \text{and} \ \ \beta'_{i} = b'_{i} -\frac{1}{2} \Longrightarrow \beta_{i} = \beta'_{i} + 1.$$
Now we can deduce the desired result. Note that the only moment when the second component of the grading representative changes along the path is when 
$$\nu_{i} \xrightarrow{[2]} \nu'_{i} \ \  \text{and} \ \  \nu_{i-1} \xrightarrow{[4]} \nu'_{i}, \  \ \text{or} \ \ \nu_{i} \xrightarrow{[4]} \nu'_{i} \ \  \text{and} \ \  \nu_{i-1} \xrightarrow{[2]} \nu'_{i}.$$
These two combinations have canceling effects on the second component of the grading representative. Furthermore, note that in the first case we have $\nu_{i-1} \notin C^{\bu}$ and $\nu_{i} \in C^{\bu}$ (i.e. the alternating path enters $C^{\bu}$), while in the second case we have the reverse (i.e. the alternating path exits $C^{\bu}$). Since the alternating path starts and ends in $\nu$ and $\nu'$, which are outside of $C^{\bu}$, the appearances of these two combinations in the alternating path can be paired together. This completes the proof. 
\end{proof}

 Now we are ready to prove Lemma \ref{Lemma:generatorinnegativehalfballstablizes}.

 \begin{proof}
     Note that in the following we can write:
        \begin{equation}\label{Equation:HI-rewritten}
    \mathcal{HI}^{-}_{m} = \{ MH(I) \cap GZ^{0}_{m} = I \cap GZ_{m}^{0} \ | \ S_{I} \in \mathcal{HS}_{m} \}.   
   \end{equation}
     Similar to the proof of Lemma \ref{Lemma:generatorinpositivehalfballstablizes}, we study the elements of $\mathcal{HS}_{m}$ using Proposition \ref{Proposition:generatorsofHFS^3}. We also use the same notations as before.
     Note that we have 
     $$S_{I_{m}} = S_{I''
     _{m}} + S_{I'_{m}} \Longrightarrow I_m = I''_m \ominus I'_m, $$
     where $\ominus$ represents the symmetric difference operation. Hence, we can write
     \begin{equation}\label{Equation:Ominussplit}
      I_{m} \cap GZ^{0}_{m} = (I''_m \cap GZ^{0}_{m}) \ominus (I'_{m} \cap GZ^{0}_{m}).   
     \end{equation}
     Based on Equation \ref{Equation:Twostabilizationformula-repeated}, the subset $$I'_{m} = f_2(x \otimes \rho_1) \otimes \xi_{m} \subseteq B^{\bu}_{[-1,0]}$$
     will be stabilized, so we need to focus on $I''_m  \cap GZ^{0}_{m}$. Furthermore, we know that $I_{m} \cap C^{\bu} = I''_{m} \cap C^{\bu}$ stabilizes (as it only depends on $f_{1}(x_0)$), and we can just focus on $(I''_m  \cap GZ^{0}_{m}) \setminus C^{\bu}$.\\

     It will be useful to consider the notation $\mathcal{I}'_{m}$ as the family of all possible subsets $I'_m$ i.e. 
     $$\mathcal{I}'_{m} := \{f_2(x \otimes \rho_1) \otimes \xi_{m} \ | \ \forall \ f : \widehat{\text{CFA}}(\mathcal{H}_{\infty} , z)  \rightarrow \widehat{\text{CFA}}(\mathcal{H}_{K},z) \ \ \A_{\infty} \ \text{homotopy equivalence}\}$$

Now we use Proposition \ref{Proposition:equivalenceclassblack}, and consider the equivalence classes of the relation $\sim$ on $I''_{m}$. Based on Proposition \ref{Proposition:betaperservation}, we know that
$$ \nu \in (I''_m  \cap GZ^{0}_{m}) \setminus C^{\bu} \Longrightarrow [\nu]_{\sim} \subseteq I''_{m} \cap GZ^{0}_{m}.$$
This means $I''_{m} \cap GZ^{0}_{m}$ consists of the disjoint union of a family of the equivalence classes of $\sim$. We can see that:
\begin{equation}\label{Equation:Notreachingblackbox}
\forall \  \nu \in (I''_m  \cap GZ^{0}_{m}) \setminus C^{\bu} \ \  \text{we have} \ \  [\nu]_{\sim} \ \cap C^{\bu} \neq \varnothing \ \ \text{for} \ m \gg0.    
\end{equation}
This follows from the fact that $GZ^{0}_{m} \subseteq B^{\bu}_{[-L_{GZ},0]}$ (presented in Corollary \ref{Corollary:GZdecomposition}), and that the length of edges are bounded above (i.e. Lemma \ref{Lemma:differentialorder}).\\

Let $T = (I''_m  \cap GZ^{0}_{m}) \setminus C^{\bu}$. Combining Equation \ref{Equation:Notreachingblackbox} and Proposition \ref{Proposition:equivalenceclassblack} gives us the following:
\begin{equation}\label{Equation:negativehalfballblocks1}
\partial S_{T} = 0 \Longrightarrow S_T \ \text{is a nullhomologous cycle lying in} \ GZ^{0}_{m}.  
\end{equation}
This means we have 
\begin{equation}\label{Equation:negativehalfballblocks2}
S_{J}  \in \mathcal{HZ}_{m} \Longleftrightarrow  S_{J}+S_{T} \in \mathcal{HZ}_{m}.
\end{equation}
Consider the family $\mathcal{NI}^{-}_{m}$ defined as follows: 
$$\mathcal{NI}^{-}_{m} : = \{T \subseteq GZ^{0}_m \setminus C^{\bu}  \  |  \ \partial S_{T} = 0\}. $$
In other words, $\mathcal{NI}^{-}_{m}$ is the family of subsets corresponding to nullhomologous cycles lying in $GZ^{0}_{m} \setminus C^{\bu}$.\\

The family $\mathcal{NI}^{-}_{m}$ stabilizes for $m \gg0$, as $GZ^{0}_m \subseteq B^{\bu}_{[-L_{GZ},0]}$ stabilizes (as seen in Corollary \ref{Corollary:GZdecomposition}) and the elements and differentials in $B^{\bu}_{[-L_{GZ},0]}$ also stabilize (as a result of Proposition \ref{Proposition:boxtensordifferential}).\\

Based on Equations \ref{Equation:Ominussplit},  \ref{Equation:negativehalfballblocks1} and \ref{Equation:negativehalfballblocks2}, we can write: 
$$\mathcal{HI}^{-}_{m} = \{ T \ominus I'_{m} \ | \ \forall \ T \in \mathcal{NI}^{-}_{m} \ , \ \forall \ I'_{m} \in \mathcal{I}'_{m} \}. $$
We showed that the families $\mathcal{NI}^{-}_{m}$ and $\mathcal{I}'_{m}$ both stabilize for $m \gg0$ as families of the subsets of $B^{\bu}_{[-L_{GZ},0]}$. Hence, we can conclude that $\mathcal{HI}^{-}_{m}$ also stabilizes under inclusions $\Phi_m$ for $m \gg0$.
 \end{proof}

 Finally, we can prove the main result of this subsection, Lemma \ref{Lemma:taupotentials}. 

 \begin{proof}
  Consider the following family of subsets 
  $$\mathcal{PI}_{m} : = \{ (I \cap B^{\bu}_{[0,L_{\tau}]}) \  \cup \ (I \cap GZ^{0}_{m}) \ | \ S_{I} \in \mathcal{HZ}_{m}\}.$$
  We can rewrite this family as follows
  $$\mathcal{PI}_{m} = \{ (MH(I) \cap B^{\bu}_{[0,L_{\tau}]}) \  \cup \ (MH(I) \cap GZ^{0}_{m}) \ | \ S_{I} \in \mathcal{HS}_{m}\}$$
  $$ = \{( I \cap (GZ^1_{m} \cup GZ_{m}^{2}) \cap B^{\bu}_{[0,L_{\tau}]})) \ \cup \ (I \cap GZ_{m}^{0}) \ | \ S_{I}\in \mathcal{HS}_{m}  \}.$$
  
  Similar to the proofs of Lemmas \ref{Lemma:generatorinpositivehalfballstablizes} and \ref{Lemma:generatorinnegativehalfballstablizes}, we analyze $\mathcal{HS}_{m}$ through the $\A_{\infty}$ homotopy equivalences. We can again rewrite the family $\mathcal{PI}_{m}$ as follows 
\begin{align}\label{Equation:Rewrittenfamily}
\mathcal{PI}_{m}&=
 \begin{aligned}[t]
  \Big\{
  &\big(I''_m \cap (GZ^1_{m} \cup GZ_{m}^{2} \big) \cap B^{\bu}_{[0,L_{\tau}]})) 
 \ \cup \  \big((I'_{m} \cap GZ^{0}_{m}) \ominus T \big ) \ \Big| \\
 & \forall \ T \in \mathcal{NI}^{-}_{m} \ , \ S_{I_m} = S_{I''_{m}} + S_{I'_m} \in \mathcal{HS}_{m}
  \Big\} \ \ \text{for} \ \ m \gg0.  
 \end{aligned}
\end{align}
It follows from Lemmas \ref{Lemma:generatorinpositivehalfballstablizes} and \ref{Lemma:generatorinnegativehalfballstablizes}, that the family $\mathcal{PI}_{m}$ stabilizes for $m \gg0$.\\

Furthermore, as a result of Corollary \ref{Corollary:placeofmaximizer}, we can write that 
$$\mathcal{P}_{m} = \{ V_{\max}(J)\ \ | \ J \in \mathcal{PI}_{m} \}  $$
We know that the family $\mathcal{PI}_{m}$ stabilizes for $m \gg0$, as a family of subsets of $B^{\bu}_{L_{\tau}}$. Moreover, the sign of relative Alexander gradings of all pairs of elements of $B^{\bu}_{L_{\tau}}$ will also stabilize for high enough $m$. As a result, the (smallest) maximizer of the Alexander grading $V_{\max}(J)$ also stabilizes. This gives us the desired result.
 \end{proof}

\subsection{Proof of Theorem \ref{Theorem:tau}}\label{Subsection:Proofoftau}\hfill\\

In this final subsection of Section \ref{Section:tau}, we finally prove Theorem \ref{Theorem:tau}.
\begin{proof}
   Recall that we have Equation \ref{Equation:tauandPm} which tells us that $\tau(K_{m})$ is the minimum of the absolute Alexander grading function $a_{m}(\cdot)$ on the stabilized set $\mathcal{P}_{m} \subseteq B^{\bu}_{L_{\tau}}$. Since both $\mathcal{P}_{m}$ and the signs of relative Alexander gradings of all pairs of elements of $B^{\bu}_{L_{\tau}}$ will stabilize for $m \gg0$, we can consider a fixed element $\nu_{\min} \in \mathcal{P}_{m}$ such that : 
   $$a_{m}(\nu_{\min}) = \underset{\nu \in \mathcal{P}_{m}}{\min}  \ a_{m}(\nu) \ \text{for} \ m \gg 0.$$
   Now we just need to show that for $m \gg0$, the absolute Alexander grading $a_{m}(\nu_{min})$ is a linear function of $m$.\\

   This fact is true for any fixed element in any fixed ball, and the argument is the following. Define the positive integer $L$ as follows 
   $$L = \max\{ L_{S} + L_{B}+2L_{A} , L_{\tau}\}.$$
   Based on the argument of the proof of Theorem \ref{Theorem:extremalhfk} (Subsection \ref{Subsection:ProofofTheoremextremalhfk}), we can deduce that there is a fixed (i.e. stabilized) subset $I \subseteq B^{\bu}_{L}$ such that $S_{I}$ is a cycle in the (knot Floer) chain complex 
   $$\widehat{\text{CFA}}(\mathcal{H}_{K}, z, w) \boxtimes \widehat{\text{CFD}}(\mathcal{H}'_{\frac{1}{m}} , z'),$$
   and it represents a non-trivial homology cycle in the absolute Alexander grading $-g(K_{m})$.\\
   
   This means we can consider a fixed element $\nu \in I \subseteq B^{\bu}_{L}$ such that 
   $$a_{m}(\nu)  = -g(K_{m}) \ \text{for} \ m \gg0.$$
   Now based on Lemma \ref{Lemma:Gradingchangebym}, we know that for $m \gg0$, the relative Alexander grading $a_{m}(\nu_{\min}) - a_{m}(\nu)$ is a linear function of $m$. From the works of Baker and Taylor \cite{Baker2016DehnFA}, we know that $g(K_{m})$ is a linear function of $m$ as well. Hence we have that 
   $$\tau(K_{m}) = a_{m}(\nu_{min}) = (a_{m}(\nu_{\min}) - a_{m}(\nu)) - g(K_{m})$$
   is also a linear function of $m$, when $m$ is sufficently large. This is the desired result. 
 \end{proof}

\section{Stabilization of thickness}\label{Section:thickness}
\subsection{Definition of $\th(K_{m})$}\hfill\\

We briefly recall the definition of thickness of a knot in this subsection. As we described the knot Floer homology $\HFKh(K_{m})$ comes with two gradings, the Maslov (or homological) grading $h_{m}$ and the Alexander grading $a_{m}$. We can combine these two gradings and define the \emph{$\delta$-grading} as $\delta_{m} = h_{m} -a_{m}$. We can define $\HFKh^{i}(K_m)$ as follows:
$$\HFKh^{i}(K_m):= \{H \in \HFKh (K_{m}) \ | \ \delta_{m}(H) = i\}$$
The thickness $\th(K_{m})$ is defined as follows: 
$$\th(K_m) = \max \{ i \in \mathbb{Z} \ | \ \HFKh^{i}(K_m) \neq 0 \}  -\min \{ i \in \mathbb{Z} \ | \ \HFKh^{i}(K_m) \neq 0 \}.$$
The proof strategy for Theorem \ref{Theorem:thickness} closely parallels that of Theorem \ref{Theorem:tau}: we show that the homology classes in the maximum and minimum $\delta_{m}$-gradings stabilize under the inclusions $\Phi_{m}$. However, the techniques employed are also closely aligned with those used in the proofs of Theorems \ref{Theorem:extremalhfk} and \ref{Theorem:extremalAlexander}. Since the arguments in Sections \ref{Section:ExtermalknotFloerstabilization} and \ref{Section:Alexanderpolynomial} are presented in detail, we omit repeated details in this section when the reasoning is analogous.\\

We now state Lemma \ref{Lemma:deltagradingmaximizer}, which encapsulates the main strategy outlined above.

\begin{lemm}\label{Lemma:deltagradingmaximizer}
There exists a positive integer $L_{\th}$ along with fixed subsets $I^{\max}$ and $I^{\min}$ of the closed ball $B^{\bu}_{L_{\th}}$, such that for all sufficently large $m$, we have: 
$$\delta_{m}([S_{I^{\max}}]) = \max \{ i \in \mathbb{Z} \ | \ \HFKh^{i}(K_m) \neq 0 \}, \ \text{and} $$ 
$$\delta_{m}([S_{I^{\min}}]) = \min \{ i \in \mathbb{Z} \ | \ \HFKh^{i}(K_m) \neq 0 \}.$$
\end{lemm}

The first step in the proof of Lemma \ref{Lemma:deltagradingmaximizer} is the construction of a \emph{shift-equivariant} basis for the space of far homology classes $\mathcal{H}_{D'_{m}}$, as defined in Subsection \ref{Subsection:Fardecomposition}. This construction is carried out in Subsection \ref{Subsection:Shiftequivariantbasis}. 

\subsection{Shift-equivariant basis of the far homology}\label{Subsection:Shiftequivariantbasis}\hfill\\

In Subsection \ref{Subsection:Fardecomposition}, we introduced the far decomposition of $\HFKh(K_m)$: 
$$\HFKh(K_m) = \mathcal{H}_{D'_m} \oplus \overline{\mathcal{H}_{D'_m}} $$
This decomposition can also be viewed as splitting $\HFKh(K_m)$ to \emph{stable} and \emph{unstable} subspaces. As shown in Lemma \ref{Lemma:Alexanderjumpsdefects-stabilization}, the complement subspaces $\overline{\mathcal{H}_{D'_m}}$ stabilize for sufficently large $m$. For the purpose of proving Lemma \ref{Lemma:deltagradingmaximizer}, we focus only on the unstable component of the decomposition—namely, the space of far homology classes $\mathcal{H}_{D'_{m}}$.\\

The main idea is that, although the subspace $\mathcal{H}_{D'_{m}}$ does not stabilize in the usual sense (for example, its dimension grows linearly with $m$, as shown in Theorem \ref{Theorem:totaldim}), it can nevertheless be computed from finite, stabilized data, up to the action of the shift maps $R_{\pm}$. More precisely, this is the content of Proposition \ref{Proposition:shiftequivarientbasis}.  

\begin{prop}\label{Proposition:shiftequivarientbasis}
  Assume that $m$ is sufficently large. There exist a minimal basis $\mathcal{B}_{m}$ for the space of far homology classes $\mathcal{H}_{D'_{m}}$, constructed from a collection of subsets $\mathcal{SB}_{m} \subseteq 2^{D'_{m}}$, satisfying the following properties: 
  \begin{itemize}
      \item For every subset $I \in \mathcal{SB}_{m}$, the corresponding element $S_{I}$ is a minimal cycle.
      \item For any $i \in \mathbb{Z}_{\geq 0}$, if $I \in \mathcal{SB}_{m}$ and $R^{i}_{\pm}(I) \subseteq D'_{m}$ then $R^{i}_{\pm}(I) \in \mathcal{SB}_{m}$ as well.
  \end{itemize}
The basis $\mathcal{B}_{m}$ is then defined as:
$$\mathcal{B}_{m} = \{[S_{I}] \ | \ I \in \mathcal{SB}_{m}\}.$$
  In other words, $\mathcal{B}_{m}$ is equivariant under the shift maps $R_{\pm}$ restricted to the subset $D'_{m} \subset \bm{C^{m}}$.
\end{prop}

Proposition \ref{Proposition:shiftequivarientbasis} should not be surprising in light of Lemma \ref{Lemma:Alexanderjumpsdefects}, and it relies on many of the same techniques. We include a sketch of the proof below. We note that our initial proof was significantly more involved. We are grateful to Professor Claudius Zibrowius for their helpful suggestions, which led to the streamlined version presented here.

\begin{proof}
    Similar to Subsection \ref{Subsection:Fardecomposition}, we define $L_{\delta}$ to be $L_{B}+2L_{A}+1$.\\
    
    We construct the basis $\mathcal{B}_{m}$ using the following procedure:
    \begin{enumerate}
        \item \textbf{Initialization}:\\
        Begin with an empty collection of far homology classes and an empty collection of associated subsets of $D'_{m}$.
        \item \textbf{Selection}:\\
        Choose a minimal cycle $S_{I}$ such that 
        $$I \cap C^{\circ}_{2L_{\delta}} \neq \varnothing,$$ and such that $[S_{I}]$ does not lie in the span of the current collection.
        \item \textbf{Equivariance Closure}:\\
        For all integers $i \in \mathbb{Z}_{\geq 0}$, such that $${R_{\pm}^{i}(I) \subseteq D'_{m}},$$ include the class $[S_{R_{\pm}^{i}(I)}]$ in the collection and record the subsets $R_{\pm}^{i}(I)$ accordingly.  
        \item \textbf{Iteration}:\\
        Repeat from Step 2 until no further choices are possible.
    \end{enumerate}
    
Let $\mathcal{B}_{m}$ denote the collection of far homology classes constructed by the procedure above. We claim that this collection is the desired shift-equivariant basis. Let $\mathcal{SB}_{m}$ be the collection of subsets.\\
  
  First, observe that the shift-equivariance of $\mathcal{B}_{m}$ is built directly into Step (3) of the construction. This step is well-defined because the boundary map is shift-equivariant outside the ball $B^{\bu}_{L_{A}+1}$, as guaranteed by Proposition \ref{Proposition:shiftbounddary}. Thus, $S_{R^{i}_{\pm}(I)}$ is a cycle and represents a homology class in $\mathcal{H}_{D'_{m}}$.\\

 To prove linear independence, assume for contradiction that there exists a minimal relation:
  $$\partial S_{T} = S_{I_1}+ \cdots +S_{I_{r}}.$$
  with $I_1,\dots,I_r \in \mathcal{SB}_{m}$. Since this is a minimal relation, it is homogeneous with respect to the Alexander grading. By Lemma \ref{Lemma:Gradingchangebyshift} no two shifts of a cycle can lie in the same Alexander grading, so the index subsets $I_1,\dots,I_r$ must have been added at distinct stages of the procedure.\\   
  
Without loss of generality, assume that $S_{I_{1}}$ s the most recently added element in this relation. Then, by construction, there exists a minimal cycle $S_{I}$ with $I \cap C^{\circ}_{2L_{\delta}} \neq \varnothing$ such that 
 $$ I_1 = R^{k}_{+}(I) \ \text{or} \ I_1 = R^{k}_{-}(I) \ \text{for some} \ k \in \mathbb{Z}_{\geq 0}.$$
Without loss of generality, assume $I_1 = R^{k}_{+}(I)$, the same argument works for the other case.\\

By shifting the full relation by $R^{k}_{-}$ we obtain:
 $$\partial S_{R^{i}_{-}(T)} = S_{R^{i}_{-}(I_1)}+ \cdots +S_{R^{i}_{-}(I_{r})} = S_{I}+ S_{R^{i}_{-}(I_2)} \cdots +S_{R^{i}_{-}(I_{r})}.$$
To justify this equality, we must verify that both sides lie entirely outside the ball $B^{\bu}_{L_{A}+1}$, so that can apply Proposition \ref{Proposition:shiftbounddary}(i.e., shift-equivariance of the boundary map). Specifically, we need:
 $$T \ \cup \ \underset{1 \leq t \leq r}{\bigcup} I_{t} \subseteq \bm{C^{m}} \setminus B^{\bu}_{L_{A}+1} \ \text{and} \ R^{i}_{-}(T) \ \cup \ \underset{1 \leq t \leq r}{\bigcup} R^{i}_{-}(I_{t}) \subseteq \bm{C^{m}} \setminus B^{\bu}_{L_{A}+1}.$$
The first inclusion follows from the assumptions and is justified using the same argument as in Equation \ref{Eqation:sidesofrelation}.\\

To obtain the second inclusion, apply  Corollary \ref{Corollary:diamofminimalrelation} to conclude:
 \begin{flalign}
 & \diam \Big(\underset{1 \leq t \leq r}{\bigcup} R^{i}_{-}(I_{t}) \Big ) = \diam \Big(\underset{1 \leq t \leq r}{\bigcup} I_{t} \Big )  \leq L_{B} \\ 
 & \Rightarrow \dist_{H}(\underset{1 \leq t \leq r}{\bigcup} R^{i}_{-}(I_{t}), C^{\bu}) \geq \dist_{H}(C^{\circ}_{2L_{\delta}},C^{\bu}) - L_{B}\\
 & \Rightarrow \dist_{H}(\underset{1 \leq t \leq r}{\bigcup} R^{i}_{-}(I_{t}), C^{\bu}) > L_{\delta} \Rightarrow \underset{1 \leq t \leq r}{\bigcup} R^{i}_{-}(I_{t}) \subseteq \bm{C^{m}} \setminus B^{\bu}_{L_{\delta}}.
 \end{flalign} 
The same argument as in Equation \ref{Eqation:sidesofrelation} then shows: 
$$R^{i}_{-}(T) \subseteq \bm{C^{m}} \setminus B^{\bu}_{L_{A}+1}.$$
Now that we have the relation 
$$\partial S_{R^{i}_{-}(T)} = S_{I}+ S_{R^{i}_{-}(I_2)} \cdots +S_{R^{i}_{-}(I_{r})},$$
we can see that this contradicts the fact that $[S_{I}]$ was not in the span of the previous basis elements at the time it was added.\\

Finally, we need to show that the collection  $\mathcal{B}_{m}$ spans the space $\mathcal{H}_{D'_{m}}$. Suppose, for contradiction, that there exists a subset $I \subseteq D'_{m}$ and a minimal cycle $S_{I}$ such that the far homology class $[S_{I}]$ is not contained in the span of the classes in $\mathcal{B}_{m}$. There exists $k \in \mathbb{Z}_{\geq 0}$ such that 
$$R^{k}_{+}(I) \cap C^{\circ}_{2L_{\delta}}  \neq \varnothing \ \  \text{or}  \ \ R^{k}_{-}(I)\cap C^{\circ}_{2L_{\delta}} \neq \varnothing.$$
Without loss of generality assume that $R^{k}_{+}(I) \cap C^{\circ}_{2L_{\delta}} \neq \varnothing$. Then by Proposition \ref{Proposition:shiftbounddary}, $S_{R^{k}_{+}(I)}$ is a minimal cycle.\\

Now consider Step (2) and (3) of the construction procedure. If $[S_{R^{k}_{+}(I)}]$ was added to $\mathcal{B}_{m}$, then by shift-equivariance, $S_{I} \in \mathcal{B}_{m}$ as well, contradicting our assumption. Otherwise, $[S_{R^{k}_{+}(I)}]$ must have been expressible as a linear combination of existing classes in the collection, i.e., there exists a minimal relation of the form
$$\partial S_{T} = S_{R^{k}_{+}(I)} + S_{I_1}+\cdots+S_{I_r}$$
for subsets $I_1,\dots,I_r$ in $\mathcal{SB}_{m}$.\\

Using the same argument as in the proof of linear independence, we may apply the shift map $R^{k}_{-}$ to both sides of the relation. This yields another valid minimal relation:
$$\partial S_{R^{k}_{-}(T)} = S_{I} + S_{R^{k}_{-}(I_1)}+\cdots+S_{R^{k}_{-}(I_r)}.$$
Since the subsets $I_1,\dots,I_r$ lie in $\mathcal{SB}_{m}$, and the collection is shift-equivariant by construction, it follows that $R^{k}_{-}(I_1),\dots,R^{k}_{-}(I_r) \in \mathcal{SB}_{m}$. Therefore, the homology class $[S_{I}]$ lies in the span of $\mathcal{B}_{m}$ contradicting our assumption.\\

This contradiction shows that no such subset $I \subseteq D'_{m}$. Hence, $\mathcal{B}_{m}$ indeed generates the space $\mathcal{H}_{D'_{m}}$, completing the proof.
\end{proof}

The procedure  used to prove Proposition \ref{Proposition:shiftequivarientbasis} yields further structural insights. Let $\mathcal{GB}_{m}$ denote the collection of all subsets $I \subseteq D'_{m}$ that were selected in Step (2) of the construction. This collection \emph{generates} $\mathcal{SB}_{m}$ in the following sense:
$$\forall \ J \in \mathcal{SB}_{m} \ \ \exists \ k \in \mathbb{Z}_{\geq 0} \ \text{such that} \ J= R^{k}_{+}(I) \ \text{or} \ J= R^{k}_{-}(I).$$
We further show that the collection $\mathcal{GB}_{m}$ stabilizes as $m \rightarrow \infty$, as stated in Corollary \ref{Corollary:generatingsetstabilizes}.

\begin{coro}\label{Corollary:generatingsetstabilizes}
    Assume $m$ is sufficiently large. Then we have 
    $$\Phi_{m}(\mathcal{GB}_{m}) = \mathcal{GB}_{m+1}.$$
    In other words, there exists a fixed collection of the subsets of $B^{\bu}_{2L_{\delta}+N_{A}L_{A}}$, denoted by $\mathcal{GB}_{\infty}$, such that 
    $$\mathcal{GB}_{m} = \mathcal{GB}_{\infty}$$
     for all sufficiently large $m$.
\end{coro}

\begin{proof}
First, observe that every subset $I \in \mathcal{GB}_{m}$ lies within the closed ball $B^{\bu}_{2L_{\delta}+N_{A}L_{A}}$. This follows from the selection criteria in Step (2), which require $I \cap C^{\circ}_{2L_{\delta}} \neq \varnothing$, and from Lemma \ref{Lemma:lengthminimal}, which ensures that $\diam(I) \leq N_{A}L_{A}$.\\

Now we only need to show that the selection criterion in the procedure also stabilizes. From the proof of Theorem \ref{Theorem:extremalhfk}, we know that the collection of minimal cycles $S_{I}$ with $I \cap C^{\circ}_{2L_{\delta}} \neq \varnothing$ stabilizes  for sufficiently large $m$. Therefore, in order to determine whether such a cycle $S_{I}$ is selected at a given stage, it remains to check whether its homology class $[S_{I}]$ lies in the span of the previously selected classes. This, in turn, reduces to analyzing whether $S_{I}$ appears in any minimal relation with other cycles already present in the collection, which are of the form $S_{R^{i}_{\pm}(J)}$, where $J \in \mathcal{GB}_{m}$ was selected in an ealier stage.\\

By Corollary \ref{Corollary:diamofminimalrelation}, and again using the fact that $I \cap C^{\circ}_{2L_{\delta}} \neq \varnothing$), any such minimal relation must be contained in $B^{\bu}_{2L_{\delta}+2L_{B}+L_{A}}$. As in the proof of Theorem \ref{Theorem:extremalhfk}, we conclude that all such relations stabilize for large $m$ and hence so does the selection process.
\end{proof}

We can also define alternative generating sets for $\mathcal{SB}_{m}$ by selecting extremal representatives of each shift orbit. In particular, we define two extremal generating sets as follows:
$$\mathcal{GB}^{-}_{\infty} = \{ R^{k}_{-}(I) \ | \ I \in \mathcal{GB}_{\infty}\ ,\  R^{k}_{-}(I)\subseteq D'_{m}  \ , \ R^{k+1}_{-}(I)\subseteq D'_{m} \}$$
$$\mathcal{GB}^{+}_{\infty} = \{ R^{k}_{+}(I) \ | \ I \in \mathcal{GB}_{\infty}\ ,\  R^{k}_{+}(I)\subseteq D'_{m}  \ , \ R^{k+1}_{+}(I)\subseteq D'_{m} \}$$
That is, $\mathcal{GB}^{-}_{\infty}$ consists of the leftmost shifts of elements of $\mathcal{GB}_{m}$ that still lie in $D'_{m}$, while $\mathcal{GB}^{+}_{\infty}$ consists of the rightmost such shifts. Both sets generate $\mathcal{SB}_{m}$ under the action of the shift maps.\\

Also let $\mathcal{HB}^{-}_{\infty}$ and $\mathcal{HB}^{+}_{\infty}$ be the collection of the corresponding homology classes i.e. 
$$\mathcal{HB}^{\pm}_{\infty} = \{ [S_{J}] \ | \ J \in \mathcal{GB}^{\pm}_{\infty} \} .$$

In Subsection \ref{Subsection:deltagradingstab}, we apply the results established here to prove Lemma \ref{Lemma:deltagradingmaximizer}. 

\subsection{Stabilization of the maximizer and minimizer of $\delta$-grading}\label{Subsection:deltagradingstab}\hfill\\

Before proving Lemma \ref{Lemma:deltagradingmaximizer}, we need to state a simple corollary of Lemma \ref{Lemma:Gradingchangebyshift}, which helps us understand how $\delta$-grading changes under shifts. 

\begin{coro}\label{Corollary:deltachangebyshift}
Let $x^1$ be an element of $B^{1}_{K}$ with $(a;b,c;d) \in \widetilde{G}$ as a representative of the grading coset $gr_K(x^1)$. Let $i,j$ be positive integers. Then 
$$x^1 \otimes \xi_i , x^1 \otimes \xi_j  \in \bm{C^{m}},$$
for any $m \geq \max \{i,j\}$, and their relative gradings are as follows :
$$\delta_m(x^1 \otimes \xi_j) - \delta_m(x^1 \otimes \xi_i) = (j-i)(M-2b-\frac{1}{2} - \omega).$$
\end{coro}

Now we are ready to prove Lemma \ref{Lemma:deltagradingmaximizer}. 

\begin{proof}
    As previously noted, the main strategy is to use the far decomposition
    $$\HFKh(K_m) = \mathcal{H}_{D'_m} \oplus \overline{\mathcal{H}_{D'_m}}.$$
    By Lemma \ref{Lemma:Alexanderjumpsdefects-stabilization}, there exists a constant $L_{\Delta}$ such that for sufficiently large $m$, the complement subspace $\overline{\mathcal{H}_{D'_m}}$ stabilizes to $\overline{\mathcal{H}'_{\infty}}$ which is supported entirely within the closed ball $B^{\bu}_{L_{\Delta}}$. That is, all homology classes in $\overline{\mathcal{H}_{D'_m}}$ admit fixed representatives contained in this ball, independent of $m$, once $m$ is large enough.\\

 To identify the maximizer and minimizer of the $\delta_{m}$-grading on the space of far homology classes $\mathcal{H}_{D'_m}$, it suffices to examine this grading on a minimal (or even homogeneous) basis. We consider the basis $\mathcal{B}_{m}$, constructed in Proposition \ref{Proposition:shiftequivarientbasis}. By Corollary \ref{Corollary:deltachangebyshift}, the $\delta_m$-grading varies linearly across each shift orbit. Therefore, it achieves its maximum and minimum at the leftmost and rightmost ends of each orbit. As a result, the extremal $\delta_{m}$-gradings over $\mathcal{H}_{D'_m}$ are realized by classes corresponding to subsets in the extremal generating sets i.e. $\mathcal{HB}^{-}_{\infty}$ and $\mathcal{HB}^{+}_{\infty}$. By Lemma \ref{Lemma:lengthminimal}, both of these collections are supported within the closed ball $B^{\bu}_{L_{\delta}+N_{A}L_{A}}$. Let $L_{\th}$ be equal to $\max \{L_{\Delta} , L_{\delta}+N_{A}L_{A}\}$.\\
 
Putting the two components together, it suffices to consider the  $\delta_{m}$-grading on the elements of
 $$\overline{\mathcal{H}'_{\infty}} \cup \mathcal{HB}^{-}_{\infty} \cup  \mathcal{HB}^{+}_{\infty}$$
since the maximizer and minimizer of $\delta_{m}$-grading over $\HFKh(K_{m})$ must lie in this collection for sufficiently large $m$. This collection is fixed and supported within the closed ball $B^{\bu}_{\th}$.\\

By Lemma \ref{Lemma:Gradingchangebym}, and using similar reasoning to earlier arguments, the sign of the relative $\delta_{m}$-grading between all the pairs of elements in this collection stabilize as $m \rightarrow \infty$. Consequently, the maximizer and minimizer also stabilize, which proves the desired result.
\end{proof}

\subsection{Proof of Theorem \ref{Theorem:thickness}}\label{Subsection:proofofthickness}\hfill\\

We can finally prove Theorem \ref{Theorem:thickness}.\\

\begin{proof}
    Assume that $m$ is sufficiently large. Based on the definition of thickness and Lemma \ref{Lemma:deltagradingmaximizer}, we have 
    $$\th(K_{m}) = \delta_m([S_{I^{\max}}]) - \delta_{m}([S_{I^{\min}}])$$
    $$= (h_m([S_{I^{\max}}]) - h_{m}([S_{I^{\min}}])) - (a_m([S_{I^{\max}}]) - a_{m}([S_{I^{\min}}])).$$
    This will be a linear function of $m$, as a result of Lemma \ref{Lemma:Gradingchangebym}.
\end{proof}

\bibliography{bibtemplate}
\bibliographystyle{smfalpha}

\end{document}